%% file: orbifolds.tex
\documentclass[12pt,psamsfonts,leqno,oneside,letterpaper]{amsart}
\usepackage[text={6.5truein,9truein},left=1truein,top=1truein]{geometry}
\usepackage{amssymb,amsmath,amscd,enumerate, microtype}

\usepackage[pdftex]{graphicx}
\usepackage{url}
\usepackage{tikz-cd}

\usepackage[colorlinks,linkcolor=blue,citecolor=blue,pdfstartview=FitH]{hyperref}
\input xy
\xyoption{all}
\SelectTips{cm}{12}

\usepackage{color}
\usepackage{mathrsfs}
\usepackage{overpic}
\input macros
\begin{document}

\title[The underlying manifold of a low-volume orbifold]{The underlying manifold \\ of a low-volume hyperbolic $3$--orbifold}
\author{David Futer}
\address{Department of Mathematics
\\
Temple University \\
Philadelphia, PA 19147}
\email{dfuter@temple.edu}

\author{Peter B. Shalen}
\address{Department of Mathematics, Statistics, and Computer Science
(M/C 249)\\
University of Illinois at Chicago\\
851 S. Morgan St.
Chicago, IL 60607-7045}
\email{petershalen@gmail.com}

\date{\today}

\begin{abstract}
This paper proves strong restrictions on the underlying topological space of a closed, orientable hyperbolic $3$--orbifold $\orbM$ whose volume satisfies a certain upper bound. For instance, if $\vol \orbM < 0.1491$, then the underlying space of $\orbM$ must be either a small Seifert fibered space, or the connected sum of two lens spaces (or of a lens space with $\SS^2 \times \SS^1$), or the gluing of one or two highly restricted Seifert fibered spaces along a single incompressible torus. If $\vol \orbM \leq 0.1571$ and the singular locus of $\orbM$ is a link, then the topology of the underlying space is restricted even further.

Our methods are primarily topological, and involve studying the underlying topology of orbifold books of $I$--bundles. Along the way, we provide an exposition of some foundational material about orbifolds that was previously absent from the literature.
\end{abstract}

\maketitle

\section{Introduction}

It follows from the Mostow--Prasad rigidity theorem
that the volume of a finite-volume hyperbolic $3$--manifold $M$ is a positive real-valued topological invariant of $M$. 
Furthermore, by a theorem of J\o rgensen and Thurston, this invariant is finite-to-one. The set of volumes has a rich structure \cite[Chapters 5--6]{Thurston:notes}: it is a closed, well-ordered set of order type $\omega^\omega$, and every sequence of manifolds $M_i$ whose volumes converge to some limit contains a subsequence whose geometric structures converge in an appropriate sense.

One consequence of the above results is that any upper bound on  $\vol M$ imposes topological restrictions on $M$. Making these restrictions explicit has been the subject of a great deal of research. 
Some of this research has taken the form of identifying the lowest-volume manifolds in various natural classes.

Celebrated results in this vein
include the proof by Gabai, Meyerhoff, and Milley  \cite{GMM:SmallestCusped, milley}  that the unique lowest-volume orientable hyperbolic $3$--manifold is the Weeks manifold $M_{\rm Weeks}$, of volume $\Vweeks \doteq \vol (M_{\rm Weeks}) = 0.9427 \ldots$.
More recent work by Gabai, Haraway, Meyerhoff, N.\ Thurston, and Yarmola   \cite{GHMTY} determines the unique orientable hyperbolic $3$--manifolds of second- and third-lowest volume. Relatedly, Cao and Meyerhoff identified the smallest limit volume as $2 \vtet =2.029 
\ldots$, and also identified the two (cusped) orientable hyperbolic $3$--manifolds realizing this volume \cite{CaoMeyerhoff}. Following the work of Gabai, Haraway, Meyerhoff, N.\ Thurston, and Yarmola   \cite{GHMTY}, the 14 lowest-volume cusped orientable $3$--manifolds are now known.

 The above results may be interpreted as classifying the orientable hyperbolic $3$--manifolds whose volumes satisfy a certain upper bound. 
 On the other hand, Culler, Shalen, and their collaborators have established topological restrictions, well short of a complete classification, of closed (say), orientable hyperbolic $3$--manifolds whose volumes satisfy a considerably larger upper bound. For a survey of some of the latter results, see \cite{shalensurvey}
 and the forthcoming updated survey \cite{new-survey}.

The results in \cite{shalensurvey, new-survey},
 illustrate the philosophy that the volume of a hyperbolic $3$--manifold provides a good measure of its topological complexity  (in an appropriate sense). This philosophy is 
 discussed and developed in the work of Gabai, Meyerhoff, and Milley \cite{GMM:Mom-Tech}; see \cite[Conjecture 0.1]{GMM:Mom-Tech} in particular. 
In a similar vein, there are many results that provide two-sided estimates for the volumes of hyperbolic $3$--manifolds in terms of their topological or combinatorial complexity. See e.g.\ Brock~\cite{Brock:Fibered} for results about fibered $3$--manifolds and the survey by Futer, Kalfagianni, and Purcell~\cite{FKP:Survey} for results about knot and link complements.

\Custom{Low-volume hyperbolic $3$--orbifolds}
This paper is devoted to the study of finite-volume orientable hyperbolic  \emph{$3$--orbifolds}.
While the formal definition of an orbifold appears in Section~\ref{orbisection}, the notion is simpler to explain in the hyperbolic setting.

Let $\Gamma$ be 
a discrete cocompact subgroup  of the group $\isomplus(\HH^3)$ of orientation-preserving isometries of the hyperbolic space $\HH^3$, and let $M$ denote the topological quotient $\HH^3/\Gamma$. In the case where 
$\Gamma$ is torsion-free, the space $M$
acquires the structure of a hyperbolic manifold, and it follows from Mostow rigidity  that
the isomorphism type of $\Gamma$, which is the same as that of $\pi_1(M)$, determines $\Gamma$ up to conjugacy.
In the general case, where $\Gamma$ may have torsion, it turns out that $\HH^3/\Gamma$, with its quotient topology,  is still an orientable topological $3$--manifold, but its topological type does not contain enough information  to recover the group $\Gamma$. Indeed, the group $\Gamma$ determines the pair $(M,\Sigma)$, where the  
 \emph{singular locus} $\Sigma\subset M$ is defined as the quotient
 of the set of points of $\HH^3$ that have non-trivial  $\Gamma$--stabilizers. Each component of  $\Sigma$ is either a tame simple closed curve or a tame trivalent graph  in  $M$. Furthermore, if $\sigma$ is either a simple closed curve component of $\Sigma$ or an (open) edge of a graph component of  $\Sigma$, one can define the  \emph{local group} $G_\sigma$, a finite cyclic group which is isomorphic to the $\Gamma$--stabilizer of an arbitrary point of $q^{-1}(\sigma)$. Likewise, if $v$ is a node of  $\Sigma$, the  \emph{local group} $G_v$ is a finite triangle group which is isomorphic to the $\Gamma$--stabilizer of an arbitrary point of $q^{-1}(v)$. A node $v \in \Sigma$ is called \emph{dihedral} if its local group $G_v$ is dihedral.

The singular locus and local groups of 
$\HH^3/\Gamma$ are encoded in a structure called a three-dimensional orbifold structure; see \S\ref{Def:OrbifoldChart}
and \S\ref{Sec:LocalGroups}.
An $n$--dimensional orbifold structure on a space $X$ is a generalization of a differentiable structure, in which each chart map, rather than being a homeomorphism whose domain is an open set in $\RR^n$, is the composition of a quotient map $q\from U\to U/G$ with a homeomorphism, where $G$ is a finite group of self-diffeomorphisms of an open set $U\subset\RR^n$. There is a natural definition of compatibility of such chart maps. 
The topological space $X$ is called the \emph{underlying space} of the orbifold. 

Many notions about manifolds have natural extensions to orbifolds. These include the notions of diffeomorphisms, suborbifolds, fundamental groups, and covering spaces. 
In particular, the fundamental group $\Gamma$  of a hyperbolic $3$--orbifold $\orbM = \HH^3 / \Gamma$ is determined up to isomorphism by the underlying space $M$, the singular locus $\Sigma$, and the assignment of local groups on $\Sigma$.
Furthermore, Mostow rigidity implies that the isomorphism type of $\Gamma$ determines $\Gamma$ up to conjugacy in $\isomplus(\HH^3)$, thereby determining $\orbM$ up to isometry. (Here and in the sequel, we typically use math-italic letters for manifold objects and calligraphic versions of the same letters for the corresponding orbifold objects; see \S\ref{Sec:naming} for details.
For more general orientable $3$--orbifolds, it is still the case that the data of $M$, $\Sigma$, and the local groups determines the orbifold up to diffeomorphism; we will neither prove nor use this statement.)

A hyperbolic $3$--orbifold $\orbM =  \HH^3/\Gamma$ has a well-defined  volume, 
denoted $\vol \orbM$,
which is equal to the volume of any fundamental domain for $\Gamma$ in $\HH^3$.

Generalizing the work of J\o rgensen and Thurston that was mentioned above, Dunbar and Meyerhoff showed in \cite{Dunbar-Meyerhoff} that the  set of finite volumes of hyperbolic $3$--orbifolds is a closed, well ordered set of order type $\omega^\omega$. Just as with manifolds, there are only finitely many hyperbolic $3$--orbifolds with a given finite volume.
These structural results lead to the research problem of identifying the lowest-volume orbifolds in various natural families.

Gehring, Marshall, and Martin have identified the unique orientable hyperbolic $3$--orbifold of lowest volume \cite{GehringMartin:MinOrbifold, MarshallMartin:MinOrbifold}. 
This orbifold has underlying space $\SS^3$ and volume approximately $0.03905$. 
Meyerhoff has identified the unique cusped orientable $3$--orbifold of lowest volume~\cite{Meyerhoff:MinCuspedOrbifold}. This orbifold has underlying space $\RR^3$
and volume  $\vtet/12 \approx 0.0846$. Adams \cite{Adams:LimitVolume} identified the unique orientable hyperbolic $3$--orbifold realizing the smallest limit volume. This orbifold has 
underlying space $\RR^3$ and volume $\voct/12 \approx 0.3053$.
Here, $\vtet \approx 1.0149$ denotes the volume of a regular ideal tetrahedron and $\voct \approx 3.6638$ denotes the volume of a regular ideal octahedron.

The present paper develops the philosophy, similar to that of \cite[Conjecture 0.1]{GMM:Mom-Tech}, that the volume of a hyperbolic $3$--orbifold is a good proxy for its topological complexity in some appropriate sense. In the case of orbifolds, the right notion of ``topological complexity'' should include the topology of the underlying space. Given a $3$--orbifold of sufficiently small volume, our results give very strong restrictions on the topology of the underlying space. In some cases, they also give restrictions on the local groups at nodes of the singular set. 

\begin{theorem}\label{Thm:FirstMain}
Let $\orbM$ be a closed hyperbolic $3$--orbifold such that $\vol \orbM < 0.1491$. Then the underlying space $M$ of $\orbM$ satisfies one of the following: 
\begin{enumerate}[\:(i)]
\item $M$ is Seifert fibered over $\SS^2$, with at most three singular fibers. 
\item $M$ is the connected sum of two lens spaces, or the connected sum of $\SS^2 \times \SS^1$ with a non-trivial lens space.
\item Cutting $M$ along 
an essential torus produces one or
two compact $3$--manifolds with boundary, 
each of which is a Seifert fibered space
 over a disk
with exactly two singular fibers, or over 
an annulus
with at most one singular fiber.
\end{enumerate}
Furthermore, if Alternative (i) does not hold, then 
every node of the singular set of $\orbM$ is dihedral.
\end{theorem}

This result will be proved as Theorem~\ref{Thm:VolWithTurnover}.

\begin{figure}
\begin{overpic}{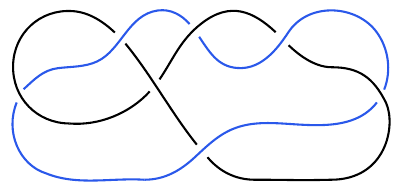}
\put(98,32){$2$}
\put(98,12){$3$}
\end{overpic}
\caption{The orbifold $\calO_L$, conjectured to have the smallest volume among hyperbolic link orbifolds. The two components of the singular locus have local groups $\ZZ/2\ZZ$ and $\ZZ/3\ZZ$.}
\label{Fig:OL}
\end{figure}

Hyperbolic link orbifolds form a natural class of $3$--orbifolds. A \emph{link orbifold} is a closed, orientable $3$--orbifold $\orbM$ whose singular locus $\Sigma$ has no nodes, meaning $\Sigma$ is a link in the underlying manifold $M$ of $\orbM$.
Atkinson and Futer have studied the problem of identifying the lowest volume hyperbolic link orbifold \cite{af}. They conjectured that the unique lowest volume hyperbolic link orbifold is an orbifold they call $\calO_L$, 
depicted in Figure~\ref{Fig:OL}, with underlying space $\SS^3$ and singular locus a two-component link \cite[Conjecture 1.4]{af}. Since $\calO_L$ is covered with degree 6 by the Weeks manifold $M_{\rm Weeks}$, its volume is $\vol(\calO_L) = \Vweeks / 6 \approx 0.1571$.
Atkinson and Futer proved that  $\calO_L$ is indeed of lowest volume among link orbifolds whose underlying space is required to be a $\ZZ/6\ZZ$--homology sphere.

In the present paper, we get the following result for link orbifolds. It has a milder volume hypothesis than Theorem \ref{Thm:FirstMain} and a stronger conclusion in case (ii).

\begin{theorem}\label{Thm:LinkMain}
Let $\orbM$ be a hyperbolic link orbifold such that $\vol \orbM \leq \vol(\calO_L) \approx 0.1571$. Then the underlying space $M$ of $\orbM$ satisfies one of the following: 
\begin{enumerate}[\:(i)]
\item $M$ is Seifert fibered over $\SS^2$, with at most three singular fibers. 
\item $M = \RR\PP^3 \# \RR\PP^3$ or $\RR\PP^3 \# L(p,q)$, where $p > 1$ is relatively prime to $6$.
\item Cutting $M$ along 
an essential torus produces one or
two compact $3$--manifolds with boundary, 
each of which is a Seifert fibered space
 over a disk with exactly two singular fibers, or over 
an annulus with at most one singular fiber.
\end{enumerate}
\end{theorem}
\EndCustom

This result represents some progress toward the conjecture that $\calO_L$ is the unique smallest-volume hyperbolic link orbifold. It will be proved as Corollary \ref{second main}.

One can get strong restrictions on the topological type of the underlying space of a closed, orientable orbifold $\orbM$ under considerably milder volume hypotheses than those of Theorems \ref{Thm:FirstMain} and \ref{Thm:LinkMain} if one rules out  \emph{essential turnovers}. A  \emph{turnover} in $\orbM$ is a two-dimensional suborbifold $\orbS$ of $\orbM$ having a $2$--sphere as underlying space, and having a singular locus consisting of three points. One says that $\orbS$ is  \emph{essential} if the inclusion of $\orbS$ in $\orbM$ induces an injection of orbifold fundamental groups.

\Theorem\label{Thm:IntroNoTurnoverMain}
Suppose that $\orbM$ is a closed, orientable hyperbolic $3$--orbifold,
and that
either 
\begin{enumerate}[\:\:(a)]
\item $\vol \orbM < \voct / 12 \approx 0.3053$, or 
\item  $\orbM$ is a link orbifold and $\vol \orbM < \voct / 6 \approx 0.6106$. 
\end{enumerate}
Then  either one of the alternatives (i)--(iii) of Theorem \ref{Thm:FirstMain} holds, or $\orbM$ contains an essential turnover. 
Furthermore, if Alternative (i) does not hold, then 
every node of the singular set of $\orbM$ is dihedral.
\EndTheorem

This will be proved in an equivalent form as Theorem \ref{Thm:NoTurnoverMain}.

Adams's result \cite{Adams:LimitVolume} that the first limit volume of orientable hyperbolic $3$--orbifolds is $\voct/12 $ implies as a consequence that there are only finitely many orbifolds satisfying the hypotheses of Theorems \ref{Thm:FirstMain} or \ref{Thm:LinkMain}. (There are no known bounds on these finite numbers, although one might conceivably prove such a bound by adapting the methods of Kobayashi and Rieck \cite{KobayashiRieck}.)
By contrast, there are infinitely many  orbifolds satisfying the hypotheses of Theorem \ref{Thm:IntroNoTurnoverMain}. In fact, Adams's result combined with orbifold Dehn filling theory \cite[Section 2]{Adams:LimitVolume} implies that all but finitely many hyperbolic $3$--orbifolds with volume less than $\voct/12 $ must have underlying space $\SS^3$. Thus  Alternative (a) of Theorem~\ref{Thm:IntroNoTurnoverMain} provides restrictions on the topological types of this finite set, while Alternative (b) provides topological restrictions on a genuinely infinite set of orbifolds that were not previously known to have any common topological structure.

\Custom{Outline of the proofs}\label{Sec:ProofOutline}

The bulk of the paper is devoted to the proof of Theorem \ref{Thm:IntroNoTurnoverMain}. Theorems \ref{Thm:FirstMain} and \ref{Thm:LinkMain} are deduced from Theorem \ref{Thm:IntroNoTurnoverMain}. A key step in this deduction, which is encapsulated  in 
Lemma~\ref{Turnover volume},
is to apply existing results from the literature to show that if the orientable hyperbolic $3$--orbifold $\orbM$ contains an essential turnover, we have $\vol \orbM \geq 0.1491$; and that if in addition  $\orbM$ is a link orbifold, we have
$\vol \orbM \geq 0.1658$. For the rest of this outline, we assume that $\orbM$ does not contain an essential turnover.

The proof of Theorem \ref{Thm:IntroNoTurnoverMain} makes strong use of the geometrization theorem  \cite{BBBMP,Morgan-Tian-geom}, which asserts that
every closed orientable $3$--manifold $M$ has a \emph{geometric decomposition} along essential spheres and tori into pieces that are either Seifert fibered or hyperbolic. 
See \S\ref{geometry and norms} for a more detailed summary. The 
proof of 
Theorem~\ref{Thm:IntroNoTurnoverMain} 
begins with the observation that under a far milder volume hypothesis, none of these pieces can be hyperbolic.

\Proposition\label{Prop:IntroGromov}
Let $\orbM$ be a closed, orientable hyperbolic $3$--orbifold whose volume satisfies $\vol
\orbM < \Vweeks \approx 0.9427$. Then none of the pieces in the geometric decomposition of $M \doteq \mani \orbM$ is hyperbolic. Consequently, $M$ is homeomorphic to
 a connected sum of graph manifolds.
\EndProposition

The proof of Proposition~\ref{Prop:IntroGromov} uses properties of the Gromov norm $|| \cdot ||$ that are reviewed in \S\ref{Gromov norm}. By Selberg's lemma, $\orbM$ has a manifold cover  $\torbM$ of 
some finite
degree $d$, so 
\[
\vol \torbM  < d \cdot \Vweeks \quad \text{which implies} \quad || \torbM || < d \cdot || M_\textup{Weeks} ||.
\]
Now, the covering map $f \from \torbM \to \orbM$ defines a map of underlying spaces  $F \from \torbM \to M$, also of degree $d$. Thus
\[
||M || \leq \tfrac{1}{d} || \torbM || <  || M_\textup{Weeks} ||,
\]
which implies that $M$ cannot have any hyperbolic pieces. See Corollary~\ref{Cor:SimpleVolBound} for details.

Following Proposition~\ref{Prop:IntroGromov}, if $\vol \orbM < \Vweeks$, then the underlying space $M$ is either a small Seifert fibered space (which is alternative (i) of Theorems~\ref{Thm:FirstMain} and \ref{Thm:LinkMain}), or else $M$ contains an essential 
surface which is a
sphere or torus. 
By making a suitable choice of such an essential surface $S \subset M$, we can arrange
 that $S$ is the underlying surface of  
a two-dimensional  suborbifold $\orbS \subset \orbM$, which is itself essential in a natural sense. 
Removing a tubular neighborhood of $\orbS$ produces a compact $3$--orbifold $\orbM \split \orbS$, with boundary two copies of $\orbS$.
We examine the topological structure of $\orbM \split \orbS$ by transporting 
 ideas from the theory of characteristic submanifolds of compact $3$--manifolds, which we briefly recall here.

Suppose that $P$ is a closed hyperbolic $3$--manifold, $S \subset P$ is an essential surface, and $Q = P \split S$ is the result of cutting $P$ along $S$.
By the work of Jaco, Shalen, and Johannson \cite{JS,Jo}, every such $Q$ contains a canonical \emph{characteristic submanifold} $V$, whose components are either $I$--bundles or solid tori. (See Proposition \ref{char-submflds, GS edition} for a more precise statement.) The components of $Q - V$ are either thickened annuli 
(across 
which the components of $V$ can be glued to form books of $I$--bundles) or acylindrical pieces with negative Euler characteristic, called $\Guts(Q)$. Each component of the guts admits a hyperbolic metric with totally geodesic boundary. 
By the work of Agol, Storm, and Thurston \cite{AST}, the volume of $P$ is bounded below by the volume of these metrics on $\Guts(Q)$, which in turn is bounded below 
in terms of
 its Euler characteristic:
\[
\vol P \geq - \voct \cdot \chi(\Guts(Q)),
\]
where $\voct$ 
is defined as above.
In particular, if some component of $Q$ is not homeomorphic to a book of $I$--bundles, then $\chi(\Guts(Q)) \leq -1$, hence
\[
\vol P \geq \voct \approx 3.6638.
\]

All of the above ideas extend (with some effort) to  $3$--orbifolds. 
In our setting, 
a characteristic suborbifold $\orbC$ of a component $\orbN$ of $\orbM \split \orbS$ 
can be defined by considering equivariant objects in a manifold cover of $\orbN$. The components of $\orbN - \orbC$ are either thickened annular $2$--orbifolds (along which the components of $\orbC$ can be glued to form orbifold books of $I$--bundles; see \S\ref{wuzza book}) or else acylindrical pieces with negative Euler characteristic that can be called $\Guts(\orbN)$. Now, the work of Agol, Storm, and Thurston, applied in appropriate covers  of $\orbM$, 
gives
\[
\vol \orbM \geq - \voct \cdot \chi(\Guts(\orbM \split \orbS)). 
\]
See Propositions~\ref{orbifold 7.2} and \ref{Prop:GromovVolEuler} for details.

In the context of the proof of Theorem 
\ref{Thm:IntroNoTurnoverMain}, we may apply these  notions concerning orbifold guts and orbifold volume  to a suborbifold $\orbS \subset \orbM$ arising from an essential sphere or torus 
$S \subset M$.
Using our running hypothesis that $\orbS$ is not a turnover, it can be deduced
that $\chi(\orbS) \leq -1/6$. This in turn implies that, if some component  $\orbN$ of $\orbM \split \orbS$ is not an orbifold book of $I$--bundles, then $\chi(\Guts(\orbN)) \leq -1/12$, and hence
\[
\vol \orbM \geq  \voct / 12 \approx 0.3053.
\]
A variant of this argument shows that if $\orbM$ is a link orbifold, and some component  $\orbN$ of $\orbM \split \orbS$ is not an orbifold book of $I$--bundles, then $\vol \orbM \geq  \voct / 6$. Since by hypothesis we have either (a) $\vol \orbM < \voct / 12$ or (b)  $\orbM$ is a link orbifold and $\vol
\orbM < \voct / 6 $, it follows that
each component of $\orbM \split \orbS$ is an orbifold book of $I$--bundles.

The remainder of the proof of Theorem 
\ref{Thm:IntroNoTurnoverMain} consists of using the orbifold book of $I$--bundles structure on the components of $\orbM \split \orbS$ to control the topology of its underlying space, namely $M\split S$.
For this reason, 
a substantial fraction of the work in this paper is devoted to orbifold $I$--bundles and books of $I$--bundles. In Section~\ref{Sec:I-bundle}, we perform a detailed study of both the underlying spaces and singular loci of orbifold $I$--bundles. We prove that the only nodes in the singular locus of an orbifold $I$--bundle must be dihedral
 (Proposition~\ref{Prop:IBundleSingLocus})
 and that the underlying space  is homeomorphic to a (manifold) $I$--bundle over a surface (Proposition~\ref{Prop:IBundleTotalSpace}). In Section~\ref{Sec:OrbifoldBooks}, we extend this analysis to orbifold books of $I$--bundles and prove (Proposition~\ref{Prop:ExpandedPage}) that under appropriate hypotheses, 
a union of underlying sets of bindings and pages in an orbifold book of $I$--bundles may be given the structure of a manifold $I$--bundle in a useful way.

To complete the proof of Theorem 
\ref{Thm:IntroNoTurnoverMain} in the presence of an essential sphere or torus $S$ in the underlying manifold $M$ of $\orbM$, we combine the results of Sections~\ref{Sec:I-bundle} and \ref{Sec:OrbifoldBooks} with some involved combinatorial arguments. The case where $S$ is a sphere is handled in Section \ref{sphere section}, while the case where $S$ is a torus is handled in Section \ref{torus section}.

We have stated that $S$ can be chosen to be the underlying manifold of an essential $2$--suborbifold $\orbS$ of $\orbM$, and we have pointed out that this forces the components of $\orbM\split\orbS$ to be orbifold books of $I$--bundles. In the setting of Section \ref{sphere section}, where $S$ is a sphere, the essentiality of $\orbS$ is achieved by choosing $S$, among all essential $2$--spheres, in a way that minimizes its number of intersections with the singular locus of $\orbM$. 
This minimization (or a slightly stronger form of it) 
has implications far beyond the essentiality of $\orbS$: it imposes strong restrictions on the orbifold books of $I$--bundles arising as  components of $\orbM \split \orbS $. These restrictions are strong enough to imply that the underlying manifold of each of these components is homeomorphic 
 either to $\SS^2 \times I$ or 
to a manifold obtained  from a non-trivial lens space by removing the
interiors of one or two closed $3$--balls (see Theorem~\ref{Thm:HomeoTypeSpherePieces}). This fact quickly implies Alternative (ii) of  Theorem \ref{Thm:FirstMain}, since $M \split S$ has at most two components.

Lemma \ref{Lem:SameSpanningSize}  provides an instructive illustration of how the minimality property of $S$ is used in restricting the structure of the books of $I$--bundles that make up $\orbM\split\orbS$.
The underlying manifold of an orbifold book of $I$--bundles tends to contain many properly embedded essential annuli that are disjoint from the singular set. In the context of Section \ref{sphere section}, 
suppose for simplicity that $M \split S$ (the underlying manifold of $\orbM\split\orbS$) is a connected $3$--manifold with boundary. If $J \subset M \split S$ is a properly embedded annulus disjoint from the singular set $\Sigma$, 
each boundary  component $C$ of $J$ has a well-defined  \emph{spanning size}, defined to be the minimum number of singular points in a disk in $\bdy (M \split S)$ bounded by $C$. In this setting, Lemma \ref{Lem:SameSpanningSize} asserts that the two boundary components of $J$ have the same spanning size. This is proved by showing that if the spanning sizes are different then one can fit together the union of $J$ with two suitably chosen disks in $\bdy (M \split S)$ to form a
sphere that can be perturbed to be embedded, and which meets
 $\Sigma$ in fewer points than $S$ does. This contradicts the minimality property of $S$.

The notion of the \emph{spanning size} of an annulus, which is well-defined by Lemma \ref{Lem:SameSpanningSize}, turns out to be an important organizing idea in the arguments of Section \ref{sphere section}. For a more systematic outline of these   arguments,  we refer the reader to the discussion immediately following Theorem~\ref{Thm:HomeoTypeSpherePieces}.

Section \ref{torus section} addresses the case where $M$ does not contain an essential sphere, but does contain an essential torus. While the arguments of this section are driven by the same philosophy as those of Section \ref{sphere section},
the details are quite different. In place of an essential sphere $S$ with a minimality property, we consider 
an essential torus $T$ with a slightly different minimality property. 

Since a simple closed curve on a torus need not bound a disk, 
the notion of spanning size does not have a direct analogue
 in this setting. However, the existence of essential curves on tori turns out to be of positive value, because it allows one to replace many combinatorial arguments by topological ones, which ultimately makes this section simpler than Section \ref{sphere section}. The conclusion of this section is given by
Proposition~\ref{Prop:TorusBooksTotalSpace}, which asserts that the underlying space of each 
component of $M \split T$
 admits a Seifert fibration, either over the disk with two singular fibers or over the annulus with at most one singular fiber. 
This implies Alternative (iii) of the conclusion of Theorem \ref{Thm:FirstMain}.
\EndCustom

\Custom{Expository background on orbifolds}
Sections \ref{orbisection} and \ref{low-dim sec}, and part of Section
\ref{bundle section}, are devoted to general definitions and results
about orbifolds. 
There are several reasons why we have needed to
include so much foundational material.

The first reason
 is that our proofs need refined versions of a number of
notions that appear in the literature, which required some
strengthening of the foundations. For example, both our treatment of maps
between orbifolds and our treatment of orbifold coverings and orbifold
fundamental groups are inspired by the account 
by Boileau,
Maillot and Porti \cite {BoileauMaillotPorti}. However, their
definitions do not appear to give functoriality of the orbifold
fundamental group with respect to orbifold maps, which turns out to be
crucial for our arguments. We have adjusted the definitions (including the definition of an orbifold map) 
in such a way that functoriality is immediate.

A second reason for the large amount of foundational material involves
category issues. We have mentioned that the underlying space $M$ of an
orientable $3$--orbifold $\orbM$ is a topological $3$--manifold.
Although such a manifold has a unique smooth structure \emph{up to
diffeomorphism}, many of our arguments require a specific smooth
structure on $M$ which is compatible with the (smooth) orbifold
structure of $\orbM$. The relevant notion is that of an {\it
$\orbM$--accordant} smooth structure, which we develop in \S\ref{wuzzaccordant}.

A third reason is that we need to use notions about orbifolds 
that the previous literature barely touches.
For instance, the
notion of an \emph{orbifold with
corners} is needed for the very definition of an orbifold book of
$I$--bundles. In working with the singular set and local groups of an
orbifold, we have found it very convenient to use the language of {\it
strata} of the underlying space.

Finally, we have found it necessary to rework 
the basic material about
suborbifolds. The relevant definitions in the literature are not mutually consistent,
and some of them permit pathological examples; compare Examples~\ref{Exa:BMPSuborbifold}.
We have also found it useful to distinguish between ``weak suborbifolds'' and ``suborbifolds
(in the strong sense).'' For example, strata are identified with
``weak suborbifolds'' which are not ``suborbifolds.''
\EndCustom

\Custom{A bibliographical note}
There is significant overlap between the results in this paper and
those in Shalen's unpublished monograph
\cite{shalen-monograph}.
Many of the results in \cite{shalen-monograph} are  supplanted
by the main results of this paper, which are considerably stronger and
have considerably simpler proofs. 
The results in
\cite{shalen-monograph} which are not subsumed by results here, in
particular those involving bounds on orbifold homology in terms of
orbifold volume, may appear elsewhere.
\EndCustom

\Custom{Acknowledgments}
We are indebted to many people for valuable discussions,
suggestions, and encouragement regarding this paper and its
predecessor
\cite {shalen-monograph}:
Ian Agol, Chris Atkinson, Francis Bonahon, Ted Chinburg, Marc Culler,
Benson Farb,  Tom Goodwillie, Tracy Hall, Moishe Kohan, Christian Lange,
Chris Leininger, Ben Linowitz, Shawn Rafalski, David Speyer, Matthew
Stover,   John Voight, and Shmuel Weinberger.

During this project, Futer was partially supported by  NSF grant DMS--2405046.
\EndCustom

\section{Preliminaries}\label{prelim section}

\Custom{Notational conventions}\label{notational}
Throughout the paper, $\RR_+$ will denote the non-negative real numbers.
For any integer $n\ge0$, the closed unit ball in $\RR^n$ will be
denoted by $\BB^n$. The unit sphere in  $\RR^{n+1}$ will be denoted $\SS^n$.
 The open unit disk in $\CC$ will be denoted by
$\DD$.

Throughout the paper, $\HH^n$ will denote $n$--dimensional hyperbolic space.
We shall denote by $\vtet$ the volume of a regular ideal tetrahedron
in $\HH^3$, and by $\voct$ the volume of a regular ideal octahedron
in $\HH^3$. Numerically, $\vtet = 1.0149 \ldots$ and $\voct = 3.6638 \ldots$.
\EndCustom

\Definition\label{half or quarter}
Let $n\ge1$ be an integer.
We define an  \emph{open half-ball} 
  in $\RR^{n-1}\times\RR_+$ centered at a point $c \in  \RR^{n-1} \times \{0\}$
to be a set which is the intersection of    $\RR^{n-1}\times \RR_+$
with an open
ball in $\RR^n$. The
 \emph{axis} of the half-ball is its intersection with the line which passes
through $c$ and is perpendicular to $\RR^{n-1} \times \{0\}$.
If $n\ge2$, we define an  \emph{open quarter-ball}
  in $ \RR^{n-2}\times \RR_+^2$
to be a set which is the intersection of $ \RR^{n-2}\times \RR_+^2$
with an open
ball in $\RR^n$ whose center lies in $\RR^{n-2} \times \{(0,0)\}$.
\EndDefinition

\Custom{Manifold conventions}\label{manifolds}
It will always be understood that a  \emph{manifold} may have a
boundary, except where we specify otherwise.

All statements involving manifolds will be understood in the smooth
category, except when the 
topological
 or PL category is explicitly mentioned. 
(Such mentions will be frequent in Sections~\ref{Sec:OrbifoldBooks}--\ref{torus section}.)
 Thus
manifolds, submanifolds, and maps (including immersions and
embeddings) between manifolds will be understood to be smooth.

As is customary in $3$--manifold topology, we shall often use smooth versions of results that have been proved elsewhere in the category of PL manifolds. In all cases this is justified by standard smoothing theorems.
\EndCustom

\Custom{Neat embeddedness}
\label{NeatEmbed}
As is standard in differential topology, we say that a submanifold $N$
of a manifold $M$ is  \emph{neatly embedded} in $M$ if
\begin{itemize}
\item $N$ is a closed subset of $M$,
\item $N\cap\bdy M=\bdy N$, and
\item $N$ and $\bdy M$ meet transversely.
\end{itemize}
\EndCustom

The following strong form of the smoothing theorem for topological
$3$--manifolds, which is based on the results of Munkres~\cite{munkres},
 will be used in the proof of Proposition \ref{smooth underlying}.

\Proposition\label{extend}
Suppose that 
$U$ is an open
subset of a topological $3$--manifold
$M$. Then every smooth structure on 
$U$ 
extends to a smooth
structure on $M$.
\EndProposition

\Proof
We express $M$ as a countable union $\bigcup_{i=1}^{\infty} V_i$, where the $V_i$ are topological open $3$--balls, and prove by induction on $n$ that the smooth structure on $U$ extends to a smooth structure on $U \cup \left(\bigcup_{i=1}^n V_i \right)$. The inductive step involves showing  that the smooth structure on $V_n \cap \left(U \cup \bigcup_{i=1}^{n-1} V_i \right)$ extends to $V_n$. Since $V_n$ embeds in $\RR^3$, it suffices to prove the proposition in the special case where $M \subset \RR^3$, which is what we now assume.

The smooth structure on $U$ 
given by the hypothesis defines a differentiable manifold
$N$ whose underlying topological space is  $U$.
On the other hand, since $M\subset\RR^3$, the manifold
$M$ admits a smooth structure; this smooth structure on $M$ restricts
to one on 
$U$, which defines another differentiable manifold $N'$
having $U$ as underlying space. The identity map of $U$ may then
be regarded as a homeomorphism $h\from N \to N'$.

Now choose a strictly positive-valued continuous  function $\epsilon$ on
$U$ 
such that $\lim_{x\to p}\epsilon(x)=0$ for every point $p \in M - U$. 
(For example, if we fix a distance function defining the topology
of $M$, we may define $\epsilon(x)$, for each $x\in U$, to be the
distance from $x$ to the 
closed set $M-U$.)

Fix a distance function $d$ defining the topology of $N'$. 
By a theorem of Munkres
\cite[Theorem 6.3]{munkres},
there is a diffeomorphism $\ell\from N\to N'$ which is an
$\epsilon$-approximation to $h$ in the sense that
$d(h(x),\ell(x))<\epsilon(x)$ for every $x\in N$. Our choice of the
function $\epsilon$ guarantees that $\ell$, regarded as a
self-homeomorphism of 
$U$, 
extends to a homeomorphism
$\overline{\ell}\from M\to M$ which fixes every point of $M - U$. Pulling back
the smooth structure
on $M$ via the homeomorphism $\overline\ell$
gives a new
smooth structure on $M$ extending the given smooth structure
on $U$.
\EndProof

The following result, Proposition \ref{tree thing}, will be used in the proof of
Lemma \ref{big ol' lemma}. In Proposition \ref{tree thing}, and in
proof and
application, the terms ``tree'' and ``graph'' are to be understood in
their topological sense: a graph is a  CW
complex of dimension at most $1$, and a tree is a simply connected graph.

\Proposition\label{tree thing}
Let $T$ be a finite tree, and let $V$ denote its vertex set. 
Suppose that each vertex $v\in V$ is assigned an integer $m_v\ge0$. For each  set $U\subset T$, set $n_U=\sum_{v\in U\cap V}m_v$. Then $T$ has a vertex $w_0$ such that, for each component $R$ of $T-\{w_0\}$, we have $n_R\le n_T/2$.
\EndProposition

\Proof
Assume that for each vertex $w$ of $T$ there is a component $R$ of $T-\{w\}$ such that $n_R> n_T/2$. This component $R$ is clearly unique for any vertex $w$, and will be denoted $R_w$. For each vertex $w$, let  $h_w$ denote the cardinality of $R_w\cap V$. Choose a vertex $z\in V$ such that $h_z$ achieves the minimum value of $h_w$.

Since $T$ is a tree, and $R_z$ is a component of $T-\{z\}$, there is a unique vertex $z' \in R_z$ that is connected to $z$ by an edge.
Let $S$ denote the component of $T-\{z'\}$ containing $z$. Since $T$ is a tree, $V$ is the disjoint union of $R_z\cap V$ and $S\cap V$. Hence $n_T=n_{R_z}+n_S$. The definition of $R_z$ gives $n_{R_z}>n_T/2$, and hence $n_S<n_T/2$.

The definition of $R_{z'}$ gives $n_{R_{z'}}>n_T/2$. Since $n_S<n_T/2$, the components $S$ and $R_{z'}$ of $T-\{z'\}$ are distinct; in particular $S\cap V$ and $R_{z'}\cap V$ are disjoint. Since we have observed that $V$ is the disjoint union of $R_z\cap V$ and $S\cap V$, we have $R_{z'}\cap V\subset R_z\cap V$. Furthermore, we have $z'\in R_z$ but $z'\notin R_{z'}$, so that $R_{z'}\cap V$ is a \emph{proper} subset of $R_z\cap V$. This implies that $h_{z'}<h_{z}$,
contradicting the definition of $z$.
\EndProof

\Custom{Geometric decompositions}
\label{geometry and norms}
Throughout the paper, we use standard notions from $3$--manifold topology, such as irreducibility, incompressibility, and Seifert manifolds. For a recent reference on this material, see e.g.~Martelli~\cite[Chapter 9]{Martelli:GeometricTopology}.

We shall make strong use of Perelman's geometrization theorem; see \cite{BBBMP} or \cite{Morgan-Tian-geom} for accounts of the proof.  One formulation of the theorem (cf.\ \cite[Theorems 1.1.4 and 1.1.6]{BBBMP}) is that if $N$ is a closed, orientable, irreducible $3$--manifold, there
is a (possibly empty) closed, orientable $2$--manifold $\boldT\subset
N$, canonical up to isotopy, such that
\begin{itemize}
\item each component of $\boldT$ is an incompressible torus, and 
\item each
component of the manifold $N'$ obtained by splitting $N$ along $\boldT$ is
diffeomorphic to either (a) a Seifert fibered $3$--manifold or (b) the compact
core of a finite-volume hyperbolic $3$--manifold. 
\end{itemize}
(The conditions (a) and (b) are mutually exclusive.) The components of
$N'$, which are well-defined up to diffeomorphism, will be called {\it
  pieces} of $N$, and the pieces that satisfy Alternative (b) will be
called  \emph{hyperbolic} pieces. For any hyperbolic piece $P$, we shall
write $\vol P$ for the volume of the hyperbolic manifold whose compact
core is diffeomorphic to $P$. If all pieces of $N$ are Seifert
fibered manifolds, we say that $N$ is a  \emph{graph manifold}.

We also recall \cite[Theorem 1.1.3]{BBBMP}
that every closed, oriented $3$--manifold $M$ is oriented-diffeomorphic
to a connected sum 
$N_1\#\cdots\#N_k$, 
where  $k\ge0$, the
$N_i$ are oriented prime $3$--manifolds, and the decomposition is
unique up to orientation-preserving diffeomorphisms and changing the
order of terms. Here it is understood that when $k=0$, the connected
sum gives $\SS^3$, which is the identity for connected sum. We consider $\SS^3$ to be \emph{trivial}. A non-trivial $3$--manifold is said to be \emph{prime}  it cannot be expressed as a
connected sum of two non-trivial summands. A prime
$3$--manifold is either irreducible or diffeomorphic to
$\SS^2\times\SS^1$.

In particular, if we are given a closed,  \emph{orientable} $3$--manifold
$M$, the summands $N_1,\ldots,N_k$ are determined up to
diffeomorphism. The pieces of those $N_i$ which are irreducible will be called  \emph{pieces} of $M$.
\EndCustom

One consequence of the geometric decomposition described in \S\ref{geometry and norms} is

\Proposition\label{because Perelman}
If a closed orientable $3$--manifold $M$ is
not hyperbolic, and $M$ contains no essential spheres or tori, then $M$
admits a Seifert fibration over $\SS^2$ with at most three singular
fibers.
\EndProposition

\Proof
Since $M$ is not hyperbolic, and contains no essential spheres or
tori, it must be Seifert fibered 
by \S\ref{geometry and norms}.
Now, it follows from \cite[Proposition 10.4.14]{Martelli:GeometricTopology} that $M$ admits a Seifert fibration over $\SS^2$ with at most three singular fibers.
\EndProof

\Custom{Gromov norm}\label{Gromov norm}
We will also make strong use of the Gromov norm. We recall 
from \cite[Definition 13.2.1]{Martelli:GeometricTopology}
that for any topological space $X$ 
and any integer $d>0$, there is a Gromov
(semi-)norm $\|\cdot\|$ on the vector space $H_d(X;\RR)$, 
that associates a number
$\|c\| \geq 0$ to every homology class
 $c\in H_d(X;\RR)$. We will consider
only the case where $X$ is a closed, orientable \emph{topological} $3$--manifold and
$c$ is the fundamental class $[X]$. We shall write $\|X\|$ for
$\|[X]\|$. 

According to \cite[Proposition 13.2.4]{Martelli:GeometricTopology}, if
$M$ and $N$ are closed 
oriented $3$--manifolds, and there is a
degree-$d$ map from $N$ to $M$, then
\Equation\label{degree and Gromov}
 \| N \| \: \geq \: d \cdot  \|M\|.
\EndEquation

If $M$ is a closed, orientable hyperbolic $3$--manifold, it is a result
due to Gromov,  which appears
as \cite[Theorem 13.2.11]{Martelli:GeometricTopology}, that
\Equation\label{genuine gromov}
\vtet  \|M\|=\vol M,
\EndEquation
where $\vtet$ is the volume of a regular ideal tetrahedron, as in \S\ref{notational}.

A folklore result, which generalizes \eqref{genuine gromov}, is that
if $M$ is a closed, orientable $3$--manifold $M$ and $P_1,\ldots,P_m$
denote its hyperbolic pieces, then 
$
\vtet \|M\| =  \sum_{i=1}^m\vol P_i.
$
We shall not quote this fact as we do not have a reference for
it. However,  we will use
the following special case, which appears as
\cite[Theorem 6.5.5]{Thurston:notes}:
\Claim
\label{all we need}
If $M$ is a closed, orientable $3$--manifold $M$ and $P$ is a
 hyperbolic piece of $M$, then
\[
\vtet  \|M\| \geq  \vol P.
\]
\EndClaim
\EndCustom

\section{Orbifolds}\label{orbisection}
This section and the next one develop some foundational material about orbifolds. 
The notion of an orbifold originated in the work of Satake~\cite{Satake:VManifold}.
Some modern references for this material include Boileau--Maillot--Porti \cite[Chapter 2]{BoileauMaillotPorti}, Choi \cite{Choi:Orbifolds}, Cooper--Hodgson--Kerckhoff \cite[Chapter 2]{CHK:OrbifoldBook}, and  Kapovich~\cite[Chapter 6]{Kapovich:HyperbolicManifolds}. For several reasons that were surveyed in the introduction --- orbifolds with corners, category issues, functoriality, and  suborbifolds --- we were unable simply to rely on the existing references, and needed to strengthen some of the foundational definitions. Nonetheless, our treatment is not entirely self-contained, and we have endeavored to point to references as much as practically feasible.

Before giving formal definitions, 
we introduce the following running example. 

\Example\label{running example}
Let $S$ be a compact surface with non-empty boundary. Then the product $M = S
\times I$ is an example of a $3$--manifold with corners.  If $G$ is a
finite group acting on $M$ by diffeomorphisms, the quotient $\orbN =
M/G$ is an example of what will be called a very good $3$--orbifold with $2$--corners. (Compare Definition~\ref{Def:OrbifoldChart}, and \S\S\ref{why corners?}
and \ref{Orbifold covers}.) 
\EndExample

\Custom{Orbifold charts}
\label{Def:OrbifoldChart}
Let $n$ and $k$ be non-negative integers.
An  \emph{$(n,k)$--orbifold chart} in a topological space $X$ is a quadruple
$( U , \tU , \Phi, G)$, 
where: 
\begin{itemize} 
\item $ U $ is an open subset of $X$;
\item $\tU$ is an open subset of 
$\RR^{n-j}\times\RR_+^j
$ for some $j$ with $0\le j\le \min(n,k)$;
\item
$G$ is a finite group of diffeomorphisms of $\tU $; and
\item $\Phi \from  \tU  \to  U $ is a continuous map that
factors through a homeomorphism from $\tU  / G$ to
$ U $. 
\end{itemize}
We may refer to
$ U $, $\tU $, $\Phi$ and $G$ respectively as the
   \emph{support} of the chart, the \emph{domain} of the chart, the  \emph{chart map} and the  \emph{group
   associated with} the chart.

If $( U , \tU , \Phi, G)$ is a chart in $X$, and $\tU'$ is an open subset of $\tU$ which is precisely $G$--invariant (in the sense that for every $g\in G$ we have either $g (\tU')= \tU'$ or $g(\tU') \cap \tU'=\emptyset$), then the stabilizer of $\tU'$ in $G$ defines a finite group $G'$ of diffeomorphisms of $\tU'$, and 
$( \Phi (\tU') , \tU' , \Phi|\tU', G')$ is a chart in $X$. A chart defined in this way will be called a \emph{subchart} of $( U , \tU , \Phi, G)$. 

We define a  \emph{transition map} from a chart $( U _1, \tU _1, \Phi_1, G_1)$ to a chart
$( U _2, \tU _2, \Phi_2, G_2)$ in $X$ to be a diffeomorphism
$\psi$ 
from the domain $\tU_1'$ of some subchart $( U_1' , \tU_1' , \Phi_1', G'_1)$ of 
$( U_1 , \tU_1 , \Phi_1, G_1)$ of 
$\tU$ to the domain $\tU_2'$ of some subchart $( U_2' , \tU_2' , \Phi_2', G_2')$ of
$( U_2 , \tU_2 , \Phi_2, G_2)$ such that 
\begin{itemize}
\item $\psi$ is equivariant with respect to some isomorphism $J\from G_1'\to G_2'$, meaning that $\psi(g x)=J(g) \big( \psi(x) \big)$ for all $x\in \tU_1$ and $g\in G_1'$; and
\item 
$\Phi_2'\circ\psi = \Phi_1'$. 
\end{itemize}

Two
$(n,k)$--orbifold
charts $( U _1, \tU _1, \Phi_1, G_1)$ and
$( U _2, \tU _2, \Phi_2, G_2)$ are said to be {\it
  compatible} if, for every $\tx _1\in  \tU_1$ and every $\tx _2 \in
\tU_2$ with $\Phi_1(\tx _1) = \Phi_2 (\tx _2)$, there is a
transition map $\psi$ between $( U _1, \tU _1, \Phi_1, G_1)$ and
$( U _2, \tU _2, \Phi_2, G_2)$, whose domain contains
$\tx_1 $, such that 
$\psi(\tx_1)=\tx_2$. 
\EndCustom

\Custom{Orbifolds with corners and manifolds with corners}
\label{Def:Orbifold}
\label{singular points}
A (smooth)   \emph{$(n,k)$--orbifold atlas} for $X$ is a collection of pairwise
compatible $(n,k)$--orbifold charts
in $X$ whose supports cover $X$. An  \emph{$n$--dimensional orbifold with
$k$--corners} 
$\orbM$
is a paracompact Hausdorff
space 
$X$ equipped with a maximal 
$(n,k)$--orbifold atlas. 
The
space $X$ is called the  \emph{underlying space} of $\orbM$
and will often be denoted by $\mani\orbM$.
By a  \emph{chart for} $\orbM$ we mean a chart in the maximal atlas
  defining $\orbM$.
A  \emph{point} of 
$\orbM$ is defined to be a point of $\mani\orbM$.

Let $\orbM$  be an $n$--orbifold with $k$--corners.
A point $x$ of $\orbM$ is called a  \emph{$j$--corner point} of $\orbM$,
for a given  
integer $j$ with $0\le j\le \min(n,k)$
if for some,
and hence for every, chart $( U , \tU , \Phi, G)$ for
$\orbM$ with $x\in U $,
the chart domain $\tU$ is an  open subset of 
$\RR^{n-j}\times \RR_+^j$, and
$\Phi^{-1}(x)
  \subset
\RR^{n-j} \times\{0\}$ (where $0$ denotes the origin of
$\RR^j$). Thus every point of $\orbM$ is a $j$--corner point for  a
unique choice of $j$.
A $0$--corner point of $\orbM$ is called an \emph{interior point}.

If $n$ and $k$ are non-negative integers, we define a (smooth) {\it
  $n$--manifold with $k$--corners} to be an   $n$--orbifold with
$k$--corners such that every point $x$ lies in the support of a chart such that the associated group $G$ is trivial.
(A more succinct characterization, using terminology to be introduced in \S\ref{Sec:LocalGroups}, is that an  $n$--manifold with corners is an   $n$--orbifold with
corners whose singular locus is empty.)
Specializing to the
case $k=1$, we recover the notion of a (smooth) $n$--manifold.
\EndCustom

\begin{custom}{Where corners arise}
\label{why corners?}
The importance of corners in this paper arises in large part from
their role in forming Cartesian products. If $n_1$, $k_1$, $n_2$ and
$k_2$ are non-negative integers, and $\orbM_i$ is an $n_i$--orbifold
with $k_i$--corners for $i=1,2$, then there is a canonically defined
$(n_1+n_2)$--orbifold $\orbM_1\times\orbM_2$ with $(k_1+k_2)$--corners
such that
$\mani{\orbM_1\times\orbM_2}=\mani{\orbM_1}\times\mani{\orbM_2}$. In
particular, if the $\orbM_i$ are orbifolds with $1$--corners (that is, simply \emph{orbifolds} according to \S\ref{orb with boundary}) then
$\orbM_1\times\orbM_2$ is an orbifold with $2$--corners.
\end{custom}

\Custom{Standard group actions}\label{invariant neighborhood}
If $n$ and $j$ are integers with $0\le j\le n$, there is a standard
identification of $\bigO(n-j)\times\bigO(j)$ with a subgroup of the
orthogonal group $\bigO(n)$. We also have a standard identification of
the symmetric group $\scrS_j$ with a subgroup of $\bigO(j)$. With the
resulting identification of $\bigO(n-j)\times\scrS_j$ with a subgroup of
$\bigO(n)$, the restriction to $\bigO(n-j)\times \scrS_j$ of the standard
action of $\bigO(n)$ on $\inter\BB^n$ leaves
$B\cap(\RR^{n-j}\times\RR_+^j)$ invariant, for any open ball
$B\subset\RR^n$ centered at $0$. This defines a
 \emph{standard action} of $\bigO(n-j)\times \scrS_j$ on
$B\cap(\RR^{n-j}\times\RR_+^j)$.
\EndCustom

\Proposition\label{general bochner}
If a compact  group $G$ acts by diffeomorphisms on an
$n$--manifold with corners $M$, and fixes a $j$--corner point $x\in M$
for some $j \leq  n$, then
$x$ has a $G$--invariant
neighborhood $V$ which can be identified diffeomorphically with 
$\inter\BB^n\cap(\RR^{n-j}\times\RR_+^j)$
in 
such a way that $x=0$, 
and the action of $G$ factors through the standard action of $\bigO(n-j)\times \scrS_j$ on
$\inter\BB^n\cap(\RR^{n-j}\times\RR_+^j)$ defined in \S\ref{invariant neighborhood}.
\EndProposition

\Proof
This is a special case of \cite[Corollary 3.27]{log-corners}, where
the quantity denoted
by $k$ in that corollary is equal to $0$. The roles of the quantities denoted
by $l$ and $n$ in \cite[Corollary 3.27]{log-corners} are played here
by $n-j$ and $j$ respectively.

The case $j=0$ is the Bochner linearization theorem \cite{bochner}.
\EndProof

From Proposition \ref{general bochner} it is easy to deduce:
\Corollary \label{Bochner corollary}
If $G$ is a finite group of diffeomorphisms  of a connected $n$--manifold
with $k$--corners 
$M$,  
and if the action of $G$  restricts to the trivial action
on some non-empty open subset of $M$, then $G$ is trivial. \NoProof
\EndCorollary

\Custom{Standard charts}
\label{about orbifolds}

It follows from 
Proposition \ref{general bochner}
that if $x$ is a $j$--corner point of an $n$--orbifold with corners
$\orbM$, there is a chart $( U , \maybeD, \Phi, G)$ 
for $\orbM$ such that 
\begin{enumerate}[\:\:(a)]
\item $\maybeD$ is the intersection of $\RR^{n-j}\times\RR_+^j$ with some open ball centered at the origin of $\RR^n$;
\item
$\Phi(0)=x$; 
\item
$G$ is a subgroup of 
$\bigO(n-j)\times\scrS_j$, whose action on $D$ is the standard one
(see \S\ref{invariant neighborhood}).
\end{enumerate}
Such a
chart will be called a  \emph{standard chart centered at $x$}.

The supports of standard charts for $\orbM$ form a basis for the topology of $\mani\orbM$. 
\EndCustom

\Custom{Local groups}\label{Sec:LocalGroups}
If $x$ is a point of an orbifold with corners $\orbM$, 
 we  define a
 \emph{local group at $x$} to be a group which
  has the form $\Stab_G(\tx)$
for some chart $( U , \tU , \Phi, G)$ with $x\in U$ and some
$\tx\in\Phi^{-1}(x)$. 
Changing the point $\tx\in\Phi^{-1}(x)$ has the effect of conjugating
the local group within $G$, and changing the chart 
$( U , \tU , \Phi, G)$ has the effect of conjugating the local group by a
 transition map. Thus a local group at $x$ is uniquely determined up
 to isomorphism. Furthermore, the isomorphism between two local groups
 at $x$ is itself canonical modulo inner automorphisms. For this
 reason, we shall refer to {\it
  the} local group at a point $x$ of $\orbM$; it may be denoted by
$G_x^\orbM$, or by $G_x$ when it is clear which orbifold is involved. 

The group associated with a standard chart centered at a
point $x$ is
canonically identified (modulo inner automorphisms) with the local
group $G_x$. In particular, if $x$ is a $j$--corner point (where $0\le
j\le n$), then $G_x$ is canonically identified, modulo conjugacy in $\bigO(n-j)\times\scrS_j$, with
a subgroup of $\bigO(n-j)\times\scrS_j$.
Furthermore, the groups associated
to two different standard charts centered at $x$ are identified with
conjugate subgroups of $\bigO(n-j)\times\scrS_j$.
Thus 
we may regard $G_x$ as a
subgroup of 
$\bigO(n-j)\times\scrS_j$
defined up to conjugacy. 
In particular,
$G_x$ is canonically identified, up to conjugacy, with a
subgroup of 
$\bigO(n)$.

The order of $G_x^\orbM$ will be called the order of $x$; it may be denoted by $\ord^\orbM(x)$,
or by $\ord(x)$ when it is clear which orbifold is involved.

A point $x$ of $\orbM$ is said to be  \emph{singular} if 
$G_x$ is non-trivial. The set of all singular points of
$\orbM$ is a closed subset of $\mani\orbM$, which is called the {\it
  singular set} of $\orbM$ and is denoted by $\Sing_\orbM$.
\EndCustom

\Custom{Orbifolds with boundary}
\label{orb with boundary}
For any $n\ge1$ one may define an $n$--orbifold with boundary to be an
$n$--orbifold with $1$--corners. In this paper we shall use the term
 \emph{orbifold} to mean an orbifold with boundary; this is consistent
with the convention regarding manifolds that was laid out in \S\ref{manifolds}.

We define an \emph{$n$--orbifold with nonsingular $2$--corners} to be an  $n$--orbifold
with $2$--corners whose $2$--corner points are all nonsingular.

We emphasize that, according to these conventions, an orbifold (without further
specification) is permitted to have $1$--corner points but not
$k$--corner points for $k > 1$. Some of the substantive results of this
paper will concern $3$--orbifolds with nonsingular $2$--corners, but $k$--corners for
$k > 2$ will not arise after the foundational material of the present section.
\EndCustom

All the observations and definitions about orbifolds with $k$--corners
that appear below
specialize to the case of orbifolds (with boundary) by setting $k=1$,
and will be used freely in that context throughout the paper.

\Custom{Naming convention}\label{Sec:naming}
Throughout the paper, we denote orbifolds with corners  by calligraphic letters, for instance $\orbM$ and $\orbN$. 
It will often be convenient to use a single letter to denote the
underlying space of an orbifold with corners.
When this arises, we typically use the
``math-italic'' counterpart of the calligraphic letter denoting the
orbifold with corners itself: e.g.\ $M=\mani\orbM$, $N=\mani\orbN$, etc.

In addition, we use a script letter  to denote a family of orbifolds
with corners
whose typical element is
denoted by the corresponding  calligraphic letter: e.g.\ $\scrM = \{
\orbM, \orbM' \ldots \}$. We also use a script letter to denote a
(possibly) disconnected orbifold whose typical component is denoted by
the corresponding calligraphic letter: e.g.\ $\scrM =  \orbM \sqcup
\orbM' \sqcup \ldots$. 
The distinction between collections and disjoint unions will be made
clear in each instance.
In the manifold setting, we use boldface letters to denote collections
of manifolds with corners and disjoint unions of manifolds with corners: e.g.\ ${\bf M} = \{ M, M' \ldots \}$.
\EndCustom

\Custom{Orientations}
\label{orientable orbifolds}
An \emph{oriented atlas} for an orbifold with corners $\orbM$ is an atlas of charts for $\orbM$, 
together with an orientation of each chart domain, with the
property that every transition map is orientation-preserving.
This implies that for each
chart,
the associated group consists of orientation-preserving
diffeomorphisms of the chart domain.

An \emph{orientation} of an orbifold with corners $\orbM$ is a maximal oriented atlas.
We say that
$\orbM$ is  \emph{orientable} if it admits an orientation.
\EndCustom

\Custom{Orbifold maps}
\label{maps and such}
Suppose that $\orbM$ and $\orbM'$ are orbifolds
with corners, and that $F$ is a continuous map from $M=\mani\orbM$ to
$M'=\mani{\orbM'}$. Let 
$( U , \tU , \Phi, G)$ and
$( U ', \tU ', \Phi', G')$ be charts for $\orbM$ and
$\orbM'$ respectively. 
A  \emph{smooth \locallift} of $F$ is a pair $(\tF, \psi)$ where $\psi \from G \to G'$ is a group homomorphism, $\tF \from  \tU\to\tU'$ is a $\psi$--equivariant smooth map,  and 
$F\circ\Phi=\Phi'\circ\tF$. A necessary condition for the existence
of a smooth \locallift\  is that $F(U)\subset U'$. We 
shall often leave the homomorphism $\psi$ implicit and simply say that $\tF$ is a local lift of $F$.

Now suppose that two \locallift s $\tF_1$ and $\tF_2$ of $F$ are
given. Thus for $i=1,2$, the \locallift\ $\tF_i$ is a map from
$\tU_i$ to $\tU_i'$, where $( U_i , \tU_i , \Phi_i, G_i)$ and
$( U_i', \tU_i', \Phi_i', G_i')$ are charts for $\orbM$ and
$\orbM'$ respectively. We shall say that $\tF_1$ and $\tF_2$ are
 \emph{compatible} if, for every $\tx _1\in  \tU_1$ and every $\tx _2 \in
\tU_2$ with $\Phi_1(\tx _1) = \Phi_2 (\tx _2)$, there exist a
transition map $\psi$ from $( U _1, \tU _1, \Phi_1, G_1)$ to
$( U _2, \tU _2, \Phi_2, G_2)$ with domain $V \subset \tU_1$, and a
transition map $\psi'$ between $( U _1', \tU _1', \Phi_1', G_1')$ and
$( U _2', \tU _2', \Phi_2', G_2')$ whose domain contains $\tF_1(V)$, such that $\tF_2 \circ\psi=\psi'\circ(\tF_1|V)$.

An  \emph{atlas of \locallift s} for $F$ is defined to be a collection
of pairwise compatible \locallift s whose domains cover $\orbM$.
A  \emph{smooth orbifold map} from $\orbM$ to $\orbM'$ is defined by a
continuous map $F\from M\to M'$ together with an atlas of \locallift s of
$F$ which is maximal with respect to inclusion.

We shall call $F$ the  \emph{underlying map} of $f$. The \locallift s of
$F$ that belong to the maximal atlas of \locallift s defining $f$ will
be called
\emph{\locallift s\ for $f$.} 
\EndCustom

\begin{example}\label{UnderlyingMapNotEnough}
In general, a continuous orbifold map is not uniquely determined by its
underlying map of spaces. For example, consider the $1$--orbifold $\RR$ and the
$2$--orbifold $\RR^2/\langle\rho\rangle$, where $\rho$ denotes the
counterclockwise rotation of $\RR^2$ through an angle $2\pi/n$ for an integer $n>1$. The identity
map of $\RR$ and the quotient map $\RR^2\to\RR^2/\langle\rho\rangle$
are chart maps defining charts for these two orbifolds, whose supports
are all of $\RR$ and all of $\RR^2/\langle\rho\rangle$,
respectively. For each $j\in\{0,\ldots,n-1\}$, we may define a  map
$\ell_j\from\RR\to\RR^2$ between chart domains by setting
$\ell_j(x)=(x,0)$ for $x\ge0$, and $\ell_j(x)=\rho^j(|x|,0)$ for
$x\le0$. Then the $\ell_j$ are all local lifts of the same map of
spaces  $F\from\mani\RR\to \mani{\RR^2/\langle\rho\rangle}$, and
these local lifts determine $n$ distinct continuous orbifold maps having $F$ as
underlying map. This phenomenon will be important in connection with
the unique path lifting property for orbifold coverings (see
\S\ref{crucial fact} below).
\end{example}

\Custom{Orbifold maps, continued}\label{maps continued}
If $f \from \orbM\to{\orbM'}$ and $f' \from  \orbM'\to{\orbM''}$ are
smooth  orbifold maps
between orbifolds with corners, we  define the  \emph{composition} of
$f$ and $f'$ to be
the unique orbifold map $f'\circ f
\from  \orbM \to{\orbM''}$ such that (a) the underlying map of $f'\circ f$
is $F'\circ F$, where $F$ and $F'$ denote the underlying maps of
$f$ and $f'$ respectively, and (b) whenever $\tF$ and
$\tF'$ are  \locallift s for $f$ and $f'$ respectively, such that the
image of $\tF$ is contained in the domain of $\tF'$, the composition
$\tF'\circ \tF$ is a \locallift\  for $f'\circ f$.
With this notion of composition, we have a
category of orbifolds with corners in which the morphisms are smooth
orbifold maps.

The isomorphisms in this category will be called  \emph{orbifold
  diffeomorphisms}.
The underlying map of an orbifold diffeomorphism is a homeomorphism,
and the local lift assigned to each point of its domain is a manifold diffeomorphism.

If $\orbM_i$ are orbifolds for $i=1,2$, there is a canonical orbifold
map $q_i\from \orbM_1\times\orbM_2\to\orbM_i$,
called a  \emph{projection,} whose underlying map is the product
projection from $\mani{\orbM_1} \times \mani{\orbM_2}$ to
$\mani{\orbM_i}$. Moreover, $\orbM_1\times\orbM_2$ is a product object
in the category of smooth orbifolds.

By replacing ``smooth'' by ``continuous'' in the definitions given
above, we obtain the notions of a  \emph{continuous \locallift} and a {\it
  continuous \locallift\ class}, and a  \emph{continuous orbifold
map}. 
We also obtain the notion of composition of continuous orbifold maps,
and we have a category of orbifolds with corners in which the morphisms are continuous
orbifold maps.

We adopt the convention that  ``\locallift s,''  
and  ``orbifold maps'' are understood to be
smooth except when ``continuous'' is specified. We shall also often
say ``map'' in place of ``orbifold map.''
\EndCustom

\Custom{Orbifold covering maps}
\label{Orbifold covers}
Let $\orbM$ and $\orbM'$ be $n$--orbifolds with $k$--corners, and let
$f\from\orbM'\to\orbM$ be an orbifold map with underlying map $F$. We
shall say that a chart 
for $\orbM$, with support $U$ and connected chart domain $\tU$, is {\it
  evenly covered} if  every component $U'$ of  $F^{-1}(U)$ is the
support of a chart 
for $\orbM'$ with chart domain $\tU$, 
such that
the identity map on $\tU$
is a \locallift\ for $f$.
We shall say that $f$ is a  \emph{covering map} if every point of
$\orbM$ has a connected open neighborhood which is the support of an
evenly covered chart.

If $\Gamma$ is a group of orbifold  diffeomorphisms of an orbifold 
with $k$--corners
$\orbM$, and the action of $\Gamma$ on $\mani\orbM$ is properly discontinuous,
then $\orbM/\Gamma$ inherits the structure of an orbifold
with $k$--corners, 
and the
quotient map from $\orbM$ to $\orbM/\Gamma$ is a
regular covering map. 
The orbifold $\orbM/\Gamma$ is a quotient object in the (smooth) category
defined in \S\ref{maps continued}, and its underlying space is
$\mani\orbM/\Gamma$.

Conversely, if $p \from \torbM\to\orbM$  is a
regular covering map of  orbifolds with corners, 
then there is a group
$G$ of orbifold diffeomorphisms of $\torbM$, acting properly
discontinuously on $\mani\torbM$, such that $p$
is the
composition of the quotient map $\torbM\to\torbM/G$ 
with an orbifold diffeomorphism from $\torbM/G$ to $\orbM$.

These
observations apply in particular to the situation where a finite
group acts by orbifold diffeomorphisms on a given orbifold with
corners; this corresponds to the case of a finite-sheeted regular covering.

An orbifold with corners $\orbM$ is said to be  \emph{good} if it has some covering space
which is nonsingular, 
i.e.\ is a manifold with corners; 
otherwise $\orbM$ is said to be  \emph{bad}. An orbifold with corners  $\orbM$ is termed  \emph{very good} if it has some finite-sheeted covering space
which is nonsingular. 
It 
will follow from the relationship between covering spaces and
fundamental groups, to be presented  in \S\ref{path subsection},
 that every finite-sheeted  cover of $\orbM$ has a further
cover $\torbM$ such that the composition $\torbM \to \orbM$ 
is
both finite-sheeted and regular. Consequently,  every very good orbifold with corners is the quotient of a manifold with corners by a finite group action. This is the setup described in Example~\ref{running example}.

For any
$n\ge2$, we define a
 \emph{hyperbolic $n$--orbifold} to be a quotient of $\HH^n$ by a
properly discontinous group of hyperbolic isometries. In particular,
a hyperbolic orbifold (or manifold) is by definition connected and
complete in its natural metric, and has no boundary.

It follows from Selberg's lemma \cite{Alperin:SelbergLemma} that every
hyperbolic orbifold $\HH^n/\Gamma$, where $\Gamma$  is a finitely
generated discrete group of isometries of $\HH^n$,
 is very good. 

The \emph{volume} $\vol \orbM$ of a hyperbolic $n$--orbifold $\orbM=\HH^n/\Gamma$ may be defined as the volume of an arbitrary (measurable) fundamental domain for $\Gamma$ in $\HH^n$, or equivalently as $(\vol\torbM)/d$, where $\torbM$ is an arbitrary finite-sheeted manifold covering of $\orbM$ and $d$ denotes the degree of the covering map. 

Up to equivalence, a connected, non-orientable orbifold $\orbM$
has a unique orientable
two-sheeted orbifold cover $\orbM'$; we shall call $\orbM'$ the {\it
  orientation cover} of $\orbM$. While the existence of $\orbM'$ is
not hard to prove by adapting the standard proof for the manifold
setting, we shall need only the case where $\orbM$ is good, and the
proof is particularly simple in this case: if $\torbM$ denotes the
universal covering of $\orbM$, which is a manifold, we may define
$\orbM'$ to be $\torbM/\Gamma$, where $\Gamma$ denotes the index--$2$
subgroup of the deck group of $\torbM$ consisting of
orientation-preserving diffeomorphisms.

For an arbitrary  good orbifold $\orbM$, which may be disconnected and
may have orientable components, we define the orientation cover to be
the two-sheeted cover which is the disjoint union of the orientation
covers of all non-orientable components of $\orbM$, together with two copies of
each orientable component of $\orbM$.
\EndCustom

\Custom{Suborbifolds}
\label{Suborbifolds}
Let $\orbM$ be an $m$--orbifold with 
corners,  and set $M=\mani\orbM$. 
Suppose that $n$ is an integer with $0\leq n \leq m$, and that $N$ is a subset of $M$ having the following property:
\Claim\label{defines a suborbifold}
For every $x\in M$ there is a chart $( U , \tU , \Phi,
G)$ for $\orbM$, with $x\in U$, such that $\Phi^{-1}(N)$ is a
(possibly empty) $n$--dimensional submanifold with corners of $\tU$.
\EndClaim
Then for every point $y\in N$, there exist a chart $( U , \tU , \Phi,
G)$ for $\orbM$, an integer $j$ with $0\le j\le n$, an open set $\tV\subset\RR^{n-j}\times\RR_+^j
$, and a diffeomorphism $h$ of $\tV$ onto a $G$--invariant open subset  of
$\Phi^{-1}(N)$ containing the set $\Phi^{-1}(y)$. By restricting the action of $G$ to $h(\tV)$ and pulling
back via $h$, we obtain a group $\barG$
of diffeomorphisms of $\tV$ which is isomorphic to a quotient of
$G$. 
The map $\omega\doteq\Phi\circ h\from\tV\to N$ then induces a
homeomorphism of $\tV/\barG$
onto an open neighborhood $V$ of $y$ in $N$.
The quadruple
$(V,\tV,\omega,\barG)$ is then an $(n,j)$--orbifold chart 
in $N$; the set
of all charts of this form is an orbifold atlas on $N$, and therefore
defines an $n$--orbifold with 
corners $\orbN$
whose underlying space is $N$. 

For any subset $N$
of $M$ such that \ref{defines a suborbifold} holds, the orbifold
$\orbN$
constructed in this way will be denoted $\orbi N$. To say that $\orbi
N$ is defined means that \ref{defines a suborbifold} holds. By
definition, we always have $\mani{\orbi N}=N$ if $\orbi N$ is defined.

A  \emph{suborbifold with corners of $\orbM$ in the weak sense}, or more briefly a
 \emph{weak suborbifold with corners of $\orbM$}, is by definition an
orbifold with corners of the
form $\orbi N$ for some set $N\subset M$ such that \ref{defines a suborbifold} holds. 

We observe that if $\orbN$  is a weak suborbifold with corners of  $\orbM$,  and $G$ is a group of diffeomorphisms of $\orbN$ whose underlying action on $\mani\orbM$ leaves $N\doteq\mani\orbN$ invariant, then the weak suborbifold $\orbi{N/G}$ of $\orbM/G$ is defined, and is canonically identified with the quotient $\orbN/G$.

In the case where $\orbM$ is an orbifold (with boundary), a
weak suborbifold $\orbN=\orbi N$ of  $\orbM$ is said to be  \emph{neatly
  embedded}  if a stronger version of \ref{defines a suborbifold}
holds, namely that for every $x\in M$ the chart $( U , \tU , \Phi,
G)$ may be chosen in such a way that the submanifold 
$\Phi^{-1}(N)$ is neatly embedded in $\tU$.

For an arbitrary $m$--orbifold with corners $\orbM$, and an arbitrary
subset $N$ of $M\doteq\mani\orbM$ satisfying \ref{defines a suborbifold} for a given $ n \leq m$,
  choosing the support $U$ to be sufficiently small ensures that
the groups denoted $G$ and $\barG$ in
the discussion above are canonically identified 
(up to inner automorphisms) 
with the local groups $G_y^\orbM$ and $G_y^\orbN$
respectively. Thus for every weak suborbifold with corners $\orbN$ of
an orbifold with corners $\orbM$, and every point
$y$ of $\orbN$, we have a surjective homomorphism  $\theta_y\from G_y^\orbM\to G_y^\orbN$ which is canonical up to composition with  inner
automorphisms of these finite groups.

  If  $\theta_y$ is an
isomorphism for every
$y\in \orbN$, we shall say that $\orbN$ is a  \emph{suborbifold of
  $\orbM$ in the strong sense.} We adopt the convention that
``suborbifold'' without qualification means ``suborbifold in the
strong sense.''

If $\orbN$ is a suborbifold of $\orbM$, then in particular $G_y^\orbN$
and $G_y^\orbM$ are isomorphic, and therefore
$\ord^\orbN(y)=\ord^\orbM(y)$,  for every $y\in\orbN$. Hence $\Sing_\orbN
= \orbN \cap \Sing_\orbM$. These facts will be used very frequently in
later sections. They need not remain true if $\orbN$ is assumed only
to be a weak suborbifold; compare the second example in \S\ref{Exa:BMPSuborbifold}.
\EndCustom

\Custom{Examples}
\label{Exa:BMPSuborbifold}\label{Exa:StrongVsWeak}
Here are two instructive examples, with the same ambient orbifold. Let $\orbM=\BB^2/C$, where $C$ is a cyclic group of rotations about $0\in \BB^2$ of some odd
order $p>1$.

For the first example, let $f \from \BB^2 \to \orbM$ be the projection, and let $g \from I \to \BB^2$ be the embedding of a diameter. Then 
according to our definitions, the image of $I$ in $M = \mani \orbM$
 is not the underlying set of even a weak suborbifold of $\orbM$, because it does not have Property \ref{defines a suborbifold}. However, $f \circ g$ is an orbifold immersion (see \S\ref{more orbifold notions} below) and its underlying map is injective; this means that the image of $I$ is  the underlying set of a ``suborbifold'' in the sense defined by Boileau--Maillot--Porti \cite[Section 2.1.3]{BoileauMaillotPorti}. 

 For the second example, let $o\in\orbM$ be the image of $0$ under the quotient map. Then the  weak suborbifold
 $\orbN\doteq\orbi{\{o\}}$ of $\orbM$ is defined but is not a
 suborbifold in the strong sense. Indeed, we have
$\ord^\orbN(o)=1\ne p=\ord^\orbM(o)$,  and $\Sing_\orbN
=\emptyset\ne\{o\}= \orbN \cap \Sing_\orbM$.
\EndCustom

\Custom{Inclusion maps}\label{inclusions}
Let  $\orbN$ be a  suborbifold with corners of an orbifold with corners $\orbM$. Set $M=\mani\orbM$ and $N=\mani\orbN$, and let $\kappa \from N\to M$ denote the inclusion map.  It follows from the discussion in \S\ref{Suborbifolds} that for every point $y\in N$ there exist charts $( U , \tU , \Phi,
G)$ and
$(V,\tV,\omega,\barG)$ for $\orbM$ and $\orbN$ respectively such that the following condition holds:
\Claim\label{for inclusion}
We have $y\in V\subset U$, and if $\tkappa\from\tV\to\tU$ denotes the inclusion map, then $(\tkappa,\theta_y^{-1})$ is a local lift of $\kappa$,
where $\theta_y^{-1} \from G_y^\orbN \to G_y^\orbM$ is the inverse of the isomorphism $\theta_y$ defined in
\S\ref{Suborbifolds}.
\EndClaim

Each point $y\in N$ and each pair of charts $( U , \tU , \Phi,
G)$ and
$(V,\tV,\omega,\barG)$ for which \ref{for inclusion} holds define a local lift of $\kappa$. The set of all local lifts defined in this way is an atlas of local lifts of $\kappa$, and therefore define an orbifold map from $\orbN$ to $\orbM$ with underlying map $\kappa$. This orbifold map will be called the \emph{inclusion map} $\orbN \hookrightarrow \orbM$. We emphasize that we have defined the inclusion map only when $\orbN$ is a suborbifold with corners in the strong sense: if $\orbN$ were only a suborbifold with corners in the weak sense, there would appear to be no naturally defined homomorphism to play the role of $\theta_y^{-1}$ in \ref{for inclusion}. 

If $\orbM'$ and $\orbM$ are orbifolds with corners, we
define an  \emph{embedding} of  $\orbM'$ in $\orbM$ to be a
smooth map from $\orbM'$ to $\orbM$ which is the composition of a
diffeomorphism of $\orbM'$ onto a suborbifold with corners $\orbN$ of
$\orbM$ with the inclusion map from $\orbN$ to $\orbM$.

If 
$f$ is an arbitrary smooth map from $\orbM$ to an orbifold with corners
$\orbN$, 
and $\orbZ$ is
any  suborbifold with corners of $\orbM$, we define $f|\orbZ\from \orbZ\to\orbN$ to be
the composition of the inclusion map $\orbZ\hookrightarrow\orbM$ with $f$.
\EndCustom

The following result will be used extensively in later sections  to recognize 
 suborbifolds (in the strong sense).

\Proposition\label{codim 1 is nice}
Suppose that $\orbN$ is a weak suborbifold with
corners  of an 
orbifold with corners $\orbM$. Suppose that
at least one of the following conditions holds:
\begin{enumerate}[\:\:(i)]
\item\label{Hyp:Codim0} $\orbN$ has codimension $0$ in $\orbM$, or
\item\label{Hyp:Codim1Bdy} $\orbN$ has codimension $1$ in $\orbM$ and $\mani\orbN$ is disjoint from
  $\mani{\inter\orbM}$, or
\item\label{Hyp:Codim1Ori} $\orbN$ has codimension $1$ in $\orbM$ and $\orbM$ is
  orientable.
\end{enumerate}
Then $\orbN$ is a  suborbifold with
corners of $\orbM$ (in the strong sense).
\EndProposition

\Proof
  According to the definition given in \S\ref{Suborbifolds}, we must
  show that the
  surjective homomorphism  $\theta_y\from G_y^\orbM\to G_y^\orbN$ has
  trivial kernel for every
$y\in \orbN$.   It follows from the discussion in
\S\ref{Suborbifolds} that we may fix a chart $( U , \tU , \Phi,
G)$ for $\orbM$, with $y \in U$, such that $G_y^\orbM$ is identified
with $G$, while $G_y^\orbN$ is identified
with the group $\barG$ of restrictions of elements of $G$ to 
a $G$--invariant open subset $\tV$ of $\Phi^{-1}(N)$, 
and $\theta_y$ is the
restriction homomorphism from $G$ to $\barG$. 
We may take $\tU$ to be connected. To show that
$\theta_y$ has trivial kernel, we must show that any element $g$ of $G$
which fixes $\tV $ pointwise is the identity. According to Corollary~\ref{Bochner corollary}
we need only show that $g$
fixes pointwise some non-empty open subset of $\tU$. This is
immediate  if hypothesis \eqref{Hyp:Codim0}  holds.

If hypothesis \eqref{Hyp:Codim1Bdy} holds, then $\tV $ is an $(n-1)$--manifold with
corners contained in the $n$--manifold with corners $\tU$, and contains
no interior points of $\tU$. Hence $\tV $ contains a $1$--corner point
$\tw$ of $\tU$. We may regard $\tw$ as a boundary point of the $G$--invariant manifold  $\tU'$ consisting of all $0$-- and $1$--corner points of $\tU$. By Proposition \ref{general bochner},
applied with $\langle g\rangle$ and
$\tU'$ playing the respective roles of $G$ and $M$, some $\langle g\rangle$--invariant neighborhood 
$Y $ of $\tw$ in $\tU'$ can be identified diffeomorphically with 
the open half-ball $(\inter \BB^n)\cap(\RR^{n-1}\times\RR_+)$
in 
such a way that $\tw = 0$, 
and the action of $\langle g\rangle$ factors through the standard action of 
 $\bigO(n-1)$. We may suppose $ Y  $ to have been chosen small enough so
that $\bdy  Y  = Y  \cap\bdy\tU'\subset\tV $. Since the standard
action of $\bigO(n-1)$ on the boundary of the open half-ball $ Y  $ is effective, and
$g$ fixes $\tV $ pointwise, it follows that $g$ also fixes $ Y  $
pointwise; thus $ Y  $ is the required open subset of $\tU$ in this case.

If hypothesis \eqref{Hyp:Codim1Ori} holds but hypothesis \eqref{Hyp:Codim1Bdy} does not, we may
choose a point $z\in\tV \cap\inter\tU$.  We choose a $\langle
g\rangle$--invariant neighborhood $T$ of $z$ in $\inter\tU$ which is
small enough to guarantee that $\tV \cap T$ is closed in the
subspace topology of $T$. If $\tV _0$ denotes the component of
$\tV \cap T$ containing $z$, we may then write $ T=T_1\cup T_2$, where
the $T_i$ are $n$--manifolds with boundary and $T_1\cap T_2=\bdy
T_1=\bdy T_2=\tV _0$. Now $g$ fixes $\tV _0\subset\tV $ pointwise,
and it preserves the orientation of $T$ since 
$\orbM$ 
is orientable
in this case. Hence $g$ cannot interchange $T_1$ and $T_2$, and it
must therefore leave each of them invariant. We can now finish the
proof exactly as in Case \eqref{Hyp:Codim1Bdy}, with $T_1$ and  $z\in\bdy T_1$ playing the
respective roles of $\tU'$ and $\tw \in \bdy\tU'$.
\EndProof

In  \S\S\ref{suborbs continued}--\ref{more orbifold notions},
$\orbM$ will denote an $n$--orbifold with
corners, and we shall set $M=\mani\orbM$.

\Custom{Suborbifolds, continued}
\label{suborbs continued}
If $U$ is an arbitrary open
subset of $M$, then 
$\orbi U$ is defined, and 
by Proposition \ref{codim 1 is nice} it
 is an $n$--dimensional suborbifold with corners (in the strong sense).
We shall call $\orbi U$ an  \emph{open
  suborbifold with corners} of $\orbM$. 

In particular, 
if $( U ,
\tU , \Phi, G)$, is any chart for an orbifold 
with corners
$\orbM$, 
there is an open suborbifold 
with corners $\orbU\doteq\orbi U$ of $\orbM$. 
The map 
$\Phi$ defines a
finite-sheeted regular orbifold covering map from $\tU$ to the open
suborbifold 
with corners
$\orbi U $ of $\orbM$. Furthermore, if 
$( U , D, \Phi, G)$ is a standard chart, then $D$ is the
universal cover of $\orbU = \orbi U $ and,
up to equivalence, is the unique connected manifold cover of $\orbU$.

For another example, the set of interior points of $\orbM$ defines an open suborbifold, called the \emph{interior of $\orbM$} and denoted $\inter \orbM$.

An  \emph{open neighborhood} of a point $x$
of $\orbM$ is an open suborbifold with corners having $x$ as a point. 

A  \emph{standard neighborhood} of a point 
$x$ of $\orbM$ is
an open neighborhood of the form $\orbU=\orbi U$, where $U$ is the support
of a standard chart  centered at $x$.

Any component $K$  of $M$ is an open subset of $\mani\orbM$, and
hence $\orbi K$ is defined. We call $\orbi K$
a  \emph{component} of $\orbM$. Note
that $\orbM$ is connected if and only if it has exactly one component.

We define a {\it
  suborbifold} or  \emph{weak suborbifold} of $\orbM$ to be,
respectively a suborbifold with corners or a weak suborbifold with
corners which is itself an orbifold. 
Note that we use this terminology even when $\orbM$ is not itself an
orbifold (that is, $\orbM$ has $k$--corner points for $k > 1$). 
This slight abuse of language is analogous
to speaking, for example, of a subfield of a ring. Likewise we may
refer to a ``submanifold'' (with or without corners) of the
general orbifold $\orbM$, meaning a suborbifold of  $\orbM$ which is itself a manifold.

If  $Z$ denotes the set of $1$--corner points of 
an orbifold $\orbM$, then $\orbi Z$
is defined. We shall refer to $\orbi Z$ as the
 \emph{boundary} of $\orbM$ and denote it $\bdy\orbM$.
It follows from Proposition \ref{codim 1 is nice} that $\bdy\orbM$
is a suborbifold (in the strong sense) of $\orbM$.

An  \emph{immersion} from $\orbM$ to an orbifold with corners $\orbM'$
is defined to be a smooth map  $f\from \orbM\to  \orbM'$ with the property that
every point of $\orbM$ has a neighborhood $\orbU$ such that
$f|\orbU\from \orbU\to\orbM'$ is an embedding.
\EndCustom

\Custom{Weak suborbifolds that are manifolds}
\label{when it's a manifold}
We will often encounter situations in which a 
subset $Y$ of $M$ has the property that $\orbi Y$ is defined and
is a manifold with corners (and possibly a manifold).
We  will then write
$Y=\orbi Y$, and refer to $Y$ as a manifold with corners (or a manifold). This is consistent with
the universal convention in manifold theory whereby a manifold is not
distinguished notationally from its underlying set or topological
space. 

In particular, if $x$ is any point of $\orbM$ then 
$\orbi{\{x\}}$ is  defined, and is a
$0$--manifold. 
\EndCustom

\Custom{Notational shortcuts}
\label{more orbifold notions}

An orbifold 
with corners
$\orbM$ is said to be  \emph{compact} or  \emph{connected} if
$\mani\orbM$ is, respectively, compact or connected.

A good many set-theoretic or general-topological operations on subsets
of $M$ can be interpreted as operations on suborbifolds with corners, in
situations where certain conditions hold. For example, if $\orbY$ and $\orbZ$
are suborbifolds with corners of $\orbM$, if we set
$Y=\mani\orbY$, $Z=\mani \orbZ$, and if $\orbi{Y\cap Z}$ is defined, we will 
denote $\orbi{Y\cap Z}$ by 
$\orbY\cap\orbZ$. 
We follow similar conventions regarding the intersection or union of
any  family of suborbifolds, for inclusion 
relations between
suborbifolds, and for set-theoretic differences of suborbifolds. 
(The intersection and union of any finite family of open suborbifolds is always defined.)

If $\orbY$ is a suborbifold with corners of  $\orbM$, if we set $Y=\mani\orbY$, and
if $\orbi
{\overline Y}$ is defined, we denote $\orbi
{\overline Y}$ by $\overline{\orbY}$ and refer to it as the {\it
  closure} of $\orbY$; we use a similar convention for frontiers of suborbifolds.

If $f$ is a smooth map of an orbifold with corners $\orbN$ into $\orbM$  defined
by a map $F\from \mani\orbN\to M$, if $\orbY$ is a suborbifold with
corners of $\orbN$, if we set $Y=\mani\orbY$, and if
$\orbi {F(Y)}$ is
defined, we will denote $\orbi {F(Y)}$ by $f(\orbY)$. 
If $f$ is an embedding of $\orbN$ in $\orbM$, then 
$f(\orbY)$
  is defined for any suborbifold $\orbY$ of $\orbN$; in particular the  \emph{image} $f(\orbM)$ of
$f$ is defined in this case. We follow a similar convention for
preimages. The preimage of an open suborbifold under an orbifold map
is always defined. 

We shall use a number of  locutions that generalize familiar locutions for
manifolds, and are self-explanatory; for example, to say that a map
$f$ from $\orbM$ to an orbifold with  corners $\orbN$ 
maps a suborbifold 
with corners $\orbY$ of $\orbM$
diffeomorphically onto a suborbifold with corners $\orbZ$ of $\orbN$ means that
$f(\orbY)=\orbZ$ and that $f|\orbY\from \orbY\to\orbZ$ is a diffeomorphism.
\EndCustom

\Custom{Straightening angles}\label{what's sta}
Let $\orbM$ be an arbitrary  $n$--orbifold with nonsingular
$2$--corners. 
Set $M=\mani\orbM$, and let $K\subset M$ denote the set of all
$2$--corner points of $\orbM$. If we set $V\doteq M-\Sing_\orbM\supset K$, then
with the conventions of \S\ref{when it's a manifold}, $V=\orbi V$ is
an open
submanifold with corners of $\orbM$. On the other hand,
$\orbW\doteq\orbi{M-K}$ is an open suborbifold (with boundary) of
$\orbM$.

Let $V'$ denote the manifold
(with boundary) obtained from $V$ 
by the standard operation of \sta\ 
(see Douady and H\'erault's appendix to \cite{borel-serre}). 
Then $V'$ has the same underlying set as $V$, and the smooth manifold
structure of $V'$ --- which may in particular be regarded as a (smooth)
orbifold structure --- agrees with the orbifold structure of 
$\orbW$ on
the open subset $(M-K)\cap V$ of $M$. Hence there is a unique 
orbifold structure on $M$ which restricts to the orbifold-with-corners
structure of $\orbW$ 
and to the smooth manifold structure of $V'$. The
resulting orbifold (with boundary)
will be said to be  \emph{obtained from $\orbM$ by \sta.} 
\EndCustom

\Custom{Strata}[\rm{cf.\ Choi~\cite[Section 4.5]{Choi:Orbifolds}}]
\label{strata def}
Let $\orbM$ be an 
orbifold with corners. 
For each
positive integer 
$p$, let $Z_p(\orbM)$ denote the subset of $\mani\orbM$ consisting of
all points whose order is $p$. 
We define a  \emph{stratum} of $\orbM$ to be a
subset of $\mani\orbM$ which is a connected component of $Z_p$ for some $p >
0$.

It follows from this definition that the strata of $\orbM$ form a
partition of $M\doteq\mani\orbM$.

 Let $( U , \maybeD, \Phi, G)$ be any standard chart
  centered at an arbitrary 
point $x\in\orbM$. Thus for some $j\le n$ we have
$D=(\inter \BB^n) \cap (\RR^{n-j}\times\RR_+^j)$ and $G\le\bigO(n-j)\times\scrS_j$.
For each subgroup $H$ of
  $G$, let $R_H\subset\maybeD$ denote the set of all points whose
  stabilizer in $G$ is precisely $H$.
Thus if 
$W_L\subset \RR^{n-j}\times\RR_+^j$  denotes the fixed set
of a subgroup
$L$ of $G$---which is the intersection of $\RR^{n-j}\times\RR_+^j$ with
a linear subspace of $\RR^n$---we have 
\Equation\label{what RH looks like}
R_H=D\cap \Big( W_H-\bigcup_{H< L\le G}W_L \Big).
\EndEquation

If $y$ is any point of $D$, and we set $H=\Stab_G(y)$ and $p=|H|$, 
then $y$ has a neighborhood $V\subset D$
such that $\Stab_G(z) \leq H$ for every $z\in V$. Hence for any $z\in
V$, we have $| \Stab_G(z) | =p$ if and only if $\Stab_G(z)=H$. This means
that 
\Equation\label{why was six}
V\cap\Phi^{-1}(Z_p)=V\cap R_H.
\EndEquation
 It follows 
 from \ref{why was six} 
that:

\Claim\label{was six}
A subset of $D$ is a component of $R_H$ for some $H\le G$
if and only if it is a component of $\Phi^{-1}(\sigma)$ for some
  stratum $\sigma$ of $\orbM$.
Hence a subset of $U$ is a component of $\Phi(R_H)$ for some $H\le G$
if and only if it is a component of $U\cap\sigma$ for some
  stratum $\sigma$ of $\orbM$.
\EndClaim

It follows from \ref{was six} and the definition of the sets $R_H$
that for any stratum $\sigma$ of $\orbM$, all points of
$\Phi^{-1}(\sigma)$ have the same stabilizer in $G$. As this holds for
 every standard chart $( U , \maybeD, \Phi, G)$  for $\orbM$, we deduce:
\Claim\label{was three}
For any stratum $\sigma$ of $\orbM$, the local groups of all points of
$\sigma$, regarded as subgroups of $\bigO(n)$, all coincide up to
conjugacy. Thus $\sigma$ has a  \emph{local group} which is a subgroup
of $\bigO(n)$, well-defined up to conjugacy. It will be denoted by $G_\sigma$.
\EndClaim

Again considering an arbitrary  stratum  $\sigma$ of $\orbM$ and an arbitrary standard chart
$( U , \maybeD, \Phi, G)$, 
observe that \eqref{what RH looks like} implies that for each
subgroup $H$ of $G$, the set $R_H$ is a manifold with corners, and
that for any $k\le n$, the set of $k$--corner points of $R_H$ is the
intersection of $R_H$ with the set of $k$--corner points of $D$. 
Note that $H$ is the largest subgroup of $G$ that acts trivially on $R_H$.
Hence, if we set $E=\Stab_G(R_H)$, 
then
$H$ is a normal subgroup of $E$, and the action of $E$ induces an
action of $E/H$ on $R_H$. This action is free (since $\Stab_G(y)=H$ for every 
$y\in R_H$), and thus the quotient is a manifold. 
This proves:

\Claim\label{was too}
For each stratum $\sigma$ of $\orbM$, 
the weak
suborbifold $\orbi\sigma$ of  $\orbM$
is defined, and
is a (smooth)
manifold
with corners. Thus according to the convention of 
\S\ref{when it's a manifold} 
we
may write $\sigma=\orbi\sigma$ and regard $\sigma$ as a
manifold
with corners. 
Furthermore, for each $k \leq n$, the set of $k$--corner points of $\sigma$ is the
intersection of $\sigma$ with the set of $k$--corner points of $\orbM$. 

In particular,  each stratum of
  $\orbM$ has a well-defined  \emph{dimension}. A stratum whose
  dimension is a given integer $d$ 
may be called a
  \emph{$d$--stratum}. 
\EndClaim

The proof of \ref{was too} shows that for 
any  stratum  $\sigma$, and any standard chart
$( U , \maybeD, \Phi, G)$ centered at a point of $\sigma$, we have
$\dim\sigma=\dim R_G$. In particular: 
\Claim\label{was four}
A stratum $\sigma$ of $\orbM$ has dimension $n$ if and only if
  $G_\sigma=\{1\}$. 
Hence
 $\Sing_\orbM$ is the union of all
  strata of $\orbM$ having dimension strictly less than $n$.
These will be called \emph{singular strata}.
\EndClaim

Finally, we observe, as a consequence of \eqref{what RH looks like}, that 
for 
\delete{any  stratum  $\sigma$, and} any standard chart
$( U , \maybeD, \Phi, G)$, the collection of all components of the sets 
$R_H$ with $H \le G$ is
finite; and that for each set $J$ in this collection, $\overline{J}-J$
  is a union of sets in the collection having lower dimension than
  $J$. Hence:
\Claim\label{missing piece}
The set of all strata of $\orbM$ is locally finite, and for each
stratum $\sigma$, the set $\overline{\sigma}-\sigma$
  is a union of strata having lower dimension than
  $\sigma$.
\EndClaim
\EndCustom

\Custom{Orbifold fundamental groups}
\label{path subsection}
The treatment of orbifold fundamental groups in this subsection is
inspired by the material in Boileau, Maillot, and Porti~\cite[Section
2.2]{BoileauMaillotPorti}. However, we have had to reorganize the
material here, and coordinate it with the material in \S\S\ref{maps and such} and \ref{maps continued}, in order to guarantee
functoriality of the orbifold fundamental group, which is used strongly
in the later sections of this paper.

The unit interval $I$ is a $1$--manifold, and in particular a
$1$--orbifold. We define an  \emph{orbifold path} in an orbifold with
corners $\orbM$ to be a continuous orbifold map (see \S\ref{maps and such}) from
$I$ to $\orbM$. The underlying map of an orbifold path $\alpha$ is a
path $A$ in the topological space $\mani\orbM$; the initial and
terminal points of $A$ will be called the initial and
terminal points of $\alpha$. The  \emph{image} of the orbifold path
$\alpha$ is defined to be $A(I)$.

We emphasize that in general an orbifold path is not uniquely
determined by its underlying map. Indeed, Example~\ref{UnderlyingMapNotEnough}
 can easily be adapted to illustrate this.

Suppose that $\alpha$ and  $\beta$ are orbifold paths in $\orbM$, that
the terminal point of $\alpha$ is a nonsingular point $c\in\orbM$,
and that $c$ is the initial point of $\beta$. Let $u$ and $v$ denote
the unique orbifold self-maps of $I$ whose underlying maps are given
by $t\mapsto t/2$ and $t\mapsto (t+1)/2$ respectively. Then there is a
unique orbifold path $\alpha\cdot\beta$ in $\orbM$,
called the  \emph{concatenation} of $\alpha$ and $\beta$,
 such that
$(\alpha\cdot\beta)\circ u=\alpha$ and $(\alpha\cdot\beta)\circ
v=\beta$ (where composition of orbifold maps is defined as in
\S\ref{maps continued}). 
The uniqueness assertion depends on the assumption that $c$ is
nonsingular.

An  \emph{orbifold homotopy} between orbifold paths $\alpha_0$ and
$\alpha_1$ in $\orbM$ is a continuous orbifold map $h\from I\times
I\to\orbM$ such that $h\circ q_i=\alpha_i$ for $i=0,1$, and $h\circ
r_i$ is a constant orbifold path for $i=0,1$, where $q_i,r_i\from I\to
I\times I$ are the orbifold maps whose underlying maps are given
respectively by $q_i(t)=(t,i)$ and $r_i(t)=(i,t)$.

If $h$ is an orbifold homotopy between orbifold paths $\alpha$ and
$\alpha'$, and $h'$ is an orbifold
 homotopy between $\alpha'$ and another orbifold path $\alpha''$, then
 there exists an orbifold homotopy $H$ between $\alpha$ and
 $\alpha''$ such that the equalities of orbifold maps $H\circ (\id\times u)=h$ and $H\circ (\id\times
 v)=h'$ hold, where $u$ and $v$ are the maps defined above. (The
 homotopy $H$ is not in general unique, although its underlying map is
 determined by those of $h$ and $h'$. The construction of the maximal
 atlas of local lifts defining $H$ requires a little care.) This shows
 that homotopy of orbifold paths is an equivalence relation.

Suppose an orbifold with corners $\orbM$ is equipped with a nonsingular basepoint
$\star$. If $\alpha$ and $\beta$ are orbifold loops based at $\star$,
i.e.\ orbifold paths having $\star$ as both initial and terminal point,
the orbifold homotopy class of $\alpha\cdot\beta$ depends only on those of
$\alpha$ and $\beta$. Thus the set 
$\pi_1(\orbM, \star)$ of all orbifold homotopy classes of orbifold loops based at
$\star$ inherits a binary operation. With respect to this operation, $
\pi_1(\orbM,\star)$ is a group. 

Now suppose that $f$ is a continuous orbifold map from an orbifold
with corners $\orbM$ to an orbifold
with corners  $\orbM'$.
Then for any orbifold path $\alpha$ in
$\orbM$, 
the composition of orbifold maps
$f\circ\alpha$ is  an orbifold path in
$\orbM'$.  Furthermore, if $h$ is an orbifold homotopy between orbifold paths
$\alpha$ and $\alpha'$ in $\orbM$, then $f\circ h$ is a homotopy
between $f\circ\alpha$ and $f\circ\alpha'$. Using these observations
and the definition of the group operation in the orbifold fundamental
group, one deduces that if $\orbM$ and $\orbM'$ are equipped with nonsingular basepoints
$\star$ and $\star'$, and if $f$
carries $\star$ to $\star'$, then $f$ induces a homomorphism
$f_* \from\pi_1(\orbM,\star)\to \pi_1(\orbM',\star')$. 
This makes
$\pi_1$ a functor from the category of orbifolds with corners,
equipped with nonsingular basepoints (and with continuous basepoint-preserving orbifold
maps as morphisms), to the category of groups.

If an orbifold with corners $\orbM$ is connected, then for any two points $x,y\in\orbM$ there
is a path with initial point $x$ and terminal point $y$. If $\alpha$
is any such orbifold path, there is a well-defined isomorphism from
$\pi_1(\orbM,x)$ to $\pi_1(\orbM,y)$ given by
$[\gamma] \mapsto [\baralpha\cdot\gamma\cdot\alpha]$. From this one
deduces that if $f\from\orbM\to\orbM'$ is an 
orbifold map between connected orbifolds with corners, replacing one
consistent pair of nonsingular basepoints for $\orbM$ and $\orbM'$ by another has
the effect of precomposing and postcomposing $f_*$ by group
isomorphisms. In particular, it makes sense to say that 
$f_*\from\pi_1(\orbM)\to \pi_1(\orbM')$ is injective, or
surjective, without specifying basepoints.

These observations apply when $\orbM$ is a connected suborbifold with corners of a connected orbifold
with corners
$\orbM'$ and $f$ is the inclusion map. We shall say that
$\orbM$
is \emph{$\pi_1$--injective} in $\orbM'$ if the inclusion homomorphism 
$f_* \from\pi_1(\orbM)\to \pi_1(\orbM')$ is injective. 
An arbitrary suborbifold $\orbY$ of an arbitrary orbifold $\orbM$ will
be termed \emph{$\pi_1$--injective} in $\orbM$ if each component of
$\orbY$ is $\pi_1$--injective in the component of $\orbM$ containing it.

An orbifold path $\alpha$ in an orbifold with corners $\orbM$ will be
termed  \emph{generic} if $\alpha^{-1}(\Sing_\orbM)$ is a finite subset
  of $(0,1)\subset[0,1]$. The following fact is crucial in relating
  orbifold fundamental groups to orbifold covering spaces:

\Claim\label{crucial fact}
Suppose that $f\from\orbM'\to\orbM$ is an orbifold covering, and that
$\alpha$ is a generic orbifold path in $\orbM$. Let $\star$ denote the
initial point of $\alpha$, and let $\star'$ be a 
point of $f^{-1}(\star) \subset \orbM'$.
Then there is a unique orbifold
path $\alpha'$ in $\orbM'$ which is a lift of $\alpha$ (in the sense
that $f\circ\alpha'=\alpha$) and has initial point $\star'$.
\EndClaim

To prove \ref{crucial fact}, one writes $\alpha$ as a concatenation of orbifold
paths $\alpha_1,\ldots,\alpha_m$ where each $\alpha_k$ has
nonsingular  initial and terminal points $i_k$ and $t_k$; we have
$i_{k+1}=t_k$ whenever $1\le k<m$,  and the  image of each $\alpha_k$
is contained in  the support of an
evenly covered chart $( U_k ,
\tU_k , \Phi_k, G_k)$. If $F$ denotes the underlying map of $f$,
and $U_1'$ denotes the component of $F^{-1}(U_1)$ containing $\star'$, it follows from the definitions that
$\alpha_1$ admits a lift to $\orbM'$ whose image is contained in
$U_1'$, and that $G_1$ acts transitively on all such lifts. Since
$\star$ is nonsingular, there is a unique lift $\alpha_1'$ of $\alpha_1$ to
$\orbM'$ having initial point $\star'$. The same argument shows that
if $1\le k<m$, and if $\alpha_k'$ is a lift of $\alpha_k$, there is a
unique lift of $\alpha_{k+1}$ whose initial point is the terminal point
of $\alpha_k'$. Now \ref{crucial fact} follows by induction.

The usefulness of \ref{crucial fact} arises from the following fact,
which 
in turn follows from Assertions \ref{was six} and \ref{what RH looks like} in the discussion of strata:

\Claim\label{homotop to generic}
Any path in an orbifold with corners, whose initial and terminal
points are nonsingular, is homotopic to a generic path.
\EndClaim

The following fact is easier to verify than \ref{crucial fact},
because there is no uniqueness assertion, and the step involving the
group action is therefore not needed:

\Claim\label{lifting homotopies}
Suppose that $f\from\orbM'\to\orbM$ is an orbifold covering, and that $h$
is an orbifold homotopy between orbifold paths $\alpha_0$ and
$\alpha_1$ in $\orbM$. If $\alpha_0'$ is a lift of $\alpha_0$ to
$\orbM'$, then $h$ admits a lift which is a homotopy from $\alpha_0'$
to some lift of $\alpha_1$.
\EndClaim

Using \ref{crucial fact}, \ref{homotop to generic}, and \ref{lifting homotopies}, it is now quite straightforward to establish the
direct generalization of the familiar facts relating
manifold covers and fundamental groups to the context of orbifold covers of connected orbifolds with
corners, equipped with nonsingular basepoints. These include the lifting criterion (for a
basepoint-preserving map between orbifolds with corners equipped with nonsingular
basepoints) and the correspondence between connected covers of an orbifold with
corners and subgroups of its fundamental group. In particular, $\orbM$ has a universal cover $\torbM$ 
that is simply connected, and there is an isomorphism
between the orbifold fundamental group $\pi_1(\orbM)$ and the deck group of  $\torbM $. Compare \cite[Theorems 2.1 and 2.2]{BoileauMaillotPorti}.
\EndCustom

\Custom{Size and weight of a subset}
\label{size 'n' stuff}
Let $\orbM$ be an orbifold
with corners, 
and let $Z$ be a subset of
$\mani\orbM$. We define the  \emph{size} of $Z$, denoted
 $\size(Z)$, to be $\card(Z\cap\Sing_\orbM)$; note that $\size(Z)$ may
 be an infinite or a finite cardinal. In the case where 
$\size(Z)$ is
 finite,  we define the  \emph{weight} of $Z$, denoted $\weighttwo(Z)$,
 to be 
$\sum_{x\in Z\cap\Sing_\orbM}\ord(x)$.
\EndCustom

\section{Low-dimensional orbifolds}\label{low-dim sec}

This section develops some key definitions and structural results about orbifolds of dimension $n \leq 3$. 
This dimension assumption makes some of the foundational results 
of Section \ref{orbisection} considerably more concrete.
In particular, Propositions
\ref{Prop:LocalStructure2Orbifold} and \ref{Prop:LocalStructure3Orbifold} describe the strata of the singular locus of a $2$--orbifold and $3$--orbifold, respectively. Notably, \S\ref{wuzzaccordant} introduces the key idea of an \emph{$\orbM$--accordant} smooth structure on the underlying space $\mani \orbM$. Several definitions, lemmas, and propositions in this section provide orbifold analogues of standard $3$--manifold notions such as essential surfaces, irreducibility, and parallelism. 

Finally, Proposition~\ref{cause gromov} and Corollary~\ref{Cor:SimpleVolBound} establish a key connection between the volume of a hyperbolic $3$--orbifold $\orbM$ and the Gromov norm of its underlying space $M \doteq \mani \orbM$. These results imply Proposition~\ref{Prop:IntroGromov} in the introduction and provide a key ingredient for the proof of the main theorems.

\Custom{Orbifolds of dimension $1$}
\label{folded}
Any compact, connected $1$--orbifold is either a manifold
diffeomorphic to $\SS^1$ or $I$, or is diffeomorphic to the quotient of
$\SS^1$ or $I$ by a reflection $\tau$. The orbifold $\SS^1/\tau$ has one
$1$--stratum and two $0$--strata, whereas $I/\tau$ has one
$1$--stratum and one $0$--stratum. We shall refer to an orbifold
diffeomorphic to $I/\tau$ as a \emph{folded \arc.}
\EndCustom

By combining 
Assertions \ref{was six} and \ref{what RH looks like}
with the classification (up to conjugacy) of finite subgroups of
$\bigO(2)$, one obtains 
the following result.

\begin{proposition}[\rm{cf.\ Cooper, Hodgson, and Kerckhoff \cite[Theorem 2.3 and the ensuing
discussion]{CHK:OrbifoldBook}}]
\label{Prop:LocalStructure2Orbifold}
For any stratum  $\sigma$  of  a $2$--orbifold  $\orbS$,
exactly one of the following alternatives holds.
\begin{enumerate}[\:\:$(1)$]
\item  
$\sigma\cap\Sing_\orbS=\emptyset$ (so that $\dim\sigma=2$ and
  $G_\sigma=\{1\}$ by 
Assertion \ref{was four}).
In this case, for any
  standard neighborhood $\orbU$
 of a point $x$ of $\sigma$, the set
  $U\doteq\mani\orbU$ is disjoint from $\Sing_\orbS$. If
  $x\in\inter\orbS$ then $U$ is an open $2$--disk. If
  $x\in\bdy\orbS$ then $U$ is homeomorphic to an open half-ball in
$\RR\times\RR_+$.
\item   $\sigma$ has dimension $0$ and is contained in $\inter\orbS$, and $G_\sigma$ is a cyclic group of rotations
  of some order 
$p > 1$. 
Furthermore, $\sigma$ is a full component of
  $\Sing_\orbS$, and if $\orbU$ is a standard neighborhood of the unique point $x$
  of $\sigma$ then $\mani\orbU$ is 
homeomorphic to
an open disk.
In this case, $x$ is called a \emph{cone point} of
  order 
$p$. 
\item $\sigma$ has dimension $1$, and $G_\sigma $ is an order-two cyclic group
  generated  by a reflection. 
For every point $x$ of $\sigma$ and every standard neighborhood $\orbU$ 
  of $x$ in $\orbS$, the intersection of $U\doteq\mani\orbU$ with
  $\Sing_\orbS$ is contained in $\sigma$. If $x\in\inter\orbS$, there
  is a homeomorphism of $U$ onto an open half-ball in $\RR\times\RR_+$, mapping 
  $U\cap\sigma$ to the boundary of the half-ball.
If $x\in\bdy\orbS$, there
  is a homeomorphism of $U$ onto 
an open quarter-ball 
  in $\RR_+ \times \RR_+$, mapping
  $x$,  $U\cap\sigma$ and $U\cap\mani{\bdy\orbS}$ into
  $(0,0)$,
  $\RR_+ \times\{0\}$ and $\{0\}\times\RR_+$ respectively.
In this case, $\sigma$ is called a \emph{mirror
    stratum}, and its points are called \emph{mirror points}.
\item $\sigma$ has dimension $0$ and is contained in $\inter\orbS$,
  and $G_\sigma$ is a dihedral group of order 
$2p$ for some $p > 1$,
  generated by reflections about two lines that meet at an angle of
  $\pi/p$. 
Furthermore, if $\orbU$ is a standard neighborhood of the
  unique point $x$ of $\sigma$, then $U\doteq\mani\orbU$ is 
homeomorphic to
an open quarter-ball 
in $\RR_+ \times \RR_+$, 
 with $x$ mapping to $(0,0)$, and 
 the two
  components of $\bdy U-\{x\}$ mapping to intervals in 
   $\RR_+ \times\{0\}$ and $\{0\}\times\RR_+$ and contained in mirror strata.
    In this case,  $x$ is called a \emph{bi-mirror point} of order
$2p$.
\end{enumerate}

We shall denote by $\mirror \orbS\subset\Sing_\orbS$  the union of all
mirror and bi-mirror strata of $\orbS$. It follows from the assertions
above that $\mirror \orbS$ may be described as the closure of
the union of  all mirror strata of $\orbS$, or as the
set of all points whose local group contains an orientation-reversing
element.  It also follows from the assertions above that $\mani\orbS$ is a topological
$2$--manifold, and that $\bdy \mani\orbS=\mani{\bdy\orbS}\cup
\mirror\orbS$. In particular, if $\orbS$ is orientable then  $\bdy
\mani\orbS=\mani{\bdy\orbS}$.
\NoProof
\end{proposition}

\Custom{\Reflective\ coverings}
\label{reflective}
Let $\orbS$ be a $2$--orbifold, set $S=\mani\orbS$, and let $\tS$
denote the topological $2$--manifold obtained by doubling $S$ along the
$1$--manifold $\mu\doteq\mirror\orbS\subset\bdy S$. Thus $\tS$ is a union of
two topological submanifolds $S_0$ and $S_1$, where each $S_i$ is equipped
with a canonical homeomorphism $h_i$ onto $S$. The $h_i$ agree on
$R\doteq S_0\cap S_1$, and map $R$ homeomorphically onto $\mu$. Let
$F\from\tS\to S$ denote the map such that $F|S_i=h_i$ for
$i=0,1$. Then there 
exist a unique $2$--orbifold $\torbS$ and a unique two-sheeted orbifold covering
map $f\from\torbS\to\orbS$ such that
$\mani\torbS=\tS$ and  $F$ is the underlying map of $f$. 
We shall call $\torbS$ the {\it
  \reflective\ covering} of $\orbS$. 

We have $R\doteq S_0\cap S_1=F^{-1}(\mirror\orbS)$, 
and the underlying map of 
the non-trivial deck transformation of $\torbS$
interchanges $S_0$ and $S_1$.

We have
$\mirror\torbS=\emptyset$, and the singular points of $\torbS$ are as follows. For any order--$p$ cone point $x \in \orbS$, the preimage $f^{-1}(x)$ consists of two order--$p$ cone points, namely $h_0^{-1}(x)$ and $h_1^{-1}(x)$. For any  order--$2p$ bi-mirror point  $x \in \orbS$, the preimage $f^{-1}(x)$ is a single order--$p$ cone point, namely $\tx = h_0^{-1}(x) = h_1^{-1}(x) \in R$. 
If
$( U , \maybeD, \Phi, G)$ is a standard chart  for
$\torbS$ centered at $\tx$, then $\Phi^{-1}(R)$ is a union of $p$
diameters of the open disk $D$. Thus $R$ is  \emph{not} in general the
underlying set of a suborbifold of $\torbS$. Likewise, $S_0$ and $S_1$
are not in general the
underlying sets of  suborbifolds of $\torbS$. 

For any small enough neighborhood $U$ of $\mirror\orbS$ in $S$, if we
set $\orbU=\orbi U$ and $\torbU=f^{-1}(\orbU)$, the covering
$f|\torbU\from\torbU\to\orbU$ is equivalent to the (orbifold)
orientation cover of $\orbU$. However, the \reflective\ covering of
$\orbS$ is not in general equivalent to its orientation covering. For
example, if $\orbS$ is a non-orientable closed $2$--manifold, regarded as a closed
$2$--orbifold with empty singular set, the orbifold orientation cover
of $\orbS$ coincides with its usual manifold orientation cover, but its
\reflective\ cover is a disjoint union of two trivial covers.
\EndCustom

By combining Assertions \ref{was six} and \ref{what RH looks like}
with the classification (up to conjugacy) of finite subgroups of
$\bigSO(3)$, one obtains:

\begin{proposition}[\rm{cf.\ \cite[Theorem 2.5]{CHK:OrbifoldBook}
and \cite[Figure 2.3]{BoileauMaillotPorti}}]
\label{Prop:LocalStructure3Orbifold}
If $\orbM$ is  an orientable $3$--orbifold with nonsingular
$2$--corners, then
for any stratum  $\sigma$  of   $\orbM$,
exactly one of the following alternatives holds.

\begin{enumerate}[\:\:$(1)$]
\item \label{ThreeOrbNonsing}
$\sigma\cap\Sing_\orbM=\emptyset$ (so that $\dim\sigma=3$ and
  $G_\sigma=\{1\}$ by 
Assertion \ref{was four}).
In this case, for any
  standard neighborhood $\orbU $ of a point $x$ of $\sigma$, the set
  $U\doteq\mani \orbU $ is disjoint from $\Sing_\orbM$. If
  $x\in\inter\orbM$ then $U$ is an open $3$--ball. If
  $x\in\bdy\orbM$ then $U$ is homeomorphic to  
an open half-ball in
  $\RR^2 \times \RR_+$, 
with $x\in\bdy U$.
 \item \label{ThreeOrbCone}
  $\sigma$ has dimension $1$, and $G_\sigma$ is a cyclic group of rotations
  of some order $p > 1$. 
Furthermore, 
if $\orbU $ is a standard neighborhood of a point $x$ of $\sigma$, then 
there is a homeomorphism $j$ from $U\doteq\mani \orbU $ to 
an open ball in $\RR^3$ (if $x\in\inter\orbM$) or to
an open half-ball in $\RR^2 \times \RR_+$ (if $x\in\bdy\orbM$).
In the respective cases, $j(U\cap\Sing_\orbM )$ is an  axis of the
ball or the  axis of the half-ball.

\item \label{ThreeOrbNode}
$\sigma$ has dimension $0$ and is contained in $\inter\orbM$, 
 and $G_\sigma$ is a spherical triangle group, given by a presentation
$$\langle \xi,\eta:\xi^{p_1}=\eta^{p_2}=(\xi\eta)^{p_3}=1\rangle$$
for some integers 
$p_1,p_2,p_3$
satisfying $1<p_1\le p_2\le p_3$ and $1/p_1+1/p_2+1/p_3>1$. 
Furthermore, 
if $\orbU$ is a standard neighborhood of the
unique point $x$ of $\sigma$, then 
$\mani\orbU\cap\Sing_\orbM$ has the form
  $\{x\}\cup\tau_1\cup\tau_2\cup\tau_3$, where for $i=1,2, 3$ the set $\tau_i$ is a component of
  the intersection of $\mani\orbU$ with a $1$--dimensional stratum of order 
$p_i$; and there is a homeomorphism
 of $\mani\orbU$ onto an open ball $E\subset\RR^3$ which
carries $x$ to the center $c$ of $E$, and carries each $\tau_i$ onto
$\rho_i-\{c\}$ for some radius $\rho_i$ of $E$. 

\end{enumerate}

In particular it follows that $M\doteq\mani\orbM$ is an orientable topological
$3$--manifold and
that $\bdy \mani\orbM=\mani{\bdy\orbM}$.
Furthermore,
every stratum of
$\orbM$ has dimension $0$, $1$ or $3$, and $\Sing_\orbM$ is homeomorphic to a $1$--complex.
\NoProof
\end{proposition}

\Custom{Cone strata and nodes}
\label{what's a node}
Let $\orbM$ be an  orientable $3$--orbifold with nonsingular $2$--corners,
and let $\sigma$ be a singular  stratum of
    $\orbM$.
If Alternative \eqref{ThreeOrbCone} 
of the conclusion  of
    Proposition \ref{Prop:LocalStructure3Orbifold} holds, 
we shall call $\sigma$ a  \emph{cone stratum
  of order 
$p$}, where $p$ is the integer given by \eqref{ThreeOrbCone}. When Alternative \eqref{ThreeOrbNode}
holds, we call the unique point of $\sigma$ a  \emph{node of type
  $(p_1,p_2,p_3)$}, where   $(p_1,p_2,p_3)$ 
is the triple given by \eqref{ThreeOrbNode}. 

We define a   \emph{dihedral
  node} of $\orbM$ to be a node of type 
$(2,2,p)$ for some $p \geq 2$.

Now, suppose that $\orbM$ is compact, and that  $\tau$ is a $1$--stratum of $\Sing_\orbM$, equipped with an orientation. If there is a continuous map 
$\alpha\from [0,1]\to\mani\orbM$ 
which restricts to
an orientation-preserving 
diffeomorphism of $(0,1)$ onto
the $1$--manifold $\tau\cap\mani{\inter\orbM} $, 
such that $\alpha(1)$ is a
node, then $\alpha(1)$ is called a \emph{terminal node} of $\tau$. 
Note that an oriented $1$--stratum can have at most one terminal node.
Note also that $\alpha(0)$ may be the same node, a different node, or a $1$--corner point of $\orbM$.

For each node
$x$ of $\orbM$, there are exactly three oriented $1$--strata of
$\orbM$ having $x$ as their terminal nodes. The orders of these
oriented $1$--strata are 
$p_1$, $p_2$ and $p_3$, where $(p_1,p_2,p_3)$
is the type of $x$. It may happen that two of these oriented strata
are defined by distinct orientations of the same underlying stratum.
\EndCustom

\Custom{Orbifold Euler characteristic}\label{Euler}
Let $\orbM$ be a compact $n$--orbifold with 
$n\le2$, or a compact, orientable $3$--orbifold with nonsingular
$2$--corners. According to Assertion \ref{missing piece}, since $\orbM$ is
compact it has only finitely many strata; and for each stratum
$\sigma$, the set $\barsigma-\sigma$ is a union of strata. 
It then follows that $M$ admits a PL
structure with the property that the closure of each stratum 
of $\orbM$ is a PL subset. (See \S\ref{folded} when
$n=1$; Proposition \ref{Prop:LocalStructure2Orbifold} and the
triangulation theorem for $2$--manifolds~\cite[Theorem 4.1]{Thomassen:Jordan-curve} when $n=2$; and
Proposition \ref{Prop:LocalStructure3Orbifold} and
the triangulation theorem for $3$--manifolds~\cite[Theorem 8.1]{moise-tame}
when $n=3$.)
Hence for each stratum $\sigma$ of $\orbM$, the sets
$\barsigma$ and
$\barsigma-\sigma$ are compact PL spaces and therefore have
well-defined Euler characteristics. We define the  \emph{Euler
  characteristic of $\orbM$} 
  to be
\Equation\label{that's what chi is}
\chi(\orbM)=\sum_\sigma\frac{ \chi(\overline \sigma) - \chi(
\barsigma-\sigma)}{ \ord(\sigma)},
\EndEquation
where $\sigma$ ranges over the strata of $\orbM$.

In the special case where $\orbS$ is a compact, orientable $2$--orbifold, 
the only singular points are (isolated) cone points, and \eqref{that's what chi is} immediately implies 
 the well-known formula
\Equation\label{hurwitz}
\chi(\orbS)=\chi(\mani \orbS )-
\sum_{x\in\Sing_\orbS}\bigg(1-\frac1{\ord(x)}\bigg).
\EndEquation

Under the general hypotheses of this subsection, 
it follows from the observations made above that $M \doteq \mani \orbM$ admits a triangulation in
which $\barsigma$ and $\barsigma-\sigma$ are subcomplexes for each
stratum $\sigma$ of $\orbM$. 
Now, definition \eqref{that's what chi is} implies
 that for such a triangulation we have 
\Equation\label{in terms of simplices}
\chi(\orbM)=\sum_\tau (-1)^{\dim\tau}\frac{1}{\ord(\tau)},
\EndEquation
where $\tau$ ranges over the open simplices of $M$, and
$\ord(\tau)$ denotes the order of the stratum containing $\tau$. In
particular 
 the right hand side of \eqref{in terms of simplices} is independent of the choice of a triangulation of
$M$ with the stated properties. Compare \cite[Definition on p.~39]{BoileauMaillotPorti} and \cite[Section 2.4]{CHK:OrbifoldBook}.

From the formula \eqref{in terms of simplices} it follows readily that 
if $\torbM$ is an $n$--sheeted
covering  of
 $\orbM$, where $n<\infty$, then
$\chi(\torbM)=n\chi(\orbM)$. The same formula implies that if $\orbY$ and $\orbZ$ are
compact suborbifolds with corners  of $\orbM$
 such that
$\orbY\cap \orbZ$ and $\orbY\cup \orbZ$ 
are defined (see \S\ref{more orbifold notions}), 
then $\chi(\orbY\cup\orbZ)=\chi(\orbY)+\chi(\orbZ)-\chi(\orbY\cap\orbZ)$.
\EndCustom

\Remark
It can  be shown
that if the dimension and orientability assumptions are removed from the discussion in
\S\ref{Euler}, then for each stratum $\sigma$ of $\orbM$, 
the sets
$\barsigma$ and
$\barsigma-\sigma$ are compact locally triangulable PL spaces.
Shmuel Weinberger has pointed out to us that, as a consequence of the results of \cite{hanner} and \cite{west}, 
such spaces are homotopy-equivalent to compact CW complexes, and therefore
have
well-defined Euler characteristics. Furthermore, if $\chi(\orbM)$ is defined
by \eqref{that's what chi is}, then the assertions in \S\ref{Euler}
involving finite coverings, unions and intersections still hold. As
we will not need these results, we omit the proofs.
\EndRemark

\Lemma\label{very good near cone stratum}
Let $\sigma$ be a cone stratum in an orientable $3$--orbifold
$\orbM$. Then $\sigma$ has an open neighborhood $\orbW$ in $\orbM$ 
that
is very good. In fact,
$\orbW$ has a degree-$\ord(\sigma)$ regular covering $\torbW$ which is
a $3$--manifold, and the covering projection $\phi\from\torbW\to\orbW$ maps
$\phi^{-1}(\sigma)$ diffeomorphically onto $\sigma$.
\EndLemma

In Proposition~\ref{smooth underlying}, which depends on Lemma~\ref{very good near cone stratum}, we will show that $\orbW$ can be chosen in a way that makes it a tubular neighborhood of $\sigma$. However, this will require an additional argument.

\Proof
It follows from Proposition \ref{Prop:LocalStructure3Orbifold} that
$\sigma$ is a locally flat 
one-dimensional 
topological
submanifold of the topological manifold
$M\doteq\mani\orbM$. Therefore, according to 
\cite[Theorem 8.1]{moise-tame},
there is a PL structure on $M$ in which $\sigma$ is a PL
submanifold. Hence, by orientability, there is an open neighborhood $W$ 
of $\sigma$ in $M$, with $W\cap\Sing_\orbM=\sigma$, such that the topological pair $(W,\sigma)$ is
homeomorphic to $(\sigma\times \inter \BB^2, \sigma\times\{0\})$. In
particular, 
if we fix a transverse orientation of $\sigma$, then
each closed curve $\gamma$ in $W-\sigma$ has a well-defined winding
number
about $\sigma$, which we will denote by 
$d(\gamma)\in\ZZ$.

If $U$ is the support of a standard chart centered at a point
of $\sigma$, and $U\subset W$, then by the definition of a cone
stratum, $U\cap\Sing_\orbM=U\cap\sigma$ 
is an unknotted open arc in the open ball $U$; hence 
for any closed curve $\gamma$ in $U-(U\cap\sigma)$, the winding number
of $\gamma$ about $U\cap\sigma$ is defined and is equal to
$d(\gamma)$. If $p$ denotes the order of the cone stratum $\sigma$,
and if $\bard(\gamma)$ denotes the image of $d(\gamma)$ in $\ZZ/p\ZZ$
for an arbitrary  closed curve $\gamma$ in $W-\sigma$, we deduce:
\Claim\label{I need this}
If $U\subset W$ is the support of a standard chart centered at a point
of $\sigma$, and $\gamma$ is a closed curve in $U-(U\cap\sigma)$, 
then $\gamma$ represents the identity in $\pi_1(\orbi U)$ if and only if $\bard(\gamma)=0$. 
\EndClaim
Now suppose that $\gamma$ is a closed curve in $W-\sigma$ which
represents the identity in the fundamental group of $\orbW\doteq\orbi W$. Then the continuous
orbifold map $\gamma\from \SS^1\to\orbW$ extends to a  continuous
orbifold map $\delta\from \BB^2 \to\orbW$. We may choose $\delta$ so
that $\delta^{-1}(\sigma)$ is a finite set $\{x_1,\ldots,x_m\}$. If
$U_1,\ldots,U_m$ are pairwise disjoint supports of standard charts for
$\orbM$ centered at $x_1,\ldots,x_m$, we may choose subdisks
$D_1,\ldots,D_m$ of $\BB^2$ centered at the $x_i$ such that
$\delta(D_i)\subset U_i$; the restrictions of $\delta$ to 
$\bdy D_1,\ldots,\bdy D_m$ then define closed curves
$\gamma_i$ in the $U_i$, and each $\gamma_i$ represents the identity
in $\pi_1(\orbi U_i)$. In view of \ref{I need this} it now follows
that $\bard(\gamma_i)=0$ for $i=1,\ldots,m$, and hence that
$\bard(\gamma)=\sum_{i=1}^m\bard(\gamma_i)=0$. This shows that 
for every closed curve $\gamma$ in $W-\sigma$ which
represents the identity in $\pi_1(\orbW)$, we have $\bard(\gamma)=0$. From
this it follows that there is a well-defined homomorphism
$h\from\pi_1(\orbW)\to \ZZ/p\ZZ$ given by $h([\gamma])=\bard(\gamma)$.
Let $\phi\from\torbW\to\orbW$ denote the $p$--fold covering of $\orbW$ defined by the
subgroup $\ker h$ of $\pi_1(\orbW)$
(see \S\ref{path subsection}).

If $U\subset W$ is the
support of a standard chart centered at a point of $\sigma$, 
another application  of \ref{I need this}  shows that 
the
composition of the inclusion homomorphism $\pi_1(\orbi
U)\to\pi_1(\orbW)$ with $h \from \pi_1(\orbW)\to \ZZ/p\ZZ$ is an isomorphism. Hence
 the induced covering $\phi^{-1}(\orbi U)$ of $\orbi U$ is diffeomorphic to the
 universal cover of $\orbi U$, which is a ball. As this holds for all
 choices of $U$, the $p$--fold cover $\torbW$ of $\orbW$ is a manifold.
It follows from the construction that $\phi$ maps
$\phi^{-1}(\sigma)$ diffeomorphically onto $\sigma$.
\EndProof

\Custom{$\orbM$--accordant smooth structures}
\label{wuzzaccordant}
Let $\orbM$ be 
an orientable $3$--orbifold, so that $M\doteq\mani\orbM$ is an
orientable topological $3$--manifold by Proposition \ref{Prop:LocalStructure3Orbifold}.
A smooth   $3$--manifold structure 
on $M$
 will be termed \emph{$\orbM$--accordant} if it 
satisfies the following conditions: 
\begin{enumerate}
\item\label{Itm:Accord1} Each stratum of $\orbM$, with the smooth structure given by
Assertion 
\ref{was too},
is a submanifold of $M$.
\item\label{Itm:Accord2} If $\sigma$ is any cone stratum
of $\orbM$, and if $p=\ord(\sigma)$,  then there exist an open subset $U$
of  $M$ with $U\cap\Sing_\orbM=\sigma$, 
such that the
self-map $(z,t)\mapsto(z^p,t)$ of $\DD\times \sigma $ is the
composition of the 
underlying map of  
some orbifold covering map $\DD\times \sigma \to \orbi U$ with some manifold diffeomorphism $\zeta\from  U\to \DD\times \sigma $. 

\end{enumerate}

We observe that if a smooth  structure on $M$ is $\orbM$--accordant, then
in particular Condition \eqref{Itm:Accord1} above must hold for every  
codimension--$0$
stratum of $\orbM$. Hence the smooth manifold structure that
$M-\Sing_\orbM$ inherits from $M$ coincides with the smooth manifold
structure that $\orbi{M-\Sing_\orbM}$ possesses as an open suborbifold
of $\orbM$. 
\EndCustom

\Proposition\label{smooth underlying}
Let $\orbM$ be 
an orientable $3$--orbifold.
Then there exists an $\orbM$--accordant smooth structure on the
topological $3$--manifold $M\doteq\mani\orbM$.
\EndProposition

\Proof
We let $\orbM'\subset\orbM$ denote the complement of the set of
nodes in $\Sing_\orbM$ 
(see \S\ref{what's a node}).
The orientability of $\orbM$ implies that the singular strata of
$\orbM'$ are all one-dimensional.
For each singular stratum $\sigma$ of $\orbM'$, Lemma
\ref{very good near cone stratum} gives
an open suborbifold $\orbW_\sigma$ of $\orbM'$, 
\delete{with $\mani{\orbW_\sigma}=\sigma$,} and
a finite-sheeted regular cover
$\phi_\sigma\from \torbW_\sigma\to \orbW_\sigma$ with cyclic  deck
group $G_\sigma$ of order $p_\sigma\doteq\ord(\sigma)$, such that $\torbW_\sigma$ is a
$3$--manifold, and  $\phi_\sigma$ maps $S_\sigma\doteq\phi_\sigma^{-1}(\sigma)$
diffeomorphically onto $\sigma$.
We may take
the $\orbW_\sigma$ to be pairwise disjoint as
$\sigma$ ranges over the cone strata of $\orbM'$. 
For each $\sigma$,
the (possibly non-compact) $1$--manifold $S_\sigma$ is then $G_\sigma$--invariant.
Fix a $G_\sigma$--invariant tubular neighborhood $V_\sigma$ of $S_\sigma$ in
$\torbW_\sigma$. 
Since $G_\sigma$ is cyclic of order $p_\sigma$ and
$\orbW_\sigma\subset\orbM$ is orientable,
$V_\sigma$ may
be diffeomorphically identified with $\DD\times S_\sigma$ in such a way that some generator of
$G_\sigma$ acts on $V_\sigma$ by $(z,s)\mapsto(e^{2\pi i/p_\sigma}z,s)$. We may then
identify $\phi_\sigma(V_\sigma)$ homeomorphically with $\DD\times S_\sigma$ in such a way
that $\phi_\sigma|V_\sigma$ is the map $(z,s)\mapsto(z^{p_\sigma},s)$. 
The standard smooth
structure on $\DD$ and the smooth structure  on $S_\sigma$ inherited
from $\torbW_\sigma$ define a smooth structure on the product
$\DD\times S_\sigma$, and therefore on  $\phi_\sigma(V_\sigma)$.

On the other hand, the open suborbifold  $\orbM'-\Sing_{\orbM'}=M'-\Sing_{\orbM'}$ of $\orbM'$ is a
smooth manifold. 
For each singular stratum $\sigma$ of $\orbM'$, the smooth structure
of   $\orbM'-\Sing_{\orbM'}$ agrees on $\phi_\sigma(V_\sigma)-\sigma$ with
the smooth structure that we have constructed on $\phi_\sigma(V_\sigma)$. 
Hence
there is a smooth structure on $M'$ that restricts to the smooth
structures that we have described on the sets $\phi_\sigma(V_\sigma)$ and on
$M'-\Sing_{\orbM'}$. By Proposition \ref{extend},
this smooth
structure on $M'$ extends to one on $M$. It follows from the
construction that this smooth structure on $M$
is $\orbM$--accordant.
\EndProof

\Remark\label{not unique}
The $\orbM$--accordant  smooth structure given by Proposition
\ref{smooth underlying} is not in general unique, and it is not clear
to us whether there is a canonical $\orbM$--accordant  smooth
structure for a given  orientable $3$--orbifold $\orbM$. At
several points in this paper it will be necessary to make a choice of 
 $\orbM$--accordant  smooth
structure; our final results are independent of these choices.
\EndRemark

\Custom{Suborbifolds and transversality}
\label{Suborbifolds-of-3orbifolds}
Let $\orbM$ be an orientable $3$--orbifold.
If $M\doteq\mani\orbM$ is  equipped with an
$\orbM$--accordant 
smooth $3$--manifold structure, 
we shall  often refer to it 
as the  \emph{underlying manifold}
of $\orbM$.
In particular, it makes sense to  speak of smooth submanifolds
of $M$, and of transversality between such submanifolds. A
smooth submanifold of  $M$ will be said to be  \emph{transverse
  to $\Sing_\orbM$} if it is transverse to all the strata of
$\Sing_\orbM$ (which are themselves smooth submanifolds of
$M$ by the definition of $\orbM$--accordance).

If $M$ is equipped with an $\orbM$--accordant smooth structure,
and if $F$ is a neatly embedded submanifold
of $M$ of 
dimension at most $2$
which is transverse to $\Sing_\orbM$, then $\orbi F$
is defined and is a neatly embedded suborbifold of $\orbM$ (in the
strong sense). To verify this, we must check that Condition
\ref{defines a suborbifold} holds with $F$ playing the role of $N$,
and that the homomorphism $\theta_y$ defined in \S\ref{Suborbifolds} is
an isomorphism for every $y\in F$.   Condition \eqref{Itm:Accord1} of \S\ref{wuzzaccordant} 
implies that \ref{defines a suborbifold} holds when the given
point
$x\in M$ is nonsingular, and that $\theta_y$ is
an isomorphism when the point $y\in F$ is a nonsingular point of $\orbM$.
Condition \eqref{Itm:Accord2} of \S\ref{wuzzaccordant} 
implies that \ref{defines a suborbifold} holds when 
$x\in M$ lies in a cone stratum of $\orbM$, and that $\theta_y$  is
an isomorphism when $y\in F$ lies in a cone stratum. In view of
transversality, the case when
$y\in F$ is a node of $\orbM$ does not arise. Thus $\orbF \doteq \orbi F$ is defined.

According to \S\ref{Suborbifolds}
we then have $\Sing_{\orbF}=
F\cap\Sing_\orbM$. 
In particular, we have $\size F=\card \Sing_{\orbF}$ 
(where $\size F$ is defined as in 
\S\ref{size 'n' stuff}, 
and is finite if $F$ is compact).
It also follows from \S\ref{Suborbifolds} that
$\ord^\orbF(x)=\ord^\orbM(x)$ for each $x\in F$; hence if $F$ is
compact
then $\weighttwo(F)$ 
is the sum of the orders of the singular points of
$\orbi F$.

We will sometimes encounter the special case in which a  submanifold $F$ of
  $M$ is  \emph{disjoint} from $\Sing_\orbM$. 
In this case,
$\orbi F$ is a manifold having the same underlying set and the same
smooth manifold structure as the given submanifold  $F$ of
$ \mani\orbM$. Thus the convention laid out in 
\S\ref{when it's a manifold},
according to which we write $F=\orbi F$, leads to no ambiguity.

The observation 
  made above about obtaining suborbifolds from submanifolds
  transverse to the singular set
has a partial converse: if $\orbM$  is an orientable $3$--orbifold, then every orientable, \neatly embedded
 $2$--suborbifold of $\orbM$
has the form $\orbi F$ for some orientable, \neatly embedded
two-dimensional submanifold $F$ of $ \mani\orbM$. (This becomes false
if one removes the hypothesis of orientability. For example, if $\rho$
is a $\pi$--rotation of $\SS^3$ and $C$ denotes the fixed circle
of $\rho$, then for any $2$--sphere $Y\supset C$ the suborbifold
$\orbN\doteq Y/\langle\rho\rangle$ of $\orbM\doteq
\SS^3/\langle\rho\rangle$ is defined, but  $\mani\orbN$ contains $\Sing_\orbM$ rather than being transverse to  $\Sing_\orbM$.)
\EndCustom

The remainder of this section primarily consists of definitions and results that extend certain familiar $3$--manifold notions to $3$--orbifolds.

\Custom{Splitting an orbifold}
\label{new splitting}

There is a notion of splitting an orbifold along a suitable
type of suborbifold, which generalizes the familiar notion of
splitting a manifold along a \neatly embedded codimension-$1$ submanifold.

Suppose that $\orbS$ is a \neatly embedded, $(n-1)$--dimensional
suborbifold of an  $n$--orbifold (with boundary) $\orbM$. 
We shall
say that an $n$--orbifold with $2$--corners $\orbM'$ is  \emph{obtained by splitting $\orbM$
  along $\orbS$,} and that an orbifold immersion
$\iota\from\orbM'\to\orbM$ is  \emph{associated with the splitting,} if 
\begin{enumerate}
\item $\iota|\iota^{-1}(\orbM-\orbS)\from \iota^{-1}(\orbM-\orbS)
\to (\orbM-\orbS)$ is an orbifold
  diffeomorphism, and
\item for every $x\in\orbS$ there is a connected open neighborhood
  $\orbV$ of $x$ in $\orbM$ such that $\iota^{-1}(\orbV)$ has exactly
  two components, and $\iota$ maps each of these components
  diffeomorphically onto the closure, relative to $\orbV$, of a component
  of $\orbV\cap(\orbM-\orbS)$.
\end{enumerate}

If $\orbM'_1$ and $\orbM'_2$ are two orbifolds with $2$--corners  obtained by splitting
$\orbM$ along $\orbS$, and if $\iota_1$ and $\iota_2$ are the
associated immersions, one can deduce directly from the definition of
splitting that there is a unique diffeomorphism $\alpha\from
\orbM'_1\to\orbM'_2$ such that $\iota_2\circ\alpha=\iota_1$. In view
of this uniqueness property of an orbifold with $2$--corners obtained
by splitting, we shall refer to  \emph{the} orbifold obtained by splitting
$\orbM$ along $\orbS$ if such an orbifold exists, and denote it by
$\orbM\split\orbS$. To say that $\orbM\split\orbS$ is defined means
that such a split orbifold exists. We shall refer to the associated
immersion as the  \emph{canonical immersion} $\iota\from
\orbM\split\orbS\to \orbM$.

Whenever  $\orbM' \doteq\orbM\split\orbS$ is defined, we have
$\Sing_{\orbM'}=\iota^{-1}(\Sing_{\orbM})$. Furthermore, 
$\iota^{-1}(\orbS)$ 
is
defined and is disjoint from
$\inter\orbM'$; by Proposition \ref{codim 1 is nice},
$\iota^{-1}(\orbS)$ 
is a suborbifold of $\bdy\orbM'$ (in the strong
sense). The set of $2$--corner points of $\orbM'\doteq\orbM\split\orbS$ is
$\mani{\bdy(\iota^{-1}(\orbS))}$. 
The canonical immersion
$\iota\from\orbM'\to\orbM$ restricts to a
two-sheeted orbifold covering map from 
$\iota^{-1}(\orbS)$
to 
$\orbS$. 
The non-trivial deck transformation of this orbifold covering will be called the \emph{canonical involution} of $\iota^{-1}(\orbS)$.

We recall that if $M$ is a manifold and $S\subset M$ is a \neatly
embedded codimension-$1$ submanifold, then
$M\split S$ is always defined.
For example, if $M$ is connected and $S$
separates $M$, we may take $M\split S$ to be the disjoint union of the
closures of the components of $M-S$, and take the canonical immersion to agree with the
inclusion on the closure of each component. In  general we may take
$M\split S$ to be the closure in $M$ of $M-R$, where $R$ is a tubular
neighborhood of $S$ in $M$; the definition of $\iota$ then requires a
little care.

While it can be shown  that, in the orbifold  context, $\orbM\split
\orbS$ exists under fairly broad conditions, we shall need only the case in which 
  $\orbM$ is an orientable $3$--orbifold and $\orbS$ is a \neatly
  embedded, orientable two-dimensional suborbifold
of $\orbM$; and in this case, there is a construction of $\orbM\split
\orbS$ which is particularly well-adapted to the applications in this
paper. Choose an $\orbM$--accordant smooth structure on 
$M\doteq\mani\orbM$, which exists by Proposition \ref{smooth underlying}. By an observation made in \S\ref{Suborbifolds-of-3orbifolds}, we have $\orbS=\orbi S$ for some \neatly
  embedded, orientable two-dimensional submanifold $S$
of the orientable $3$--manifold $M$. By the discussion above, $M'\doteq
M\split
S$ is defined. If $J\from M'\to M$ denotes the canonical
immersion, it is straightforward to check that 
there exist a unique
$3$--orbifold $\orbM'$ and a unique orbifold immersion
$\iota\from\orbM'\to\orbM$ such that $\mani\orbM'=M'$ and  $J$ is the
underlying map of $\iota$. 
Furthermore, Conditions
(1) and (2) above hold with these choices of $\orbM'$ and
$\iota$. Thus $\orbM\split\orbS=\orbM'$ is defined in this situation, and it
follows from the uniqueness pointed out above that $\orbM'$ is
independent of the choice of an $\orbM$--accordant smooth structure. 
\EndCustom

\Lemma\label{split injective}
Let $\orbS$ be a \neatly
  embedded, orientable, $\pi_1$--injective two-dimensional suborbifold
of a connected, orientable $3$--orbifold $\orbM$. 
Let $\orbM'$ denote a component of $\orbM \split \orbS$, and let 
$\iota \from \orbM' \to\orbM$ denote the canonical
immersion. Then the induced homomorphism
$\iota_*\from\pi_1(\orbM') \to \pi_1(\orbM)$ 
is injective.
\EndLemma

\Proof
When $\orbM$ is a manifold, this is standard, 
and one of the standard proofs involves constructing the universal
covering of the given manifold as a tree of spaces
(cf.\ Scott and Wall \cite{scott-wall}). We shall adapt this proof to the orbifold
setting, using
the orbifold versions of standard results in
covering space theory (see \S\ref{path subsection}).

We shall denote by $\tau$ the canonical involution of $\iota^{-1}(\orbS)$ (see \S\ref{new splitting}).

 Any orbifold covering
space of $\orbM$ induces an orbifold covering of  $\orbM\split \orbS$. To
prove the lemma, it suffices to 
construct an orbifold covering space $\torbM$ of $\orbM$ such that every
component of the induced orbifold
covering space of $\orbM\split\orbS$ is simply connected. (It is
not necessary to show that $\torbM$  itself is simply connected,
although this could be done.)

For $n\ge0$ we shall recursively define orbifolds $\torbM_n$, orbifold
immersions $p_n\from\torbM_n\to\orbM$, and orbifold embeddings
$j_n\from\torbM_n\to\torbM_{n+1}$ such that $p_{n+1}\circ j_n=p_n$. For
each $n$, each component of the suborbifold $p_n^{-1}(\orbS)$ of
    $\torbM_n$ will be either properly embedded in $\torbM_n$ or
    contained in $\bdy\torbM_n$. 
Furthermore, if  the  components of $p_n^{-1}(\orbS)$ that are
    contained in $\bdy\torbM_n$ are faithfully indexed as $(\torbS_\alpha)_{\alpha\in J_n}$,
where $J_n$ is some index set,
then for each  $\alpha\in J_n$ 
 here will be a simply connected 
suborbifold $\orbE_\alpha$ of $\torbM_n$, containing $\torbS_\alpha$,
 such that
    $p_n|\orbE_\alpha\from\orbE_\alpha\to\orbM$ is the composition with
    $\iota$ of
    some covering map $p_\alpha$ from $\orbE_\alpha$ to a component of $\orbM\split\orbS$.

We take $\torbM_0$ to be a simply connected orbifold
covering of some component $\orbM_0$ of $\orbM\split\orbS$, and define
$p_0$ to be the composition of the covering map with $\iota$. 

Now suppose that
$\torbM_n$ and $p_n$ have been constructed. For each $\alpha\in J_n$, set $\orbS_\alpha=p_\alpha(\torbS_\alpha)$ and $\orbT_\alpha=\tau(\orbS_\alpha)$.
The simple connectivity of $\orbE_\alpha$ and the 
$\pi_1$-injectivity of $\orbS_\alpha$ imply that $\torbS_\alpha$ is simply
connected. Likewise, if $\orbM'_\alpha$ denotes the component of
$\orbM\split\orbS$ containing $\tau(\orbS_\alpha)$, and if we fix a simply
connected covering $q_\alpha\from\orbF_\alpha\to\orbM'_\alpha$ and a
component $\torbT_\alpha$ of $q_\alpha^{-1}(\tau(\orbS_\alpha))$, then
$\torbT_\alpha$ is also simply connected; hence for each 
$\alpha\in J_n$ we may fix a lift
$\ttau_\alpha\from\torbS_\alpha\to\torbT_\alpha$ of $\tau|\orbS_\alpha\from\orbS_\alpha\to\orbT_\alpha$ which
is an orbifold diffeomorphism.

Consider the
disjoint union
of $\tM_n\doteq\mani\torbM_n$ with all spaces of the form
$\orbF_\alpha$, where $\alpha$ ranges over $J_n$. Define $\tM_{n+1}$ to be the quotient space of this disjoint
union obtained by identifying $\mani\torbS_\alpha$ with $\mani{\torbT_\alpha}$, via the
underlying homeomorphism of
$\ttau_\alpha$, for each  $\alpha\in J_n$. There is a
well-defined map $P_{n+1}\from \tM_{n+1}\to \mani\orbM$ which restricts
to $P_n$ on $\tM_n$ and to $q_\alpha$ on  $\orbF_\alpha$, for each $\alpha\in J_n$. There exist a
unique orbifold $\torbM_{n+1}$ and a unique orbifold immersion
$p_{n+1}\from\torbM_{n+1}\to\orbM$ such that
$\mani{\torbM_{n+1}}=\tM_{n+1}$ and $P_{n+1}$ is the underlying map
of $p_{n+1}$. The natural map from $\tM_n$ to $\tM_{n+1}$ is the
underlying map of an orbifold embedding $j_n$. With these definitions,
$\torbM_{n+1}$, $p_{n+1}$ and $j_n$ have the properties required to complete the
recursion step. (If  the  components of $p_{n+1}^{-1}(\orbS)$ that are
    contained in $\bdy\torbM_{n+1}$ are faithfully indexed as $(\torbS_\beta)_{\beta\in J_{n+1}}$,
then for each  $\beta\in J_{n+1}$ there is a unique $\alpha\in J_n$ such that $\torbS_\beta\subset \orbF_\alpha$. We may then take $\orbE_\beta=\orbF_\alpha$ and $p_{\beta}=q_\alpha$.)
This completes the recursive construction.

The  directed system defined by the $\torbM_n$ and the
$j_n$ has a well-defined direct limit $\torbM$ in the category of orbifolds.
The maps $p_n$ then define a covering map $\torbM \to \orbM$ which has the
required properties.
\EndProof

\Custom{Special types of orbifolds}
\label{spherical etc}
An orbifold (of any dimension) will be termed  \emph{discal} or {\it
  spherical} if it is covered by a ball
or a sphere respectively. A
$2$--orbifold will be termed  \emph{toric} or  \emph{annular} if it is
covered by a torus or annulus respectively. A  \emph{solid toric
  orbifold} is a $3$--orbifold which is covered by a solid torus. 

  These definitions apply to
both orientable and non-orientable orbifolds. 

If $\orbN$ is an orientable discal
$2$--orbifold then  $\mani\orbN$ is a disk, and 
$\Sing_\orbN$ either is empty or consists of a single cone point.
If $\orbN$ is an orientable annular
$2$--orbifold then either $\mani\orbN$ is an annulus and
$\Sing_\orbN=\emptyset$, or $\mani\orbN$ is a disk and
$\Sing_\orbN$ consists of two cone points of order $2$. 

If $\orbN$ is an orientable discal
$3$--orbifold, 
then the Orbifold Theorem  
 \cite{BoileauMaillotPorti,
  CHK:OrbifoldBook} implies that
 $N\doteq\mani\orbN$ is a topological $3$--ball.
 If $\orbN$ is 
a  solid toric
$3$--orbifold, then the Orbifold Theorem similarly implies that
$N\doteq\mani\orbN$ is  a (topological)  solid torus or  $3$--ball.
Furthermore, if $N$ is a solid torus then $\Sing_\orbN$ is either  the
empty set or a single cone stratum which is a
core curve of $N$. If $N$ is a ball, then
$\Sing_\orbN$ is contained in a (tame) properly embedded disk $D \subset N$,  and consists of either two disjoint cone strata or three cone strata forming the letter $H$.

A  \emph{turnover} is an orientable $2$--orbifold $\orbS$ such that
$\mani\orbS$ is a $2$--sphere and $\Sing_\orbS$ has cardinality $3$. If
the orders of the singular points of $\orbS$ are $p$, $q$ and $r$, we
may refer to $\orbS$ as a  \emph{$(p,q,r)$-turnover.} If $\orbM$ is a
$3$--orbifold, we shall often use the phrase ``turnover in $\orbM$'' to mean a suborbifold of $\orbM$ which is a turnover.
\EndCustom

\Custom{\Clean\ pairs}
\label{Clean pairs}
We define a \emph{\clean\ pair} to be an ordered pair $(\orbY,\orbZ)$ such
that 
\begin{enumerate}
\item\label{Clean1} $\orbY$ is a compact, orientable $3$--orbifold with $2$--corners, 
\item\label{Clean2} $\orbZ$ is a
$2$--dimensional suborbifold (with boundary) of $\orbY$, 
\item\label{Clean3} $\mani{\orbZ}$ is disjoint from
 $\mani{\inter\orbY}$, and 
 \item\label{Clean4} $\mani{\bdy\orbZ}$ is the set of all $2$--corner points of $\orbY$.
\end{enumerate}
If $(Y,Z)$ is a \clean\ pair such that $Y$ is a manifold with corners
(and hence $Z$ is a manifold), we shall call $(Y,Z)$ a  \emph{clean
  manifold pair}.

If $(\scrY,\scrZ)$ is a \clean\ (orbifold) pair, then
$(\orbY,\scrZ\cap\orbY)$ is a \clean\ pair for any
component $\orbY$ of $\scrY$. A \clean\ pair of this form  is called a 
  \emph{component} of $(\scrY,\scrZ)$. We say that $(\scrY,\scrZ)$ is
  \emph{connected} if it has exactly one component, or equivalently if
 $\scrY$ is connected.

If $(\orbY,\orbZ)$ is a \clean\ pair then
$\orbY$ has nonsingular $2$--corners. To see this, note that by
\S\ref{Suborbifolds} we have $
\orbZ \cap \Sing_\orbY
=\Sing_\orbZ
$.  If $K$ denotes the set of $2$--corner points of $\orbY$, we
have $K=\mani{\bdy\orbZ}$ by the definition of a clean pair. Hence
$K\cap\Sing_\orbY=\mani{\bdy\orbZ}\cap\Sing_\orbZ$. The latter
intersection is empty since $\orbZ$ is an orientable $2$--orbifold, and
the assertion follows.

If $\orbY$ is a compact, orientable $3$--orbifold with $2$--corners and
$\orbZ$ is a 
two-dimensional
weak suborbifold of $\orbY$,
and if Conditions \eqref{Clean1}, \eqref{Clean3} and \eqref{Clean4} of the definition above hold, then
$(\orbY,\orbZ)$ is a clean pair. 
Indeed, Proposition \ref{codim 1 is nice} implies that $\orbZ$ is a suborbifold of $\orbY$ in the strong sense, hence Condition \eqref{Clean2} holds also.
\EndCustom

\Custom{Essentiality and Irreducibility}\label{Pi1Injective}
For the following set of definitions, suppose $\orbM$ is a compact orientable $3$--orbifold.

Let $\orbS_0$ and $\orbS_1$ be 
compact
$2$--suborbifolds 
of a $3$--orbifold $\orbM$, 
such that
$\bdy\orbS_i\subset\bdy\orbM$ for $i=0,1$.
We shall say that $\orbS_0$ and
$\orbS_1$ are  \emph{parallel} if 
there is an embedding 
 $f \from \orbS_0 \times [0,1] \to \orbM$, such
that $f$ maps $\orbS_0 \times \{i\}$
diffeomorphically onto  $\orbS_i$ for
$i=0,1$, and $f( \bdy\orbS_0 
\times [0,1] ) \subset \bdy \orbM$. The image 
(see \S\ref{Suborbifolds})
of such an embedding
is called a  \emph{parallelism} between $\orbS_0$ and $\orbS_1$. A 
\neatly\ embedded 
$2$--suborbifold
  of $\orbM$ is said to be  \emph{boundary-parallel}
if it is parallel to some $2$--suborbifold contained in $\bdy \orbM$.

A connected $2$--suborbifold $\orbS$ of 
$\orbM$ is called \emph{essential} if $\orbS$ is neatly embedded,
orientable
and $\pi_1$--injective, is not the boundary of a discal
$3$--suborbifold, and is not boundary-parallel.
A possibly disconnected $2$--suborbifold $\orbR$ is called \emph{essential} if every component of $\orbR$ is essential.

As a special case, we have the notion of an  \emph{essential turnover} in
a $3$--orbifold.

We say that $\orbM$ is \emph{irreducible} if $\orbM$ is connected and every closed  $2$--suborbifold of $\orbM$ with strictly positive Euler characteristic is the boundary of 
a discal $3$--suborbifold of $\orbM$. 
Since every bad $2$--orbifold has positive Euler characteristic by \cite[Corollary 2.28]{CHK:OrbifoldBook}, we observe that an irreducible $3$--orbifold cannot contain any bad, essential $2$--suborbifolds.
We also observe that every hyperbolic $3$--orbifold is irreducible.

We say that $\orbM$ is \emph{boundary-irreducible} if $\bdy \orbM$ is $\pi_1$--injective.
\EndCustom

\Proposition\label{when injective}
{\rm (cf.\ \cite[Corollary
3.20]{BoileauMaillotPorti}).}
Let $\orbW$ be an orientable two-dimensional closed suborbifold of the interior of
a very good, orientable $3$--orbifold
$\orbM$. Then the following conditions are equivalent:
\begin{itemize}
\item $\orbW$ is $\pi_1$--injective in $\orbM$.
\item for every two-dimensional 
orientable
discal suborbifold $\orbD$ of $\orbM$
  such that $\orbC\doteq\bdy\orbD=\orbD\cap\orbW$,
there is  a discal suborbifold of $\orbW$ whose boundary is $\orbC$.
\end{itemize}
\EndProposition

\Proof
First suppose that $\orbW$ is  $\pi_1$--injective in $\orbM$. Let
$\orbD$ be any two-dimensional 
discal suborbifold of $\orbM$
  such that $\orbC\doteq\bdy\orbD=\orbD\cap\orbW$. Then $\orbC$ is
  a simple closed curve disjoint from the singular set of $\orbM$, and
  represents (up to conjugacy and inversion) an element of finite
  order in $\pi_1(\orbM)$. Since $\orbW$ is  $\pi_1$--injective, $\orbC$ 
  represents  an element of finite
  order in $\pi_1(\orbW)$, and 
is therefore the  boundary of 
a  discal suborbifold of $\orbW$.

Now  suppose that $\orbW$ is not $\pi_1$--injective in $\orbM$.
Set $\orbM'=\orbM\split\orbW$, 
denote the canonical immersion
from $\orbM'$ to $\orbM$ by $\iota$,  and set
$\orbW'=\iota^{-1}(\orbW)$
(which, as observed in \S\ref{new splitting}, is a
well-defined  suborbifold of $\bdy\orbM'$).
Then it follows from 
Lemma \ref{split injective}
that $\orbW'$ is not $\pi_1$--injective in
  $\orbM'$. 
Fix an orbifold loop $\alpha$ in  $\orbW'$  which is orbifold-homotopically trivial 
in $\orbM'$ but not in $\orbW'$. 

Since $\orbM$ is  very good, we may fix a finite-sheeted regular cover
$p\from N\to\orbM'$ such that $N$ is a manifold. Let $G$ denote the covering
group. Since $\alpha$  is
orbifold-homotopically trivial 
in $\orbM'$, it follows from \ref{lifting homotopies} and \ref{homotop to generic}
that $\alpha$ admits a lift $\talpha$ which is a loop in
$N$. Since $\alpha([0,1])\subset\orbW'$, we have $\talpha([0,1])\subset
V\doteq p^{-1}(\orbW')\subset\bdy N$. Since $\alpha$  is
orbifold-homotopically trivial 
in $\orbM'$ but not in $\orbW'$, the 
orbifold loop 
$\talpha$  is
orbifold-homotopically trivial 
in $N$ but not in $V$. 
Hence by
Meeks and Yau's
equivariant loop theorem  \cite[Theorem 3]{MY-sphere},
$N$ contains a non-empty, $G$--invariant, \neatly 
embedded
submanifold $\boldE$, each
component of which is an essential disk. Let $\boldD$ denote the
frontier in $N$ of a 
$G$--invariant 
regular neighborhood of
$\boldE$. Then $\boldD$ is itself $G$--invariant and its components are
essential disks; and in addition, no component of $\boldD$ is
invariant under an element of $G$ which reverses its orientation. It
now follows that
$p(\boldD)$ is a well-defined 
orientable, \neatly embedded $2$--suborbifold of $\orbM'$  having
boundary in $\orbW'$, and that
each component of $p(\boldD)$  is
discal.  Since
the components of $\boldD$ are essential, no boundary component of
$p(\boldD)$ is the boundary of a discal suborbifold of
$\orbW'$. Hence the image under $\iota$ of an
arbitrary component of $p(\boldD)$
is an orientable discal
weak 
suborbifold  $\orbD$ of $\orbM$
  such that $\orbC\doteq\bdy\orbD=\orbD\cap\orbW$,
and
there is  no discal suborbifold of $\orbW$ whose boundary is $\orbC$.
Since  $\orbD$ has codimension $1$ in  $\orbM$,  and  $\orbM$ is
orientable, it follows from Proposition \ref{codim 1 is nice} that
$\orbD$ is in fact a suborbifold of $\orbM$ (in the  strong sense). 
\EndProof

\Proposition\label{if three}
Let $\orbM$ be a closed, orientable, very good
$3$--orbifold. Let $S\subset\mani \orbM$ be a
$2$--sphere which is transverse to $\Sing_\orbM$ 
and has
size\
at most $3$  (see \S\ref{size 'n' stuff}).
 Then $\orbS\doteq\orbi S $ is
$\pi_1$--injective in $\orbM$. 
\EndProposition

\Proof
According to Proposition \ref{when injective},
$\pi_1$--injectivity 
of $\orbS\doteq\orbi S $ in $\orbM$
is equivalent to the assertion
that if $C$ is any simple closed curve in
$S-(S\cap\Sing_\orbM)$ 
such that $C = \orbi C $
 bounds 
an orientable
discal suborbifold of $\orbM$,
then $C$ bounds a discal suborbifold of $\orbS $. But since $S$ is
a sphere of size\ at most $3$, 
for \emph{every}
simple closed curve $C$ in $S-(S\cap\Sing_\orbM)$ there is a disk
bounded by $C$ that contains at most one point of $\Sing_\orbM$, and hence
$C $ 
bounds a discal suborbifold of $\orbS $. 
\EndProof

\Proposition\label{cause gromov}
Let $\orbM$ be a closed, orientable hyperbolic $3$--orbifold. Observe
that by Proposition~\ref{Prop:LocalStructure3Orbifold}, $M \doteq \mani
\orbM$ is a closed topological $3$--manifold,
so that the Gromov norm $\| M \|$ is defined by \S\ref{Gromov norm}.
Then $\vol \orbM \geq \vtet \| M \|$ (where $ \vtet$  is defined as in \S\ref{notational}).
\EndProposition

\Proof
Since $\orbM$ is hyperbolic, it is very good. Let $\torbM$ be a
finite-sheeted covering space of $\orbM$ which is a hyperbolic
manifold, and let $d$ denote the degree of the covering map
$p\from \torbM\to\orbM$. 
Then  the underlying map of
$p$ is a degree-$d$ map from
$\torbM$ to $M$. 
It now follows from \eqref{degree and Gromov}, \eqref{genuine gromov}, 
and the multiplicativity  of volume under covers that
\[ 
d \cdot \vtet \|M\|
\: \leq \:  \vtet \| \torbM \| 
\: = \:  \vol \torbM
 \: = \: d \cdot \vol \orbM,
\]
which implies the result.
\EndProof

We recall that Gabai, Meyerhoff, and Milley \cite{GMM:SmallestCusped, milley} have identified the unique lowest-volume orientable hyperbolic  $3$--manifold. This is the Weeks manifold $M_{\rm Weeks}$, of volume $\Vweeks \doteq \vol (M_{\rm Weeks}) = 0.9427 \ldots$.
Combining their result with Proposition~\ref{cause gromov} and the properties of the Gromov norm yields the following corollary, which was stated in the introduction as Proposition~\ref{Prop:IntroGromov}.

\Corollary\label{Cor:SimpleVolBound}
Let $\orbM$ be a closed, orientable hyperbolic $3$--orbifold with $\vol
\orbM < \Vweeks$. Then 
$M \doteq \mani \orbM$ is homeomorphic to
 a connected sum of graph manifolds.
\EndCorollary

\Proof
Let us fix a smooth structure on
$M$. Assume that the smooth, orientable $3$--manifold $M$ is not a
graph manifold, so that by \S\ref{geometry and norms} it has at least
one 
hyperbolic piece, say $P$. Then
by combining the result due to Gabai, Meyerhoff, and
Milley mentioned above with Assertion \ref{all we need},
Proposition~\ref{cause gromov}, and the hypothesis of the present
corollary, we find
\[
\Vweeks \leq \vol P 
 \leq \vtet \|M\| \leq \vol \orbM < \Vweeks,
\]
a contradiction.
\EndProof

Corollary~\ref{Cor:SimpleVolBound}, which is much easier than Theorems \ref{Thm:FirstMain} and \ref{Thm:LinkMain}, can already serve as a proof of concept for the idea that upper bounds on volume imply restrictions on the topology of $M = \mani \orbM$. As explained in Section~\ref{Sec:ProofOutline}, this corollary is also a key ingredient in the proof of those theorems. Indeed, given a hyperbolic orbifold $\orbM$ of volume less than $\Vweeks$, Corollary~\ref{Cor:SimpleVolBound} implies that the underlying space $M$ is either small Seifert fibered, reducible, or toroidal. If $M$ is small Seifert fibered, we already have conclusion (i) of Theorems \ref{Thm:FirstMain} and \ref{Thm:LinkMain}. The case where $M$ is reducible is handled in Section~\ref{sphere section}, and the case where $M$ is toroidal is handled in Section~\ref{torus section}.

Before proceeding to an analysis of essential spheres and tori in $M$, we need to develop some background on orbifold $I$--bundles and books of $I$--bundles. This is carried out in Sections~\ref{bundle section}--\ref{Sec:OrbifoldBooks}.

\section{Orbifold fibrations}
\label{bundle section}

The next two sections of the paper are devoted to orbifold fibrations over orbifolds. This section develops general background in all dimensions, while Section~\ref{Sec:I-bundle} specializes to two-dimensional bases and one-dimensional fibers. The main result of this section is Proposition~\ref{FiberPreservingAction}, which provides a way to obtain an orbifold fiber bundle as the quotient of a manifold fiber bundle by a finite group.

Our definitions in this section are modeled on those of Cooper, Hodgson, and Kerckhoff \cite[Section 2.7]{CHK:OrbifoldBook}, although some of the terminology has changed and our treatment is more formal. Meanwhile, our definitions differ in content from those of 
Boileau, Maillot, and Porti \cite[Section 2.4]{BoileauMaillotPorti}. See Remark~\ref{Rem:NoGenericFiber} for more detail.

\Definition\label{Def:OrbifoldFiberBundle}
Suppose that $\orbP$ is  an orbifold with $2$--corners, and $\orbF$ and $\orbH$ are orbifolds (with
boundary).
A smooth orbifold
map
$q \from \orbP \to \orbH$
 is said to be an
\emph{orbifold $\orbF$--fibration} if
for every  $x \in \orbH$, there exist
\begin{itemize}
\item  
an open neighborhood $\orbU$ of $x$,
\item a manifold $\torbU$,
\item  a finite group $\frakG$, equipped with  actions on $\torbU$ and  $\orbF$, such that the diagonal action of $\frakG $ on
the product orbifold with $2$--corners $\torbU\times \orbF $ (see
\S\ref{why corners?}) is effective,
\item a 
smooth map
 $\phi\from \torbU\to\orbU$ 
which induces an orbifold diffeomorphism
  between $\torbU/\frakG $ and $\orbU$, and
\item a 
smooth map
$\xi$ from 
$\torbU\times \orbF$ to the open suborbifold 
with $2$--corners
 $q^{-1}(\orbU)$ of $  \orbP$
which induces an orbifold diffeomorphism
  between $(\torbU\times \orbF )/\frakG $  and $q^{-1}(\orbU)$,
\end{itemize}
such that 
the following diagram, where $\tq\from 
\torbU\times \orbF \to\torbU$ denotes projection to the first factor, 
commutes:
\begin{center}
\begin{tikzcd}
\torbU \times \orbF \arrow[r, "\xi"] \arrow[d, "\tq"] & q^{-1}(\orbU) \arrow[d, "q"] \\
\torbU \arrow[r, " \phi"] & \orbU
\end{tikzcd}
\end{center}

A quintuple $(\orbU,\torbU,\frakG,\phi,\xi)$ such that the bulleted
conditions above hold will be called a  \emph{local trivialization of}
the orbifold fibration $q$  \emph{near} the point $x$.
This notation leaves the actions implicit.

  We observe that
  if $q$ is an orbifold fibration, 
the 
underlying
map of  
$q$ 
is surjective.
We also observe
that if
$(\orbU,\torbU,\frakG,\phi,\xi)$ is a  local trivialization of
$q$, then
$\xi\from \torbU\times \orbF \to q^{-1}(\orbU)$
is a regular covering whose deck group is $\frakG$; and $\phi\from \torbU\to\orbU$ is a regular covering whose deck group is
some (possibly proper) quotient of
$\frakG$.

An orbifold $\orbF$--fibration $q\from \orbP\to\orbH$ may also be referred
to as an orbifold $\orbF$--fibration of $\orbP$ with base $\orbH$. An
orbifold 
with $2$--corners $\orbP$
equipped with an orbifold fibration over some base will be
called an \emph{orbifold bundle.}
\EndDefinition

\Custom{Standard local trivializations}
\label{Rem:StandardNbdFibration}

Suppose that $q\from\orbP\to\orbH$ is an orbifold
$\orbF$--fibration, and that
$(\orbU,\torbU,\frakG,\phi,\xi)$  is a local trivialization of $q$
near a point $x\in\orbP$.
Then for any connected open neighborhood $\orbU'\subset\orbU$ of $x$,
there is another local trivialization $(\orbU',\torbU',\frakG',\phi',\xi')$
near $x$ given by taking
$\torbU'$ to be a component of $\phi^{-1}(\orbU')$,  setting
$\frakG'=\Stab_\frakG(\orbU')$,  
$\phi'=\phi|\torbU'$, and $\xi'=\xi|(\torbU'\times \orbF)$, and
restricting  the action of $\frakG$ on $\torbU$ to obtain an action of $\frakG'$ on $\torbU'$.  (Since
the diagonal action of $\frakG'$ on $\torbU\times \orbF $ is effective, it
follows from Corollary~\ref{Bochner corollary} that its diagonal action  on
$\torbU' \times \orbF $ is also effective.)
 
It follows that there is a local trivialization $(\orbU,\torbU,\frakG,\phi,\xi)$ of $q$ such
that $\orbU$ is a standard
neighborhood of $x$ (as defined in \S\ref{Suborbifolds}), and $\torbU$
is connected.

According to an observation made in \S\ref{Orbifold covers}, the standard neighborhood $\orbU$ of $x$  has only one connected manifold cover up to
equivalence. Hence 
 there is a standard chart 
 $(U, D, \Phi,G)$ centered at $x$, where $U = \mani\orbU$,
the group $G$ is the quotient of $\frakG$ that acts effectively on
$\torbU$, and
the  open ball or half-ball $D = \orbi D$ may be
 smoothly and $\frakG$--equivariantly identified with $\torbU$ in such
 a way that $\Phi$ is the 
underlying map of 
$\phi$. Under this identification the given local
trivialization becomes 
$(\orbU,D,\frakG,\phi,\xi)$. A local trivialization of the latter
form, associated with a suitable standard chart $(U, D, \Phi,G)$
centered at $x$,  will be termed  \emph{standard}.

If it happens that
 the action of $\frakG$ on
$D = \torbU$ is effective, then $\frakG$ is also identified with $G_x$.
\EndCustom

\Custom{Fibers and vertical weak suborbifolds}
\label{vertical weak}
If $f\from \orbP\to\orbH$ is an  orbifold $\orbF$--fibration, 
with underlying
map $F\from \mani\orbP\to\mani\orbH$, then 
for any weak
suborbifold $\orbZ$ of $\orbH$, 
it follows from the definitions that
$Y\doteq F^{-1}(\mani\orbZ)$
is  the underlying set of a weak suborbifold-with-corners
of $\orbP$. Following the convention introduced in
  \S\ref{more orbifold notions},
we denote the weak suborbifold-with-corners
$\orbi Y$ of $\orbP$ by $f^{-1}(\orbZ)$.

A weak suborbifold 
$\orbN$ of  $\orbP$ is called \emph{vertical} if $\orbN =
f^{-1}(\orbZ)$ for a
weak
 suborbifold $\orbZ$ of $\orbH$.

In particular, 
for every $x\in\orbH$, we have a
well-defined 
vertical
weak suborbifold $f^{-1}(x)$ of $\orbP$. This
weak 
suborbifold 
is orbifold-diffeomorphic
to a quotient of $\orbF$ by a finite group action, and is called
the  \emph{fiber} over $x$.
\EndCustom

\Remark\label{Rem:NoGenericFiber}
Our definition of an orbifold $\orbF$--fibration is equivalent to the
one given in Cooper, Hodgson, and Kerckhoff \cite[Section 2.7]{CHK:OrbifoldBook}. Our definition is weaker than
the one given in Boileau, Maillot, and Porti \cite[Section 2.4]{BoileauMaillotPorti}, where it is
assumed that
near every point $x\in\orbH$ there is a local trivialization
$(\orbU,\torbU,\frakG,\phi,\xi)$ such that
the action of  $\frakG $ on $\orbU$ is
effective. In \cite{CHK:OrbifoldBook} the orbifold $\orbF$ is
referred to as a ``generic fiber,'' but we have avoided this term
because, in general, the map $q\from \orbP\to\orbH$ need not have any fiber
which is 
orbifold-diffeomorphic to $\orbF$. 
(For example, if $\torbF \to \orbF$ is a finite characteristic cover, then any
$\orbF$--fibration of a given orbifold $\orbP$ is also
an $\torbF$--fibration. In this situation, a product $\orbH \times
\orbF$ will not have any fiber diffeomorphic to $\torbF$.
This also shows that an orbifold $\orbF$--fibration
$q\from\orbP\to\orbH$, where $\orbF$, $\orbP$ and $\orbH$ are manifolds, is not necessarily
an $\orbF$--fibration in the usual sense.)
\EndRemark

\Example\label{Ex:SeifertManifold}
Let $M$ be an orientable Seifert fibered $3$--manifold. 
Identifying every Seifert fiber to a point produces a quotient $q \from M \to \orbH$, where $\orbH$ is a $2$--orbifold. Then $q \from M \to \orbH$ satisfies Definition~\ref{Def:OrbifoldFiberBundle} with $\orbF = \SS^1$. The regular Seifert fibers are exactly the preimages of nonsingular points in $\orbH$, while the singular Seifert fibers are the preimages of cone points in $\orbH$.
\EndExample

\Lemma\label{fibrations surject pi one}
Let $q \from \orbP \to \orbH$ be an orbifold $\orbF$--fibration. Then
for every generic orbifold path $\alpha$ in $\orbH$, and for every point $z$ in the fiber over the
initial point of $\alpha$, there is a lift of $\alpha$ to $\orbP$
having initial point $z$. Furthermore, if
$\orbP$ and $\orbF$ are connected, then the homomorphism $q_* \from \pi_1(\orbP) \to
\pi_1(\orbH)$ (see \S\ref{path subsection}) is surjective. 
\EndLemma

\Proof
We begin by proving the first assertion under the additional
assumption that the image of $\alpha$ is contained in $\orbU$ for some
local trivialization $(\orbU,\torbU,\frakG,\phi,\xi)$ of $q$. In this
case, we may select a point $(x,y)\in \torbU\times\orbF$ such that
$\xi(x,y)=z$. It follows from the assertion \ref{crucial fact} that $\alpha$
admits a lift $\talpha$ to the orbifold covering
$\phi\from\torbU\to\orbU$ having initial point $x$. Now let
$s\from\torbU\times\orbF\to\orbF$ denote the projection to the second
factor. Since $\torbU \times \orbF$ is a product object in the smooth category of
orbifolds (see \S\ref{maps continued}), there is an orbifold path
$\beta$ in $\torbU\times\orbF$ such that $\tq\circ\beta=\talpha$ and
$s\circ\beta$ is the constant path at $y$. Then $\xi\circ\beta$ is a lift of $\alpha$ to $\orbP$
having initial point $z$.

To prove the assertion in the general case, we write the given generic
orbifold path $\alpha$ as a composition $\alpha_1\cdots\alpha_n$,
where each $\alpha_i$ is a generic path whose image is contained in $\orbU_i$ for some
local trivialization $(\orbU_i,\torbU_i,\frakG_i,\phi_i,\xi_i)$ of
$q$. Using the special case  that has been
proved, we may recursively choose lifts $\talpha_1,\ldots,\talpha_n$
of $\alpha_1,\ldots,\alpha_n$ such that $\talpha_1$ has initial point
$z$, and for any $i$ with $1<i\le n$, the initial point of $\talpha_i$
is the terminal point of $\talpha_{i-1}$. Then the composition
$\talpha_1\cdots\talpha_n$ is defined and is a lift of $\alpha$ to $\orbP$
having initial point $z$.

To prove the second assertion, it suffices to show that if $\orbF$ is
connected, and $z$ is a
point of $\orbP$ such that $w\doteq q(z)$ is nonsingular, then $q_* \from \pi_1(\orbP,z) \to
\pi_1(\orbH,w)$ is surjective (see \S\ref{path subsection}). 
By Assertion \ref{homotop to generic}, any element $g$ of
$\pi_1(\orbH,w)$ is represented by a generic loop
$\alpha$ based at $w$. By the first assertion of the present lemma,
the path $\alpha$ admits a lift $\talpha$ to $\orbP$ whose initial point is
$z$. Since $\alpha$ is a loop, the terminal point $z'$ of $\talpha$
lies in the same fiber $\orbF_0$ as $z$. Since $\orbF$ is connected, $\orbF_0$ is
connected, and hence by an observation made in \S\ref{path subsection}
there is an orbifold path $\gamma$ in $\orbF_0$ with initial point $z'$
and terminal point $z$. If we set
$h=[\talpha\cdot\gamma]\in\pi_1(\orbP,z)$, then $q_*(h)=g$.
\EndProof

\Lemma\label{Lem:InducedCover}
Let $q \from \orbP \to \orbH$ be an orbifold $\orbF$--fibration, where
$\orbP$ and $\orbF$ are connected. Then, for every connected cover
$\theta  \from \torbH \to \orbH$, there 
exist an 
orbifold $\orbF$--fibration $\tq \from \torbP \to \torbH$ and a cover $\upsilon  \from \torbP \to \orbP$ such that $\theta $ and $\upsilon $ have the same degree and the following diagram commutes:
\[
\begin{tikzcd}
\torbP \arrow[r, "\upsilon "] \arrow[d, "\tq"] & \orbP \arrow[d, "q"] \\
\torbH \arrow[r, "\theta "] & \orbH
\end{tikzcd}
\]
Such an orbifold $\orbF$--fibration $\tq \from \torbP \to
\torbH$ will be said to be \emph{induced} by the orbifold
$\orbF$--fibration $q \from \orbP \to \orbH$.
\EndLemma

\begin{proof}
Let $\upsilon  \from \torbP \to \orbP$ be the cover of
$\orbP$ corresponding to the subgroup $q_*^{-1}(\theta _*(\pi_1
\torbH))$. Now, the lifting criterion (see \S\ref{path subsection}) ensures that $q \circ \upsilon  \from \torbP \to \orbH$ lifts to a smooth map $\tq \from \torbP \to \torbH$.

Since  $\orbF$ is connected, 
it follows from Lemma~\ref{fibrations surject pi one} that
$q_* \from \pi_1 \orbP \to \pi_1 \orbH$ is surjective. Thus 
\[
\big[ \pi_1 \orbH : \theta _*(\pi_1 \torbH) \big] = \big[ \pi_1 \orbP : q_*^{-1}(\theta _*(\pi_1 \torbH)) \big],
\]
hence $\theta $ and $\upsilon $ have the same degree. 

It remains to check that $\tq \from \torbP \to
\torbH$ is an orbifold $\orbF$--fibration. This can be shown
by verifying that for every $\tx \in \torbH$, a
standard local trivialization of $q$ near $x = \theta (\tx)
\in \orbH$ 
gives rise to
a standard local trivialization of $\tq$ near $\tx$.
\end{proof}

The  material in \S\S\ref{what's fiber preserving} and \ref{collars and actions} will be needed for
the statement and  proof of Proposition
\ref{FiberPreservingAction}.

\Custom{Fiber-preserving actions}
\label{what's fiber preserving}
Suppose that $r\from M\to Z$ 
is a smooth
fibration, where $M$ and $Z$ are manifolds (with boundary). A
self-diffeomorphism $g$ of $M$ is said to \emph{preserve the
  fibration $r$} if
$g\cdot F$ is a fiber of $r$ for each fiber $F$ of $r$.
\EndCustom

\Custom{Boundary collars and group actions}\label{collars and actions}
A  \emph{parametrized boundary collar} for a manifold 
$M$ is defined to be an embedding $c\from (\bdy
M) \times [0,1) \to M$ such that $c(x,0)=x$ for each $x\in\bdy
M$. One can construct a parametrized boundary collar for any manifold 
$M$ by fixing a vector field
$Y$ on $M$ which is
inward-pointing at every point of $\bdy M$,
integrating $Y$, and restricting the resulting flow to a
small neighborhood of $\bdy M$. 

If we are given an action by diffeomorphisms of a finite
group $G$ on $M$, then by averaging 
a vector field
$Y_0$  which is
inward-pointing at every point of $\bdy M$, we obtain a
$G$--invariant vector field $Y$ with the same property; the construction
referred to above, with this choice of $Y$, then gives
a parametrized boundary collar $c$ for $M$
which is {\it
$G$--equivariant} in the sense that
$g\cdot c(x,t)=c(g\cdot x,t)$ for each $(x,t)\in (\bdy
M) \times [0,1) $. 

Now suppose that we are given a smooth fibration $r \from M\to Z$ of the
compact manifold $M$ over a closed manifold $Z$, whose fiber is a
compact manifold with boundary, and that the action of $G$ on $M$
preserves the fibration (\S\ref{what's fiber preserving}).
Then the 
inward-pointing vector field $Y_0$ used in that construction given
above may
be taken (by a partition-of-unity argument) to be 
everywhere tangent to the fibers of $r$.
With this choice of $Y_0$, the
$G$--equivariant parametrized boundary collar $c_0$ given by the
construction will be
compatible with the fibration $r$ in the sense that $c_0(x\times[0,1))$ is
contained in a fiber of $r$ for each $x\in\bdy M$.
\EndCustom

\Proposition\label{FiberPreservingAction}
Suppose that $r\from M\to Z$ 
is a smooth
fibration with fiber $F$, where $M$ is a compact manifold with
  $2$--corners and $F, Z$ are 
compact
manifolds (with boundary).
Let $G$ be a finite group of  diffeomorphisms of $M$, and
suppose that each $g\in G$ sends fibers to fibers, 
so that there is an induced action of $G$ on  $Z$. 
Let $q\from M/G  \to Z/G$ denote the
unique orbifold map such that the diagram
\[
\begin{tikzcd}
M \arrow[r, ""] \arrow[d, "r"] & 
M/G \arrow[d, "q"] 
 \\
Z \arrow[r, ""]
& Z/G
\end{tikzcd}
\]
commutes, where the horizontal arrows indicate quotients by the $G$--action.
Then $q$ 
is an
orbifold $F$--fibration. 
\EndProposition

\Proof
We first consider the case in which $Z$ and $F$ are both closed, so that $M$ is closed.

We fix a $G$--invariant Riemannian metric on $M$.

We  set $\orbN=M/G$ and $\orbH=Z/G$, and we denote by $p$ the quotient map from
$Z$ to $\orbH$. We set $m=\dim M$
and $n=\dim Z$, so that $F$ is a manifold of dimension $k\doteq m-n$. 

Let a point $x\in\orbH$ be given. Let us choose a point $\tx\in
p^{-1}(x)\subset Z$. Set $F_0=r^{-1}(\tx)\subset M$, so that $F_0$ is
orbifold-diffeomorphic to $F$. 

Let $J\le G$ denote the
stabilizer of $\tx$ 
under the action of $G$ on $Z$, which is also the stabilizer of $F_0$ under the action of $G$ on 
$M$.
For a small enough
positive number $\epsilon$, 
the proof of the standard tubular neighborhood theorem for smooth
submanifolds shows that
the $\epsilon$--neighborhood $V=V_\epsilon$
of $F_0$ in $M$ (with respect to the chosen Riemannian metric) is the
domain of a well-defined nearest-point map $t\from V\to F_0$, and $t$ is a
(manifold) fibration of $V$ with base $F_0$, whose fiber is an open ball of
dimension $m-k=n$. The $G$--invariance of the
metric on $M$ implies that $V$ is $J$--invariant, and that $t$ is a
$J$--equivariant 
map. Since the action of $G$ on $Z$ is induced by its
action on $M$, the map $r|V\from V\to Z$ is also $J$--equivariant. Hence, if
the product $Z\times F_0$ is endowed with the product $J$--action, we may
define a $J$--equivariant map $\eta\from V\to Z\times F_0$ by
$\eta(z)= (r(z),t(z))$. 

Note that $\eta$ is a smooth map between $m$--manifolds. We claim:

\Claim\label{invertible}
For every point 
$z\in F_0$, 
the linear map $d\eta_z\from T_zV\to
T_{\eta(z)}(Z\times F_0)$ is invertible.
\EndClaim

To prove \ref{invertible}, we let $D$ denote the $n$--ball
$t^{-1}(z)\subset V$. 
We write $T_z V=A\oplus B$, where $A$ and
$B$ are canonically isomorphic to $T_z D$ and $T_{t(z)} F_0$  respectively.
We also write $T_{\eta(z)} (Z\times F_0) = A' \oplus B$, where $A'$ and
$B$ are canonically isomorphic to $T_{r(z)} Z$ and $T_{t(z)} F_0$
respectively. 
Since $t|F_0$ is the identity map and $r|F_0$ is constant, 
$d\eta_z$ maps $B$ isomorphically onto $B$.

Since $r$ is a fibration, $dr_z\from T_zV\to T_{r(z)}Z$ is surjective. But
$dr_z|B$ is the zero map, and hence $dr_z|A \from  A \to A' = T_{r(z)}Z$ is surjective. Since
$A$ and $A'$ are both $n$--dimensional, this means that $dr_z|A$
maps $A$ isomorphically onto  $A'$. But $t|D$ is
constant, and hence $dt_z|T_zD=0$; it follows that $d\eta_z$ maps $A$ isomorphically onto $A'$. 
Thus \ref{invertible} is established.

It follows from \ref{invertible} and the compactness of $F_0$ that some
open neighborhood of $F_0$ in $V$ is mapped diffeomorphically by $\eta$
onto an open neighborhood of $\{\tx\} \times F_0$ in $Z\times F_0$. Again
using the compactness of $F_0$, we see that the latter neighborhood
contains one of the form $\torbU\times F_0$ where $\torbU$ is an open neighborhood
of $\tx$ in $Z$; and since $J$ fixes $\tx$, we may assume after
replacing $\torbU$ by a possibly smaller neighborhood of $\tx$ that $\torbU$ is
$J$--invariant, and is disjoint from its image under every element of $G-J$. Some open neighborhood $W$ of $F_0$ is mapped
diffeomorphically onto the $J$--invariant open set $\torbU\times F_0$, and the
$J$--equivariance of $\eta$ then implies that $W$ is $J$--invariant and
that $\eta|W\from W\to \torbU\times F_0$ is $J$--equivariant. 

Let $\psi\from \torbU\times F_0\to\orbN$ denote the composition of the
diffeomorphism $(\eta|W)^{-1}\from \torbU\times F_0\to W\subset M$ with the
quotient map $M\to M/G=\orbN$. Then the conditions of
Definition \ref{Def:OrbifoldFiberBundle} hold with $\orbN$, $F$,
$p(\torbU)$, $J$ and $p|\torbU$ playing the respective roles of
$\orbP$, $\orbF$, $\orbU$, $\frakG $ and $\phi$.
This shows that $q$ is an orbifold $F$--fibration,
and completes the proof in the case where $Z$ and $F$ are both closed.

This argument does not quite work in  the case in which $Z$ is closed but
$\bdy F\ne\emptyset$, because for an arbitrary $G$--invariant
metric on $M$, the proof of the tubular neighborhood theorem does not show
that the map $t$ appearing in the argument is a fibration for an
arbitrary point $x\in\orbH$; to guarantee this we must exercise a little more care in
choosing the $G$--invariant metric. 
For this purpose we use 
the construction of \S\ref{collars and actions} to produce a
$G$--equivariant parametrized boundary collar in $M$ which is
compatible with the fibration $r$, in the sense defined in that subsection.

Now define a metric $\rho_0$ on 
$c_0((\bdy
M)\times[0,1))$ 
by pushing forward, via $c_0$, the product of some
$G$--invariant metric on $\bdy M$ with the standard metric on
$[0,1)$. There is a metric $\barrho_0$ on $M$ which agrees with
$\rho_0$ on some neighborhood of $\bdy M$, and averaging
$\barrho_0$ via the action of $G$ gives a $G$--invariant metric $\rho$ on $M$ which
again agrees with
$\rho_0$ on some neighborhood of $\bdy M$. This implies that
the restriction of the metric $\rho$ to some neighborhood of $\bdy M$ can itself be obtained by pushing forward 
a product metric via some 
parametrized boundary collar $c\from (\bdy
M)\times[0,1)\to M$ which is  compatible with the fibration $r$. This
is enough to guarantee that the map
$t$ appearing in the argument given above is a fibration for an
arbitrary point $x\in\orbH$. This  completes the proof in the case in which $Z$ is closed but
$\bdy F$ may be non-empty.

To prove the proposition in the general case, we consider certain doubled
manifolds $DM$ and $DZ$. Here $DM$ is obtained by doubling $M$ along
the manifold (with boundary) 
$E\doteq r^{-1}(\bdy Z)$.
The manifold $DM$ is the union of two manifolds with $2$--corners
$M_1$ and $M_2$,
each of which is identified diffeomorphically with $M$ 
and therefore
has a submanifold $E_i$ identified with $E$; and we have
$M_1\cap M_2=
E_1=E_2$.
The manifold $DZ$ is the double of $Z$ along its full boundary; it is
the union of two submanifolds
$Z_1$ and $Z_2$,
each of which is identified diffeomorphically with $Z$, and $Z_1\cap Z_2=\bdy
Z_1=\bdy Z_2$. There is a doubled map $Dr\from DM\to DZ$ which, for
each $i$,
restricts to a map from  $M_i$ to $Z_i$ that is identified with $r$;
and there is a doubled action of $G$ on $DM$ which, for
each $i$,
restricts to an action on  $M_i$ that is identified with the
action of $G$ on $M$. The hypotheses of the proposition continue to
hold when $M$, $Z$ and $r$ are replaced by  $DM$, $DZ$ and $Dr$, and
$G$ acts on $DM$ by the doubled action. By the case of the proposition
already proved, the map $Dq\from 
(DM)/G  \to (DZ)/G$ induced by $Dr$ is an orbifold fibration. Now
since $Z_1/G$ is a suborbifold of $(DZ)/G$, and
$M_1/G=(Dq)^{-1}(Z_1)$, the restricted map $q_1=Dq|(M_1/G)\from M_1/G\to
Z_1/G$ is also an orbifold fibration. The conclusion follows.
\EndProof

\section{Orbifold $I$--bundles}\label{Sec:I-bundle}

This section focuses attention on orbifold fiber bundles whose bases are compact $2$--orbifolds and whose fibers are compact $1$--orbifolds. For most of the section, we will furthermore assume that the total space $\orbP$ is orientable.

We will prove several structural results about the total space and singular locus of such an $I$--bundle $\orbP \to \orbH$. Notably, Proposition~\ref{Prop:IBundleSingLocus} describes a detailed dictionary between the singular locus of $\orbP$ and the singular locus of the base $\orbH$, and proves in particular that any node in $\Sing_\orbP$ is dihedral. Proposition~\ref{Prop:IBundleTotalSpace} proves that the underlying space $P = \mani \orbP$ is homeomorphic to a (manifold) $I$--bundle over a surface.

\Custom{Orbifold $I$--fibrations}\label{Def orbifold I-bundle}
By specializing the notion of an orbifold $\orbF$--fibration (resp.\ orbifold $\orbF$--bundle) defined in \S\ref{Def:OrbifoldFiberBundle} to the case where $\orbF$
is the $1$--manifold $I=[0,1]$, we obtain the notion of an
\emph{orbifold $I$--fibration} 
(resp.\ \emph{orbifold $I$--bundle}).

If $q \from \orbP\to\orbH$ is an $I$--fibration,
the \emph{vertical boundary of $\orbP$}, denoted $\vertbdy \orbP$, is
defined to be
$q^{-1}(\bdy \orbH)$,
which according to  \S\ref{vertical weak} is  \emph{a
priori} a weak
suborbifold of $\orbP$. But since $\bdy\orbH$ has codimension $1$
in $\orbH$, the codimension of $\vertbdy\orbP$ in $\orbP$ is also
equal to $1$; hence, by Proposition \ref{codim 1 is nice},
$\vertbdy\orbP$ is a suborbifold of $\orbP$ (in the strong sense).
 The \emph{horizontal boundary of $\orbP$},
denoted $\horizbdy \orbP$, is the $\bdy I$--bundle over $\orbH$
obtained by restricting $q$. 
It follows that every non-interior point of $\orbP$ belongs to either $\vertbdy \orbP$ or $\horizbdy \orbP$, and the intersection $\vertbdy \orbP \cap \horizbdy \orbP$ is exactly the set of $2$--corner points of $\orbP$.
If $\orbP$ is orientable, then  $\horizbdy
\orbP$ is the orientation cover of $\orbH$
(see \S\ref{Orbifold covers}).
In particular, if we assume  $\orbP$ is orientable, then $\horizbdy \orbP$ is disconnected if and only if $\orbH$ is orientable.
\EndCustom

\Example\label{Ex:FiberedDisk}
Let $A = \SS^1 \times I$.
Then the projection $q \from A \to \SS^1$ is an $I$--fibration.
Now, suppose $\lambda \from A \to A$ is an involution acting by reflection on both $\SS^1$ and $I$.
The quotient $\orbP = A / \langle \lambda \rangle$ is a $2$--orbifold
with underlying space $\BB^2$ and two cone points of order
$2$. Furthermore, $\orbP$ admits an $I$--fibration over 
$\orbJ \doteq \SS^1 / \langle \lambda \rangle$. Compare \cite[Example
2.34]{CHK:OrbifoldBook}. 
\EndExample

\Custom{Fibrations and cleanliness}
\label{clean fibration}
If a compact, orientable $3$--orbifold $\orbP$ is equipped with an
$I$--fibration, then $(\orbP,\vertbdy\orbP)$ is a \clean\ pair. Indeed,
we observed in \S\ref{Def orbifold I-bundle} that $\vertbdy\orbP$ is
a suborbifold of $\orbP$ (in the strong sense), establishing Condition \eqref{Clean2}. Conditions \eqref{Clean3} and \eqref{Clean4}
of the definition in \S\ref{Clean pairs}  follow from the
definition of a fibration.
\EndCustom

\Proposition\label{fibered annular}
Let $\orbA$ be a compact connected $2$--orbifold which is an $I$--bundle 
over a closed $1$--orbifold $\orbJ$. Then $\orbA$ is annular.
\EndProposition

\Proof
Since the fibers of $\orbA$ are quotients of $I$, and have non-empty
boundary by \S\ref{folded},
we have
$\bdy \orbA \neq \emptyset$.
Thus it follows from \cite[Theorem 2.22]{CHK:OrbifoldBook} that $\orbA$ is very good.

Consequently, $\orbA$ 
is covered by a compact orientable $2$--manifold $\torbA$ that is itself an $I$--bundle over a closed base. 
Thus $\chi(\torbA) = 0$, 
hence $\torbA$ is an annulus, and the result follows.
\EndProof

We remark that in Proposition~\ref{fibered annular}, if $\orbA$ is orientable, then $\orbA$ is either an annulus, or the quotient of an annulus by an involution as in Example~\ref{Ex:FiberedDisk}. We will not need this stronger conclusion.

\Corollary\label{Cor:VertBoundaryStructure}
Let $\orbP$ be a compact $3$--dimensional orbifold $I$--bundle.
Then every component  of $\vertbdy \orbP$  is an annular $2$--orbifold. 
\EndCorollary

\Proof
The $I$--fibration $q \from \orbP \to \orbH$ induces an $I$--fibration of each
component of $\vertbdy \orbP$ over a component of $\bdy \orbH$. 
Observe that each component of $\bdy \orbH$ is closed. The conclusion now follows from
Proposition  \ref{fibered annular}.
\EndProof

\Lemma\label{Lem:SameGroup}
Let  $q \from \orbP \to \orbH$ be an orbifold $I$--fibration, 
where $\orbP$ is an orientable $3$--orbifold.
Let a point $x \in \orbH$ be given and let
$(\orbU,\torbU,\frakG,\phi,\xi)$ be a standard local trivialization of
$q$  near $x$.
Then an element $g$ of $\frakG$ reverses the orientation of $I$ if and only if $g$ reverses the orientation of $\torbU$. 
Furthermore, $\frakG$ acts effectively on $\torbU$, and can therefore be identified
with the local group  $G_x^\orbH$ by  \S\ref{Rem:StandardNbdFibration}.
In particular, $\frakG$ is either cyclic or dihedral by Proposition
\ref{Prop:LocalStructure2Orbifold}.
\EndLemma

\Proof
Since $\orbP$ is orientable, the diagonal action of $\frakG$ 
must preserve the orientation of the product $\torbU \times I$. This
implies the first conclusion of the lemma.

Now suppose that some non-trivial element
$g$ of $\frakG$ acts trivially on $\torbU$. Since the action of
$\frakG$ on $\torbU\times I$ is effective, $g$ must act non-trivially
on $I$. But any non-trivial periodic diffeomorphism of $I$ is conjugate
to a reflection; hence  $g$ reverses the orientation of $I$; as it acts
trivially on  $\torbU$, this contradicts the first assertion.
This shows that
$\frakG$ acts effectively on $\torbU$. As indicated in the statement, the
remaining assertions follow from 
\S\ref{Rem:StandardNbdFibration} and Proposition
\ref{Prop:LocalStructure2Orbifold}.
\EndProof

\Convention\label{no more Gamma}
According to Lemma \ref{Lem:SameGroup}, if  $q \from \orbP \to \orbH$ is an orbifold $I$--fibration, 
where $\orbP$ is an orientable $3$--orbifold, and if
$(\orbU,\torbU,\frakG,\phi,\xi)$ be a standard local trivialization of
$q$  near a point $x \in \orbH$, then $\frakG$ and $\torbU$ are canonically
identified with the local group $G_x$ and a disk or half-disk $D$. In such a
situation we shall therefore generally denote the  local
trivialization in question by $(\orbU,D,G_x,\phi,\xi)$.
\EndConvention

\Lemma\label{Lem:FoldedIntervals}
Let  $q \from \orbP \to \orbH$ be an orbifold 
$I$--fibration, where $\orbP$ is an orientable $3$--orbifold.
Let a point $x\in\orbH$ be given.
If $x \in  \mirror \orbH$, 
the fiber $q^{-1}(x)$    is a folded \arc.
If $x \notin  \mirror \orbH$, the
 fiber $q^{-1}(x)$  is diffeomorphic to $I$.
\EndLemma

\Proof
By \S\ref{Rem:StandardNbdFibration} and \S\ref{no more Gamma},
we may fix a standard local trivialization 
$(\orbU, D, G_x,\phi,\xi)$ near $x$.
The fiber $q^{-1}(x)$ is diffeomorphic to $\big( \{x\}\times I \big)/G_x$, and
is therefore a folded arc if and only if the action of the local group $G_x$ on
$I$ is not orientation-preserving. 
By  Lemma \ref{Lem:SameGroup}, this is in turn equivalent to saying
that the action of $G_x$ on
$D$ is not orientation-preserving, i.e.\ that 
$x \in \mirror \orbH$. 
\EndProof

\Custom{$I$--bundles have nonsingular corners}\label{why nonsingular}
Recall that if a compact,
orientable $3$--orbifold $\orbP$ admits an $I$--fibration, then $(\orbP,
\vertbdy \orbP)$ is a \clean\ pair
by
\S\ref{clean fibration}, 
and hence that $\orbP$ has nonsingular $2$--corners by
\S\ref{Clean pairs}.

Recall also from 
\S\ref{what's a node}
that in a compact, orientable $3$--orbifold with 
nonsingular
corners, every singular stratum is either a
node or a cone stratum. 
For the  $I$--bundle case, we have the following stronger statement.
\EndCustom

\Proposition\label{Prop:IBundleSingLocus}
Let  $q$
be an $I$--fibration of a compact, orientable $3$--orbifold $\orbP$ over
a $2$--orbifold $\orbH$.
Then for every singular stratum $\sigma$ of $\orbP$, exactly one of the
following alternatives holds:
 \begin{enumerate}[\:\:(i)]
 \item\label{Itm:Vert} 
$\sigma$ is a cone stratum of some order $n$, and its closure in $\orbP$ is a
fiber over
a cone point of order $n$ or a bi-mirror point of order $2n$. Such a stratum $\sigma$ will be
called \emph{vertical}. 
\item\label{Itm:Horiz} 
$\sigma$ is a cone stratum of order $2$,  its image under $q$ is a
mirror stratum $\tau$ of  
$\orbH$, and $q|\sigma$ is a diffeomorphism between $\sigma$ and
$\tau$ (which are smooth manifolds by 
Assertion \ref{was too}).
Such a stratum $\sigma$ will be
called  \emph{horizontal}. 
\item\label{dihedral alternative}
$\sigma=\{u\}$
for some dihedral node $u$,
which is a terminal node of two oriented horizontal $1$--strata (of
order $2$), 
and of one oriented vertical $1$--stratum which
projects to a bi-mirror point of $\orbH$. 
\end{enumerate}
\EndProposition

It follows from Proposition \ref{Prop:IBundleSingLocus} that a singular
stratum $\sigma$ of $\orbP$ is a  vertical stratum if and only if
$\overline{\sigma}$ is a vertical weak suborbifold of $\orbP$ (in the
sense of \S\ref{vertical weak}). However, a vertical stratum $\sigma$
need not itself be a vertical weak suborbifold, as its closure may
contain a dihedral node. This mild abuse of language will
not lead to any ambiguities.

\begin{proof}[Proof of Proposition \ref{Prop:IBundleSingLocus}]
We denote by $Q\from\mani\orbP\to\mani\orbH$ the 
underlying map of 
$q$.
    
For each point $x \in \orbH$, 
we fix a standard local trivialization 
$(\orbU_x, D_x, G_x,\phi_x,\xi_x)$ near $x$, following the conventions of \S\ref{Rem:StandardNbdFibration}
and \S\ref{no more Gamma}.
By
\S\ref{Rem:StandardNbdFibration}, there is a standard chart $(U_x,D_x,\Phi_x,G_x)$ centered at $x$, where
 $U_x=\mani\orbU_x$.
For each $x \in \orbH$, we set $\orbW_x = q^{-1}(\orbU_x)$, an open suborbifold of $\orbP$.

According to Proposition \ref{Prop:LocalStructure2Orbifold}, every
point of $\orbH$ is a nonsingular point, a cone point, a mirror point
or a bi-mirror point. Hence for a singular stratum $\sigma$ of $\orbP$,
at least one of the following alternatives must hold:
\begin{enumerate}
\item\label{Itm:SomeNonsing} For some $z\in\sigma$, the point $Q(z)$ is nonsingular.
\item\label{Itm:SomeCone} For some $z\in\sigma$, the point $Q(z)$ is a cone point.
\item\label{Itm:EveryMirror} For \emph{every} $z\in\sigma$, the point $Q(z)$ is a mirror point.
\item\label{Itm:SomeBiMirror} For some $z\in\sigma$, the point $Q(z)$ is a bi-mirror
  point.
\end{enumerate}

If \eqref{Itm:SomeNonsing} holds, and if we choose a point $z\in\sigma$ such that $x\doteq
Q(z)$ is nonsingular, then $G^\orbH_x=\{1\}$. Hence $\orbW_x$, 
which by Definition~\ref{Def:OrbifoldFiberBundle} is orbifold-diffeomorphic to 
 $(D_x\times \orbF )/G_x$, is a manifold, and $z\in
  \orbW_x$ is a nonsingular point; this is a contradiction,
  since $\sigma$ is a singular stratum. Thus Case \eqref{Itm:SomeNonsing} cannot occur.

If \eqref{Itm:SomeCone} holds, and if we choose a point $z\in\sigma$ such that $x\doteq
Q(z)$ is 
a cone point of order 
$n > 1$, then 
$G_x \cong \ZZ / n\ZZ$ acts on $D_x$
by
rotations. 
Since every $g \in G_x$ preserves the orientation
  of $D_x$, Lemma~\ref{Lem:SameGroup} implies that $G_x $ acts
  trivially on the $I$ factor of $D_x \times I$. Hence  $G_x $ acts
  on the cylinder $D_x \times I$ 
by rotations about the axis $\{0\}\times I$. Thus the singular set
of  $\orbW_x$ is the closed arc  $Q^{-1}(x)$. It follows that
$Q^{-1}(x)$ is  a vertical cone stratum of order $n$ of $\orbP$, in the sense defined in
the statement of the Proposition. Since this vertical stratum contains $z$,
it must coincide with $\sigma$. 
Thus Alternative \eqref{Itm:Vert} of the conclusion of the proposition holds in this case.

If \eqref{Itm:EveryMirror} holds, then since $\sigma$ is connected, $Q(\sigma)$ must be
contained in a single mirror stratum $\tau$ of $\orbH$.
Define $S = \barsigma\cap Q^{-1}(\tau)$, and observe that $\sigma \subset S$, hence $S \neq \emptyset$.

Consider an arbitrary point $w$ of
$S$. Since $v\doteq Q(w)\in\tau$,
the point $v$ is a mirror point. Hence
the generator $g$ of $G _v\cong \ZZ/2\ZZ$ acts on $D_v$ by a
  reflection; by  Lemma~\ref{Lem:SameGroup}, $g$ acts on $I$ by a reflection as well. The
  $3$--dimensional action of $g$ is
a composition of reflections through two perpendicular planes,
  which is a $\pi$--rotation about a horizontal open line segment in $D_v \times I$.
Thus the singular locus of $\orbW_v$ is an open arc $A$ which
contains $w$ and is contained
in an order-$2$ cone stratum $\sigma'$ of $\orbP$; and $Q$ maps $A$
diffeomorphically onto an arc in the mirror stratum  of $\orbH$
containing $v$, namely $\tau$. Furthermore, since $A$ is the full singular locus of $\orbW_v$, and $Q|A$ is one-to-one, the only singular point in the fiber $Q^{-1}(v)$ is $w$ itself.

In particular, $A$ is a
neighborhood of $w$ relative to  $\Sing_\orbP$; since $w\in
S\subset\barsigma$, the arc $A$ must meet $\sigma$. Hence
$\sigma'=\sigma$, which shows that $\sigma$  is a cone stratum of order
$2$, and that $w\in\sigma$.
Since $w$ was an arbitrary point of $S$, we have shown that
$S\subset\sigma$. 
We have already observed the reverse inclusion $\sigma\subset S$. Since each point $w \in S = \sigma$ is the only singular point of its fiber, it follows that $Q | \sigma$ is one-to-one and  is therefore a diffeomorphism to its image.

It follows that $Q(\sigma)$ is an open subset of $\tau$. But since $\sigma=S=\barsigma\cap
Q^{-1}(\tau)$, it follows that $Q(\sigma)=\tau\cap Q(\barsigma) = Q(\barsigma)$, hence $Q(\sigma)$ is also closed in $\tau$. We conclude that $Q | \sigma$ is a diffeomorphism onto the stratum $\tau$, verifying that $\sigma$ is a horizontal stratum in the sense defined in the statement
of the Proposition; 
that is, Alternative \eqref{Itm:Horiz} of the conclusion holds.

Finally, suppose that \eqref{Itm:SomeBiMirror} holds.
Let us choose a point $z\in\sigma$ such that $x\doteq Q(z)$ is a
bi-mirror point. Then $G_x$ is a dihedral group of 
some finite order $2n$. 
There is a cyclic index--$2$ subgroup $C$ of $G_x$ that acts
on $D_x$ by rotations; since the action of $C$ preserves the
orientation of $D_x$, it must act trivially on $I$ by Lemma
\ref{Lem:SameGroup}. Thus $C$ acts on $D_x \times I$ by rotations about a vertical axis. On the other hand, every element of $G_x - C$
reverses the orientation of $D_x$, and by Lemma
\ref{Lem:SameGroup} it must also reverse the orientation of
$I$. Hence such an element acts on $D_x \times I$ as a product of
two reflections, which is  a $3$--dimensional $\pi$--rotation about a
horizontal axis. From this one deduces that  $G_x$ acts on $D_x \times I$ as the orientation-preserving symmetries of a prism over a regular $n$--gon.
The singular set 
of the orbifold $\orbW_x$ is the union of a
closed arc $\alpha=Q^{-1}(x)$ with two half-open arcs $\beta_1$ and
$\beta_2$, such that
$\alpha\cap\beta_1=\alpha\cap\beta_2=\beta_1\cap\beta_2=\{ u \}$ for
some point $u $. If we set $\alpha'=\alpha-\{u \}$, and $\beta_i'=\beta_i-\{u \}$ for $i=1,2$, then the singular strata of  $\orbW_x$ are
$\{u \}$, $\alpha'$, $\beta_1'$ and $\beta_2'$. The map
$Q$ 
maps each $\beta_i'$ diffeomorphically onto an open arc in a mirror
stratum of $\orbH$.

Since $\alpha'$ and $\{u \}$ do not meet the frontier of $\orbW_x$ in $\orbP$, they are singular strata of $\orbP$. 
Each $\beta_i$ is
contained in a singular stratum $\gamma_i$ of $\orbP$. 
Since $z\in Q^{-1}(x)=\alpha=\alpha'\cup\{u \}$, 
the stratum $\sigma$
is equal to either $\{u \}$ or $\alpha'$, which is a vertical stratum over a
bi-mirror point in the language of the proposition.

Let us summarize what we have shown so far in this case:

\Claim\label{what if four}
If \eqref{Itm:SomeBiMirror} holds for a given singular stratum $\sigma$ of $\orbP$, then
$\sigma$ is contained in a single fiber $Q^{-1}(x)$ for a bi-mirror point $x \in \orbH$.
Furthermore, either
\begin{enumerate}[\:\:(a)] 
\item $\sigma$ is a vertical stratum over $x$, or
\item 
$\sigma = \{u \}$, and $u $ has a neighborhood in
$\Sing_\orbP$ which is the union of a
closed arc $\alpha=Q^{-1}(x)$ with two half-open arcs $\beta_1$ and
$\beta_2$, such that
$\alpha\cap\beta_1=\alpha\cap\beta_2=\beta_1\cap\beta_2=\{u \}$. Furthermore,
if we set $\alpha'=\alpha-\{u \}$, and $\beta_i'=\beta_i-\{u \}$ for $i=1,2$, then 
$\alpha'$ is a vertical stratum of $\orbP$, while each $\beta_i'$ is
contained in a
stratum and is mapped diffeomorphically by $Q$ onto an open arc in a mirror
stratum of $\orbH$.
\end{enumerate}
\EndClaim

To complete the proof of the proposition, we
note that 
 Alternative (a) of \ref{what if four} implies Alternative \eqref{Itm:Vert} of the
conclusion 
of the proposition.
Now suppose that (b) holds, and
for $i=1,2$, let $\gamma_i$ denote the stratum of $\orbP$ that
contains $\beta_i$. (Note that $\gamma_1$ and $\gamma_2$ are not
necessarily distinct.)
For each $\gamma_i$, one of the alternatives \eqref{Itm:SomeNonsing}--\eqref{Itm:SomeBiMirror} must hold.
We have seen that \eqref{Itm:SomeNonsing} cannot occur, and that \eqref{Itm:SomeCone} implies
that the given stratum $\gamma_i$ is 
contained in a fiber, which is
impossible since $q$ maps $\beta_i'$ diffeomorphically to its image. If
\eqref{Itm:SomeBiMirror} holds 
for $\gamma_i$, then  \ref{what if four} likewise implies that $\gamma_i$ is contained in a fiber, which we have seen is impossible.
Hence each stratum $\gamma_i$ must satisfy \eqref{Itm:EveryMirror}, and is therefore 
a horizontal stratum, in the sense defined in the statement of the
proposition. In particular the $\gamma_i$ are cone strata of order $2$; and it
follows that $u $ is a dihedral node,
which is a terminal node of two oriented horizontal $1$--strata of
$\orbP$, and of one oriented vertical $1$--stratum which
projects to 
a bi-mirror point of $\orbH$. 
This is Alternative \eqref{dihedral alternative} of the conclusion of the proposition.
\end{proof}

\Remark\label{Seifert}
The structural information provided by Proposition \ref{Prop:IBundleSingLocus} is closely related to the well-known structure of orbifold Seifert fibered spaces \cite[Section 2.8]{CHK:OrbifoldBook}. For the sake of this discussion, an \emph{orbifold Seifert fibered space} is an orientable $3$--orbifold $\orbM$ equipped with an $\SS^1$--fibration over a $2$--orbifold $\orbH$. (This generalizes Example~\ref{Ex:SeifertManifold}, where $\orbM$ was required to be a manifold.) In an orbifold Seifert fibered space $\orbM$, it is once again true that every $1$--stratum of $\Sing_\orbM$ is either vertical or horizontal, with every vertical $1$--stratum projecting to a cone point and every horizontal $1$--stratum projecting homeomorphically to a mirror stratum of $\Sing_\orbH$.
Furthermore, just as in Proposition \ref{Prop:IBundleSingLocus}, every $0$--stratum is a dihedral node projecting to a bi-mirror point.
Compare \cite[Pages 41--42]{CHK:OrbifoldBook}. 
\EndRemark

  \Custom{\Clasibun s}
  \label{classical}
We define a  \emph{\clasibun} to be a topological $3$--manifold (with
boundary) $P$ equipped with a
locally trivial fibration $Q\from P\to H$ in which the base  $H$ is a topological
$2$--manifold (with
boundary), and the fiber is the unit interval $I$. 
In this situation,
$Q\from P\to H$ may be referred to as a
\emph{classical $I$--fibration}. 
A subset of $P$
is called  \emph{vertical} if it is saturated in the fibration $Q$. We define the {\it
  vertical boundary} of $P$ to be the submanifold $\vertbdyclas
P\doteq Q^{-1}(\bdy H)$
of $\bdy P$, and its  \emph{horizontal boundary, } denoted $\horizbdyclas
P$, to be its
associated $\bdy I$--bundle. 

We note that smooth and PL structures play no role in this definition.
\EndCustom

\Lemma\label{Lem:IBundleTotalSpaceOriBase}
Let $q \from \orbP \to \orbH$ be an orbifold $I$--bundle,
where $\orbP$ is compact, connected, orientable, and $3$--dimensional. Suppose in addition that
$\mirror\orbH=\emptyset$.
Then the 
underlying
 map $Q \from \mani \orbP \to \mani \orbH$ of $q$ 
is a classical $I$--fibration.
\EndLemma

\begin{proof}
Let $P \doteq \mani \orbP$ and $H \doteq \mani \orbH$.  Since 
$\mirror\orbH=\emptyset$,
Proposition~\ref{Prop:LocalStructure2Orbifold} says that the singular
locus $\Sing_\orbH$ consists entirely of cone points; hence
Proposition~\ref{Prop:IBundleSingLocus} implies that $\Sing_\orbP$ consists
entirely of vertical strata. Thus each stratum $\sigma$ of $\orbP$ is
the fiber over some cone point of $\orbH$, which will be denoted
$x_\sigma$.
In particular we have $Q^{-1}(\Sing_\orbH)=\Sing_\orbP$.

By \S\ref{Rem:StandardNbdFibration}, for each cone point $x\in\orbH$,
we may fix a standard local trivialization near $x$, which by
\S\ref{no more Gamma} we may denote
$(\orbU_x, D_x, G_x,\phi_x,\xi_x)$. We may choose the local
trivializations in such a way that  $\orbU_x\cap\orbU_{x'}=\emptyset$ 
whenever $x$ and $x'$ are distinct cone points.

Since $\mirror \orbH =
\emptyset$, Lemma~\ref{Lem:FoldedIntervals} says that for every
$x\in\orbH$, the fiber $\sigma_x\doteq q^{-1}(x)$ is
diffeomorphic to $I$;
  that is, the action of $G_x$ on $I$ is trivial.
Hence $(D_x\times I )/G_x $  may be canonically identified with
$(D_x /G_x)\times I $. On the other hand, according to
the definition of a local trivialization, $\phi_x$ induces an orbifold
diffeomorphism between $D_x /G_x$ and $\orbU_x$. Under these
identifications,  $\xi_x$ becomes a smooth map from $D_x\times I$
to $q^{-1}(\orbU_x)$, which induces an orbifold diffeomorphism $\eta_x$
  between $\orbU_x\times I $  and $q^{-1}(\orbU_x)$. The
  commutativity of the diagram in \S\ref{Def:OrbifoldFiberBundle} then
  says that $q\circ\eta_x\from \orbU_x\times I\to\orbU_x$ is the projection
  to the first factor. Thus $Q \from P \to H$  is a locally trivial fibration of topological spaces.
\end{proof}

As a consequence of Proposition~\ref{Prop:IBundleSingLocus} and Lemma~\ref{Lem:IBundleTotalSpaceOriBase}, we can prove the following useful fact about Euler characteristics.

\Lemma\label{Lem:SameEuler}
Let $q \from \orbP \to \orbH$ be an orbifold $I$--fibration,
where $\orbP$ is compact, connected, orientable, and $3$--dimensional. 
Then $\chi(\orbP) = \chi(\orbH)$.
\EndLemma

\Proof
Set $P = \mani \orbP$ and $H = \mani \orbH$. 
Let $Q \from \mani \orbP \to \mani \orbH$ denote the 
underlying map of 
$q$. 

We first consider the case in which $\mirror\orbH=\emptyset$. 
In this case,  $\Sing_\orbH$ consists of finitely many cone points. It
follows from Proposition~\ref{Prop:IBundleSingLocus} that
for every  cone point $x$  of $\orbH$, the set $Q^{-1}(x)$ is a cone
stratum of $\orbP$, and furthermore that $Q^{-1}(\Sing_\orbH)$ is the
full singular set of $\orbP$. Since $\orbP$ is connected, $\orbH$ is
also connected; this implies that $H-\Sing_\orbH$ is the unique smooth
stratum of $\orbH$, while $P-\Sing_\orbP=Q^{-1}(H-\Sing_\orbH)$ is the unique smooth
stratum of $\orbP$.

In particular, the assignment $\sigma\mapsto Q^{-1}(\sigma)$ is a
bijection between the strata of $\orbH$ and those of
$\orbP$. By \eqref{that's what chi is}, we have
$$
\chi(\orbH)=\sum_\sigma\frac{ \chi(\overline \sigma) - \chi(
\barsigma-\sigma)}{ \ord(\sigma)} \qquad {\rm and} \qquad
\chi(\orbP)=\sum_\sigma\frac{ \chi\big(\overline{Q^{-1}( \sigma)}\big) - \chi\big(
\overline{Q^{-1}(\sigma)}-Q^{-1}(\sigma)\big)}{ \ord(Q^{-1}(\sigma))},
$$
where $\sigma$ ranges over the strata of $\orbH$. Hence, to establish
the conclusion in this case, it suffices to show that for every
stratum $\sigma$ of $\orbH$, we have
\Equation\label{what we need}
\chi\big(\overline{Q^{-1}(\sigma)}\big)=\chi(\barsigma), \quad \chi\big(
\overline{Q^{-1}(\sigma)}-Q^{-1}(\sigma)\big)=\chi(
\barsigma-\sigma),
\quad{\rm and}\quad
\ord(Q^{-1}(\sigma))=\ord(\sigma).
\EndEquation

If $\sigma$ is a singular stratum of $\orbH$, then $\sigma$ consists
of  a cone
point of some order $p$. By Proposition~\ref{Prop:IBundleSingLocus},
the cone stratum $Q^{-1}(\sigma)$ also
has order $p$, and $Q^{-1}(\sigma)$
is a closed topological arc since $\mirror\orbH=\emptyset$.  The sets
$\barsigma-\sigma$ and
$\overline{Q^{-1}(\sigma)}-Q^{-1}(\sigma)$
are empty. This establishes 
\eqref{what we need} when $\sigma$ is a singular stratum.

If $\sigma$ is the nonsingular stratum of $\orbH$, then $Q^{-1}(\sigma)$
is the nonsingular stratum of $\orbP$; we have
$\barsigma=H$ and $\overline{Q^{-1}(\sigma)}=P$. The equality
of $\chi\big(\overline{Q^{-1}(\sigma)}\big)$ and $\chi(\barsigma)$ therefore
follows from Lemma~\ref{Lem:IBundleTotalSpaceOriBase}. 
To establish the equality
$\chi\big(
\overline{Q^{-1}(\sigma)}-Q^{-1}(\sigma)\big)=\chi(
\barsigma-\sigma)$ in this situation, we note that 
$\barsigma-\sigma=\Sing_\orbH$ and that
$\overline{Q^{-1}(\sigma)}-Q^{-1}(\sigma)=\Sing_\orbP$. Since the
components of $\Sing_\orbP$, which are closed arcs, are in bijective
correspondence with the points that make up $\Sing_\orbH$, the desired
equality follows. Finally, the orders of $\sigma$ and $Q^{-1}(\sigma)$
are both equal to $1$. Thus \eqref{what we need} is proved for all
strata, and the proof of the lemma in the case
$\mirror\orbH=\emptyset$ is complete.

Now suppose that $\mirror\orbH \neq \emptyset$. Let $\torbH\to\orbH$ denote the 
  \reflective\ covering of $\orbH$ (see \S\ref{reflective}). Lemma~\ref{Lem:InducedCover} provides an induced $I$--fibration
$\tq \from \torbP \to \torbH$ and a two-fold covering $\torbP \to \orbP$. Since $\mirror\torbH = \emptyset$, the case of the lemma that we have already proved, combined with the multiplicativity of Euler characteristic under covers, implies
\[
2 \chi(\orbP) = \chi(\torbP) = \chi(\torbH) = 2 \chi(\orbH).
\]
Thus $\chi(\orbP) = \chi(\orbH)$, completing the proof.
\EndProof

\Lemma\label{hands-on part}
Suppose that $P_0$ is a \clasibun, and that $A$ is a two-dimensional, compact 
vertical submanifold of $\vertbdyclas P_0$ (so that each component of
$A$ is a topological annulus or disk). 
Suppose that $ T \from A \to A$ is
an involution which leaves invariant each fiber of 
$P_0$ contained in $A$,
and reverses the orientation of each such
fiber. 
Let $f \from P_0 \to E$ be the quotient map that identifies every orbit of $T$ to a point.
Then $E $ is a topological $3$--manifold, and
may be given the structure of
a \clasibun\ in such a way that $f$ 
carries $\overline{(\vertbdyclas P_0)-A}$ onto a vertical
submanifold of $\vertbdyclas E$.
\EndLemma

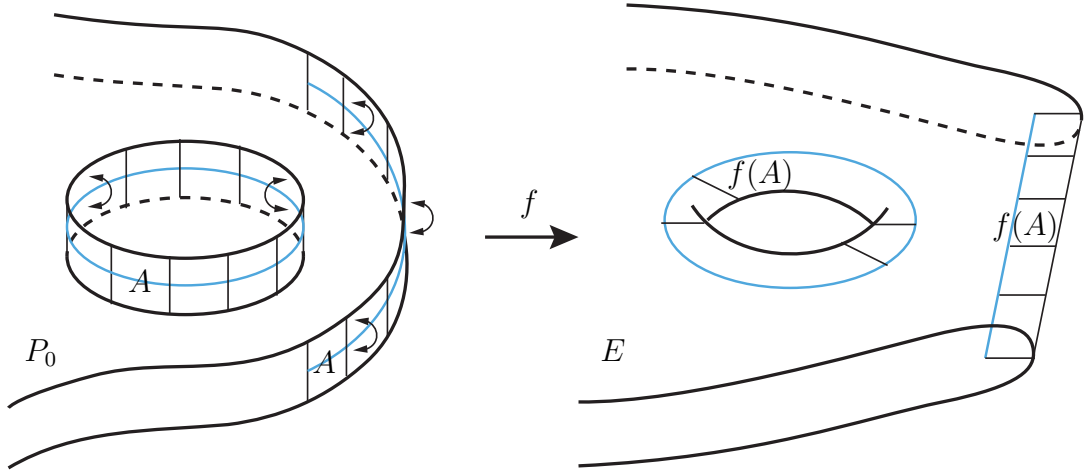
\begin{figure}
\begin{overpic}[width=6in]{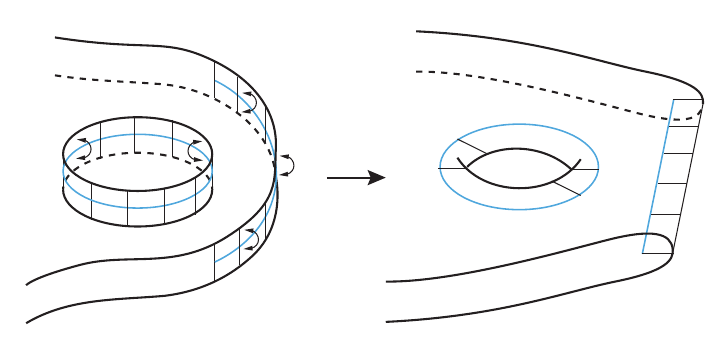}
\put(48,25){$f$}
\put(14,18){$A$}
\put(30,11){$A$}
\put(89,23){$f(A)$}
\put(66,27.5){$f(A)$}
\put(5,12){$P_0$}
\put(55,12){$E$}
\end{overpic}
\caption{The quotient map $f \from P_0 \to E$ of Lemma~\ref{hands-on part}. The involution $T \from A \to A$ is depicted by curved double arrows, and the fixed point set of $T$ appears in blue. The homeomorphism $\eta\from P_0\to E$ can be chosen to be the identity outside a small neighborhood of $A$.   }
\label{Fig:Folding}
\end{figure}

\Proof
Let  $Q \from P_0\to H_0$ denote the 
locally trivial fibration that defines the \clasibun\ structure of
$P_0$. 
The hypothesis implies that there is a compact $1$--manifold
$\alpha\subset\bdy H_0$ such that $A=Q^{-1}(\alpha)$. The
hypothesis also implies that $A$ may be identified homeomorphically
with $\alpha\times I$ 
in such a way that $Q|A$ is the projection to the
first factor, and $ T (x,t)=(x,1-t)$ for every $(x,t)\in A$.

From this description of $T$, it follows that $E$ is a topological
$3$--manifold, and that there is a homeomorphism
$\eta\from P_0\to E$ which agrees with the quotient map $f \from P_0\to E$
outside an arbitrarily small neighborhood of $A$ in $P_0$. 
See Figure~\ref{Fig:Folding}.
In
particular, if $N$ is a regular neighborhood of $A$ relative to
$\vertbdyclas P_0$, which is a vertical submanifold of $P_0$,  we may take $\eta$ to be the identity on
$\overline{(\vertbdyclas P_0)-N}$. 
It now follows that
$Q'\doteq Q\circ\eta^{-1}\from E\to H_0$ is a classical $I$--fibration. 
In terms of the $I$--bundle structure of $E$ defined by $Q'$, the map 
$f$ carries $\overline{(\vertbdyclas P_0)-N}$ onto a vertical
submanifold of $\vertbdyclas E$. Hence, by precomposing $Q'$ with a
self-homeomorphism of $E$ that maps $A$ onto $N$, we obtain a
classical $I$--fibration with the required properties.
\EndProof

\Proposition\label{Prop:IBundleTotalSpace}
Let $q \from \orbP \to \orbH$ be an orbifold $I$--bundle,
where $\orbP$ is compact, connected, orientable, and $3$--dimensional.
Then $P \doteq \mani \orbP$,
regarded as a topological $3$--manifold,
may be given the structure of
a \clasibun\ in such a way that $\mani{\vertbdy \orbP}$ is a vertical
submanifold of $\vertbdyclas P$.
\EndProposition

\Proof
If  $\mirror\orbH=\emptyset$, 
then the result follows from Lemma~\ref{Lem:IBundleTotalSpaceOriBase}. 

Now, assume 
that $\mirror\orbH\ne\emptyset$.
Let $\theta \from\torbH\to\orbH$ denote the 
  \reflective\ covering of $\orbH$ (see \S\ref{reflective}).
Set $H=\mani\orbH$ and $\tH=\mani\torbH$, and let $\Theta \from\tH\to H$
denote the underlying
map of $\theta $.
According to \S\ref{reflective}, we may write $\tH$ as a union of
two topological $2$--submanifolds $H_0$ and $H_1$, such that
$\Theta $ maps $H_i$  homeomorphically  onto $H$ for $i=0,1$.

Lemma~\ref{Lem:InducedCover} provides an induced $I$--fibration
$\tq \from \torbP \to \torbH$ and a covering map $\upsilon  \from \torbP \to \orbP$.
According to that lemma,
we have $\deg \upsilon =\deg \theta =2$; hence  $\torbP$ has a
unique non-trivial deck transformation $\delta$, which is an
involution.
Lemma~\ref{Lem:InducedCover} also implies that there is a commutative diagram
\Equation\label{involution-diagram}
\begin{tikzcd}
 \torbP \arrow[r, "\delta"] \arrow[d, "\tq"] & \torbP \arrow[d, "\tq"] \\
\torbH \arrow[r, "\lambda"] & \torbH
\end{tikzcd}
\EndEquation
where $\lambda$ denotes the non-trivial deck transformation
of $\torbH$.

Let $Q \from P \to H$, $\tQ \from \tP \to \tH$,  $\Upsilon  \from \tP \to P$, $\Lambda\from\tH\to\tH$ and
$\Delta\from\tP\to\tP$ denote the 
underlying maps of 
$q$, $\tq$,  $\upsilon $, $\lambda$ and $\delta$
respectively. 

For $i=0,1$, set $P_i=\tQ^{-1}(H_i)$. Since
$\tH=H_0\cup H_1$, we have $\tP=P_0\cup P_1$. 
According to \S\ref{reflective}
we have $R\doteq H_0\cap H_1=\Theta ^{-1}(\mirror\orbH)$, 
and the involution
$\Lambda$ of $\tH$  interchanges $H_0$ and $H_1$.
It follows
that $A\doteq P_0\cap P_1=\tQ^{-1}(R)$, 
and, in view of the commutativity of \eqref{involution-diagram},  that the involution $\Delta$
of $\tP$  interchanges $P_0$ and $P_1$. In particular $\Delta$ leaves $A$ 
invariant, and therefore restricts to a self-homeomorphism
of $A$ which  \emph{a priori} is periodic with period at most $2$. From \S\ref{reflective}
we also have that $\Lambda$ fixes $R$ 
pointwise, which with the commutativity of \eqref{involution-diagram},
implies that  $\Delta$
leaves the set  $E_\tx\doteq \tQ^{-1}(\tx)\subset A$ invariant for each point $\tx\in R$.

If $\tx$ is any point of $R$, then $x\doteq
\theta (\tx)\in \mirror\orbH$, so that Lemma \ref{Lem:FoldedIntervals}
gives that the orbifold fiber $q^{-1}(x)$ is a folded arc. On the
other hand, since $\mirror\torbH=\emptyset$,  Lemma \ref{Lem:FoldedIntervals}
gives that the orbifold fiber $\tq^{-1}(\tx)$ is diffeomorphic to
$I$. Hence $ \Delta |E_\tx$ is an 
orientation-reversing  involution of $E_\tx$ for each $\tx\in R$. In
particular $ \Delta | A $ is an involution.

According to \S\ref{reflective}, we have
$\mirror\torbH=\emptyset$. 
Lemma~\ref{Lem:IBundleTotalSpaceOriBase} then implies that $\tQ \from \tP \to \tH$ is a locally trivial fibration of topological spaces. Hence if we set
$P_0=\tQ^{-1}(H_0)$, then
$\tQ$ restricts to a locally trivial fibration  $Q_0\from P_0\to H_0$, and $P_0$ thus acquires the structure of a
\clasibun. From the perspective of this structure, $A$ is a vertical
subset of $P_0$, and the sets of the form $E_\tx$ for $\tx\in R$ are
the fibers of $P_0$ contained in $A$. Thus $\Delta | A $ is an involution
of the vertical set $A$ which restricts to an orientation-reversing
involution of each fiber contained in $A$. Furthermore, Proposition
\ref{Prop:LocalStructure2Orbifold} implies that $\mirror\orbH$ is a
one-dimensional topological submanifold of the boundary of the
topological $2$--manifold $H$;
since $\Theta $ maps $H_0$  homeomorphically  onto $H$, it follows that 
$R$ is a
one-dimensional topological submanifold of $\bdy H_0$, and
hence that $A$ is a two-dimensional topological submanifold of 
$\vertbdyclas P_0$. Thus $P_0$, $A$ and $ T \doteq \Delta | A$ satisfy the
hypotheses of Lemma  \ref{hands-on  part}.

We have $\Upsilon \circ \Delta=\Upsilon $. Since $\Upsilon  \from \tP \to P$ is
surjective, and 
since $\Delta$  interchanges $P_0$ and $P_1$
and leaves $A$ invariant, 
the map $\Upsilon _0\doteq
\Upsilon |P_0\from P_0\to P$ 
induces a continuous bijection from
the quotient space $E$ defined in Lemma \ref{hands-on part} to
$P$. Since $E$ is compact and $P$ is Hausdorff, this induced
map is a homeomorphism. It therefore follows from Lemma \ref{hands-on part} that $P$ 
may be given the structure of
a \clasibun\ in such a way that 
 $\Upsilon _0(\overline{(\vertbdyclas P_0)-A})$ is a vertical
submanifold of $\vertbdyclas P$.

Now Proposition \ref{Prop:LocalStructure2Orbifold} implies that
$\mani{\bdy\orbH}=\overline{(\bdy H)-\mirror\orbH}$.
Hence 
\[
\overline{(\vertbdyclas P_0)-A}
= (Q\circ \Upsilon _0)^{-1}(\mani{\bdy\orbH})=\Upsilon _0^{-1}(\mani{\vertbdy\orbP}),
\]
which with the surjectivity of $\Upsilon _0$ shows that
 $\Upsilon _0(\overline{(\vertbdyclas P_0)-A})=\mani{\vertbdy
  \orbP}$. It now follows that $\mani{\vertbdy\orbP}$ is a vertical
submanifold of $\vertbdyclas P$.
\EndProof

\section{Orbifold books of $I$--bundles}
\label{Sec:OrbifoldBooks}

In $3$--manifold theory, a book of $I$--bundles is constructed by taking some number of $I$--bundles (\emph{pages}) and gluing their vertical boundaries to solid tori (\emph{bindings}). (See e.g.\ \cite[Definition 4.2]{GS-homotoping} for details.) In this section, we adapt this theory to the setting of $3$--orbifolds, leading to the careful definition of an orbifold book of $I$--bundles in \S\ref{wuzza book}. This leads to the proof, in Proposition~\ref{nodes of a book}, that any node in the singular locus of an orbifold book of $I$--bundles is dihedral.

In the course of studying the pages and bindings of an orbifold book of $I$--bundles, we divide the complement of (the underlying space of) a page into components, some of which are labeled \emph{harmless}. Harmless components are topological $3$--balls with connected frontier (\S\ref{Def:Harmless}).  In Proposition~\ref{Prop:ExpandedPage}, we will attach harmless components to the underlying space of a page, and show that the resulting $3$--manifold admits the structure of a manifold $I$--bundle. This sort of manifold $I$--bundle will turn out to be very useful in the topological arguments of Sections~\ref{sphere section} and~\ref{torus section}.

\Custom{Conventions about smoothness and transversality}\label{AccordantConvention}
In the remainder of the paper, 
whenever we consider an orientable $3$--orbifold  $\orbM$, it will be  understood that
$\mani\orbM$ is equipped with
an $\orbM$--accordant smooth  structure, which exists in view of
Proposition \ref{smooth underlying}.

In \S\ref{Suborbifolds-of-3orbifolds}, we described
the correspondence between orientable, \neatly embedded
 $2$--suborbifolds of $\orbM$ and \neatly embedded
two-dimensional submanifolds  of $M\doteq \mani\orbM$ transverse to $\Sing_\orbM$. This correspondence will be used very
freely, as will the connection --- also explained in \S\ref{Suborbifolds-of-3orbifolds} ---
between the singular structure of such suborbifolds of $\orbM$ and the 
size and weight (see \S\ref{size 'n' stuff}) of the corresponding two-dimensional submanifolds  of
$ \mani\orbM$. 
We will also freely use the fact (see \S\ref{new splitting}) that an orientable $3$--orbifold can be split along a neatly
embedded $2$--suborbifold. This includes the special
case that for a \clean\ submanifold $K$ of $M$ with frontier transverse
to $\Sing_\orbM$, we have a well-defined suborbifold $\orbi K$. 
\EndCustom

\Custom{\Clean\ suborbifolds}\label{clean}

We shall say that a suborbifold with corners 
$\orbY$ of a compact, orientable
$3$--orbifold $\orbN$
is  \emph{\clean}  if
\begin{enumerate}
\item $\Fr_{\orbN}\orbY$ 
    is defined and is a \neatly
    embedded suborbifold of $\orbN$, and
  \item\label{clean pair part} $(\orbY, \Fr_{\orbN}\orbY)$ is a \clean\ pair 
 (see \S\ref{Clean pairs}).
  \end{enumerate}
  Note that \eqref{clean pair part} includes the information that
  $\dim\orbY=3$ and that $\dim\Fr_{\orbN}\orbY=2$.

If $\orbY$ is a \emph{weak} suborbifold of a compact, orientable
$3$--orbifold $\orbN$ such that $\Fr_{\orbN}\orbY$ 
    is defined and is a \neatly
    embedded \emph{weak} suborbifold of $\orbN$, and if Condition \eqref{clean pair part} above holds, then it follows from an observation made in 
\S\ref{Clean pairs} that $\orbY$ and $\Fr_{\orbN}\orbY$  are suborbifolds of $\orbN$ (in the strong sense), and hence that $\orbY$ is a clean suborbifold. This will make it simpler to recognize clean suborbifolds.

Specializing to the case in which $\orbN$ is a manifold, we obtain the notion of a
\clean\ submanifold with corners. 

Note that if
$\orbY$ is a clean suborbifold with corners of a compact, orientable $3$--orbifold
 $\orbN$ (so that  $N\doteq\mani\orbN$ has a fixed $\orbN$--accordant smooth structure by 
the convention of
\S\ref{AccordantConvention}),
then the induced smooth manifold structure on
 $Y\doteq\mani\orbY$ is $\orbY$--accordant, and $Y$ 
 is a clean submanifold with corners of $N$. To go
in the other direction, by analogy with   \S\ref{Suborbifolds-of-3orbifolds}, if $Y$ is a clean submanifold with corners of $N$, and $\Fr_{N} Y$ is transverse to $\Sing_\orbN$, then $\orbi Y$ is defined. Furthermore, by Proposition~\ref{codim 1 is nice},
$\orbi Y$ is a suborbifold in the strong sense.
It then follows readily that $\orbi Y$ is a clean suborbifold.

It follows from \S\ref{Clean pairs}
that if $\orbY$ is a \clean\  suborbifold with corners of a compact, orientable
$3$--orbifold $\orbN$ then $\orbY$ has nonsingular corners.

If $\orbY$ is \cleansub $\orbN$, 
then with the conventions established in
\S\ref{Suborbifolds}, the suborbifolds (with boundary)
$\Fr_\orbN\orbY$,
$\orbY\cap\bdy\orbN$
and
${\orbN-\orbY}$
of $\orbN$, and the suborbifold with corners
$\overline{\orbN-\orbY}$, are defined.
Note  that $\Fr_\orbN\orbY$ is
a \neatly embedded $2$--suborbifold of $\orbN$ and that 
$\overline{\orbN-\orbY}$ is \cleansub $\orbN$.
Note also that  $\bdy(\orbY\cap\bdy\orbN)=
\bdy(\Fr_\orbN\orbY)$. 

If $\scrY$ is a family of \clean\  suborbifolds
with
$2$--corners in $\orbN$, and if for all $\orbY,\orbY'\in\scrY$, the
suborbifold $\orbY\cap\orbY'$ is defined and is a union of common
frontier components of $\orbY$ and $\orbY'$, then
$\bigcup_{\orbY\in\scrY} \orbY$ is defined and is \cleansub
$\orbN$.

We have the following immediate consequence of Lemma \ref{split injective}: 

\Claim\label{because split injective}
If $\orbY$ is \cleansub an orientable $3$--orbifold $\orbN$, and $\Fr_\orbN\orbY$ is
$\pi_1$--injective in $\orbN$, then $\orbY$ and $\overline{\orbN-\orbY}$ are
$\pi_1$--injective in $\orbN$.
\EndClaim
\EndCustom

\Custom{Bindinglike and pagelike suborbifolds}
\label{PagelikeBindinglike}

Let $\orbN$ be a compact, orientable $3$--orbifold.

A suborbifold
with corners
$\orbB$ 
of $\orbN$ is said to be  \emph{bindinglike} if
\begin{itemize}
\item $\orbB$ is
  \cleansub $\orbN$,
\item the orbifold obtained by \sta\ of
  $\orbB$ 
(which is defined since the clean suborbifold $\orbB$ has nonsingular corners  by \S\ref{clean}; see \S\ref{what's sta})
is a solid toric orbifold, and 
\item $\Fr_\orbN\orbB$ 
is a non-empty
disjoint union of annular $2$--orbifolds which are $\pi_1$--injective in
$\orbB$.
\end{itemize}

This definition 
implies that 
$\orbB$ is connected and that 
$\orbB\cap\bdy\orbN$ 
is also a non-empty
disjoint union of annular $2$--orbifolds  which are $\pi_1$--injective in
$\orbB$.

According to the discussion of solid toric orbifolds
in \S\ref{spherical etc}, if $\orbB$ is a bindinglike suborbifold of $\orbN$ then
$B \doteq \mani\orbB$ either is  a topological ball,
i.e.\ is homeomorphic 
 to $\BB^3$, or is  a topological solid torus,
i.e.\ is homeomorphic to $\BB^2\times\SS^1$.
Furthermore, if  $B$ is
a topological solid torus, then the components of $\Fr B$ are annuli with non-zero
winding number in $B$.

A  suborbifold
with corners
$\orbP$  
of $\orbN$ is said to be  \emph{pagelike} if 
$\Fr_\orbN\orbP$ is defined and neatly embedded in $\orbN$, and
$\orbP$ 
admits an  $I$--fibration over a 
  connected $2$--orbifold 
with non-positive (orbifold) Euler characteristic,
such that
  $\Fr_\orbN\orbP$ is the vertical boundary of the $I$--bundle $\orbP$.
It then follows from \S\ref{clean fibration} that $(\orbP,\Fr_\orbN \orbP)$ is a clean pair. By an observation made in \S\ref{clean}, this in turn implies that $\orbP$ and $\Fr_{\orbN}\orbP$  are suborbifolds of $\orbN$ in the strong sense, and hence that $\orbP$ is a clean suborbifold. 

Note that this definition
implies that $\orbP$ is connected,
and has non-positive Euler
characteristic (by Lemma~\ref{Lem:SameEuler}), 
that $\orbP\cap\bdy\orbN$ is the horizontal
boundary of the $I$--bundle $\orbP$, and that   $\Fr_\orbN\orbP$ is $\pi_1$--injective in
$\orbP$. Note as well that  $\Fr_\orbN\orbP$ may be empty.
\EndCustom

The properties of being a bindinglike suborbifold and of being a
pagelike suborbifold are not mutually exclusive. In fact, a pagelike
suborbifold is bindinglike if and only if its Euler characteristic is
$0$,
or equivalently, if and only if its orbifold fundamental group is
virtually cyclic.
This equivalence is easily deduced from Lemma \ref{Lem:SameEuler}.
A \emph{purely \bindinglike} suborbifold  of  $\orbN$ is a \bindinglike\ suborbifold 
which is not \pagelike.
A \emph{purely \pagelike} suborbifold  of  $\orbN$ is a \pagelike\ suborbifold 
which is not bindinglike. 

\Custom{Orbifold books of $I$--bundles}
\label{wuzza book}
We define an \emph{orbifold book of $I$--bundles} to be a compact,
connected, orientable $3$--orbifold $\orbN$ equipped with two
\clean\ suborbifolds with corners
$\scrB$ and $\scrP$ such that 
\begin{itemize}
\item every component of $\scrB$ is purely
bindinglike, and  every
  component of $\scrP$ is purely
pagelike; and
\item $\scrB\cup\scrP=\orbN $
and $\scrB\cap\scrP=\Fr_\orbN\scrB=\Fr_\orbN\scrP$.
\end{itemize}

The components of $\scrB$ and $\scrP$ are respectively
 called the \emph{bindings}  and the \emph{pages} of $\orbN$. 

Note that according to this definition, each component of
$\scrB\cap\scrP$ is an   
orientable, annular,  \neatly embedded suborbifold of $\orbN$,
and $\scrP$ admits
an $I$--fibration in which $\scrP\cap\bdy\orbN$ is the
horizontal boundary and $\scrB\cap\scrP$ is the vertical boundary.

Note also that every binding has a
virtually cyclic fundamental group, and that no page has a
virtually cyclic fundamental group.

It follows from an observation made in \S\ref{clean} that any union of
pages and bindings of a book of $I$--bundles $\orbN$ is a clean
suborbifold with corners of $\orbN$. Furthermore, it follows from
Assertion \ref{because split injective} and the discussion in 
\S\ref{PagelikeBindinglike} that 
every page or binding of a book of $I$--bundles $\orbN$ --- and more
generally  every union of pages and bindings of $\orbN$ --- is
$\pi_1$--injective in $\orbN$, 
and so is the frontier of every page or binding. These observations
will be used extensively in the following sections.

When the structure of an orbifold book of $I$--bundles on $\orbN$ has
been fixed, we denote the suborbifolds $\scrB$ and $\scrP$
by  $\scrB_\orbN$ and $\scrP_\orbN$ respectively. Furthermore,
if $\orbS$ is any union of components of $\bdy \orbN$,
we set
$\scrA_{\orbN,\orbS}=\scrB_\orbN\cap\orbS$ and
$\scrH_{\orbN,\orbS}=\scrP_\orbN\cap\orbS$;
  these are well-defined $2$--suborbifolds of $\orbS$. Note  that every component of $\scrA_{\orbN,\orbS}$ has
 Euler characteristic $0$,  that every component of $\scrH_{\orbN,\orbS}$ has
strictly negative Euler characteristic, and that $\scrA_{\orbN,\orbS}\cup \scrH_{\orbN,\orbS}=\orbS$.
We also set 
$\boldC_{\orbN,\orbS}=\scrA_{\orbN,\orbS}\cap
\scrH_{\orbN,\orbS}=\Fr_{\orbS}\scrA_{\orbN,\orbS}=\Fr_{\orbS}\scrH_{\orbN,\orbS}$,
a well-defined $1$--manifold. It can be deduced from Lemma~\ref{split injective} that $\boldC_{\orbN,\orbS}$ is $\pi_1$--injective in the orbifold $\orbN$.

We mention as a mnemonic device that $\scrA$ stands for ``annular,'' $\scrH$ stands for ``horizontal boundary,'' and~$\boldC$ stands for ``curves.''
\EndCustom

\Proposition\label{nodes of a book}
If $\orbN$ is a book of $I$--bundles then every node of the
$3$--orbifold $\orbN$ is dihedral. 
\EndProposition

\Proof
Set $\scrB=\scrB_\orbN$ and $\scrP=\scrP_\orbN$. Let
$\scrF$ denote the common frontier of $\scrB$ and $\scrP$. By
\S\ref{wuzza book}, $\scrF$ is an orientable, \neatly embedded
suborbifold of $\orbN$; hence by \S\ref{Suborbifolds-of-3orbifolds},
$\mani\scrF$ is a neatly embedded two-dimensional submanifold of
$\mani\orbN$, transverse to $\Sing_\orbN$.
This implies that every node of $\orbN$ is either a
node of $\scrB$ or a node of $\scrP$. It follows from
the description of solid toric orbifolds given in
 \S\ref{spherical etc} that the components of $\scrB$ only have dihedral nodes.
On the other hand, the components of
$\scrP$ are pagelike, which by definition implies that they admit
$I$--fibrations over $2$--orbifolds; according to Proposition
\ref{Prop:IBundleSingLocus}, it follows that every node of $\scrP$
is dihedral.
\EndProof

\Lemma\label{saddy}
Let $\orbN$ be an orbifold book of $I$--bundles, and let $\orbJ$ be a
connected
  two-dimensional
suborbifold of   $\bdy \orbN$ with
$\bdy \orbJ\subset\boldC_{\orbN,\bdy\orbN}$. Suppose that
$\chi(\orbJ)=0$. Then $\orbJ$ is a component of $\scrA_{\orbN,\bdy\orbN}$.
\EndLemma

\Proof
Let us set $\scrA=\scrA_{\orbN,\bdy\orbN}$,
$\scrH=\scrH_{\orbN,\bdy\orbN}$, and $\boldC = \boldC_{\orbN,\bdy\orbN}$.
Since $\bdy \orbJ\subset\boldC$, 
the orbifold $\orbJ$ is a
union of components of $\scrA$ and of $\scrH$. Since the
components of  $\boldC = \scrA\cap
 \scrH$ are simple closed curves, $\chi(\orbJ)$ is the sum of
 the Euler characteristics of components of  $\scrA$ and of
 $\scrH$ contained in $\orbJ$. We observed in \S\ref{wuzza book} that every component of $\scrA$ has
 Euler characteristic $0$, while every component of $\scrH$ has
strictly negative Euler characteristic. Hence $\orbJ$ is a
union of components of $\scrA$, and since it is connected, the
conclusion follows.
\EndProof

\Lemma\label{just four simpler}
Let $\orbN$ be
an orbifold book of $I$--bundles. Let $\orbB$ be a binding of
$\orbN$, and suppose that  the frontier of $B\doteq\mani \orbB$ in
$N\doteq\mani \orbN$ has a component $F$
which is an annulus, and that each component of $\bdy F$ bounds a disk
in $\bdy N$ 
whose \textweighttwo\ (see 
\S\ref{size 'n' stuff}) 
is $4$. Then $B$ is a topological $3$--ball and $F$ is the full frontier of $B$ in $N$.
\EndLemma

\Remark\label{what weight 4 means}
Under the hypotheses of Lemma \ref{just four simpler}, it follows from
the definition of an orbifold book of $I$--bundles that a component $C$
of
$\bdy F$ is $\pi_1$--injective in $\orbN$. 

In general, 
if $D\subset\orbN$ is a disk, transverse to $\Sing_\orbN$ and bounded by a
curve $C$ which is $\pi_1$--injective in $\orbN$, 
the condition that $D$ has \textweighttwo\ $4$ is
equivalent to saying that $D$ has exactly two singular points, each of which has
  order $2$. 
This is because, if  $D$ had exactly one singular point
  and its order were $4$, then  $\orbi D$ would be a discal orbifold,
  a contradiction to the $\pi_1$--injectivity of $C$.
\EndRemark

\Proof[Proof of Lemma \ref{just four simpler}]
Let $C_1$ and $C_2$ denote the components of $\bdy F$. By
hypothesis, each $C_i$ bounds a \textweighttwo--$4$ disk
$D_i\subset\bdy N$. 
It follows from Remark \ref{what weight 4 means}
 that $\orbD_i\doteq\orbi {D_i}$ has Euler
characteristic $0$.
Furthermore, since $F$ is a frontier component of 
$B = \mani \orbB$,
we have  $C_i\subset\boldC_{\orbN,\bdy\orbN}$.
Hence by Lemma \ref{saddy}, 
each $\orbD_i$ is a component of $\scrA\doteq\scrA_{\orbN,\bdy\orbN}$; that is, $D_i$ is
a component of $\boldA \doteq \mani \scrA$
for $i=1,2$. Since the
boundaries $C_1$ and $C_2$ of $D_1$ and $D_2$ are distinct,
the components $D_1$ and $D_2$ of $\boldA$ are distinct
and therefore disjoint. Hence $W\doteq D_1\cup F\cup D_2$ is a
topological
$2$--sphere. 
We may regard $\boldB\doteq\mani{\scrB_\orbN}$ as a topological
manifold, each of whose components is a topological ball or solid
torus, and $B$ is one of these components.
Since $F\subset \bdy B\subset\bdy\boldB$ and
$D_i\subset \boldA \subset\bdy\boldB$, we have $W\subset\bdy\boldB$; and
since $F\subset B$, the component of $\boldB$ containing $W$ must be
$B$. Since $W$ is a topological sphere,  $B$ must be a topological ball, and $W$ must be the entire
boundary of $B$. In particular we have $\Fr_NB=F$.
\EndProof

\Custom{Harmless components}
\label{Def:Harmless}
Let  $\orbP$ be a page of an orbifold book of $I$--bundles $\orbN$. 
Set $N=\mani\orbN$ and $P=\mani\orbP$.
Note that any component  $V$ of 
$\overline{N -P}$ is a union of underlying sets of
pages and bindings. As observed in \S\ref{wuzza book}, this implies that
$\orbi V$ is defined and is a \clean\ suborbifold with corners of $\orbN$, so that
$V$ is a \clean\ submanifold with corners  of $N$.

We
shall say that $V$ is \emph{\harmless} if 
$V$ is  a topological ball and $\Fr_{N}V$ 
is connected. 

Note that since $\orbP$ is a page of the orbifold book of $I$--bundles
$\orbN$, each component of $\Fr_{\orbN}\orbP$ is an annular orbifold, and
hence each component of $\Fr_{N}P$ is 
an annulus or a disk. 
It follows
that the
frontier of each \harmless\ component of 
$\overline{N-P}$
is  an  annulus or a disk.

Hence if $V$ is a harmless
   component of $\overline{N-P}$, then $V$ is diffeomorphic to either 
$\BB^2\times I$ or $\BB^3 \cap (\RR^2 \times \RR_+)$. Under the diffeomorphism, $\Fr_N V$ corresponds to $\bdy \BB^2 \times I$ or $\BB^2 \times \{0\}$, respectively.
\EndCustom

\Proposition\label{Prop:ExpandedPage}
Let $\orbP$ be a page of an orbifold book of $I$--bundles $\orbN$. 
Set $N=\mani\orbN$ and $P=\mani\orbP$.
Let $\boldV$ be the union of some number of 
harmless components  of  
$\overline{N-P}$.
  Let $\hatP $ denote the clean submanifold  with corners
    $P\cup \boldV$ 
 of $N$.
Then the following hold:

\begin{enumerate}[\:\:$(1)$]
 \item \label{ExpandedPageFrontier} 
$\Fr_N\hatP = \Fr_N P - \Fr_N \boldV$.
\item\label{ExpandedPageIsBundle} $\hatP $,
regarded as a topological $3$--manifold,
may be given the structure of
a \clasibun\ in such a way that 
$\Fr_N\hatP$ is a vertical subset of $\vertbdyclas \hatP$.
\item\label{ExpandedPageEuler} The Euler characteristic of $\hatorbP \doteq \orbi \hatP$ is strictly negative,
and each component of
$\Fr_\orbN\hatorbP=\orbi {\Fr_{N} \hatP }$
 is an annular $2$--orbifold.
\end{enumerate}
\EndProposition

\Proof
Set $\orbF =\Fr_\orbN \orbP$ and $F=\mani\orbF=\Fr_NP$.
According to the definition of a
harmless component, each component of $\boldV$ is a  component  of  
$\overline{N-P}$ having connected frontier. This proves conclusion \eqref{ExpandedPageFrontier}.

Since $\orbN$ is an orbifold  book of $I$--bundles, the page $\orbP$
of $\orbN$ is equipped with an
orbifold
 $I$--bundle structure in which $\orbF $
is the vertical boundary. It therefore follows from  Proposition
\ref{Prop:IBundleTotalSpace} that the
topological
$3$--manifold 
$P$ may be given the structure of a \clasibun\ 
in such a way that 
$F $
is a
vertical subset of 
$\vertbdyclas P$.

Consider an arbitrary component $V$ of $\boldV$. Since $V$ is harmless,
it is homeomorphic to $\BB^3$, and its frontier is a single annulus or
a single disk. Since $\Fr_{N }V$ is a component of $F $, it is a
vertical submanifold of
$\vertbdyclas P$; thus the \classical\ fibration of $ P$
restricts to a 
locally trivial fibration of $\Fr_{N }V$ over a $1$--manifold, with
fiber $I$. Since $N=\mani\orbN$ is orientable, this fibration of
$\Fr_NV$ is trivial.

Now observe that any trivial topological
 fibration, with fiber $I$, of a topological annulus or disk contained in $\bdy \BB^3$ extends to
a trivial fibration of $ \BB^3$. Hence the fibration 
of $\Fr_{N }V$
obtained by restriction extends to a trivial $I$--fibration of $V$. Since this
holds for every $V$, the $I$--fibration of $ P$ extends to a \classical\
$I$--fibration of $\hatP $. 
By Conclusion~\eqref{ExpandedPageFrontier},
$\Fr_{N }\hatP $ is a union of components of $F $; hence 
$\Fr_{N }\hatP $ is a vertical submanifold of $ P $, and is therefore
a vertical submanifold of $\hatP  $. 
Since $\Fr_{N }\hatP $ is
vertical in $\hatP$ and is contained in the boundary of the
topological $3$--manifold $\hatP$, it is in fact contained in 
$\vertbdyclas \hatP$.
This establishes Conclusion~\eqref{ExpandedPageIsBundle}.

Next
note that $\hatorbP\doteq \orbi \hatP $ is a union  of pages
and bindings; hence we may write $\hatorbP=\scrB'\cup\scrP'$,
where $\scrB'$ is a union of bindings of $\orbN$ and $\scrP'$ is a
union of pages. Since $\orbP\subset\scrP'$, we have
$\scrP'\ne\emptyset$; and by the definition of an orbifold book of
$I$--bundles, each component of $\scrP'$ has strictly negative
orbifold Euler characteristic. Hence $\chi(\scrP')<0$. Each
component of $\scrB'$ 
becomes a solid toric orbifold upon \sta, 
and hence
$\chi(\scrB')=0$. Each component of $\scrB'\cap\scrP'$ is a
component of the vertical boundary of the $I$--bundle $\scrP'$ and
is therefore annular; thus $\chi(\scrB'\cap\scrP')=0$. It now
follows from the discussion in \S\ref{Euler} that
$\chi(\hatorbP)=\chi(\scrP')+\chi(\scrB')-\chi(\scrP'\cap\scrB')<0$. 

Finally, by Conclusion~\eqref{ExpandedPageFrontier},
each component of
$\Fr_{\orbN} \hatorbP $
is also a component of
  $\Fr_\orbN\orbP=\vertbdy\orbP$,
and is therefore annular
 by
Corollary \ref{Cor:VertBoundaryStructure}. This establishes Conclusion~\eqref{ExpandedPageEuler}.
\EndProof

\section{Essential spheres}\label{sphere section}

In this section, we consider a closed hyperbolic $3$--orbifold $\orbM$ such that $M \doteq \mani \orbM$ contains an essential $2$--sphere $S$. The main result is Corollary~\ref{Cor:HomeoTypeReducible}, which can be paraphrased as follows: if, for an appropriate choice of $S$, each component of $\orbM \split \orbi S$ is an orbifold book of $I$--bundles, then $M$ is either the connected sum of two lens spaces or the connected sum of $\SS^2 \times \SS^1$ with a lens space. In fact, Corollary~\ref{Cor:HomeoTypeReducible} will follow as an immediate consequence of Theorem~\ref{Thm:HomeoTypeSpherePieces}, which describes the topological structure of each component of $M \split S$. 
The proof of Theorem~\ref{Thm:HomeoTypeSpherePieces} will occupy the bulk of this section. We point the reader to the discussion immediately after the statement for a top-level summary of the argument.

\Custom{Smoothing subsurfaces}
\label{what's smoothing}

We will frequently encounter situations in which $\orbM$ is a closed,
orientable $3$--orbifold, and we have a topological surface
$F$ in $M\doteq\mani \orbM$ and a closed $1$--manifold $L\subset F\cap (M-\Sing_\orbM)$
such that 
\begin{enumerate}[\:\:(1)]
\item $L$ and
$F-L$ are smooth submanifolds of $M$, and $F-L$ is transverse to $\Sing_\orbM$; and 
\item
 there exist a tubular
neighborhood $J$ of $L$ in $M-\Sing_\orbM$, and
a diffeomorphism of $J$ onto $\BB^2\times L$ carrying $F\cap
J$ onto $V\times L$, where $V$ is a union of two line radii in
$\BB^2$. 
\end{enumerate}
Such a surface $F$ is called \emph{smoothable}.
If $F$ is closed and smoothable, there is a topological ambient isotopy of
$M$ which is supported on $J$ and carries $F$ onto a smooth
closed surface $F'$ which is transverse to $\Sing_\orbM$. We shall say that such a surface $F'$ is \emph{obtained  by smoothing $F$.} We observe that in this situation, 
$\size(F')=\size(F)$.
\EndCustom

\Custom{Minimal essential spheres}
\label{minimal}
Specializing the definition of an essential suborbifold given in
\S\ref{Pi1Injective}, we shall say that a (smooth) $2$--sphere 
$S$ in a closed orientable $3$--manifold $M$ is \emph{essential} if $S$ does not bound a $3$--ball in $M$.

Let $\orbM$ be a closed, orientable $3$--orbifold,
and set $M=\mani\orbM$.
 We define an \emph{$\orbM$--minimal essential sphere} in $M$ to
 be a $2$--sphere $S\subset M$ such that 
\begin{enumerate}[\:\:(1)]
\item $S$ is essential in $M$, 
\item $S$ is transverse to $\Sing_\orbM$, and 
\item
for every
$2$--sphere $S'\subset M$ which is essential in $M$ and
transverse to $\Sing_\orbM$, we have $\size(S')\ge\size(S)$. 
(See \S\ref{size 'n' stuff}. The sizes of $S$ and $S'$ are finite by transversality.)
\end{enumerate}
We define a \emph{strongly $\orbM$--minimal essential sphere} in $M$ to be an $\orbM$--minimal essential sphere $S$ in $M$ such that for every $\orbM$--minimal essential sphere $S'$ in $M$, we have $\weighttwo(S')\ge\weighttwo(S)$.

Note that if $\orbM$ is a closed, orientable $3$--orbifold and $M\doteq\mani \orbM$ contains an essential $2$--sphere, then $M$ contains 
a strongly $\orbM$--minimal essential sphere.
\EndCustom

\Lemma\label{it's incompressible}
Let 
$\orbM$ be a closed, orientable, hyperbolic $3$--orbifold, and let
$S$ be an $\orbM$--minimal essential $2$--sphere in $M\doteq\mani
\orbM$. 
Then $\orbS \doteq \orbi S $ is essential in $\orbM$. 
\EndLemma

\Proof
We first observe that $\orbS$ is not the boundary of a three-dimensional discal
suborbifold of $\orbM$.  Indeed, if $\orbH$ were a discal suborbifold
with $\bdy \orbH=\orbS$, then by \S\ref{spherical etc}, $
H\doteq
\mani \orbH\subset
\mani \orbM$ would be a $3$--ball with 
$\bdy H=S$, 
and this is impossible since $S$ is essential
in $M$.

It remains
to show that 
$\orbS$ is $\pi_1$--injective in $\orbM$. 
Since $\orbM$ is hyperbolic, it is very good. 
According to
Proposition \ref{when injective},
proving the $\pi_1$--injectivity of $\orbS$ amounts to showing that
if $\orbD$ is a two-dimensional
orientable
  discal suborbifold of
$\orbM$ with $\orbD \cap \orbS = \bdy \orbD$, then $\bdy \orbD$ bounds a discal suborbifold of $\orbS$. 
Since $\orbD$ is an orientable discal suborbifold, $ D \doteq \mani
\orbD$ is a 
disk transverse to $\Sing_\orbM$
  (see \S\ref{Suborbifolds-of-3orbifolds}),
 with $\size(D)\le 1$ and $\size(\bdy D) = 0$.

Now the curve $C\doteq\bdy D$ bounds two disks $ D_1,  D_2 \subset
S$. 
For $i = 1, 2$, 
set $n_i =\size ( D_i)$,
so that $\size(S) = n_1 + n_2$. 
According to \S\ref{what's smoothing}, the sphere $S_i$
obtained by smoothing
$D \cup  D_i$ has the same size as $D \cup  D_i$; hence
$\size S_i = n_i + \size( D) \le n_i + 1$. 
Since $S$ is essential, at least one of the $S_i$ is essential, so we
may assume after relabeling that $S_1$ is essential. The minimality
of $S$ then gives $\size(S) \le \size(S_1)$, i.e.\ $n_1 + n_2 \le
\size(S_1) \le n_1 + 1$. So $n_2 \le 1$, and it follows that $\orbi{
  D_2}$
  (which is defined by virtue of transversality)
is a discal suborbifold of $\orbS$ bounded by $\bdy \orbD$.
\EndProof

The following theorem is the main result of this section.

\Theorem\label{Thm:HomeoTypeSpherePieces} 
Let $\orbM$ be a closed, orientable, hyperbolic $3$--orbifold.
Let $S$ be
a 
strongly
$\orbM$--minimal essential $2$--sphere in $M\doteq\mani \orbM $. Let $\orbN$ be a component of $\orbM \split \orbi S $. 
Assume that $\orbN$ admits the structure of an orbifold book of $I$--bundles. Then $ N \doteq\mani
\orbN $ is homeomorphic 
 either to $\SS^2 \times I$ or 
to a manifold obtained  from a non-trivial lens space by removing the
interiors of one or two closed $3$--balls.
\EndTheorem

The proof of Theorem~\ref{Thm:HomeoTypeSpherePieces} will occupy the remainder of this section. We begin with some
definitions and general lemmas. In particular, Lemmas \ref{therefore}
and \ref{what about} will help us identify disks  in $N$ that cut off
$3$--balls in $N$. In Lemmas~\ref{not more than one, now}
and~\ref{more four}, we gain some structural information about the
bindings in the book of $I$--bundles $\orbN$.

In \S\ref{Subsec:BigPage}, we introduce the crucial notion of a
\emph{big page}: roughly speaking, this will be a page in $\orbN$
whose intersection with $S$ cuts $S$ into disks whose size is at most
$\frac{1}{2} \size(S)$. Big pages anchor our structural understanding
of $\orbN$. In Lemma~\ref{big ol' lemma}, we will prove that a big
page always exists. In Lemma~\ref{both verses}, we will show that
components of the complement of a big page are frequently harmless
(see \S\ref{Def:Harmless}).
 By adjoining harmless $3$--balls to the underlying space of a big page, we will flesh out enough of the topology of $N$ to prove Theorem \ref{Thm:HomeoTypeSpherePieces}.

Theorem  \ref{Thm:HomeoTypeSpherePieces} has the following
immediate corollary, which 
essentially covers the reducible case of Theorem \ref{Thm:IntroNoTurnoverMain}.

\Corollary\label{Cor:HomeoTypeReducible}
Let $\orbM$ be a closed, orientable, hyperbolic $3$--orbifold.
Suppose that 
$M\doteq\mani \orbM $ 
contains  a  
strongly
$\orbM$--minimal essential $2$--sphere $S$ such that each component of
$\orbM \split \orbi S $ 
admits the structure of
a book of $I$--bundles. Then 
$M$ 
is homeomorphic to either a connected sum of two non-trivial lens spaces, or the connected sum of $\SS^2\times \SS^1$ with a possibly trivial lens space.
\NoProof
\EndCorollary

We now proceed to setting up the tools needed to prove Theorem  \ref{Thm:HomeoTypeSpherePieces}.

\Definition\label{was-height def}
Let $\orbM$ be a closed, orientable $3$--orbifold, and let $S$ be a
$2$--sphere in 
$M\doteq\mani \orbM$
which is transverse to $\Sing_\orbM$. Let $N = M \split S$, and let $\iota \from N \to M$ be the canonical immersion.
If $C$ is any simple closed curve in 
$\bdy N$, disjoint from $\Sing_\orbM$, 
then 
$C$ bounds two disks  $D_1, D_2 \subset \bdy N$. 
  Recall from \S\ref{new splitting} that
$N$ is the underlying space of a
$3$--orbifold $\orbN =  \orbM \split \orbi S$, 
so that $\size(D_i)$ is well-defined.
We define the \emph{\textspanningsize} of $C$ to be
\[
\spanningsize(C) = \min \{ \size(D_1), \: \size(D_2) \}.
\]
Note that  $\size(D_i) = \size(\iota(D_i))$, and therefore $\spanningsize(C)\le\size(S)/2$.
\EndDefinition

\Lemma\label{Lem:SameSpanningSize}
Let $\orbM$ be a closed, orientable $3$--orbifold. Set $M=\mani \orbM$.
Suppose that $S$ is an $\orbM$--minimal essential $2$--sphere in $M$. 
Let $N$ be a component of $M \split S$, with canonical immersion $\iota\from  N \to M$.
Suppose that $J$ is a \neatly embedded annulus in $N$ with $\size(\iota(J))=0$. 
Then the two components of $\bdy J$ have 
the same \textspanningsize.
\EndLemma

\Proof
Let $n = \size(S)$. Let $\gamma$ and $\gamma'$ denote the components of $\bdy J$.
Set $s = 
\spanningsize(\gamma)$ and $s' = \spanningsize(\gamma')$.
By symmetry, it is enough to show that $s'\geq s$.

Let $\Psi$ denote the component of $\bdy N = \iota^{-1}(S)$ containing
$\gamma$.
Then $\gamma$ bounds two disks in $ \Psi$; they may be labeled
$\Delta_1$ and $\Delta_2$ in such a way that
$\size(\Delta_1)=s$ and
$\size(\Delta_2)=n-s$. 
Furthermore, $\gamma'$ bounds a disk $E\subset \bdy N$ with $\size(E)=s'$.

Note that $J\cup E$ is a
smoothable topological disk whose boundary is $\gamma$, and that
$\size(J\cup E)=s'$. By a small non-ambient isotopy in $N$ we may
replace $J\cup E$ by a 
smooth,
 \neatly embedded disk $\Delta_3$ in $N$ with
$\bdy \Delta_3=\gamma$ and $\size(\Delta_3)=s'$. 
(This is a  minor variant on the smoothing operation described in
\S\ref{what's smoothing}.)

The disks $\Delta_i$ for $i=1,2,3$ all have boundary $\gamma$, and any two of
them intersect precisely in $\gamma$. Furthermore,
 the map
$\iota|(\Delta_1 \cup  \Delta_2 \cup \Delta_3)$ is
one-to-one. 
So if we set $D_i=\iota(\Delta_i)$, and $C=\iota(\gamma)$, then the
disks $D_i$ all have boundary $C$,
and any two of them intersect precisely in $C$. We have
$D_1\cup D_2 = \iota(\Psi) = S$.
By  \S\ref{what's smoothing}, for $i=1,2$, the sphere $S_i$  obtained by smoothing $D_i \cup
D_3$ has the same size as $D_i \cup
D_3$.
Since $S$ is essential in $M$, at least one of the spheres $S_1$ and $S_2$ is essential.
Since $S$ is minimal and has size\ $n$, at least one of the inequalities $\size(S_1)\ge n$, $\size(S_2)\ge n$ holds.

But we have $$\size(S_1)=\size(D_1)+\size(D_3)=\size(\Delta_1)+\size(\Delta_3)=s+s',$$ and similarly
$\size(S_2)=(n-s)+s'$. Furthermore, the definition of size\ gives $s\le n/2$, so that $s\le n-s$. Thus in any event we have $s'+(n-s)\ge n$, so we conclude that $s'\ge s$ as required.
\EndProof

\Custom{Spanning size of  an annulus or binding}
\label{what's the point}
Let $\orbM$ be a closed, orientable $3$--orbifold. Set $M=\mani \orbM$.
Suppose that $S$ is an $\orbM$--minimal essential $2$--sphere in
$M$. Let $N$ be a component of $M \split S$, 
  so that according to the discussion in \S\ref{new splitting}, 
$N$ is canonically identified with $\mani\orbN$ for some component
$\orbN$ of $\orbM\split\orbi S$.  
If $J$ is a \neatly embedded annulus in $N$ with $\size(\iota(J))=0$, 
then by Lemma \ref{Lem:SameSpanningSize}, the two components of $\bdy J$ have 
the same \textspanningsize. This common \textspanningsize\ will be referred to as the
\emph{\textspanningsize\ of $J$}.

Now suppose that $\orbN $  is
equipped with the structure of
an orbifold book of $I$--bundles, and
let $\orbB $ be a 
binding of $\orbN$ such that $B \doteq \mani\orbB$ is a topological solid torus.
Then every component of 
$\Fr_{N}B $ or $(\bdy N)\cap B $ 
is a size--$0$ annulus. If $A$
is any component of 
$(\bdy N)\cap B $, 
then the two components $C_1$ and $C_2$ of $\bdy A$ also have the same
\textspanningsize\ because they are isotopic in $\bdy
N-(\Sing_\orbN\cap\bdy N)$; hence if $J_i$ denotes the component
of $\Fr_{N}B $ containing $C_i$, we have $\spanningsize(J_1)=\spanningsize(J_2)$.
Since $\bdy B $ 
is connected, it follows that all components of 
$\Fr_{N}B $ 
have the same \textspanningsize. Their common \textspanningsize\ will be
called the \emph{\textspanningsize\ of $\orbB$}.
Note that we have defined the \textspanningsize\ of a binding
  $\orbB$ only in
the case where $\mani\orbB$ is  a
topological solid torus
rather than a topological ball (see \S\ref{PagelikeBindinglike}).
\EndCustom

\Lemma\label{therefore}
Let $\orbM$ be a closed, orientable, hyperbolic $3$--orbifold.  Suppose that $W$ is a (smooth) $2$--sphere in $M\doteq\mani \orbM$, transverse to $\Sing_\orbM$, such that $\orbi W $ is a toric $2$--orbifold. Suppose that the inclusion homomorphism $\pi_1(\orbi W)\to\pi_1(\orbM)$ has infinite image. Then  there is a ball 
$J\subset M$ 
such that $\bdy J=W$, and $\orbi J $ (which is defined by \S\ref{clean}) is a solid toric $3$--orbifold.
\EndLemma

\Proof
Since $\orbM$ is closed and hyperbolic, it has a finite-sheeted regular cover $\torbM$ which is a closed hyperbolic $3$--manifold. Let $p\from \torbM\to \orbM$ denote the covering projection, and $G$ the covering group. Fix a component $T$ of $p^{-1}(\orbi W )$. Since $\orbi W $ is toric, its finite-sheeted manifold covering $T$ is a torus. Since the inclusion homomorphism $\pi_1(\orbi W )\to\pi_1(\orbM)$ has infinite image, and the covering $p\from \torbM\to \orbM$ is finite-sheeted,
the inclusion homomorphism $\pi_1(T)\to\pi_1(\torbM)$ has infinite image.

Since $\torbM$ is closed and hyperbolic, $\torbM$ is irreducible and $T$ is compressible in $\torbM$. We recall that a compressible torus in a hyperbolic manifold either is contained in a ball or bounds a solid torus. If $T$ were contained in a ball in $\torbM$ then
the inclusion homomorphism $\pi_1(T)\to\pi_1(\torbM)$  would be trivial, a contradiction. Hence $T=\bdy V$ for some solid torus $V\subset\torbM$. 

If $g$ is an arbitrary element of $G$, we have either $g\cdot T=T$ or $(g\cdot T)\cap T=\emptyset$. Hence one of the following four alternatives holds: 
\begin{enumerate}[\:\:(i)]
\item\label{Vdisjoint} $(g\cdot V)\cap V=\emptyset$;
\item\label{Vcontains}  $g\cdot V\subset V$; 
\item\label{Vcontained} $g\cdot V\supset V$; or 
\item\label{Vlens} $(g\cdot V)\cup V=\torbM$. 
\end{enumerate}
If \eqref{Vlens} holds then $\torbM$ has a cyclic fundamental group, which is impossible since $\torbM$ is a closed hyperbolic $3$--manifold. If \eqref{Vcontains} holds, and if the inclusion given by \eqref{Vcontains} is strict, then $g^n\cdot V$ is a proper subset of $V$ for every integer $n>0$; this is impossible, since $g$ has finite order. Hence \eqref{Vcontains} implies that $T$ is invariant under $g$; and by the same argument, with $g^{-1}$ in place of $g$, \eqref{Vcontained} also implies that $T$ is invariant under $g$. 

We conclude that for each $g\in G$ we have either $(g\cdot V)\cap
V=\emptyset$ or $g\cdot V= V$. Hence $p$ maps 
$V$ 
onto a suborbifold of $\orbM$ which is (orbifold-) diffeomorphic to the
quotient of $V$ by 
$\Stab_G(V)$.
In particular $\orbJ\doteq p(V)$ is a solid toric orbifold. We have
$\bdy\orbJ=p(\bdy V)=p(T)=\orbi W $. Now set $J=\mani
{\orbJ}$. By orientability we have $\bdy J=\mani {\bdy
  \orbJ}=W$. Since $\orbJ$ is a solid toric orbifold and $W$ is a
sphere, it follows
  by \S\ref{spherical etc}
that $J$ is a ball.
\EndProof

\Lemma\label{what about}
Let $\orbM$ be a closed, orientable, hyperbolic $3$--orbifold.
Suppose that $S$ is a strongly $\orbM$--minimal essential $2$--sphere
in $M\doteq\mani \orbM$. Let $\orbN$ be a component of  $\orbM
\split \orbi S $.
Suppose that $F$ is a \neatly embedded disk in
$ N \doteq\mani \orbN$, transverse to $\Sing_\orbN$,  such that
$\orbi F $ 
(which is defined by \S\ref{Suborbifolds-of-3orbifolds})
is an annular orbifold and is $\pi_1$--injective in
$\orbN$. Then there is a
  topological
ball $K\subset  N $ such that
$\Fr_{ N }K=F$, and $\orbi K $ is a bindinglike suborbifold of $\orbN$.
\EndLemma

\Proof
According to \S\ref{new splitting},
$N$ is
canonically identified with
 a component of 
$M\split S$.
Let $\iota \from  N \to M$ denote the 
restriction to $N$ of the canonical immersion of 
$M\split S $ in $M$. 
Since $F$ is a disk and $\orbF \doteq \orbi F $ is an orientable annular
suborbifold, $F\cap\Sing_\orbN =\Sing_\orbF $ consists of two points, both of order
$2$. Now the curve $C\doteq\bdy F$ is contained in some component
$\tS$ of $\bdy N $, and bounds two disks $D_1,
D_2 \subset \tS$. Set $n_i =\size (D_i)$ for $i = 1, 2$
and $n = \size(S) = n_1 + n_2$. 
According to  \S\ref{what's smoothing}, for $i = 1, 2$, the sphere $S_i$
obtained by smoothing
$\iota (F \cup D_i)\subset M$ has the same size as $\iota (F \cup
D_i)$; hence
$\size S_i = n_i + \size(F) = n_i + 2$. 
Since $S$ is essential
  in $M$,
at least one of the $S_i$ is essential,
so we may assume after relabeling that $S_1$ is essential. 
The
minimality of $S$ then gives $\size(S) \le \size(S_1)$, i.e.\ $n_1 +
n_2 \le \size(S_1) = n_1 + 2$. So $n_2 \le 2$. If $n_2$ were at most
$1$, the curve $\bdy F=\bdy D_2$ would represent the conjugacy
class of an element of finite order in the orbifold fundamental group
$\pi_1(\orbN)$, contradicting the $\pi_1$--injectivity of the annular
orbifold $\orbi F $.
Hence $n_2=2$. It now follows that $\size(S_1)=n$, so that $S_1$ is $\orbM$--minimal. The strong $\orbM$--minimality of $S$ then implies that $\weighttwo(S)\le\weighttwo(S_1)$. 

Since $n_2=2$, the $2$--orbifold $\orbi  {D_2}$ has two singular
points, whose orders are integers strictly greater than $1$; we denote 
these integers 
by $p$ and $q$. We then have 
\[
\weighttwo(S)=\weighttwo(D_1)+\weighttwo(D_2)=\weighttwo(D_1)+p+q. 
\]
On the other hand, since the two singular points of $\orbi F $ are of order $2$, we have
\[
\weighttwo(S_1)=\weighttwo(D_1)+\weighttwo(F)=\weighttwo(D_1)+4. 
\]
The inequality
$\weighttwo(S)\le\weighttwo(S_1)$ then becomes $\weighttwo(D_1)+p+q\le
\weighttwo(D_1)+4$, so that $p+q\le4$ and hence $p=q=2$. 
Since $S_2$ is obtained by smoothing $\iota(F\cup D_2)$, it follows that
$\orbi  {S_2}$ 
has four singular points, all of order $2$. Since $S_2$ is a sphere, it follows that $\orbi  {S_2}$ is a toric orbifold.
Furthermore, since $\orbi F \subset\orbi { S_2}$ is annular and $\pi_1$--injective in $\orbN$,  the inclusion homomorphism $\pi_1(\orbi  
{S_2})\to\pi_1(\orbM)$ has infinite image. We may therefore apply
Lemma \ref{therefore}, with $S_2$ playing the role of $W$, to deduce
that there is a topological ball $J\subset M$ such that $\bdy J=S_2$, and
$\orbi J $ is a solid toric $3$--orbifold.
Hence $\iota(F \cup D_2)$ bounds a topological $3$--ball $K^* \subset
M$, and $K^*$ is isotopic to $J$ by a topological ambient isotopy
of $M$  which is constant on $\Sing_\orbM$. In particular, \sta\ of  the
orbifold with corners $\orbK^*\doteq\orbi{K^*}$ yields a solid toric orbifold.

Since  $S$ is essential in $M$, we cannot have $S\subset K^*$. Hence 
there is a topological $3$--ball $K \subset N$ whose frontier is $F$,
such that $\iota$ maps $K$ diffeomorphically onto $K^* \subset M$. Now $\orbi K$ is a well-defined \clean\ suborbifold with corners
 of $\orbN$, and is orbifold-diffeomorphic to
$\orbK^*$; thus \sta\ of  
$\orbi{K}$ yields a solid toric orbifold. It
then
follows from the definition in \S\ref{PagelikeBindinglike} that $\orbi K$ is bindinglike.
\EndProof

\Custom{Standing hypotheses and conventions}
\label{standing}
For the remainder
of the section, the hypotheses of Theorem~\ref{Thm:HomeoTypeSpherePieces} will be assumed to
hold. Thus it will be understood that we are given a closed,
orientable, hyperbolic $3$--orbifold  $\orbM$ and a 
strongly
$\orbM$--minimal essential $2$--sphere $S$ in $M\doteq\mani \orbM
$; we shall write $\orbS=\orbi S$, which is a well-defined suborbifold
of $\orbM$ by 
\S\ref{Suborbifolds-of-3orbifolds}.
It will further be understood that we are given 
 a component  $\orbN$ of $\orbM \split \orbi S $, so that by
\S\ref{new splitting},
$\orbN=\orbi
 N$ for some component $N$ of $M\split S$. It will be assumed that $\orbN$
is equipped with the structure of an orbifold book of $I$--bundles. We
shall write $n=\size(S)$.
\EndCustom

\Lemma\label{not more than one, now}
Let $\scrbstar$ denote
the union of all bindings of $\orbN$ whose underlying spaces are
topological solid tori, and whose  \textspanningsize\
is $n/2$
(see \S\ref{what's the point}).
Then each component of $\bdy N $ contains at most one component of $\mani {\scrbstar}\cap\bdy N $. 
\EndLemma

\Proof
Suppose there is a component $\tS$ of $\bdy  N $ which contains
two distinct components of $\mani \scrbstar\cap\bdy N $, say
$A_1$ and $A_2$. Let $D_1$ denote the component of
$\overline{\tS-A_1}$ that contains $A_2$. Then $D_1$ contains a unique
component $D_2$ of $\overline{\tS-A_2}$. Both $D_1$ and $D_2$ are
disks. Since each $A_i$ is a component of the intersection of
$\bdy  N $ with a  topological solid torus which is the
underlying manifold of a binding $\orbB_i$,
and $\orbB_i$ has \textspanningsize\ $n/2$, 
the boundary curves of $A_i$ have \textspanningsize\ $n/2$
and hence 
$\size(D_i)=n/2$.
Now if $J$ denotes the annulus $\overline{D_1-D_2}$, we
have $\size(D_1)=\size(D_2)+\size(J)$, and hence
$\size(J)=0$. 
It follows that $\orbJ\doteq\orbi J$ has Euler characteristic
$0$. Since $\bdy J=\bdy A_1\cup\bdy
A_2\subset\boldC_{\orbN,\bdy\orbN}$, 
it follows from Lemma \ref{saddy} that $\orbJ$ is
a component of $\scrA\doteq\scrA_{\orbN,\bdy\orbN}$. 
Since $J$ shares a boundary curve
with the component $A_1 \ne J$ of $\boldA\doteq\mani \scrA$, 
this is
a contradiction. 
\EndProof

\Lemma\label{more four}
Let $\orbB$ be a binding of
$\orbN$ such that $B\doteq\mani \orbB$ is a topological ball, and let $F$ be a
component of $\Fr_NB$ which is an annulus. Then 
each component of $\bdy F$ bounds a disk of \textweighttwo\ $4$ in $\bdy
  N$. 
(As discussed in Remark \ref{what weight 4 means}, such a disk must have
exactly two cone points, each of which has order $2$.)
\EndLemma

\Proof
According to Definition \ref{PagelikeBindinglike}, 
\sta\ of $\orbB$ yields a
solid toric orbifold, and $\orbF\doteq\orbi F$ is $\pi_1$--injective in
$\orbN$. Since $B=\mani \orbB$ is a topological ball, we have $\size(\bdy
B)=4$, and each singular stratum of $\orbB$
  is a cone stratum of order $2$ (see \S\ref{spherical etc}).
Since
$F\subset\bdy B$ is a size--$0$ annulus, and $\orbF$ is
in particular $\pi_1$--injective in
$\orbB$, 
any
boundary component $C$ of $F$ bounds a disk $D\subset\bdy B$ such
that 
$\weighttwo(D)=4$. 
Furthermore,
  $\orbD\doteq\orbi D$
is $\pi_1$--injective in $\orbB$; since the
$\pi_1$--injectivity of $\orbF$ in $\orbN$ implies that $\orbB$ is
$\pi_1$--injective in  $\orbN$, the orbifold $\orbD$ is also
$\pi_1$--injective in  $\orbN$. We can modify $D$ by a small non-ambient isotopy in $B$ to obtain
a \neatly embedded disk $D_1$ in $B$, disjoint from $\Fr_NB$ and
transverse to $\Sing_\orbB$,  such that  
$\weighttwo(D_1)=4$, and
  $\orbD_1\doteq\orbi{D_1}$
is $\pi_1$--injective in  $\orbN$. Note  that $D_1$ is
\neatly embedded in $N$. We may therefore apply Lemma \ref{what about}, with $D_1$ playing the role of $F$,
to obtain  a topological ball $K\subset  N $
such that $\Fr_{ N }K=D_1$ and $\orbi K$ is bindinglike, so that straightening the angles of
$\orbi K $ yields a solid toric $3$--orbifold.
  We therefore have $\weighttwo(\bdy K)=8$.
Now $E\doteq\overline{(\bdy
  K)-D_1}$
is a disk in $\bdy N$ whose boundary is isotopic in
$(\bdy N)-(\Sing_\orbN\cap\bdy N)$ to $C$, and
$\weighttwo(E)=\weighttwo(
  \bdy K)-\weighttwo(D_1)=4$.
Hence $C$ bounds a \textweighttwo--$4$ disk in $\bdy N$.
\EndProof

\Custom{Big pages}\label{Subsec:BigPage}
Recall that our goal is to prove Theorem~\ref{Thm:HomeoTypeSpherePieces},
which describes the topological structure of $N$ 
under the standing assumptions given in \S\ref{standing}.
This topology will be anchored by (the underlying space of) a \emph{big page}, which we define now.

A page $\orbP$  of
$\orbN$ is called \emph{\Sbig} if there exist a 
boundary component $\tS$ of $N\doteq \mani \orbN$ 
and a component $H$ of $\mani \orbP\cap\tS$, such that every component of $\overline{\tS-H}$ has size\ at most $n/2$. 
\EndCustom

\Remark\label{Rem:BigIsRare}
It is straightforward to show that big pages are quite rare.
Indeed, suppose $\orbP$ is a big page and $H$ is the component of $\mani \orbP \cap \tS$ stipulated in the definition. Let $D_1, \ldots, D_k$ be the components of $\overline{\tS-H}$, and let $C_i = H \cap D_i$. Observe that if $k \geq 2$ and $1 \leq i \neq j \leq k$, then $C_i$ and $C_j$ co-bound an annulus in $\tS$, which can be shown to have non-zero size by Lemma~\ref{saddy}. Thus at most one $D_i$ can have size equal to $n/2$.
If $\size(D_i) < n/2$, then $D_i$ cannot intersect the underlying space of any  big page. If $\size(D_i) = n/2$, then another application of Lemma~\ref{saddy} implies that $D_i$ has at most one component of intersection with the underlying space of any big page. Thus $\tS$ intersects at most two underlying spaces of big pages. A more careful version of this argument implies that $\orbN$ can contain at most two big pages; see the final paragraph of  \S\ref{ChaiProof}.

Despite the difficulty of satisfying Definition~\ref{Subsec:BigPage}, the next lemma shows that big pages always exist.
\EndRemark

\Lemma\label{big ol' lemma}
The book of $I$--bundles
$\orbN$ has at least one \Sbig\ page.
\EndLemma

\Proof
Choose a component $\tS$ of $\bdy  N $, 
and set $\tcalS=\orbi \tS$. 
Using the notation of \S\ref{wuzza book}, set 
$\scrA=\scrA_{\orbN,\tcalS}$, $\scrH=\scrH_{\orbN,\tcalS}$,
and $\boldC = \boldC_{\orbN,\tcalS}$. 
Then $\boldH \doteq\mani \scrH$ and $\boldA \doteq\mani
\scrA$ are $2$--manifolds with boundary whose union is $\tS$,
and $\boldC=\boldH\cap\boldA $
is a closed $1$--manifold. 
A pair of components $A \subset \boldA$ and $H \subset \boldH$ are called \emph{adjacent} if $A \cap H \neq \emptyset$. Note that $A \cap H \subset \boldC$, and since every component of $\boldC$ separates $\tS$,  adjacency is equivalent to saying that $A \cap H$ is a component of $\boldC$.
Let $T$ be the adjacency graph of components of $\boldA$ and $\boldH$. That is, every vertex of $T$ corresponds to a component of $\boldA$ or a
component of $\boldH$, and two vertices are connected by an edge if and only if the corresponding components are adjacent. Since every simple closed curve on $\tS$ is separating, $T$ is necessarily a tree.

For each vertex $v$ of $T$,
let us denote by $K_v$ the component of $\boldA$ or of $\boldH$
corresponding to $v$, and let us set $m_v=\size(K_v)$. 
Given this definition of $m_v$, we may define 
$n_U=\sum_{v\in U\cap V }m_v$ for each  set $U\subset T$, as in the statement of Proposition~\ref{tree thing}. Note that $n_T=\size(\tS)=n$. Proposition~\ref{tree thing} now gives a vertex $w_0\in V $ such that, for each component $R$ of $T-\{w_0\}$, we have $n_R\le n_T/2=n/2$. Then $L\doteq K_{w_0}$ is a component of either $\boldH$ or $\boldA$. 

If $Z$ is any component of $\overline{\tS-L}$, there is a component $R$ of $T-\{w_0\}$ such that $Z=\bigcup_{v\in R\cap V }K_v$. Hence $\size(Z)=\sum_{v\in R\cap V }\size(K_v)=n_R\le n/2$. Thus:

\Claim\label{i knew it}
Every component of $\overline{\tS-L}$ has size\ at most $n/2$.
\EndClaim

We now claim:

\Claim\label{i still know it}
There is a component $H$ of $\boldH$ such that every component of $\overline{\tS-H}$ has size\ at most $n/2$.
\EndClaim

To prove \ref{i still know it}, first note that in the case where $L$ is a component of $\boldH$, the assertion follows immediately from \ref{i knew it} upon setting $H=L$.

We now turn to the case in which $L$ is a component of $\boldA$. Then $\orbi  L$ is an annular orbifold. Hence either (a) $L$ is an annulus and $\size(L)=0$, or (b) $L$ is a disk, $\size(L)=2$, and each of the two singular points of $\orbi  L$ is of order $2$. If (a) holds, 
let $H$ be a component of $\boldH$ adjacent to $L$.
Since (a) holds, it now follows that $H$ is isotopic to $L\cup H$ in $\tS\cap( N -\Sing_\orbN)$. Hence any component of $\overline{\tS-H}$ is isotopic in $\tS\cap( N -\Sing_\orbN)$ to a component of $\overline{\tS-(L\cup H)}$, which is in turn contained in a component of $\overline{\tS-L}$ and therefore has size\ at most $n/2$ by \ref{i knew it}. Since an isotopy in $\tS\cap( N -\Sing_\orbN)$ obviously preserves size, this shows that every component of $\overline{\tS-H}$ has size\ at most $n/2$. This establishes \ref{i still know it} in the subcase where (a) holds.

There remains the subcase in which (b) holds. In this subcase,
$\overline{\tS-L}$ is a disk and has size
 $n-2$.  
Hence
\ref{i knew it} gives $n-2\le n/2$, so  that $n\le4$. 

If $n\le3$ then $\size(\overline{\tS-L})\le1$. This means that
$\orbi{\overline{\tS-L}}$ is a discal orbifold, so that $\bdy L$
defines an element of finite order in $\pi_1(\orbN)$; this is a
contradiction, since $\bdy L$ is a component of 
$\boldC=\orbi
\boldC$,
which is
 $\pi_1$--injective in $\orbN$ by
\S\ref{wuzza book}. 
Hence $n=4$.

Now let $H$ denote the component of $\boldH$ 
adjacent to $L$. 
The component $L$ of $\overline{\tS-H}$ has size\ $2$. If $Y$ is any component of $\overline{\tS-H}$ distinct from $L$, we have 
\[
4 = n = \size(\tS)\ge\size(Y)+\size(L)=\size(Y)+2,
\]
so that $\size(Y)\le2$. Thus every component of $\overline{\tS-H}$ has size at most $2=n/2$. This completes the proof of \ref{i still know it}.

Now if $H$ is the component of $\boldH$ given by \ref{i still know it}, there is a page $\orbP$ of $\orbN$ such that $L$ is a component of $\mani \orbP\cap\tS$. It then follows from \ref{i knew it}, and the definition of a \Sbig\ page, that $\orbP$ is a \Sbig\ page of $\orbN$. This gives the conclusion of the lemma.
\EndProof

\Lemma\label{middle-sized lemma}
Let $\orbP$ be a page of $\orbN$,
and set
$ P=\mani\orbP$. 
Assume that  some component of $\Fr_{  N } P$ is an annulus of \textspanningsize\ $n/2$. Then $\orbP$ is \Sbig.
\EndLemma

\Proof
Fix a  component $F$ of $\Fr_{ N } P$ which is an annulus of \textspanningsize\ $n/2$,  fix a component $C$ of $\bdy F$, and let $\tS$ denote the component of $\bdy  N $ containing $C$. Then $C$ is 
contained in the boundary
of some component $H$ of $ P\cap\tS$. According to the definitions, we have $\spanningsize(C)=\spanningsize(F)=n/2$. Hence each of the two disks contained in $\tS$ and bounded by $C$ has size\ equal to $n/2$. These disks may be labeled $D$ and $D'$, where $D$ is a component of $\overline{\tS-H}$ and $D'\supset H$. Each component of $\overline{\tS-H}$ either is equal to $D$ or is contained in $D'$; hence each component of $\overline{\tS-H}$ has size\ at most $n/2$. It now follows from the definition that $\orbP$ is \Sbig.
\EndProof

\Lemma\label{EveryCompBigSize} 
Let a \Sbig\ page $\orbP$  of $\orbN$
be given. Let $\tS$ be any 
boundary component of $ N \doteq \mani \orbN$, and let $H$ be any
component of the intersection of
$ P\doteq\mani \orbP$ with $\tS$. 
Then every component of $\overline{\tS-H}$ has size\ at most $n/2$.
\EndLemma

\Proof
The orbifold $I$--bundle structure of $\orbP$ implies that
$ P\cap\tS$ has at most two components. If it has only one
component, the assertion is immediate from the definition  of a \Sbig\ page. Now suppose it has two components, $H_0$ and $H_1$. Let $\tS_i$ denote the component of $\bdy  N $ containing $H_i$. (The $\tS_i$ are not necessarily distinct.) Since $\orbP$ is \Sbig, we may suppose the $H_i$ to be labeled in such a way that every component of $\overline{\tS_0-H_0}$ has size\ at most $n/2$. It remains to show  that every component of $\overline{\tS_1-H_1}$ has size\ at most $n/2$.

Let $q \from \orbP \to \orbH$ be the fibration map of the $I$--bundle $\orbP$.
Since the horizontal boundary $\horizbdy \orbP$ has two components $\orbi {H_0}$ and
$\orbi {H_1}$, it follows from   \S\ref{Def orbifold I-bundle} that the base $\orbH$ is orientable, and that each of $\orbi {H_0}$ and
$\orbi {H_1}$ is
orbifold-diffeomorphic to $\orbH$. In particular, 
$w\doteq\size(H_0)=\size(H_1)$. Each component of $\vertbdy \orbP $ must meet both components of $\horizbdy \orbP$.
In particular no component of $\Fr_{ N } P = \mani {\vertbdy \orbP}$ has connected boundary;
since  the components of $\Fr_{ N } P$ are underlying surfaces of annular orbifolds, they must be size--$0$ annuli. Hence $H_0$ and $H_1$ have the same number of boundary curves, which we denote by $m\ge0$, and we may write $\bdy H_i=\bigcup_{1\le j\le m}C_i^j$, where $C_0^j$ and $C_1^j$ are joined by a size--$0$ annulus in $ N $ for $1\le j\le m$. According to 
Lemma \ref{Lem:SameSpanningSize}, we have
$\spanningsize(C_0^j)=\spanningsize(C_1^j)$ for every $j$.

For all $i$ and $j$,
let $D_i^j$ denote the component of
$\overline{\tS_i - H_i}$ containing $C_i^j$. For $1\le j\le m$ we have
$\size(D_0^j)\le n/2$, and we must show that $\size(D_1^j)\le
n/2$. This is vacuously true if $m=0$, so we may assume $m>0$. Since
the disks $D_1^j\subset\tS_1$ are pairwise disjoint, we have
$\sum_{1\le j\le m}\size(D_1^j)\le \size\tS_1=n$; hence there is at
most one index $j$ for which $\size(D_1^j)>n/2$. Thus, after
re-indexing if necessary, we may assume that $\size(D_1^j)\le n/2$ for
$1\le j<m$.  We now
  need only prove that $\size(D_1^m)\le n/2$.

For $i=0$ and $1\le j\le m$, and also for $i=1$ and $1\le j<m$, since $\size(D_i^j)\le n/2$, the definition of \textspanningsize\ gives $\size(D_i^j)=\spanningsize(C_i^j)$. Note also that for $i=0,1$ we have $w+\sum_{1\le j\le m}\size(D_i^j)=\size(\tS_i)=n$. Hence
\begin{align*}
w+\sum_{1\le j\le m}\spanningsize(C_0^j)&=w+\sum_{1\le j\le m}\size(D_0^j) \\
&= w+\sum_{1\le j\le m}\size(D_1^j)\\
&=w+ \Big(\sum_{1\le j<m}\spanningsize(C_1^j) \Big) +\size(D_1^m)\\
&=w+ \Big( \sum_{1\le j< m}\spanningsize(C_0^j) \Big)  +\size(D_1^m),
\end{align*}
which gives $\size(D_1^m)=\spanningsize(C_0^m)$. The definition of \textspanningsize\ now implies that $\size(D_1^m)\le n/2$, as required.
\EndProof

Our next step on the way to proving Theorem~\ref{Thm:HomeoTypeSpherePieces} is to show that when $\orbP$ is a big page, the components of $N - \mani \orbP$ are often harmless.

\Lemma\label{both verses}
Let $\orbP$ be a \Sbig\
page of $\orbN$, and set $ P=\mani\orbP$. Let $F$ be a component of
$\Fr_{  N } P$ such that either 
\begin{enumerate}[\:\:(a)]
\item $F$ is an annulus  with
$\spanningsize(F)<n/2$, or 
\item 
the binding of $\orbN$ containing  $\orbi F$ has a 
topological ball as its  underlying space.
\end{enumerate}
Let $V$ denote the component of $\overline{ N - P}$ containing
$F$. Then $V$ is \harmless, and is disjoint from 
the underlying manifold of
every \Sbig\ page of $\orbN$ other than $\orbP$.
\EndLemma

\Proof
First we observe that if there is a
\Sbig\ page $\orbQ\ne \orbP$ of $\orbN$ such that  $\mani
\orbQ$ has non-empty intersection with 
$V$, then since $\mani\orbQ$ is connected and disjoint from $P$ we must have
$\mani\orbQ\subset V$. It therefore suffices to prove that (I) $V$ is
\harmless, and (II) there is no big page $\orbQ\ne \orbP$ of $\orbN$
such that $\mani\orbQ\subset V$.

Consider the case in which (a) holds. We set $h=\spanningsize(F)$, so that 
$h<n/2$. Let $C_1$ and $C_2$ denote the components of $\bdy F$. For $i=1,2$ let $H_i$ denote the component of $ P\cap\bdy  N $ containing $C_i$, and let $\tS_i$ denote the component of $\bdy N $ containing $H_i$. (The $H_i$ are not necessarily distinct, and in particular the $\tS_i$ are not necessarily distinct.) 
For $i=1,2$ let $D_i$ denote the component of $\overline{\tS_i-H_i}$ containing $C_i$. Then $D_i$ is a disk, and it follows from Lemma \ref{EveryCompBigSize} that $\size(D_i)\le n/2$. The definitions of the \textspanningsize\ of a simple closed curve, and  of the \textspanningsize\ of an annulus, then imply that  $\size(D_i)$ is equal to $h$ (and is in particular strictly less than $n/2$).

We claim:
\Claim\label{snores} 
For $i=1,2$, the interior of $D_i$ is disjoint from 
the underlying manifold of every \Sbig\ page of $ \orbN$, 
including $\orbP$.
\EndClaim

If \ref{snores} is false, then $\inter D_i$ contains a component 
$H$ of $\mani \orbQ\cap\tS_i$ for some \Sbig\ page $\orbQ$ of $\orbN$. But then the disk $\overline{\tS_i-D_i}$ is contained in some component $E$ of $\overline{\tS_i-H}$. According to Lemma \ref{EveryCompBigSize} we have $\size(E)\le n/2$. But $\size(E)\ge\size(\overline{\tS_i-D_i})=n-h>n/2$, a contradiction. This proves \ref{snores}.

According to \ref{snores}, $\inter D_i$ is in particular disjoint from $P$ for $i=1,2$. Hence:
\Claim\label{shores} 
For $i=1,2$, the disk $D_i$ is a component of $\overline{(\bdy N )-( P\cap \bdy N )}$. 
\EndClaim

The disks $D_1$ and $D_2$, which by \ref{shores} are components of
$\overline{(\bdy N )-( P\cap \bdy  N )}$, 
are distinct because their boundaries $C_1$ and $C_2$ are
distinct. Hence $D_1\cap D_2=\emptyset$. Since the annulus $F$ is
\neatly embedded in $ N $, we have $D_i\cap F=C_i$ for $i=1,2$, and
$ D_1\cup F\cup D_2$ is a smoothable 
$2$--sphere. 
By a minor variant of the smoothing operation described in \S\ref{what's smoothing},  $ D_1\cup F\cup D_2$ is non-ambiently isotopic to  a smooth $2$--sphere $R$ contained in the interior of $N$ and transverse to the singular set of
$\orbN$, with the same
 size as $ D_1\cup F\cup D_2$.

Since $\size(D_1)=\size(D_2)=h$ and $\size(F)=0$, we
have $\size(R)=2h$.  The
canonical immersion from $ N $ to $M$ maps $R$
diffeomorphically 
to a smooth 
sphere $R'$ in $M$ which is
transverse to the singular set of $\orbM$ and has 
size  $2h$. Since $2h<n$ and $S$ is 
$\orbM$--minimal, the sphere $R'$ cannot be essential in
$M$; hence $R'$ bounds a smooth  ball $W\subset M$. 
Now $\bdy W=R'$ is disjoint from $S$, and since $S$ is
essential it cannot be contained in $\inter W$. Hence $W$ is
disjoint from $S$. It now follows that 
$ D_1\cup F\cup D_2$ 
bounds a topological ball $K\subset
 N $. Thus we 
have $\Fr_{ N }K=F$. 
Since $F$ is a component of $\Fr_{ N } P$, either (i) $K$ is a component of $\overline{ N - P}$ or (ii) $K\supset P$. But (ii) would imply  $ P\cap\bdy N \subset K\cap\bdy N =D_1\cup D_2$, which contradicts \ref{shores}. Hence (i) holds, i.e.\ 
$K=V$. In particular 
$V$ is a topological
ball and its frontier is the single annulus $F$. Thus 
$V$ is \harmless, and Assertion (I) is established in this case.

To prove Assertion (II) in this case, suppose that there is a
\Sbig\ page $\orbQ\ne \orbP$ of $\orbN$ such that  $Q\doteq\mani
\orbQ\subset V=K$.
Since $\orbQ$ is a page of the book of $I$--bundles
$\orbN$, we have $Q\cap\bdy N \ne\emptyset$. But we
have $K\cap\bdy N  =D_1\cup D_2$, and hence
$Q $ must meet the interior of 
either $D_1$ or $D_2$. 
This contradicts \ref{snores}, 
and completes the proof of the lemma in the case where (a) holds.

We now turn to the case in which (b) holds.
Consider the
subcase in which $F$ is a disk.
In this subcase, 
since $\orbi F $ is $\pi_1$--injective in $\orbN$ 
(see \S\ref{wuzza book}),
it follows from  Lemma \ref{what about} that there is a topological ball 
$K\subset N$ 
such that $\Fr_{ N  }K=F$, and $\orbi K $ is a bindinglike suborbifold of $\orbN$.
Hence $\pi_1(\orbi K )$ is virtually cyclic.
Since $F$ is a component of $\Fr_{ N  } P $, either (i) $K$ is a component of $\overline{ N  - P }$ or (ii) $K\supset P $.   If (ii) holds, then  since $\orbP$ is a $\pi_1$--injective suborbifold of $\orbN$ by
\S\ref{wuzza book},
$\orbP$ is in particular a $\pi_1$--injective suborbifold of  $\orbi K
$,
 and hence the page $\orbP$ of $\orbN$ has a virtually cyclic (orbifold) fundamental group. This contradicts
an observation made in \S\ref{wuzza book}.
Hence (i) holds, i.e.\ 
$K=V$. In particular 
$V$ is a topological ball and its frontier is the single disk $F$. Thus 
$V$ is \harmless, and 
Assertion (I) is established in this subcase.

To prove Assertion (II) in this
subcase,
suppose that for some \Sbig\ page $\orbQ\ne \orbP$ of $\orbN$ we have $\mani \orbQ \subset V$. 
Since $\orbQ$ is a $\pi_1$--injective suborbifold of $\orbN$ by
\S\ref{wuzza book},
$\orbQ$ is in particular a $\pi_1$--injective suborbifold of  $\orbi K
$, and hence the page $\orbQ$ of $\orbN$ has a virtually cyclic
(orbifold) fundamental group. Again this contradicts
an observation made in \S\ref{wuzza book}.

There remains the subcase in which $F$ is an annulus. Let $\orbB$
denote the binding of $\orbN$ containing $\orbi F$. 
Since (b) holds,
$B\doteq\mani \orbB$ is a topological ball. It therefore follows from Lemma
\ref{more four}
and Remark \ref{what weight 4 means}
that 
every component of $\bdy F$ bounds a disk in $\bdy N$ that has \textweighttwo\ $4$.
Thus, by Lemma \ref{just four simpler}, $F$ is the full frontier of
$B$ relative to $N$. 
It follows that $V=B$ in
this case, so that $V$ is a topological ball whose frontier is $F$. This means that
$V$ is \harmless, and 
Assertion (I) is established
in this subcase. 
Assertion (II) follows immediately, because $V$ is the underlying space of the binding $\orbB$,
and so cannot contain the underlying space of any page.
\EndProof

\Custom{Global structure: the proof of Theorem \ref{Thm:HomeoTypeSpherePieces}}\label{ChaiProof}

In \S\S\ref{not more than one, now}--\ref{both verses}, we
have seen how the hypotheses laid out in \S\ref{standing} give
information about the book of $I$--bundles $\orbN$. We shall now assemble
these facts to prove Theorem~\ref{Thm:HomeoTypeSpherePieces}, which describes the
topological type of $N=\mani\orbN$ under these hypotheses. 
In outline, the proof runs as follows. For every big page $\orbP_j$ of $\orbN$, we adjoin certain harmless components to the underlying space $P_j$ to form a set $\hatP_j$ that is itself a \clasibun. It turns out that $N$ is the union of the \clasibun s $\hat P_1, \ldots, \hatP_k$ with
certain solid torus bindings. We then show that every $\hatP_j$ is a \clasibun\ over $\SS^2$, $\RR\PP^2$, or $\BB^2$, and in the latter case $N$ is formed by attaching one or two 
$2$--handles to a solid torus.  
The theorem then follows easily.

\Proof[Proof of Theorem \ref{Thm:HomeoTypeSpherePieces}]
As in \S\ref{standing}, we set 
$n=\size(S)$.
Let $\orbP_1,\ldots,\orbP_k$ be the big pages of $\orbN$, and observe that $k \geq 1$ by Lemma \ref{big ol' lemma}.
For
$j=1,\ldots,k$, 
we set $ P_j=\mani{\orbP_j}$, and we
let $\boldF_j$ denote the set of all components $F$  of $\Fr_{ N 
} P_j $ such that either 
\begin{enumerate}[\:\:(a)]
\item $F$ is an annulus  with
$\spanningsize(F)<n/2$, or 
\item
the underlying space of the binding  of $\orbN$ 
containing  $\orbi F$ is homeomorphic to a ball.
\end{enumerate}
For each $F\in\boldF_j$, let $V_F$ denote the component of $\overline{ N  - P_j}$ containing $F$. Then  Lemma \ref{both verses} implies:
\Claim\label{zeero}
For each $j\in\{1,\ldots,k\}$ and each $F\in\boldF_j$, the component
$V_F$ of $\overline{ N  - P_j}$ is \harmless, and is disjoint
from $ P_{j'}$ 
for every $j'\ne j$.
\EndClaim

Let an index $j\in\{1,\ldots,k\}$ be given.
According to the definition of a harmless component,
 for each 
$F\in\boldF_j$, the frontier of $V_F$ is
connected and is therefore equal to $F$. In particular, if we fix an
index $j$, then for any distinct elements $F$ and $F'$ of $\boldF_j$ we
have $V_F\ne V_{F'}$. For each $j$ we set $ \hatP 
_j= P_j\cup\bigcup_{F\in\boldF_j}V_F$. We apply 
 Proposition \ref{Prop:ExpandedPage}, letting $\orbP_j$ play the role of
 $\orbP$, and letting $\boldV_j\doteq \bigcup_{F\in\boldF_j} V_F$
 play  the role of   $\boldV$.   
 Proposition \ref{Prop:ExpandedPage}\eqref{ExpandedPageIsBundle} implies 
that the  manifold $ \hatP_j$ may be given the structure of
a \clasibun\ in such a way that $\Fr_{ N  } \hatP_j$ is a vertical
subset of $\vertbdyclas \hatP_j$. As we observed in \S\ref{Def:Harmless}, 
$\hatorbP_j\doteq\orbi { \hatP _j}$ is a well-defined \clean\ suborbifold
with corners of $\orbN$, so that $\hatP _j$ 
is a clean submanifold with corners of $N$.
Furthermore, Proposition \ref{Prop:ExpandedPage}\eqref{ExpandedPageFrontier} and the definition of $\boldV_j$ imply that
$\Fr_{ N  }\hatP _j$ is the union of all components of $\Fr_{ N 
} P_j$ that do not belong to $\boldF_j$, i.e.\ do not satisfy (a) or (b). A component $E$ of $\Fr_{ N 
} P_j$ fails to satisfy (b) if and only if  the underlying set 
of the
binding containing $\orbi E$ 
is
a  topological solid torus, which by 
\S\ref{spherical etc}
implies
in particular that $E$ is
an annulus. An annulus component $E$ of $\Fr_{ N 
} P_j$ fails to satisfy (a) if and only if  
$\spanningsize(E)=n/2$. Hence we have shown:

\Claim\label{jeepers}
 $\Fr_{ N  }\hatP _j$ is the union of all components of $\Fr_{ N 
 } P_j$ which are annuli of \textspanningsize\ $n/2$
and are contained in 
topological
solid 
tori that are underlying manifolds of bindings.
\EndClaim

We claim:

\Claim\label{fish and chips}
The 
sets
$ \hatP  _1,\ldots, \hatP  _k$ are pairwise disjoint.
\EndClaim

To prove \ref{fish and chips}, we must
 check that when  $j,j'\in\{1,\ldots,k\}$ are distinct indices, three conditions hold: first, $ P_j\cap  P_{j'}=\emptyset$; second, $ P_{j'}\cap V_F=\emptyset$ for each $F\in\boldF_j$; and third, $V_F\cap V_{F'}=\emptyset$ for all $F\in\boldF_j$ and $F'\in\boldF_{j'}$.
If  $F\in\boldF_j$ and $F'\in\boldF_{j'}$, it
is immediate that $ P_j\cap  P_{j'}=\emptyset$ since $\orbP_j$ and $\orbP_{j'}$ are distinct pages of $\orbN$. According to Lemma \ref{both verses}, $V_F$ is  disjoint from every \Sbig\ page of $\orbN$ other than $\orbP_j$; in particular $V_F\cap \orbP_{j'}=\emptyset$. Finally, to show that the connected, proper subsets $V_F$ and $V_{F'}$ of $ N  $ are disjoint, 
it is enough to prove that each is disjoint from the frontier of the other; and by symmetry, we need only prove that  $V_F$ is disjoint from $\Fr_{ N  }V_{F'}$. But this follows from the fact that 
$V_F$ is disjoint from $\orbP_{j'}$, as the latter set contains $\Fr_{ N  }V_{F'}$. Thus
\ref{fish and chips} is established.

Next, we claim:

\Claim\label{Claim:IfEmptyFrontier}
If $\Fr_N \hatP_j = \emptyset$ for some $j$, then $N$ is homeomorphic to either $\SS^2 \times I$ or the $3$--manifold obtained by removing 
the interior of a closed
 $3$--ball from $\RR \PP^3$.
\EndClaim

To prove \ref{Claim:IfEmptyFrontier}, assume that $\Fr_N \hatP_j = \emptyset$. Since $N$ is connected, it follows that $\hatP_j = N$ (and that $j=k =1$). Thus $N$ is homeomorphic to a \clasibun\ over some closed surface of positive Euler characteristic. 
Thus $ N  $ is homeomorphic
either to the trivial \clasibun\ 
$\SS^2 \times I$ or to a twisted \clasibun\ over
$\RR \PP^2$. In the second case, $N$ is homeomorphic to the $3$--manifold obtained by removing 
the interior of a closed
 $3$--ball from $\RR \PP^3$. 
 
 According to \ref{Claim:IfEmptyFrontier}, the conclusion of the theorem holds if
$\Fr_N \hatP_j = \emptyset$ for some $j$. Thus, for the remainder of the proof, we 
assume:
\Claim\label{non-empty frontiers}
$\Fr_N \hatP_j \neq \emptyset$ for every $j$.
\EndClaim

For $j=1,\ldots,k$, let $H_j$ denote the base of the
\clasibun\ 
$\hatP _j$.
Then $H_j$ is connected since $\hatP _j$ is connected.
If for some $j$ the surface $H_j$  is closed, then we have 
$\vertbdyclas \hatP_{j}=\emptyset$. In particular,
$\Fr_{ N  } \hatP_{j}
\subset\vertbdyclas \hatP_{j}$ is empty, 
contradicting
\ref{non-empty frontiers}. This proves:

\Claim\label{not closed}
For $j=1,\ldots,k$ we have $\bdy H_j\ne\emptyset$.
\EndClaim

Now set $\hatboldP = \hatP  _1\cup\cdots\cup \hatP  _k$, which is a disjoint union
according to \ref{fish and chips}. The 
\clasibun\ structures on the
$\hatP _j$ define a \clasibun\ structure
 on $\hatboldP$. 
As in the statement of Lemma \ref{not more than one, now}, we
let
$\scrbstar$ 
denote the union of all bindings 
of $\orbN$ which have 
topological
solid tori as underlying 
sets
and have  \textspanningsize\
$n/2$.
Set $\boldbstar=\mani{\scrbstar}$.
We claim:

\Claim\label{and whaddya got}
We have $N=\hatboldP\cup\boldbstar$ and
$\hatboldP\cap\boldbstar=\Fr_{ N } \hatboldP=\Fr\boldbstar$. Furthermore, each component of
$\hatboldP\cap\boldbstar$ is an annulus which is vertical in the
\clasibun\ $\hatboldP$ 
and has non-zero winding number on the 
(topological solid torus) 
component of $\boldbstar$ containing it.
\EndClaim

To prove \ref{and whaddya got}, first note that  
if $F$ is any component of $\Fr_{ N  }\hatboldP$, then
by \ref{fish and chips},
$F$ is a component of $\Fr_{ N  }\hatP_j$ for some $j\le k$;
and according to \ref{jeepers},
$F$ is therefore an annulus  component of 
$\Fr_N  P_j$ 
having \textspanningsize\ $n/2$. Let $\orbB $ denote the binding of $\orbN$
such that $B\doteq\mani \orbB  $ 
shares the frontier component $F$ with $ P_j$.
Then \ref{jeepers} also implies that 
$B  $ is a topological solid torus
(so that $\spanningsize(\orbB  )$ is defined by \S\ref{what's the point}),
and that $\spanningsize(\orbB  )=\spanningsize(F)=n/2$. 
Thus by definition $\orbB $ is a component of ${ \scrbstar}$. 
This shows that $\Fr_{ N  }\hatboldP\subset\Fr_{ N  }\boldbstar$.

Conversely, suppose that $F$ is a component of $\Fr_{ N  }
\boldbstar$. Then $F$
is an annulus of \textspanningsize\ $n/2$, and is a frontier component of
$\mani \orbP $ for some page $\orbP$ of $\orbN$. By Lemma
\ref{middle-sized lemma}, $\orbP$ is a \Sbig\ page of $\orbN$, so that
$\orbP=\orbP_j$ for some $j\in\{1,\ldots,k\}$. 
Since $F$ is a \textspanningsize--$n/2$ component of $\Fr_NP_j$, and
since  the
component of $\boldbstar$ containing $F$ is a topological solid torus by the
definition of $\boldbstar$, it follows from
\ref{jeepers} that $F$ is a component of $\Fr_N\hatP _j$. 
This establishes the reverse inclusion $\Fr_{ N  }\boldbstar\subset\Fr_{ N  }\hatboldP$.

Thus we have $\Fr_{ N  }\boldbstar = \Fr_{ N  }\hatboldP$. 
To prove that this set coincides with $\hatboldP \cap \boldbstar$, it suffices to show that 
if  $F$ is a component of $\Fr_{ N  }\boldbstar = \Fr_{ N  }\hatboldP$, and $W$ is  a tubular neighborhood of $F$ in $N$, then $(W-F)\cap\hatboldP$ and $(W-F)\cap \boldbstar$ are contained in different components of $W-F$. But \ref{jeepers} implies that $(W-F) \cap \hatboldP$ is contained in some $P_j$, which is the underlying space of a page $\orbP_j$, while $(W-F) \cap \boldbstar$ is contained in the underlying space of a binding $\orbB \subset \scrbstar$. Thus 
$(W-F)\cap\hatboldP$ and $(W-F)\cap \boldbstar$ are disjoint, 
and we have shown that $\hatboldP\cap\boldbstar=\Fr_{ N} \hatboldP=\Fr_{ N  }\boldbstar$. 

Since
$\hatboldP\cap\boldbstar=\Fr_{ N 
}\hatboldP=\Fr_{ N  }\boldbstar$, the
compact  subset
$\hatboldP\cup\boldbstar$ of $ N  $ is an open subset of $ N $. Since we have seen (via a crucial application of 
Lemma \ref{big ol' lemma}) that $k\ge1$, we have $\emptyset \ne P_1\cup\cdots \cup P_k\subset\hatboldP$, and in particular
$\hatboldP\cup\boldbstar\ne\emptyset$. Since $ N  $ is connected it
follows that $\hatboldP\cup\boldbstar= N  $, and the 
first sentence of \ref{and whaddya got} is proved. 

To prove the second sentence of \ref{and whaddya got}, consider an arbitrary component $F$ of
$\hatboldP\cap\boldbstar$. Then according to 
the first sentence of \ref{and whaddya got},
$F$ is both a component of the
frontier of some component $B$ of $\boldbstar$, and a component of the
frontier of  $\hatP _j$ for some $j$.
The definition of
$\boldbstar$ implies that $\orbB\doteq\orbi B$ is a binding of
$\orbN$, 
and that $B$ is a topological solid torus.
According to an observation made in \S\ref{wuzza book},
$F$ is an annulus which has non-zero winding number in
the topological solid torus $B$. 
On the other hand, since $F$ is a component of
$\Fr_N\hatP _j$, our choice of the $I$--bundle structure on $\hatP _j$ guarantees
that $F$ is a vertical 
subset
of $\bdy \hatP _j$ and therefore of
$\bdy\hatboldP$.
The proof of \ref{and whaddya got} is now complete.

Now we claim: 
\Claim\label{one annulus}
For each component $K$ of $\bdy N$, the intersection $K\cap\boldbstar$ is a single annulus.
\EndClaim

To prove \ref{one annulus}, we first note that \ref{and whaddya got} directly implies that each component of $K\cap\boldbstar$ is an annulus. Furthermore, it follows from Lemma \ref{not more than one, now} that $K\cap\boldbstar$ has at most one component. It remains only to show that $K\cap\boldbstar\ne\emptyset$. If we assume that $K\cap\boldbstar$ is empty, then since $N=\hatboldP\cup\boldbstar$ by \ref{and whaddya got}, we  have $K\subset\hatboldP$. But the component $K$ of $\bdy N$ is a $2$--sphere, and in particular a closed surface. Hence $K$ must be a full boundary component of $\hatP_{j}$ for some $j\in\{1,\ldots,k\}$. 
Thus the \clasibun\ $\hatP_{j}$ has a boundary component which is a sphere; and the base $H_{j}$ of the \clasibun\ has non-empty boundary by \ref{not closed}. It follows that $H_j$ must be a disk. But then $\hatP_j$ is a topological ball whose full boundary is the sphere $K$; this contradicts the essentiality of the sphere $S$ in $M$. Thus \ref{one annulus} is proved.

Now consider an arbitrary component $K$ of $\bdy N$. By \ref{one annulus}, $K\cap\boldbstar$ is an annulus, and it follows from \ref{and whaddya got} that $K\cap\hatboldP=\overline{K-(K\cap \boldbstar)}$. 
Since $K$ is a sphere, it follows that $K\cap\hatboldP$ is the union of two disjoint disks. In particular:
\Claim\label{They're disks}
Every component of $\hatboldP\cap\bdy N$ is a disk.
\EndClaim

Now we claim:
\Claim\label{each hatpj is a ball}
For each $j\le k$, the manifold with corners $\hatP_j$ is a topological ball, and $\hatP_j\cap\boldbstar$  is a single annulus.
\EndClaim

To prove \ref{each hatpj is a ball}, we first recall that 
$\Fr_N \hatP_j \neq \emptyset$ for every $j$
by \ref{non-empty frontiers}. We choose a component $F$ of $\Fr_N \hatP_j$, which by \ref{and whaddya got} is an annulus. Let $C_1$ and $C_2$ denote the components of $\bdy F$. For $i=1,2$, it follows from \ref{and whaddya got} 
that $C_i$ is contained in a component $D_i$ of $\hatboldP\cap\bdy N$. By \ref{They're disks}, each $D_i$ is a disk. In particular we have $\bdy D_i=C_i$ for $i=1,2$. Since $D_i\subset\hatboldP$ and $C_i\subset\hatP_j$, we have $D_i\subset\hatP_j$. It now follows that $R \doteq D_1\cup F\cup D_2$ 
is a topological $2$--sphere contained in $\bdy \hatP_j$. But $\hatP_j$ is a \clasibun\ over the compact surface $H_j$, which by \ref{not closed} has non-empty boundary. It follows that the manifold obtained by \sta\ of $\hatP_j$ is irreducible. Since $\bdy\hatP_j$ contains the $2$--sphere $R$, we deduce that $\hatP_j$ is a topological ball with boundary $R$. We have $\hatP_j\cap\boldbstar= R \cap\boldbstar=F$. This completes the proof of \ref{each hatpj is a ball}.  

It follows from \ref{each hatpj is a ball} that
$ N  $ is homeomorphic 
to a topological manifold obtained from $\boldbstar$ by
adding $k$ topological $2$--handles.
As $ N  $ is connected, it follows that $\boldbstar$ is connected,
i.e.\ is a single topological solid torus. A topological manifold
obtained by adding $k$ topological $2$--handles to a solid torus along 
annuli with non-zero winding number
 is homeomorphic to 
a manifold obtained from a (possibly trivial) lens space by removing
 the interiors of $k$ disjoint closed $3$--balls. 
In particular, $k$ is the number of boundary components of $N$, and is therefore equal to $1$ if $S$ separates $M$, and to $2$ otherwise.
If $k=1$, then $N$ is homeomorphic to a manifold obtained from a lens space by
removing  the interior of a closed $3$--ball; and in this case, since
$S$ is essential, the lens space in question must be non-trivial.  If
$k=2$, then $N$ is homeomorphic 
to a manifold obtained from a (possibly trivial) lens space by
removing the interiors of two disjoint closed $3$--balls.
In either case, the conclusion of the theorem follows.

(We remark that the above proof shows that $N$ has at most two big pages. Indeed, if $\Fr_N \hatP_j = \emptyset$ for some $j$, then there is exactly one big page, and otherwise the number $k$ of  big pages is equal to the number of components of $\bdy N$.)
\EndProof
\EndCustom

\section{Essential tori}\label{torus section}

In this section, we consider a closed hyperbolic $3$--orbifold $\orbM$ such that $M = \mani \orbM$ is irreducible but contains an essential torus $T$. 
In \S\ref{TorusComplexity} and Lemma~\ref{it's still incompressible}, we  build an essential torus $T$ that is minimal in an appropriate sense, and show that it defines an essential $2$--orbifold $\orbi T$. The main result of this section is Proposition~\ref{Prop:TorusBooksTotalSpace}, which says that if $N$ is a component of $M \split T$ and $\orbN = \orbi N$
is an orbifold book of $I$--bundles, then $N$ admits a Seifert fibration, either over the disk with two singular fibers or over the annulus with at most one singular fiber.

The proof of Proposition~\ref{Prop:TorusBooksTotalSpace} shares a common feature with the arguments of Section~\ref{sphere section}. Recall that the proof in Section~\ref{sphere section} proceeded
 by finding a big page $\orbP$, and then adjoining harmless $3$--balls to the underlying space $P$ of this page. 
 In the current section, we will consider a \emph{dominant} page $\orbP$, whose underlying space $P$ contains a non-trivial curve in every component of  $P \cap \bdy N$. In particular, the key definition is topological rather than combinatorial. The existence of a dominant page is established in Lemma~\ref{Lem:DominantExists}. In Lemma~\ref{doughnut hole}, we prove that the topological non-triviality of $P$ persists even after attaching harmless components to form a space $\widehat P$. After these preliminary steps, the proof of Proposition~\ref{Prop:TorusBooksTotalSpace}  amounts to showing that $\widehat P$ and 
all the components of $N - \widehat P$ are solid tori. 
This is enough  to find the claimed Seifert fibration.

 \Custom{Complexity and $\orbM$--minimal tori}\label{TorusComplexity}
 Let $\orbM$ be a closed, orientable, hyperbolic $3$--orbifold such that
$M\doteq\mani \orbM $ is irreducible. 
According to the conventions laid out in \S\ref{AccordantConvention},
$\orbM$ is understood to be equipped with an $\orbM$--accordant
smooth structure.
Suppose that $M$ contains a $\pi_1$--injective torus $T$ that is transverse to $\Sing_\orbM$. Define the \emph{complexity} of $T$ to be
\[
\comp(T) = 
\begin{cases}
\size(T) & \text{if $T$ is separating,} \\
2 \size(T) & \text{if $T$ is non-separating}.
\end{cases}
\]
A torus $T \subset M$ is called \emph{$\orbM$--minimal} if $T$ 
\begin{enumerate}
\item\label{Itm:Comp1} is $\pi_1$--injective, 
\item\label{Itm:Comp2} is transverse to the singular locus, and 
\item\label{Itm:Comp3} has the smallest complexity among all tori in $M$ satisfying \eqref{Itm:Comp1} and \eqref{Itm:Comp2}.
\end{enumerate}
Observe that if $M$ contains a $\pi_1$--injective torus, then it contains an $\orbM$--minimal one.
 \EndCustom
 
The above definition of complexity  has the feature that if $R$  is a well-chosen tubular neighborhood of $T$ in $M$, the size of 
the boundary of each component
of $ \overline{M-R}$ is equal to $\comp(T)$. This feature
proves to be crucial in the proof of Proposition~\ref{Prop:TorusBooksTotalSpace}, specifically Claims~\ref{they separate}--\ref{how's this?}.

\Lemma\label{it's still incompressible}
Let 
$\orbM$ be a closed, orientable, hyperbolic $3$--orbifold such that
$M\doteq\mani \orbM$ is irreducible, and let
$T$ be an $\orbM$--minimal
torus in $M$.
Then $\orbi T $ is essential in $\orbM$. 
\EndLemma

\Proof
Set $\orbT = \orbi T $.
We first observe that since $T$ is a torus,
$\orbT$ 
cannot be spherical, and thus cannot bound a discal $3$--orbifold. 
It remains
to show that 
$\orbT$ is $\pi_1$--injective in $\orbM$. According to
Proposition \ref{when injective},
this amounts to showing that
if $\orbD$ is a two-dimensional orientable discal suborbifold of $\orbM$ with $\orbD \cap \orbT = \bdy \orbD$, then $\orbC\doteq\bdy \orbD$ bounds a discal suborbifold of $\orbT$.

Since $\orbD$ is an orientable discal suborbifold, $ D \doteq \mani
\orbD$ is a 
disk, transverse to $\Sing_\orbM$,
 with $\size(D)\le 1$ and $\size(\bdy D) = 0$.
(See \S\ref{Suborbifolds-of-3orbifolds}.)
 
Now since $T$ is $\orbM$--minimal and therefore $\pi_1$--injective, 
the curve $C\doteq\mani \orbC=\bdy D$ 
bounds a disk
$D'\subset T$. 
Then
 $S \doteq D\cup D'$ is a 
smoothable 
topological
$2$--sphere
with $S\cap T=D'$. 
Since $M$ is irreducible, $S$ bounds a
topological ball $H\subset M$.
The $\pi_1$--injectivity of $T$ 
in $M$ implies that 
$T$ cannot be
contained in
$H$; hence $H\cap T=D'$. It follows that
the torus $T_1$ obtained by smoothing $(T-D')\cup D$ (which
by \S\ref{what's smoothing} has the same size as $(T-D')\cup D$)
is smoothly 
isotopic to $T$ in
$M$. Thus $T_1$  is $\pi_1$--injective, which is property \eqref{Itm:Comp1} of \S\ref{TorusComplexity}.
By construction, $T_1$ also satisfies property \eqref{Itm:Comp2} of \S\ref{TorusComplexity}.

The minimality of $T$ now implies $\comp(T) \leq \comp(T_1)$. Since $T$ and $T_1$ are isotopic, they are either both separating or both non-separating; hence $\size(T) \leq \size(T_1)$.
Thus
\[
\size(T) \leq \size(T_1)=\size(\overline{T-D'})+\size(D) = \size(T) - \size(D') + \size(D).
\]
Hence
$\size(D')\le\size(D)\le1$. Thus $\orbi {D'}$ is the
required discal suborbifold of $\orbT$ bounded by $\orbC$.
\EndProof

The next definition and two lemmas will be applied to a $3$--orbifold $\orbN$ obtained by splitting a given closed $3$--orbifold $\orbM$ along $\orbT$, where $\orbT = \orbi T$ is the $2$--orbifold obtained from an $\orbM$--minimal torus $T$.

\Custom{Dominant pages}
Let $\orbN$ be an orbifold book of $I$--bundles, and let $N \doteq \mani \orbN$.
A page $\orbP$ of $\orbN$ is called \emph{dominant} if for every component $K$ of $\mani \orbP \cap \bdy N$, the inclusion homomorphism $\pi_1(K)\to\pi_1(N)$ is
non-trivial. 
\EndCustom

\Lemma\label{Lem:DominantExists}
Let $\orbN$ be an orbifold book of $I$--bundles, and
  set $N = \mani \orbN$.
Suppose that $N$ is irreducible and that every component of $\bdy N$ is a $\pi_1$--injective torus. Then $\orbN$ has a dominant page $\orbP$.
\EndLemma

\Proof
Choose a component $T$ of $\bdy N $, and set $\orbT=\orbi T$.
Using the notation of \S\ref{wuzza book}, set
$\scrA=\scrA_{\orbN,\orbT}$, $\scrH=\scrH_{\orbN,\orbT}$, and $\boldC = \boldC_{\orbN,\orbT}$.
Then $\boldH\doteq\mani \scrH$ and $\boldA\doteq\mani \scrA$
are $2$--manifolds with boundary whose union is $T$,
and
$\boldC = \boldH\cap\boldA=\Fr_{T}\boldH=\Fr_{T}\boldA$
is a closed $1$--manifold.

We claim:

\Claim\label{Claim:PartiallyDominant}
There is a component $H$ of $\boldH$ such that the inclusion homomorphism $\pi_1(H)\to\pi_1(T)$ is
non-trivial. 
\EndClaim

To prove \ref{Claim:PartiallyDominant}, we consider two cases. First, suppose that
 some component $C$ of $\boldC$ is
homotopically non-trivial on the torus $T$. 
If $H$ denotes the component of 
$\boldH$
containing $C$, the inclusion homomorphism $\pi_1(H )\to\pi_1( T  )$ is
non-trivial, as desired.

Next, suppose that
 every component of $\boldC$ is
homotopically trivial in $T$.
In this case,  
$\boldC$ is a finite union of simple closed curves
in the torus $T$,  each of which bounds a disk in $T$. Hence
there is
a compact submanifold $\Delta$ of $T$, each component of which is a
disk, such that $\bdy\Delta\subset\boldC\subset\Delta$.
It follows that
$L\doteq\overline{T-\Delta}$ is connected,  that 
the inclusion homomorphism $\pi_1(L )\to\pi_1( T  )$ is
surjective, and that
$\inter L$ 
is disjoint from
$\boldC$. 
Since the $2$--manifold $L\subset T$ is connected, has boundary
contained in $\boldC$ and has interior disjoint from $\boldC$, we deduce
that $L$ is a
component of either $\boldA$ or $\boldH$. 
But every component $A$ of $\boldA$ is either a disk or an annulus, so in particular $\pi_1(A)$ cannot surject onto $\pi_1(T)$. 
Thus $L$ must be a component of  $\boldH$, and setting $H = L$ completes the proof of \ref{Claim:PartiallyDominant}.

Now, let $\orbP$
  denote
the unique page of $\orbN$ whose underlying space $P$ contains $H$. 
  Set
  $\boldK= P\cap \bdy N \subset \boldH$, so that $H$ is a component of $\boldK$. 
We claim:

\Claim\label{Claim:ConnectOrClassic}
Either $\boldK$ is connected, or $ P$ may be given the structure of
a trivial
\clasibun\ 
 in such a way that 
$\boldK=\horizbdyclas P$.
\EndClaim

To prove \ref{Claim:ConnectOrClassic}, recall  
from 
\S\ref{PagelikeBindinglike}
that the page $\orbP$ has the structure of an orbifold $I$--bundle in
which 
$\orbi \boldK=\horizbdy\orbP$  and 
$\Fr_\orbN\orbP=\vertbdy\orbP$.
In particular, according to Corollary \ref{Cor:VertBoundaryStructure},
each component of $\Fr_\orbN\orbP$
is an annular orbifold, and hence each component of $\Fr_NP$  is an annulus or a disk.

It follows from Proposition  \ref{Prop:IBundleTotalSpace} that
the topological $3$--manifold $ P$ may be given
the structure of a
\clasibun\ 
in such a way that
$\Fr_NP$ is a vertical subset of $\vertbdyclas P$.
In particular,
 if $V$ denotes the vertical boundary of the 
\clasibun\  $P$, 
we have $\Fr_NP\subset V$, and $Z\doteq \overline{V-\Fr_NP}$ is
also a vertical submanifold. If $Y$ denotes the horizontal boundary 
 of the \clasibun\  $P$,  we have $\boldK=Y\cup Z$. 
In a \clasibun, each
 component of a
 vertical
  subset
meets each component of the horizontal boundary;
 hence if $Z\ne\emptyset$, or if $Y$ is connected, then $\boldK$ is
 connected. But $Y$ is disconnected if and only if $P$ is a trivial
\clasibun;
and $Z = \emptyset$ if and only if $\boldK=Y$. 
This proves \ref{Claim:ConnectOrClassic}.

To prove the lemma, we claim that $\orbP$ is a dominant page.
That is, for every component $K$ of
 $\boldK = P \cap \bdy N$, the inclusion homomorphism $\pi_1(K )\to\pi_1( N  )$ is
  non-trivial.
Observe that if $K = H$, the assertion follows 
from \ref{Claim:PartiallyDominant} and the $\pi_1$--injectivity
 of $T$ in $N$. In particular, the assertion is true if $\boldK$ is
connected. If $\boldK$ is not connected, then by \ref{Claim:ConnectOrClassic},
$ P$ may be given the structure of
a trivial 
\clasibun\ 
in such a way that 
$\boldK=\horizbdyclas P$.
In this case, $\boldK$ has two
components, one of which is $H$. We need only prove the assertion in
the subcase when
$K$ is the other component of $\boldK$. But then the 
\clasibun\ 
structure of $ P$ gives a homeomorphism of $K$ onto $H$ which,
regarded as a map from $K$ to $N$, is homotopic to the inclusion
$K\hookrightarrow N$. Since the inclusion homomorphism $\pi_1(H)\to\pi_1( N  )$ is
non-trivial, 
we conclude that the inclusion homomorphism $\pi_1(K )\to\pi_1( N  )$ is
also non-trivial, 
and the lemma is proved.
\EndProof

\Lemma\label{doughnut hole}
Let $\orbN$ be an orbifold book of $I$--bundles, and let $N \doteq \mani \orbN$. Suppose that $N$ is irreducible and that every component of $\bdy N$ is a $\pi_1$--injective torus.  Let $\orbP$ be a dominant page in $\orbN$, which exists by Lemma~\ref{Lem:DominantExists}. Set 
$P = \mani \orbP$, and let $\hatP$  denote
 the union of $ P$ with all \harmless\
components 
of $\overline{ N  - P }$ (see \S\ref{Def:Harmless}).
Then every component of  
$\bdy(\hatP\cap \bdy N)=\bdy(\Fr_{ N  }\hatP)$ 
is $\pi_1$--injective in $ N  $.
\EndLemma

\Proof
Following the notation of 
Lemma~\ref{Lem:DominantExists}, 
set $\boldK = P \cap \bdy N$. We claim:

\Claim\label{i just did}
If $Q$ is any component of $\bdy\boldK$ which is a homotopically trivial
simple closed curve in $N$, then $Q$  bounds a disk which is a component of
$\overline{\bdy N -\boldK}$.
\EndClaim

To prove \ref{i just did}, first note that since 
$Q\subset\bdy N$ is homotopically
trivial in $N$, and $\bdy N$ is $\pi_1$--injective in $N$, the curve $Q$ 
bounds a disk $D\subset\bdy N$. Since 
$Q$ is a component of $\bdy\boldK$, either (a) $D$
is a component of
$\overline{\bdy N -\boldK}$, or (b) $D$
contains some component $K$ of $\boldK$. But since $D$
is a disk, (b) would imply that the inclusion homomorphism $\pi_1(K )\to\pi_1( N  )$ is
trivial, 
contradicting the assumption that $\orbP$ is dominant.
Hence (a) must hold,
and \ref{i just did} is established.

We also claim:
\Claim\label{kiddly divy}
If $F$ is any component of $\Fr_NP$ such that the inclusion
homomorphism $\pi_1(F)\to\pi_1(N)$ is trivial, then there is a smoothable 
$2$--sphere
$S\subset N$ which is the union of $F$ with some subsurface of
$\bdy N$.
\EndClaim

To prove \ref{kiddly divy}, suppose that $F$ is a component of $\Fr_NP$
such that the inclusion
homomorphism $\pi_1(F)\to\pi_1(N)$ is trivial. Then each component of
$\bdy F$ is a simple closed curve in $\bdy N$, homotopically trivial in
$N$, which is also a component of $\bdy\boldK$. 
By \ref{i just did}, each such component bounds  
a disk which is a component of
$\overline{\bdy N -\boldK}$.

In the
case where $F$ is an annulus, let $Q_1$ and $Q_2$ denote its boundary
components. Each $Q_i$  bounds a disk $D_i$ which is a component of
$\overline{\bdy N -\boldK}$. Since $Q_1\ne
Q_2$, the components $D_1$ and $D_2$ of $\overline{(\bdy
  N)-\boldK}$ are distinct and therefore
disjoint. The conclusion
of \ref{kiddly divy}
 then follows upon setting $S=D_1\cup F\cup D_2$.

In the
case where $F$ is a disk,
$Q\doteq\bdy F$ bounds a disk $D$ which is a component of
$\overline{\bdy N -\boldK}$. 
In this case the  conclusion 
of \ref{kiddly divy}
follows upon setting $S=D\cup F$. 

Now we claim:
\Claim\label{who wants to know}
If $F$ is any component of $\Fr_NP$ such that the inclusion
homomorphism $\pi_1(F)\to\pi_1(N)$ is trivial, then the component of
$\overline{N -P}$ 
containing $F$ is \harmless. 
\EndClaim

To prove \ref{who wants to know}, suppose that $F$ is a component of $\Fr_NP$ such that the inclusion
homomorphism $\pi_1(F)\to\pi_1(N)$ is trivial, and let $S$ be the
smoothable $2$--sphere given by \ref{kiddly divy}. Since $N$ is irreducible, $S$
bounds a topological ball $J\subset N$. 
According to \ref{kiddly divy}, $S=\bdy
J$ is the union of $F$ with some subsurface of
$\bdy N$. It now follows that $\Fr_NJ=F$. 

Since $\Fr_NJ=F$ is a component of $\Fr_NP$, either $J$ is  a component of $\overline{N-P}$, or 
$J \supset P$. But if
$J\supset P$ then in particular 
$\boldK\ne\emptyset$ and is contained in the topological ball $J\subset
N$; this is impossible since the inclusion homomorphism $\pi_1(K)\to\pi_1( N  )$ is 
non-trivial for each component $K$ of $\boldK$, by the definition of  a dominant page. 
Hence $J$ is  a component of $\overline{N-P}$. Since $J$
is a topological ball and $F=\Fr_NJ$ is connected, it now follows from the
definition that $J$ is a harmless component of $\overline{N-P}$, and
\ref{who wants to know} is proved.

To complete the proof of the lemma, suppose
that $C$ is a component of  $\bdy(\hatP\cap\bdy N
)=\bdy(\Fr_{ N  }\hatP)$.
The component $F$ of $\Fr_{ N 
}\hatP$ containing $C$ is in particular a component of $\Fr_NP$.
Assume that $C$
is not $\pi_1$--injective in $ N  $.  Since 
$C$ is contained in the $\pi_1$--injective submanifold $\bdy N$ of $N$,
the  simple closed curve
$C\subset T\subset N$ is homotopically trivial in $N$. Since 
$F$ is a disk or an
annulus, it follows that the inclusion homomorphism
$\pi_1(F)\to\pi_1(N)$ is trivial. Hence by \ref{who wants to know},
$F$ is contained in a harmless component of $\overline{N-P}$. But this
contradicts the fact that $F$ is a component of $\Fr_{N} \hatP$.
\EndProof

We are now ready to prove the main result of this section.

\Proposition\label{Prop:TorusBooksTotalSpace}
Let $\orbM$ be a closed, orientable, hyperbolic $3$--orbifold such that $M\doteq\mani \orbM $ is irreducible.
Suppose that $T$ is an $\orbM$--minimal torus in $M$. 
Let $N$ be a component of $M \split T$, and assume 
that $\orbN \doteq \orbi N$ admits the structure of  a book of $I$--bundles. 
Then either
\begin{enumerate}[\:\:(i)]
\item 
$T$ is separating, and 
$N$ admits a Seifert fibration over a disk with exactly two
singular fibers; or
\item $T$ is non-separating, 
and $N$ admits a Seifert fibration over an annulus  with at most
one singular fiber.
\end{enumerate}
\EndProposition

The proof of Proposition~\ref{Prop:TorusBooksTotalSpace} is organized as follows. 
Lemma~\ref{Lem:DominantExists} guarantees a dominant page $\orbP \subset \orbN$ with underlying space $P$.
We then form a $3$--manifold $\widehat P$ by taking the union of $P$ with all harmless components of $\overline{N - P}$, and prove that $\widehat P$ is a topological solid torus. In a series of claims, we will then prove that all of $N$ can be obtained by gluing one or two solid tori to $\widehat P$ along annuli. It will follow that $N$ has the claimed Seifert fibered structure.

\Proof[Proof of Proposition~\ref{Prop:TorusBooksTotalSpace}]
Let $R$  be a tubular neighborhood of $T$ in $M$. Then $N$ is diffeomorphic to a component $N_0$ of $\overline{M-R}$. We shall show that the conclusions of the proposition hold with $N_0$ in place of $N$, thereby proving the proposition.

We may suppose $R$ to have been chosen in such a way that $\bdy N_0$ is transverse to $\Sing_\orbM$, and that each component of $\bdy N_0$ has the same size as $T$. Since $\bdy N_0$ has one component if $T$ is separating, and has two components if $T$ is non-separating, the definition of complexity now implies:
\Equation\label{size is complexity}
\size(\bdy N_0)=\comp(T). 
\EndEquation

We set $\orbNnought=\orbi\Nnought$, and  fix a structure of a book of $I$--bundles on $\orbNnought$.

First note that 
$T   $ is $\pi_1$--injective in $M $ by the definition of an $\orbM$--minimal torus. It follows that $\Nnought $ is irreducible and that 
every component of  
$\bdy \Nnought $ is a $\pi_1$--injective torus
 in $\Nnought $.
In particular:

\Claim\label{sickle}
 $ \Nnought  $ is not a solid torus. 
Therefore
 $ \Nnought  $ admits no Seifert fibration over a disk with at most one singular fiber.
\EndClaim

We also observe that since $T$ is essential in $M$, 
the manifold case of Lemma~\ref{split injective} implies
that
\Claim\label{that too}
$N_0$ is $\pi_1$--injective in $M$.
\EndClaim

Now according to Lemma \ref{Lem:DominantExists},  we may fix a 
dominant page $\orbP$  of $\orbNnought$.
We set $P=\mani \orbP$.
We apply Proposition \ref{Prop:ExpandedPage}\eqref{ExpandedPageIsBundle}, taking $\boldV$ to be the union of
all harmless components of $\overline{\Nnought -P}$.
This shows that $\hatP 
\doteq P \cup \boldV
$ may be given the structure of a
\clasibun\ in
such a way that 
$F\doteq\Fr_{ \Nnought   }\hatP$
is a vertical submanifold of 
$\bdy\hatP$. 

We let 
$W$ denote the $3$--manifold with corners  $\overline{ \Nnought   - \hatP}$. It may happen that $W$ is empty. On the other hand,  since $\Nnought$ is connected and $\hatP\supset P\ne\emptyset$, we have that
\Claim\label{cooties}
Each component of $W$ contains a component of $F$. 
\EndClaim
We set
$C=\bdy F=\bdy(\hatP\cap\bdy \Nnought   )$. It follows from Lemma \ref{doughnut hole} that
\Claim\label{ronly honly bing}
 $C$
is $\pi_1$--injective in $ \Nnought   $. 
\EndClaim

It follows from \ref{ronly honly bing} that no component of $F$ is  a disk.
The components of $F$ must therefore be annuli, and, again by \ref{ronly honly bing}, these annuli are $\pi_1$--injective. But a vertical annulus in the boundary of a \clasibun\ must be a component of the vertical boundary.  
 Hence, if $E$ denotes the vertical boundary of the \clasibun\ $\hatP$, we may state: 
\Claim\label{frugal me}
$F$ is a union of components of  $E$, and each component of $F$ is a $\pi_1$--injective annulus in $N_0$.
\EndClaim

It follows from \ref{ronly honly bing}
that the
boundary components of $\hatP\cap\bdy \Nnought   $ are homotopically
non-trivial simple closed curves in 
$\Nnought $, and in particular in $\bdy \Nnought   $. 
Since 
$\bdy \Nnought $ consists of tori,
it follows that each component of $\hatP\cap\bdy \Nnought   $ or
$W\cap\bdy \Nnought   $
 is a torus or an annulus. Since the components of
 $F$ are annuli by
  \ref{frugal me},
  it now follows that:
\Claim\label{spell it out}
All of the
boundary components of the topological manifolds $\hatP$ and $W$ are
smoothable tori. 
\EndClaim

In particular we have
$\chi(\hatP)=0$, so that the base of the \clasibun\ $\hatP$ is a
topological torus,  Klein bottle,  annulus, or M\"obius band. 

If the base of  $\hatP$ is a topological torus or  Klein bottle, then
$E=\emptyset$, and in particular  \ref{frugal me} implies
$F=\emptyset$. Hence $ \Nnought   =\hatP$, i.e.\ $\Nnought$ is either a
trivial \clasibun\ over $T^2$ or a twisted \clasibun\ over the topological Klein
bottle. 
If $\Nnought$ is a
trivial \clasibun\ over $T^2$, then
$\Nnought$ admits a Seifert fibration over an
annulus with no singular fibers.
If  $\Nnought$ is a twisted \clasibun\ over the Klein
bottle, then
$\Nnought$ admits a
Seifert fibration over a disk with two singular fibers (see
\cite[Theorem 2.3(b)]{hatcher3M}).
Thus the conclusion of the proposition holds if  the base of  $\hatP$ is a topological torus or Klein bottle.

For the rest of the proof, we assume that the base of $\hatP$ is a topological annulus or M\"obius band. 
Since $M $ is orientable, $\hatP$  is either a trivial \clasibun\
over a
topological annulus or a twisted \clasibun\ over a topological M\"obius band. In
particular $\hatP$ is a topological solid torus,  $E$ has at most two components, and if $E$ has two components then each of them is an annulus whose winding number in $\hatP$ is $1$. By \ref{frugal me}
$F$ is a union of components of  $E$. We must have
$F\ne\emptyset$, for otherwise $ \Nnought   $ would be equal to the
topological solid torus $\hatP$, a contradiction to \ref{sickle}. 
Hence:
\Claim\label{pour chanter en choeur}
$\hatP$ is a topological solid torus, and  either (a) $F$ is a single annulus, or (b) $F$ is the union of
two disjoint annuli and $\hatP$ is a parallelism (see \S\ref{Pi1Injective})  between these two annuli. 
\EndClaim

According to Proposition~\ref{Prop:ExpandedPage}\eqref{ExpandedPageEuler}, each component of $F$ is the underlying surface of an annular $2$--orbifold. Since  every component of $F$ is an annulus, we have $\size(F)=0$, and therefore 
$\size(\bdy\hatP)+\size(\bdy W) =\size(\bdy \Nnought)$. 

On the other hand,
Proposition~\ref{Prop:ExpandedPage}\eqref{ExpandedPageEuler} also gives that 
$\chi(\orbi  \hatP)<0$, which implies that  $\chi(\orbi  {\bdy\hatP})=2\chi(\orbi  \hatP)<0$. Since $\hatP$ is a topological solid torus, we have $\chi(\bdy\hatP)=0$. Hence  $\size(\bdy\hatP)>0$. It now follows that $\size(\bdy W)<\size(\bdy  \Nnought   )$,
which with \eqref{size is complexity} gives
\Equation\label{and therefore}
\size(\bdy W)<\comp(T). 
\EndEquation

We claim:
\Claim\label{they separate}
Every  component of $\bdy W$  is a separating
smoothable
 torus in $M$.
\EndClaim

To prove \ref{they separate}, first note that each component of $\bdy W$ is a 
smoothable
torus by \ref{spell it out}. Next, note that since the $3$--submanifold $W$ of the closed $3$--manifold $M$ is compact, $\bdy W$ cannot have \emph{exactly} one non-separating component. Assume  that $\bdy W$ has at least two non-separating components. Choose two such components, and smooth them to obtain tori $T_1$ and $T_2$, with $\size (T_1)\le\size(T_2)$.
Using the definition of complexity and \eqref{and therefore}, we find
$$\comp(T_1)=2\size(T_1)\le\size(T_1)+\size(T_2)\le\size(\bdy W)<\comp(T).$$
Since $T$ is an $\orbM$--minimal torus, and $T_1\subset M$ is a torus transverse to $\Sing_\orbM$, it follows that $T_1$ is not $\pi_1$--injective in $M$. Since $M$ is irreducible, any non-$\pi_1$--injective torus in $M$ either bounds a solid torus in $M$ or is contained in a $3$--ball in $M$. In either case, $T_1$ must separate $M$, and we have a contradiction. This proves \eqref{they separate}.

On the other hand, we claim:
\Claim\label{connected complement}
For every component $K$ of $W$, the manifold $\overline{M-K}$ is connected. 
\EndClaim

Let $K$ be a component of $W = \overline{N_0-\hatP}$.
Note that $\overline{N_0-K}$ is the union of $\hatP$ with all components of $\overline{N_0-\hatP}$ distinct from $K$. Each such component meets $\hatP$, since $N_0$ is connected. Hence $\overline{N_0-K}$ is connected. But $\overline{M-N_0}$ is connected by construction, and $\overline{N_0-K}\cap\overline{M-N_0}$ contains $P \cap\bdy N_0$, which is the underlying set of the horizontal boundary of an orbifold $I$--bundle $\orbP$, and is therefore 
non-empty. 
It follows that $\overline{M-K}=\overline{N_0-K}\cup\overline{M-N_0}$ is connected, and \ref{connected complement} is proved.

We now claim more:
\Claim\label{how's this?}
Each component of $W$ is a topological solid torus.
\EndClaim

To prove \ref{how's this?}, 
let $K$ be any component  of $W$. 
It follows from \ref{they separate} and \ref{connected complement} that
 $\bdy K$ is 
 a single smoothable torus, which in particular  
separates $M$. Let $T'$ denote a torus obtained by smoothing $\bdy K$. By definition
 $T'$ is transverse to the singular set of $\orbM$. 

By \eqref{and therefore}
we have 
\[
\size(T')=\size(\bdy K)\le\size(\bdy W)<\comp(T).
\]
 But $T'$ separates $M$, so we have $\comp(T')=\size(T')$, hence $\comp(T')<\comp(T)$. If $T'$ were $\pi_1$--injective in $M $, we would now have a contradiction to the $\orbM$--minimality of $T$. Hence $T'$ is not $\pi_1$--injective in $M $. Since $M $ is irreducible, $T'$ either bounds a 
solid torus in $M $ or is contained in a 
ball in $M $. But $T'$ is topologically isotopic to $\bdy K$, which by \ref{cooties}  contains a component of $F\supset C$;
and by \ref{ronly honly bing}, $C$ is a $\pi_1$--injective closed one-dimensional submanifold of $\Nnought$, and hence of $M $ by \ref{that too}. 
Hence $T'$ cannot be contained in a ball. Since $T'$ is isotopic to the full boundary of $K $, it follows that either $K$ or $\overline{M -K}$ is a topological solid torus. 
But $\overline{M -K}$ contains $T$, which 
is an $\orbM$--minimal torus and is thus 
$\pi_1$--injective in $M $. Hence $K$ must be a topological solid torus, and \ref{how's this?} is proved.

To establish the conclusion of the proposition, it suffices to show that $N_0$ has a Seifert fibration such that either (i) the base is a disk and there are two singular fibers, or (ii) the base is an annulus and there is at most one singular fiber. Indeed, if (i) holds then $\bdy N_0\cong\bdy N$ is connected, and hence $T$ is separating; and if (ii) holds then $\bdy N_0\cong \bdy N$ has two components, and hence $T$ is non-separating.

First consider the case in which Alternative (a) of \ref{pour chanter en choeur} holds. Since $ \Nnought   $ is connected, it then follows from \ref{how's this?} that $W$ is a single topological solid torus. The topological solid tori $\hatP$ and $W$ meet in $F$, which by
  \ref{frugal me}
  is an annulus, $\pi_1$--injective in $ \Nnought    $ and therefore in both $\hatP$ and $W$. Hence $ \Nnought   =\hatP\cup W$ admits a Seifert fibration over the disk with at most two singular fibers. It follows from \ref{sickle} that this Seifert fibration must have exactly two singular fibers, and conclusion (i) follows in this case.

Now  consider the case in which Alternative (b) of \ref{pour chanter en choeur} holds. 
Let $F_1$ and $F_2$ denote the components of $F=\hatP\cap W=\bdy \hatP\cap \bdy W$,
which by
  \ref{frugal me}
are  annuli, and are $\pi_1$--injective in $ \Nnought   $ and hence in $W$.
Since $ \Nnought $ is connected, it then follows from  \ref{how's this?} that either 
(1) $W$ is the disjoint union of two topological solid tori $W_1$ and $W_2$, and we have $F_i\subset W_i$ for $i=1,2$, or
(2) $W$ is a single topological solid torus.
Since Alternative (b) includes the assertion that $\hatP$ is a parallelism between $F_1$ and $F_2$, the manifold $ \Nnought  $ is homeomorphic to a topological
manifold obtained from $W$ by gluing the annuli $F_1$ and $F_2$ together by some homeomorphism.
 In Subcase (1) it follows that $ \Nnought  $ admits a Seifert fibration over the disk with at most two singular fibers, and by \ref{sickle}  this Seifert fibration must have exactly two singular fibers. In Subcase (2) it follows that $ \Nnought  $ admits a Seifert fibration over
the annulus with at most one singular fiber.
\EndProof

\section{Characteristic suborbifolds}\label{characteristic section}

This section  generalizes the theory of characteristic submanifolds to the setting of $3$--orbifolds. 
The theory of  characteristic submanifolds was originally developed by  Johannson \cite{Jo} and Jaco and Shalen 
\cite{JS}.
According to the theory 
(which is reviewed in \S\ref{GS review})
every simple $3$--manifold $Q$
contains a characteristic submanifold $\boldV$
whose components are $I$--bundles and solid tori. The components of $Q - \boldV$ are either thickened annuli or acylindrical pieces with negative Euler characteristic. In fact, $\boldV$ is uniquely determined up to isotopy by slightly more precise versions of these properties; see Proposition~\ref{char-submflds, GS edition}.
The key topological fact is that $Q$ is a book of $I$--bundles if and only if 
$\chi(Q - \boldV) = 0$.

We begin by extending these definitions from $3$--manifolds to $3$--orbifolds, using manifold covers. If $\orbN$ is a very good, simple $3$--orbifold, then a characteristic suborbifold $\scrC \subset \orbN$ is obtained by constructing a regular manifold cover $\torbN$, finding an equivariant position for the characteristic submanifold $\boldV \subset \torbN$, and then projecting back down. See Proposition~\ref{CharSuborbifold} for details. While we do not concern ourselves with the uniqueness of characteristic suborbifolds,  Corollary~\ref{characteristic corollary} does establish their existence. Extending the manifold case, Proposition~\ref{nothing like a book} shows that $\orbN$ has the structure of a book of $I$--bundles if and only if $\chi(\orbN - \scrC) = 0$.

Our interest in characteristic suborbifolds stems from the fact that they lead to 
either geometric information (volume estimates) or topological information (the structure of a book of $I$--bundles).
Suppose  that $\orbM$ is a closed hyperbolic $3$--orbifold, 
that
$\orbS \subset \orbM$ is a two-dimensional essential suborbifold, 
that $\orbN$ is a component of $ \orbM \split \orbS$, and 
that $\scrC$ is a characteristic suborbifold of  $\orbN$. 
By applying the work of Agol, Storm, and Thurston~\cite{AST} in appropriate covers, we show in Theorem~\ref{orbifold AST} that 
\[
\vol \orbM \geq - \voct \cdot \chi(\orbN - \scrC).
\]
In particular, we either get a non-trivial lower bound on $\vol \orbM$, or learn that every component of $\orbM \split \orbS$ has the structure of a book of $I$--bundles.

\Custom{Definitions}\label{GS review}
We shall
generalize to the context of orbifolds some definitions that
are stated in Guzman and Shalen \cite[Section 3]{GS-homotoping} in the context of
manifolds. 

We shall say that a $3$--orbifold is \emph{simple} if
it is compact, orientable, irreducible
and boundary-irreducible, 
 is not discal, and has a fundamental group containing no rank--$2$ abelian
subgroup. 
Since irreducible $3$--orbifolds are connected by Definition~\ref{Pi1Injective}, every simple $3$--orbifold is also connected.
We observe that if a closed $2$--orbifold $\orbR$   in a
closed hyperbolic $3$--orbifold $\orbM$ is essential, then every component of
$\orbM\split\orbR$ is simple.

 Let $\orbQ$ be a compact, orientable $3$--orbifold.
The notion of a bindinglike (or purely  bindinglike)
suborbifold $\calK$ of $\calQ$, and that of a pagelike (or purely  pagelike)
suborbifold of $\calQ$, were defined in \S\ref{PagelikeBindinglike}.
(The definition of ``pagelike'' in \S\ref{PagelikeBindinglike}, when specialized to the manifold case, differs slightly from
the definition of a ``pagelike submanifold'' $P$ in \cite{GS-homotoping}, in that it requires non-positive Euler characteristic. However, this does not affect our quotation of any result from \cite{GS-homotoping}.)
According to \S\ref{PagelikeBindinglike},  if $\calK$ is a pagelike or bindinglike suborbifold  of $\calQ$ then
the components of $\Fr_\calQ\calK$ are \neatly embedded 
annular $2$--orbifolds 
in $\calQ$. We
shall say that $\calK$ has \emph{essential frontier} if all the
components of $\Fr_\calQ\calK$ are essential 
annular orbifolds 
in $\calQ$. 

A \emph{thickened annular orbifold} in $\calQ$ is a 
\clean\ 
suborbifold with corners
$\calY$ of $\calQ$ which is 
diffeomorphic to $\orbA \times[0,1]$, for some annular orbifold $\orbA$,
under a diffeomorphism that maps $\Fr_\orbQ\orbY$ onto
$\orbA \times\{0,1\}$.

By specializing the notions of simple orbifold, 
(purely) \pagelike\ or \bindinglike\  suborbifold (with essential frontier),
essential annular orbifold,   and thickened annular orbifold
to the manifold case, one obtains the
notions of simple manifold, 
 (purely) pagelike or bindinglike submanifold (with essential frontier),
essential annulus, and thickened annulus.
After translating between the PL and smooth categories, these
specialized notions coincide with the notions for manifolds that are used in
\cite{GS-homotoping}.
\EndCustom

\Custom{Acylindrical pairs}\label{only manifolds}
Next,
we review some notions involving manifolds which
are defined in \cite{GS-homotoping}. Unlike the notions discussed in
\S\ref{GS review}, these notions will not need to be generalized to
orbifolds.

An \emph{acylindrical pair} is a 
\clean\ manifold pair $(\boldL ,\boldE)$ 
satisfying the following conditions: 
\begin{enumerate}
\item each component of
the manifold obtained by \sta\ of
 $\boldL $ is irreducible, no  component of $\boldL$ is a 
topological ball, 
and no component has a fundamental group with a rank--$2$ free abelian subgroup;
\item each component of $\boldE$ is an annulus and is $\pi_1$--injective
  in $\boldL $;
\item no component (see \S\ref{Clean pairs}) of $(\boldL ,\boldE)$
  is  diffeomorphic
to a pair of the form $(L,E)$, where $L$ is an $I$--bundle
over a (possibly non-orientable) surface and $E$ is the vertical
boundary of the $I$--bundle $L$;
\item no component of $\boldL $ is a
topological solid torus; and
\item for every \neatly embedded, $\pi_1$--injective annulus
  $\Delta$ in $ \boldL -\boldE$, either
(a) $\Delta$ is boundary-parallel in $\boldL-\boldE$, or
(b) there is a component $E$ of $\boldE$ such that  $\Delta$ is
parallel  to $E$ in $(\boldL -\boldE)\cup E$. See \S\ref{Pi1Injective} for the notion of \emph{parallel}.
\end{enumerate}

We remark that the notion of an acylindrical pair is extremely similar to the notion of an \emph{acylindrical pared manifold} $(M, P)$ that is commonly used in the literature on Kleinian groups. See e.g.\ Namazi and Souto~\cite[Sections 3.1--3.2]{NamaziSouto:Density}.
\EndCustom

\Custom{Characteristic submanifolds}\label{char-submflds}
Our treatment of characteristic submanifolds follows the point of view of 
Guzman and Shalen
 \cite[Section 3]{GS-homotoping}, which is based on
Jaco--Shalen \cite{JS}. In \cite{GS-homotoping}, 
the authors
consider a class of manifolds which includes all manifolds that are simple in the sense of this paper;
and for each manifold $Q$ in this
class the characteristic submanifold is described as a certain \clean\
submanifold with corners $\boldV \subset Q$, 
 which is well-defined up to smooth isotopy. 

When
$Q$ is simple, $\boldV$ may be thought of as
a \clean\ submanifold with corners, each of whose components is either bindinglike or pagelike
and satisfies certain nondegeneracy conditions; and
among all such submanifolds, $\boldV$ is maximal in a suitable sense.

While the formal definition of the characteristic submanifold will not
be quoted directly in this paper, we will make strong use of the
following characterization of the characteristic submanifold of a
simple manifold, which is proved in
\cite{GS-homotoping}. 
\EndCustom

\Proposition\label{char-submflds, GS edition}
Let $Q$ be a simple $3$--manifold with non-empty boundary, 
and let $\boldV$ be a
\clean\  submanifold
  with corners
of $Q$. Then $\boldV$ is
isotopic to the characteristic submanifold of
$Q$ if and only if it has the following properties.
\begin{enumerate}
\item Each component of $\boldV$ 
is either \bindinglike\ or \pagelike, and has essential frontier
relative to $Q$.
\item
For each component $L$ of $\overline{Q-\boldV}$, 
either the pair $(L,\Fr_Q L)$  is acylindrical, or $L$ is a \ta;
and in the latter case,  the components of $\Fr_Q L$ are contained in
distinct components of $\boldV$, of which one is purely
\bindinglike\ and the other is purely \pagelike.
\end{enumerate}
\EndProposition

\Proof
This is a direct translation of \cite[Proposition 3.12]{GS-homotoping} to the smooth category. 
\EndProof

For most of this section, the characterization of Proposition~\ref{char-submflds, GS edition} can be
  treated as a definition. However, 
we note that Proposition~\ref{CharSuborbifold} and Remark~\ref{requark} 
rely on a theorem of Meeks and Scott about the characteristic submanifold \cite[Theorem 8.6]{MeeksScott}, which uses the original definition.

We will need the following consequence of Proposition~\ref{char-submflds, GS edition}:

\Lemma\label{what about manifolds}
Let $Q$ be a simple $3$--manifold with non-empty boundary, and let $\boldV$ denote its
characteristic submanifold. Let $L$ be a component of $\overline{Q -  \boldV}$. Then:
\begin{enumerate}[\:\:$(1)$]
\item Every component of $\bdy L$ (where we think of $L$ as a topological manifold) has non-positive Euler characteristic.
\item $\chi(L) \leq 0$.
\item If $\chi(L)=0$, then $L$ is a \ta; and  
the components of $\Fr_Q L$ are contained in
distinct components of $\boldV$, of which one is purely
\bindinglike\ and the other  is purely \pagelike.
\end{enumerate}
\EndLemma

\Proof
According to Proposition \ref{char-submflds, GS edition},
each component
  of $\Fr_Q\boldV$ is 
an essential annulus in $Q$.
In particular, the boundary
  components of $\overline{Q -  \boldV}\cap\bdy Q$ are $\pi_1$--injective simple closed
  curves in $\bdy Q$.
Furthermore, the simplicity of $Q$ implies that no component of
$\bdy Q$ is a $2$--sphere. Hence each component of $L\cap\bdy
Q$ has non-positive Euler characteristic. 
Again using that 
the frontier
components of $\boldV$ are annuli, we deduce:
\Claim\label{easy part}
Each boundary component of the topological manifold $\overline{Q -  \boldV}$ has non-positive
Euler characteristic.
\EndClaim

In particular, the claim implies Conclusion (1). Furthermore, 
since  $\chi(L)=\chi(\bdy L)/2$, 
claim \ref{easy part} also establishes Conclusion (2).

To prove the last assertion of the lemma, suppose that $L$ is a
component of $\overline{Q -  \boldV}$ such that
  $\chi(L)=0$. Then $\chi(\bdy L)=0$, and in view of \ref{easy part} it follows that every component of $\bdy L$ is a
  topological torus. 

The components of $\Fr_QL$ are among the components of $\Fr_Q\boldV$; hence they are
 essential annuli in $Q$, and  in particular are
$\pi_1$--injective annuli. The
$\pi_1$--injectivity of $\Fr_QL$, together with the simplicity of $Q$, implies that $L$ is irreducible and
$\pi_1$--injective in $Q$ and that $\bdy L\ne\emptyset$. Since $Q$
is simple, any component of $\bdy L$ fails to be $\pi_1$--injective
in $Q$, and therefore fails to be $\pi_1$--injective in $L$. But it is a standard
consequence of the Loop Theorem that a compact, orientable,
irreducible $3$--manifold, whose boundary contains a torus which is not
$\pi_1$--injective, is a solid torus. 
Hence $L$ is a topological solid torus.

According to Proposition \ref{char-submflds, GS edition}, 
either the pair $(L,\Fr_Q L)$  is acylindrical,
or $L$ is a \ta. Since $L$ is a topological solid torus, the
definition implies that $(L,\Fr_Q L)$  is not acylindrical. Thus
$L$ is a \ta. It also follows from 
Proposition \ref{char-submflds, GS edition} that
the components of $\Fr_Q L$ are contained in
distinct components of $\boldV$, of which one is purely
\bindinglike, and the other is purely \pagelike.
This
completes the proof of Conclusion (3).
\EndProof

We now turn our attention to orbifolds.

\Lemma\label{still simple}
Let $\orbN$ be a simple $3$--orbifold, and suppose that $\torbN$ is a
regular,
finite-sheeted covering of $\orbN$ which is a manifold. Then $\torbN$
is simple.
\EndLemma

\Proof
The compactness, orientability, 
boundary-irreducibility and non-discality of $\torbN$, 
as well as the non-existence of  rank--$2$ abelian
subgroups in $\pi_1(\torbN)$, follow immediately from the
corresponding properties of $\orbN$.
It remains to show that  $\torbN$ is irreducible. Since
$\orbN$ is irreducible, it follows from the equivariant sphere theorem \cite[Theorem 7]{MY-sphere}
that $\pi_2(\torbN)=0$. 
Hence 
every $2$--sphere in $\torbN$ bounds a contractible, compact
$3$--manifold $H\subset\torbN$
(see \cite[proof of Theorem 2]{Milnor}).
By the solution to the Poincar\'e Conjecture \cite{Morgan-Tian},
$H$ is a $3$--ball, and irreducibility is established.
\EndProof

\Proposition\label{CharSuborbifold}
Let $\orbN$ be a simple $3$--orbifold with non-empty boundary, suppose that $\torbN$ is a regular, finite-sheeted manifold cover of $\orbN$
(so that $\torbN$ is simple by Lemma \ref{still simple}, and therefore
has a well-defined characteristic submanifold, up to isotopy,
according to \S\ref{char-submflds}).
Then there is a  
clean suborbifold 
with corners
$\bigorbC \subset \orbN$ 
whose preimage in $\torbN$ is isotopic to the characteristic submanifold of $\torbN$.
\EndProposition

\Proof
 Let $G$ 
denote
the deck group of the 
covering space $\torbN$ of $\orbN$,
and let $p\from\torbN\to\orbN$ denote the covering projection.
Let $\boldV$ denote the characteristic submanifold of
$\torbN$. By a theorem of Meeks
and Scott~\cite[Theorem 8.6]{MeeksScott}, 
we may choose $\boldV$ within its isotopy class in such a way that it is
$G$--invariant. Then 
by an observation made in \S\ref{Suborbifolds}, $\bigorbC\doteq\orbi{\boldV / G}$ is defined (as a weak suborbifold with corners of $\orbN$). Since $\boldV$ is a clean submanifold of $\torbN$, 
the  frontier of  $\boldV=p^{-1}(\bigorbC)$ 
    is a \neatly
    embedded submanifold of $\torbN$, and
$(\boldV,\Fr_{\torbN}\boldV)=(p^{-1}(\bigorbC),p^{-1}(\Fr_{\orbN}\bigorbC))$ is a clean pair. It then  follows formally from the definitions that 
$\Fr_{\orbN}\bigorbC$ is defined and is a neatly embedded suborbifold of $\orbN$, and that
$(\bigorbC,\Fr_{\orbN}\bigorbC)$ is a clean pair. In view of the discussion in \S\ref{clean},
it follows
that $\bigorbC$ is a clean suborbifold with corners of $\orbN$, and is in particular a suborbifold with corners in the strong sense. 
\EndProof

\Custom{Characteristic suborbifolds}
\label{called characteristic part one}
Let $\orbN$ be a simple $3$--orbifold with non-empty boundary.
We define \emph{a
  characteristic suborbifold} of $\orbN$ to be a clean 
suborbifold with corners
$\bigorbC \subset \orbN$ with the property that
there exists a finite-sheeted regular manifold cover $\torbN$ of $\orbN$
(which is simple by Lemma \ref{still simple}), such that
the preimage of $\bigorbC$ in  $\torbN$  is isotopic to the characteristic submanifold of
$\torbN$.
\EndCustom

The following result follows immediately by combining  Lemma \ref{still simple}
and Proposition \ref{CharSuborbifold}:

\Corollary\label{characteristic corollary}
Every very good, simple $3$--orbifold $\orbN$ with non-empty boundary has a characteristic suborbifold.
\NoProof
\EndCorollary 

\Remarks\label{called characteristic part two}
It is a consequence of the Orbifold Theorem, included in \cite[Corollary
3.28]{BoileauMaillotPorti}, that every irreducible
 $3$--orbifold is very good.  Given that every simple $3$--orbifold is irreducible by \S\ref{GS review}, it follows that every simple  $3$--orbifold is very good. Using this fact,
the hypothesis ``very good'' could be removed from Corollary
\ref{characteristic corollary}. We will not need to do this, because
in all applications of  Corollary \ref{characteristic corollary},
$\orbN$ will be a suborbifold of a hyperbolic $3$--orbifold; hence the fact that $\orbN$ is very good will follow directly from observations that were
made in 
\S\ref{Orbifold covers}.

Similarly, in
this paper, we will not address the issue of uniqueness of a characteristic suborbifold
up to orbifold isotopy, because we will not need such a uniqueness
result. An argument for the uniqueness of $\bigorbC$ is sketched 
in Bonahon and Siebenmann's paper \cite{bon-sie} (see the paragraph
on page 445 beginning ``More generally \ldots''). 
\EndRemarks

\Proposition\label{Prop:CharSuborbifoldPieces}
Suppose that $\bigorbC$ is a characteristic suborbifold of a simple
$3$--orbifold $\orbN$ with non-empty boundary. Then
\begin{enumerate}[\:\:$(1)$]
\item \label{EveryFrontAnnular} every component of $\Fr_\orbN\bigorbC$
is a $\pi_1$--injective annular $2$--orbifold in
$\orbN$;
\item \label{EveryPiecePageOrBinding}
every component of $\bigorbC$ is
either pagelike or bindinglike as a suborbifold of $\orbN$;
\item \label{EveryCompNonPos}
every component of 
$\overline{\orbN -  \bigorbC}$ 
has non-positive Euler characteristic; and
\item \label{ZeroEulerProduct}
if $\orbY $ is a component of 
$\overline{\orbN -  \bigorbC}$ 
with $\chi(\orbY )=0$, then 
$\orbY $ is a thickened annular orbifold;
and  the components of $\Fr_\orbN \orbY $ are contained in
distinct components of $\bigorbC$, of which one is purely
\bindinglike\ and the other is purely \pagelike.
\end{enumerate}
\EndProposition

\Proof
Since $\bigorbC$ is a characteristic
suborbifold of $\orbN$, we may fix
a finite-sheeted regular covering space $p \from \torbN\to\orbN$
such that 
$\torbN$ is a manifold 
(and is therefore simple by 
Lemma \ref{still simple}), and such that  $\boldV\doteq p^{-1}(\bigorbC)$ is isotopic to the characteristic
submanifold of $\torbN$.  
Let $G$ be the deck group of the cover $p$.

If
$\orbF$ is any component of $\Fr_\orbN \bigorbC$, we may choose a component
$\torbF$ of $p^{-1}(\orbF)$; 
it follows from Proposition \ref{char-submflds, GS edition}
that the component $\torbF$ of $\Fr_{\torbN}\boldV$ is
an essential annulus in $\torbN$, and is in particular a $\pi_1$--injective annulus. 
Hence $\orbF$ is annular; and the inclusion homomorphism
$\iota\from \pi_1(\orbF)\to\pi_1(\orbN)$, whose domain group is infinite cyclic
or infinite dihedral, restricts to an injection on some finite-index
subgroup. It follows that $\iota$ is itself injective, and 
Assertion \eqref{EveryFrontAnnular} is proved.

 Proposition \ref{char-submflds, GS edition} also gives:
\Claim\label{components above bigorbC} Every component of 
$p^{-1}(\bigorbC)$
is a bindinglike or pagelike submanifold of $\torbN$.
\EndClaim

We claim:
\Claim\label{down to business}
Let $\orbC$ be any component of $\bigorbC$, and let $\torbC$ be any component
of $p^{-1}(\orbC)$. Then $\orbC$ is bindinglike if and only if
 $\torbC$ is bindinglike, and $\orbC$ is pagelike if and only if
 $\torbC$ is pagelike.
\EndClaim

In the following proof of \ref{down to business}, we shall implicitly use the facts that $\orbC$ is a clean suborbifold of $\orbN$, and that $\boldV$ is a clean submanifold of $\torbN$.
First 
suppose that $\orbC$ is
pagelike. Then $\orbC$ has an $I$--bundle structure in which
$\orbC\cap\bdy\orbN=\horizbdy\orbC$. 
The $I$--bundle
structure of $\orbC$ induces an $I$--bundle structure on $\torbC$ whose
horizontal boundary is $\torbC\cap\bdy\torbN$, and hence $\torbC$
is pagelike.

Conversely, suppose that $\torbC$
is pagelike. Then $\torbC$
may be given the
structure of an $I$--bundle  over a surface $F$
in such a way that 
$\torbC\cap\bdy \torbN= \horizbdy\torbC$. 
Since $\torbC$ is simple, $F$
cannot be a sphere or projective plane.
We may fix 
a
finite-sheeted covering $\torbC'$ of $\torbC$ such that (a) the $I$--bundle
structure which $\torbC'$ inherits from $\torbC$ is trivial, and (b)
the covering projection $p' \from \torbC'\to\orbC$ is regular. Let $G'$ denote the deck group of $p'$.

The horizontal
boundary of the trivial $I$--bundle $\torbC'$ is the preimage of
$\orbC\cap\bdy \orbN$ under the covering projection $p'$.
Hence the deck group $G'$ leaves 
$\horizbdy \torbC'$ invariant. 
Furthermore, since $F$ is not a sphere or projective plane, the  base
of the $I$--bundle $\torbC'$ is
not a sphere or projective plane.
It then follows
from Meeks and Scott's result
\cite[Theorem 8.1]{MeeksScott} that, after possibly modifying
the $I$--fibration of $\torbC'$, without changing its
vertical boundary, we may assume that the action of $G'$ preserves the
$I$--fibration. 
(The assertion about vertical boundaries is not stated in \cite[Theorem 8.1]{MeeksScott}, but it follows from the proof of that theorem, which is derived from \cite[Theorem 2.2]{MeeksScott} by a doubling argument.)
It now follows from Proposition \ref{FiberPreservingAction} that
the $I$--fibration of $\torbC'$ then induces an $I$--fibration of
$\orbC$ in which the horizontal boundary is $\orbC\cap\bdy\orbN$. 
Hence the vertical boundary of $\orbC$ in this fibration is $\Fr_\orbN\orbC$. This
shows that $\orbC$ is pagelike,
and the first assertion of \ref{down to business} is proved.

To prove the second assertion of \ref{down to business}, note that the manifold  $\torbC_0$ obtained by \sta\ of $\torbC$ is a finite-sheeted covering of the manifold $\orbC_0$ obtained by \sta\ of $\orbC$. Hence $\torbC_0$ is a solid torus if and only if $\orbC_0$ is a solid toric orbifold. Since the components of $\Fr_\orbN \orbC$ are annular by \eqref{EveryFrontAnnular}, and the components of $\Fr_{\torbN} \torbC$ are therefore annuli, the second assertion follows, and the proof of \ref{down to business} is complete.

Assertion \eqref{EveryPiecePageOrBinding} of the present proposition follows immediately from \ref{components above bigorbC} and \ref{down to business}.

We record the following direct consequence of \ref{down to business}:

\Claim\label{and purely so}
If $\orbC$ is a component of $\bigorbC$, and some component of
$p^{-1}(\orbC)$ is purely pagelike or purely bindinglike, then $\orbC$
is, respectively, purely pagelike or purely bindinglike.
\EndClaim

To prove Assertion \eqref{EveryCompNonPos}, we need only note that if $\orbY $ is a component of 
$\overline{\orbN -  \bigorbC}$, then every component of
$p^{-1}(\orbY)$ is a component of $\overline{\torbN-\boldV}$, and
therefore has non-positive Euler characteristic by Lemma \ref{what about manifolds}; 
it follows that $\chi(\orbY)\le0$.

To prove Assertion \eqref{ZeroEulerProduct}, suppose that $\orbY $ is a component of 
$\overline{\orbN -  \bigorbC}$ 
with $\chi(\orbY )=0$.
If we choose a component $L$ of
$p^{-1}(\orbY)$, then
$\chi(L)=0$. Hence  
Lemma \ref{what about manifolds} gives that 
$L$ is a \ta, and that the components of $\Fr_{\torbN}L$ are contained
in distinct components $\Vp$ and $\Vb$ of $\boldV$, where $\Vp$ is
purely pagelike and $\Vb$ is purely bindinglike.

Since $L$ is a \ta, it may be
given the structure of a trivial $I$--bundle over an annulus in such a
way that $\Fr_{\torbN} L= \horizbdy L$.
It follows from \cite[Theorem 8.1]{MeeksScott} that, after possibly modifying
the $I$--fibration of $L$, without changing its horizontal or
vertical boundary, 
we may assume that the action on $L$ of
$\Stab_G(L)$
preserves the $I$--fibration. 
(Recall that $G$ denotes the deck group of $p$.)
It now follows from Proposition \ref{FiberPreservingAction} that
the $I$--fibration of $L$ induces an $I$--fibration
of $\orbY$, whose base is an annular orbifold.
But since  $\Vp$ is
purely pagelike and $\Vb$ is purely bindinglike,
no element of $\Stab_G(L)$ can interchange the two 
components of $\horizbdy L$. Hence
the $I$--bundle $\orbY$ is trivial, i.e.\ $\orbY$ is a thickened annular
orbifold. 

It now also follows that the components of $\Fr_{\orbN}\orbY$ are
contained in distinct components of $\bigorbC$, which may be labeled
as $\orbC_{\rm p}$ and $\orbC_{\rm b}$ in such a way that $\Vp$ and
$\Vb$ are components of $p^{-1}(\orbC_{\rm p})$ and $p^{-1}(\orbC_{\rm
  b})$ respectively. Hence  by \ref{and purely so},
$\orbC_{\rm p}$ is purely pagelike and $\orbC_{\rm b}$
is purely bindinglike. This establishes Assertion \eqref{ZeroEulerProduct}.
\EndProof

\Proposition\label{nothing like a book}
Let $\orbN$ be a simple $3$--orbifold with non-empty boundary, and suppose that
$\bigorbC$ is a characteristic suborbifold of $\orbN$. Then we have
$\chi(\overline{\orbN - \bigorbC})\le0$, and if $\chi(\overline{\orbN - \bigorbC})=0$ then $\orbN$ admits the
structure of a book of $I$--bundles.
\EndProposition

\Proof
According to Assertion \eqref{EveryCompNonPos} of
Proposition \ref{Prop:CharSuborbifoldPieces}, each
component of $\overline{\orbN - \bigorbC}$ has non-positive
  Euler characteristic. 
Thus
  $\chi(\overline{\orbN - \bigorbC})\le0$.

Now suppose that $\chi(\overline{\orbN - \bigorbC})=0$. Then
each component of $\overline{\orbN - \bigorbC}$ must have Euler
characteristic $0$. 
According to Assertions \eqref{EveryFrontAnnular}, \eqref{EveryPiecePageOrBinding} and \eqref{ZeroEulerProduct} of
Proposition \ref{Prop:CharSuborbifoldPieces},
 every component of $\Fr_\orbN\bigorbC$
is a $\pi_1$--injective annular $2$--orbifold in
$\orbN$;
every component of $\bigorbC$ is a
 pagelike or bindinglike suborbifold of $\orbN$;
and  every  component of 
$\overline{\orbN -  \bigorbC}$ 
is a 
\ta\ whose frontier components are contained in
distinct components of $\bigorbC$, of which one is purely
\bindinglike\ and the other is purely \pagelike. It now follows from
Definition~\ref{wuzza book} that  $\orbN$ admits the
structure of a book of $I$--bundles. 
\EndProof

\Custom{The Gromov volume}
\label{gromov volume}
If $Q$ is a compact, orientable $3$--manifold, we define the \emph{Gromov
  volume} of $Q$, denoted $\volG(Q)$, to be the number
$\vtet\cdot\|DQ\|/2$, where $DQ$ denotes the double of $Q$, 
and $\|\cdot\|$ denotes the Gromov
norm 
(see \S\ref{Gromov norm}). 
This definition is inspired by \eqref{genuine gromov}, which implies
that if $Q$ is closed and hyperbolic then $\volG(Q)=\vol Q$.

Suppose that $d$ is a positive integer and that $\tQ$ is a $d$--sheeted cover
of a
compact, orientable topological $3$--manifold $Q$. Then the double
$D\tQ$ of $\tQ$ is a $d$--sheeted cover of $DQ$, and according to \cite[Section 0.2]{gromov}
we have $\|D\tQ\|= d\cdot\|DQ\|$. Hence
\Equation\label{gromov multiplicative}
\volG(\tQ)=d\cdot\volG(Q).
\EndEquation

Using \eqref{gromov multiplicative}, we can extend the definition of
$\volG$ to all very good, orientable compact $3$--orbifolds: for any
such orbifold $\orbN$ we set $\volG(\orbN)=\volG(\torbN)/d$, where
$\torbN$ is an arbitrary finite-sheeted manifold cover of $\orbN$ and
$d$ is the degree of the covering. If $\torbN_1$ and $\torbN_2$ are
two finite-sheeted manifold covers of $\orbN$, and $d_i$ denotes the
degree of the covering $\torbN_i \to \orbN$ for $i=1,2$, we may choose a 
finite covering  $\torbN_3$ of $\orbN$ whose covering
projection factors through those of $\torbN_1$ and $\torbN_2$. If
we let $n_i$ denote the degree of $\torbN_3 \to \orbN_i$
then $d_1 n_1= d_2 n_2$. But by \eqref{gromov multiplicative}  we have
$n_1\volG(\torbN_1)=\volG(\torbN_3)=n_2\volG(\torbN_2)$. Hence
$\volG(\torbN_1)/d_1=\volG(\torbN_2)/d_2$, and we have shown that
$\volG(\orbN)$ is well-defined.

The definitions given above do not require the
manifolds involved to be  connected. (A closed, 
orientable manifold with $n$ components
has $2^n$ fundamental classes, but they all have the same Gromov
norm. Furthermore, the existence of a common finite-sheeted covering $\torbN_3$ of
$\torbN_1$ and $\torbN_2$ does not depend on connectedness.) We remark that if $\mathscr{N}$ is a disconnected $3$--orbifold, then $\volG (\mathscr{N}) = \sum_\orbN \volG(\orbN)$, where $\orbN$ ranges over the connected components of $\mathscr{N}$.
\EndCustom

\Proposition\label{orbifold 7.2}
Let $\orbM$ be a closed, orientable hyperbolic $3$--orbifold,
and let $\orbR$ be an essential two-dimensional suborbifold of $\orbM$.
Then
$\vol \orbM \geq  \volG(\orbM \split \orbR)$.
\EndProposition
  
We remind the reader that a hyperbolic $3$--orbifold $\orbM$ is required to be connected, whereas $\orbM \split \orbR$ may have several components.

\Proof
In the special case where $\orbM$ is a manifold (so that
$\Sing_\orbM=\emptyset$ and $\orbR$ is an essential surface in
$\orbM$), this is a result of Agol, Storm, and Thurston~\cite{AST}. See, specifically, the first inequality in the statement of \cite[Theorem 9.1]{AST}. 

To prove the general case of the proposition, fix a 
regular
manifold cover
$p\from \torbM\to\orbM$ of some degree $n<\infty$. 
Then $\torbR\doteq
p^{-1}(\orbR)$ is an essential $2$--manifold in $\torbM$,
and $\torbM\split\torbR$ is an $n$--sheeted manifold cover of $\orbM \split \orbR$.
Thus we have
\[
n\cdot\vol\orbM = \vol\torbM
\geq \volG(\torbM\split\torbR) = n \cdot \volG(\orbM \split \orbR).
\]
Here, the inequality follows from the manifold case of the
proposition, and the last equality follows from \S\ref{gromov volume}. Now, dividing the first and last terms by $n$ gives the conclusion.
\EndProof

\Proposition\label{Prop:GromovVolEuler}
Suppose that $\orbN$ is a simple $3$--orbifold with non-empty boundary, and that 
$\bigorbC$ is a
characteristic suborbifold of $\orbN$. 
Then $\volG(\orbN)\ge\voct\cdot(-\chi(\overline{\orbN-\bigorbC}))$.
\EndProposition

\Proof
According to the definition of a characteristic suborbifold,
we may fix a regular manifold cover $p\from  Q \to\orbN$ such
that    $\boldV\doteq p^{-1}(\bigorbC)$ is (up to isotopy) the
characteristic submanifold of $Q$. 
In the notation of Agol, Storm, and Thurston \cite{AST},
the union of all components of $\overline{Q  -  \boldV}$ that are not thickened annuli is denoted by $\Guts(Q)$. (This rephrasing of \cite[Definition 2.1]{AST} essentially matches our characterization of the characteristic submanifold $\boldV$ in Proposition~\ref{char-submflds, GS edition}.)

Note that
$\chi(\Guts(Q))=\chi(\overline{Q -  \boldV})$.
If $d$ denotes the degree of the covering $Q$ of $\orbN$, then
$\overline{Q -  \boldV}$ is a (possibly disconnected) $d$--sheeted covering of
$\overline{\orbN-\bigorbC }$, and hence $\chi(\Guts(Q))=d\cdot\chi(\overline{\orbN-\bigorbC })$.

Agol, Storm, and Thurston point out 
in the short proof of \cite[Theorem 9.1]{AST} that,
as a consequence of \cite[Theorem 6.5.5]{Thurston:notes} and Miyamoto's inequality \cite{Miyamoto},
every simple $3$--manifold $Q$ satisfies
$\volG(Q)\ge\voct(-\chi(\Guts(Q)))$. 
Thus in the context of the present
argument, using the definition of the Gromov volume of the orbifold
$\orbN$ (see \S\ref{gromov volume}), we have
\[
\volG(\orbN) = \tfrac{1}{d}
\volG(Q)
\geq \tfrac{1}{d}
\voct \cdot (-\chi(\Guts(Q))) = \voct\cdot(-\chi(\overline{\orbN-\bigorbC })),
\]
as required.
\EndProof

\Theorem\label{orbifold AST}
Let $\orbM$ be a closed, orientable hyperbolic $3$--orbifold. Let
$\orbR$ be a non-empty, 
essential 
$2$--suborbifold of $\orbM$.
Then each component $\orbN$  of $\orbM \split
\orbR $ 
is simple and very good, and hence
has a characteristic suborbifold $\bigorbC_\orbN$ by Corollary \ref{characteristic corollary}.
Furthermore, 
we have
$$\vol\orbM\ge\voct\cdot\sum_\orbN(-\chi(\overline{\orbN-\bigorbC_\orbN})),$$
where $\orbN$ ranges over the components of $\orbM\split\orbR$.
\EndTheorem

\Proof
Since $\orbM$ is hyperbolic and $\orbR$ is essential, it follows from
an observation made in
\S\ref{GS review} that each component $\orbN$ of $\orbM \split \orbR$ is simple. Likewise, since $\orbM$ is hyperbolic, it follows from the discussion in
\S\ref{Orbifold covers} 
that $\orbM$ is  very good, hence every $\orbN$ is very good. The displayed
inequality now follows 
immediately from Propositions \ref{orbifold 7.2} and
\ref{Prop:GromovVolEuler},
together with the observation that orbifold Euler characteristic is
additive over components.
\EndProof

\Remark\label{requark}
Note that Proposition \ref{Prop:GromovVolEuler} and Theorem \ref{orbifold AST}
involve quantities of the form $\chi(\overline{\orbN  - 
  \bigorbC})$, where $\orbN$ is a simple $3$--orbifold with non-empty boundary and $\bigorbC$ is a characteristic suborbifold of $\orbN$. For the arguments in
this paper, we do not need to know that the quantity $\chi(\overline{\orbN  - 
  \bigorbC})$ is an invariant of the simple $3$--orbifold $\orbN$,
i.e.\ that it is independent of  $\bigorbC$. The assertion of invariance
of  $\chi(\overline{\orbN  - 
  \bigorbC})$
is weaker than the
assertion that $\bigorbC$ is unique up to orbifold isotopy; as we 
mentioned in \S\ref{called characteristic part two}, a proof of the stronger
assertion is sketched in \cite{bon-sie}. Here, for the benefit of the
interested reader, we sketch a proof of the weaker assertion
that $\chi(\overline{\orbN  - 
  \bigorbC})$ is an invariant of  $\orbN$. Using arguments similar to
those invoked in \S\ref{gromov volume} to show
the Gromov volume of an orbifold is well-defined, one can reduce the
proof of invariance of $\chi(\overline{\orbN  - 
  \bigorbC})$ to showing that if $p\from \tQ\to Q$ is an $n$--fold covering of
simple $3$--manifolds, and if $\boldV$ and $\boldtV$ denote the
respective characteristic submanifolds of $Q$ and $\tQ$, then $\chi(\overline{\tQ  - 
  \boldtV})=n\cdot \chi(\overline{Q  - 
  \boldV})$. For this purpose it is enough to know that $p^{-1}(\boldV)$ is
isotopic in $\tQ$ to $\boldtV$; and the latter fact is not hard to deduce
from \cite[Theorems 8.1 and 8.6]{MeeksScott} and the characteristic submanifold theory.
\EndRemark

\section{Proofs of the main results}\label{Sec:MainResults}

This section proves the main theorems of the paper, which were stated in the Introduction. We will first prove 
Theorem~\ref{Thm:IntroNoTurnoverMain}, which assumes the closed hyperbolic $3$--orbifold $\orbM$ has no essential turnovers, and then derive Theorems~\ref{Thm:FirstMain} and \ref{Thm:LinkMain} from it. 

If $\orbM$ has no essential turnovers, the plan is as follows.  Suppose that $\vol \orbM < \voct/12$, or else that $\orbM$ is a link orbifold and $\vol \orbM < \voct/6$. Then, by Corollary~\ref{Cor:SimpleVolBound}, the underlying space $M = \mani \orbM$ cannot have any hyperbolic pieces. If $M$ is reducible or toroidal, we find an essential suborbifold $\orbR \subset M$ by following the minimization procedures of either Section~\ref{sphere section} or Section~\ref{torus section}. The volume bound on $\orbM$, combined with the no-turnover hypothesis and Theorem~\ref{orbifold AST}, will imply that every component of $\orbM \split \orbR$ admits the structure of an orbifold book of $I$--bundles (Proposition \ref{why a book}). Then, the main results of Sections~\ref{sphere section} and~\ref{torus section} imply that $M$ has the claimed topological structure. This proves Theorem~\ref{Thm:NoTurnoverMain}, which is a restatement of Theorem~\ref{Thm:IntroNoTurnoverMain}.

If $\orbM$ has an essential turnover $\orbS$, then a theorem of Adams and Schoenfeld~\cite{AdamsSchoenfeld} shows that, after an isotopy, we may take $\orbS$ to be 
totally geodesic. The work of Miyamoto~\cite{Miyamoto} gives lower bounds for the volume of hyperbolic $3$--manifolds with totally geodesic boundary, and extends via covers to $3$--orbifolds. Miyamoto's work, combined with other results from the literature, implies that every hyperbolic $3$--orbifold containing an essential turnover has $\vol \orbM \geq 0.1491$; see Lemma~\ref{Turnover volume}. Combined with Theorem~\ref{Thm:NoTurnoverMain}, this yields Theorem~\ref{Thm:FirstMain}.

Finally, Theorem~\ref{Thm:LinkMain} about link orbifolds has both milder hypotheses on volume and a stronger topological conclusion. The milder volume hypotheses suffice by Lemma~\ref{Turnover volume}. The stronger topological conclusion comes from combining Theorem~\ref{Thm:NoTurnoverMain}, a covering argument, the work of Agol \cite{Agol:Pants}, and a theorem of Atkinson and Futer~\cite{af}.

\Proposition\label{what if no book}
Let $\orbM$ be a closed, orientable hyperbolic $3$--orbifold containing
no essential turnovers. Let
$ R$ be a closed orientable surface in $M\doteq\mani
\orbM$  which is transverse to $\Sing_\orbM$. 
Suppose that the two-dimensional suborbifold $\orbR\doteq\orbi R $ is essential in $\orbM$. Let $\orbN$ be a component of $\orbM \split \orbR $, and suppose that 
$\orbN$
does not admit the structure of 
an orbifold book of $I$--bundles.
Let $\bigorbC$ be a characteristic suborbifold of $\orbN$.  
Then $\chi(\overline{\orbN - \bigorbC})\le-1/12$. 
Furthermore, if $\orbM$ is a link orbifold then
 $\chi(\overline{\orbN - \bigorbC})\le-1/6$.
\EndProposition

\Proof
Under the hypotheses of the proposition, 
let $\orbZ$ denote the $3$--orbifold obtained by \sta\ 
of
$\overline{\orbN - \bigorbC}$. Set $Z=\mani\orbZ$.
Observe that  $\bdy Z$ can be identified with the smooth surface obtained by smoothing the boundary of the topological manifold $\mani {\overline{\orbN - \bigorbC}}$.

We first claim:

\Claim\label{big enough}
Every component of $\bdy \orbZ$ has non-positive Euler characteristic.
\EndClaim

To prove \ref{big enough}, recall from Proposition~\ref{CharSuborbifold} and Definition~\ref{called characteristic part one} that $\orbN$ has a regular, finite-sheeted manifold cover $\torbN$, such that $\scrC$ is covered by the characteristic submanifold $\boldV$ of $\torbN$. Thus every component of $\overline{\orbN - \bigorbC}$ is covered by a component of $\overline{\torbN - \boldV}$.
Now, by Lemma~\ref{what about manifolds}, every boundary component of the manifold obtained by \sta\ of  $\overline{\torbN - \boldV}$ has non-positive Euler characteristic, and \ref{big enough} follows.

Next, we claim:

\Claim\label{four or more}
Every 
component of $\bdy Z $ which is a topological $2$--sphere
has size\ at least $4$.
\EndClaim

To prove \ref{four or more}, suppose that $S$ is a $2$--sphere component of $\bdy Z$, and that 
$\size(S)\le 3$. 
By Proposition \ref{if three},  $\orbi S $ is $\pi_1$--injective in $\orbM$, and by \ref{big enough} we have $\chi(\orbi S )\le0$. If $\chi(\orbi S )=0$ then $\orbi S $ is a $\pi_1$--injective toric suborbifold of $\orbM$, a contradiction to hyperbolicity. If $\chi(\orbi S )<0$ then $\orbi S $ is an essential turnover
in $\orbM$, a contradiction to the hypotheses. This completes the proof of \ref{four or more}.

Next, we recall that
\Equation\label{chi oh my}
\chi(\bdy \orbZ)=2 \chi(\orbZ).
\EndEquation
On the other hand, the orbifold Euler characteristic formula 
\eqref{hurwitz} implies that every component $\orbW$ of $\bdy \orbZ$ satisfies
\Equation\label{snub ya}
\chi(\orbW)=\chi(\mani \orbW )-
\sum_{x\in\Sing_\orbW}\bigg(1-\frac1{\ord(x)}\bigg).
\EndEquation
Observe that each term of the sum 
in 
\eqref{snub ya} is at least $1/2$.

Now if $\bdy \orbZ$ has a component $\orbW$ such that 
\begin{enumerate}[\:\:(a)]
\item $W\doteq\mani \orbW $ has genus $g > 1$, or 
\item $W $ has genus $1$ and $\size(W )>0$, or 
\item $W $ has genus $0$ and $\size(W )\ge5$, 
\end{enumerate}
then 
it follows from \eqref{snub ya} that $\chi(\orbW)\le-1/2$; hence by \ref{big enough} we have $\chi(\bdy \orbZ)\le -1/2$, and by \eqref{chi oh my} we have $\chi(\orbZ)\le-1/4<-1/6$. Thus both conclusions of the proposition are true if (a), (b) or (c) holds. For the rest of the proof we assume that none of the alternatives (a), (b) or (c) holds. In view of \ref{four or more}, 
we may now assume that:

\Claim\label{four, basically}
Every component of $\bdy Z $ is either a size--$4$ topological $2$--sphere or a size--$0$  topological  torus.
\EndClaim

Now by hypothesis $\orbN$ 
does not admit the structure of a
book of $I$--bundles, which by
Proposition \ref{nothing like a book} implies
that $\chi(\orbZ)<0$. Hence 
$\bdy \orbZ$ has a component $\orbW_0$ such that $\chi(\orbW_0)<0$. By
 \ref{four, basically}, $\mani {\orbW_0}$ is a sphere of size\ $4$. If
 each of the singular points of $\orbW_0$ had order $2$, it would follow
 from \eqref{snub ya} that $\chi(\orbW_0)=0$, a contradiction. Hence
 some singular point 
$x_0$ of $\orbW_0$ has order at least $3$. Since each of the other three singular points is of order at least $2$, \eqref{snub ya} gives $\chi(\orbW_0)\le-1/6$; \eqref{big enough} then gives $\chi(\bdy \orbZ)\le-1/6$, and \eqref{chi oh my} gives $\chi(\orbZ)\le-1/12$. This proves the first assertion of the proposition.

To prove the second assertion, suppose that $\orbM$ is a link
orbifold. Then the component of $Z \cap\Sing_\orbM$ containing
$x_0$ is an arc, having $x_0$ as one endpoint. Let $x_1$ denote the
other endpoint. Then $x_1\in\Sing_{\bdy \orbZ}$ and $\ord x_1=\ord
x_0\ge 3$. If $x_1$ lies in the component $\mani{ \orbW_0}$ of $\bdy
Z $, then $\orbW_0$ has two singular points of order at least $3$,
while each of the other two has order at least $2$. In this case
\eqref{snub ya} gives $\chi(\orbW_0)\le-1/3$; 
hence  $\chi(\bdy \orbZ)\le-1/3$ and
$\chi(\orbZ)\le-1/6$, as required.   If $x_1\in\mani{ \orbW_1}$ for some component $\orbW_1\ne \orbW_0$ of $\bdy \orbZ$, then $\orbW_1$ has a singular point of order at least $3$, while each of the other three has order at least $2$. In this case \eqref{snub ya} gives $\chi(\orbW_1)\le-1/6$; 
hence  $\chi(\bdy \orbZ)\le-1/3$ and
$\chi(\orbZ)\le-1/6$. Thus the proof of the second assertion is complete.
\EndProof

\Proposition\label{why a book}
Suppose that $\orbM$ is a closed, orientable hyperbolic $3$--orbifold containing
no essential turnovers, and that
either 
\begin{enumerate}[\:\:(a)]
\item
$\vol \orbM < \voct / 12$, or 
\item $\orbM$ is a link orbifold and $\vol
\orbM < \voct / 6$. 
\end{enumerate}
Let $R $ be a closed, orientable surface in
$M\doteq\mani \orbM$ which is transverse to $\Sing_\orbM$. Suppose
that the two-dimensional suborbifold 
$\orbR\doteq\orbi R $ 
is essential in $\orbM$. Then every component of  $\orbM \split \orbR $ 
admits  the structure of
an orbifold 
book of $I$--bundles.
\EndProposition

\Proof
Let us define a number $\alpha$ by setting $\alpha=1/6$ if $\orbM$ is
a link orbifold, and $\alpha=1/12$ otherwise. Then the volume hypothesis
implies that $\vol \orbM <\alpha\,\voct$.

For each component $\orbN$ of $\orbM \split \orbR $, 
we use 
Corollary~\ref{characteristic corollary}
to choose 
a characteristic suborbifold $\bigorbC_\orbN \subset \orbN$.
(The hypothesis that $\orbN$ is very good follows from the hyperbolicity of $\orbM$.)

Assume that the conclusion is false, so that $\orbM \split \orbR $ has
a component $\orbN_0$ which does not
admit  the structure of
an orbifold 
book of $I$--bundles. Since $\orbM$ contains no essential turnovers,
it follows from Proposition \ref{what if no book} that
$\chi(\overline{\orbN_0 - \bigorbC_{\orbN_0}})\le-\alpha$. For any component
$\orbN\ne\orbN_0$ of $\orbM \split \orbR $, we have
$\chi(\overline{\orbN - \bigorbC_\orbN})\le0$. Hence
$\sum_\orbN(-\chi(\overline{\orbN - \bigorbC_\orbN}))\ge\alpha$,
where $\orbN$ ranges over the components of $\orbM\split\orbR$.
It therefore follows from 
Theorem \ref{orbifold AST} that 
$\vol \orbM\ge\alpha\voct$, a contradiction.
\EndProof

The following result was stated in the introduction as Theorem~\ref{Thm:IntroNoTurnoverMain}.

\Theorem\label{Thm:NoTurnoverMain}
Suppose that $\orbM$ is a closed, orientable hyperbolic $3$--orbifold containing
no essential turnovers, and that
either $\vol \orbM < \voct / 12$, or  $\orbM$ is a link orbifold and $\vol
\orbM < \voct / 6$. Set $M = \mani \orbM $. Then   one of the following alternatives holds.

\Alternatives
\item $M$ admits a Seifert fibration 
over $\SS^2$
and having at most
three singular fibers. 

\item $M$ is homeomorphic to  a connected sum of two non-trivial lens spaces, or the connected sum of $\SS^2\times \SS^1$ with a non-trivial lens space.

\item
$M$ contains 
a separating essential torus $T$ such that each component of $M \split T$ 
admits a Seifert fibration over $\BB^2$
having exactly two singular fibers.

\item
$M$ contains 
a non-separating essential torus $T$  such that  $M \split T $ 
admits a Seifert fibration 
over an annulus
having at most one singular fiber.
\EndAlternatives

\noindent Furthermore, if Alternative (i) does not hold, then every node of
$\Sing_\orbM$ is dihedral.
\EndTheorem

\Proof
We begin by observing that the volume hypotheses of the present theorem are stronger than those of Corollary~\ref{Cor:SimpleVolBound}, because $\vol \orbM < \voct/6 =0.61\ldots < \Vweeks$. 
Thus, by Corollary~\ref{Cor:SimpleVolBound}, $M$ cannot be a
hyperbolic manifold. Hence, according to Proposition \ref{because Perelman},
we must be in one of the following cases:

{\bf Case I.} $M$ admits a Seifert fibration over $\SS^2$ with at most three singular fibers.

{\bf Case II.} $M$ is reducible, i.e.\ it contains an essential $2$--sphere. 

{\bf Case III.} $M$ is irreducible and contains an essential torus. 

In Case I, Alternative (i) of the conclusion of the theorem holds.

Next, observe that the volume hypotheses of the present theorem include those of Proposition \ref{why a book}. Once we find an essential two-dimensional suborbifold  $\orbR \subset \orbM$ (as we will in cases II and III), we can apply that proposition.

In Case II, it follows from an observation made in \S\ref{minimal} that
$M$ contains a strongly $\orbM$--minimal $2$--sphere $S$. According
to Lemma \ref{it's incompressible}, 
$\orbS\doteq\orbi S $ 
is essential in $\orbM$. Now
we apply Proposition \ref{why a book},
taking $R=S$, to deduce that
each
component of $\orbM \split \orbS$ admits the structure of an orbifold
book of $I$--bundles. It then follows from Corollary
\ref{Cor:HomeoTypeReducible}
that $M$ is homeomorphic to either (a) a connected sum of two non-trivial lens spaces, or (b) the connected sum of $\SS^2\times \SS^1$ with a possibly trivial lens space. In Subcase (a), Alternative (ii) of the present theorem holds. In Subcase (b), 
either Alternative (i) or Alternative (ii) holds.

In Case III, it follows from 
an observation made in \S\ref{TorusComplexity}
that  $M $ contains an $\orbM$--minimal torus $T$ (which is in particular an essential torus).
 According to Lemma \ref{it's still incompressible}, 
$\orbT = \orbi T $ 
is essential in $\orbM$.
Now, we apply   Proposition \ref{why a book},
  this time taking  
$R=T$, to deduce that
each component of $\orbM \split \orbT$ admits the structure of an orbifold
book of $I$--bundles. If  $T$ 
separates $M$, then Proposition~\ref{Prop:TorusBooksTotalSpace} implies that
each component of $M \split T$ 
admits a Seifert fibration 
over $\BB^2$
having exactly two singular fibers.
If  $T$ does not
separate $M$, then Proposition~\ref{Prop:TorusBooksTotalSpace} implies that
$M \split
T $ 
admits a Seifert fibration 
over an annulus
having at most one singular fiber.
Thus one of the alternatives (iii) and (iv) of the present theorem holds in this case.

This completes the proof that, under the hypotheses of the theorem,
one of the alternatives (i)--(iv) must hold. 

To complete the proof of
the theorem, it suffices to show that in Cases II and III, every node of
$\orbM$ is dihedral. To this end, we observe that, in each of
these cases, the arguments given above provide a closed $2$--manifold
$R \subset M$, transverse to $\Sing_\orbM$, such that $\orbR \doteq \orbi R$ is
essential in $\orbM$, and each component of $\orbM\split\orbR$ admits
the structure of a book of $I$--bundles. By transversality, every node of $\orbM$ is  a
node of $\orbM\split\orbR$. But by Proposition
\ref{nodes of a book}, every node of $\orbM\split\orbR$ is
dihedral.
\EndProof

We will handle orbifolds containing 
essential turnovers using 
a theorem of Miyamoto \cite{Miyamoto}, which requires the following definition.

\Custom{Orbifolds with geodesic boundary}\label{Def:TotallyGeodesic}
Fix a dimension $n \geq 2$. A \emph{truncated hyperbolic space} is a closed, connected submanifold of $\HH^n$ whose boundary is a (non-empty) disjoint union of totally geodesic copies of $\HH^{n-1}$. A \emph{hyperbolic $n$--orbifold with totally geodesic boundary} is a very good $n$--orbifold $\orbM$ such that each component of $\orbM$ is the quotient of a truncated hyperbolic space by a 
discrete
group of isometries.
We remark that a hyperbolic orbifold with totally geodesic boundary does not satisfy the definition of ``hyperbolic orbifold'' given in \S\ref{Orbifold covers}, because the universal cover is not all of $\HH^n$. This departure from our usual convention (whereby boundaries are implicitly permitted) will only be needed in Theorem~\ref{MiyamotoOrbifold} and Lemma~\ref{Turnover volume}.
\EndCustom

\Theorem[Miyamoto \cite{Miyamoto}]\label{MiyamotoOrbifold}
Let $\orbN$ be a hyperbolic $3$--orbifold with totally geodesic boundary. Then $\vol \orbN \geq (\voct / 2) \cdot | \chi(\bdy \orbN) |$.
\EndTheorem

\begin{proof}
As volume and Euler characteristic are additive over components, we
may assume that $\orbN$ is connected. Let $\torbN$ be a
$d$--sheeted manifold cover of $\orbN$. In \cite[Theorem
4.2]{Miyamoto}, Miyamoto proved that
$\vol \torbN \geq (\voct/4\pi) \area(\bdy
\torbN)$.
Since volume and Euler characteristic are both multiplicative in the cover $\torbN \to \orbN$, we obtain
\[
d \cdot \vol \orbN = \vol \torbN \geq \frac{\voct}{4\pi} \area(\bdy \torbN) = \frac{\voct}{2}  | \chi(\bdy \torbN) | = d \cdot \frac{\voct}{2}  | \chi(\bdy \orbN) |,
\]
as desired.
\end{proof}

\Lemma\label{Turnover volume}
Let $\orbM$ be a closed hyperbolic $3$--orbifold that contains an essential turnover $\orbS$. Then we have the following volume estimates:
\begin{enumerate}[\:\:$(1)$]
\item\label{TurnoverVolSmallest} In all cases, $\vol \orbM \geq 0.1491$.
\item\label{TurnoverVolNot7} If $\orbS$ is not a $(2,3,7)$ turnover, then $\vol \orbM \geq  \voct / 20 \approx 0.1831$.
\item\label{TurnoverVolLink} If $\orbM$ is a link orbifold, then $\vol \orbM \geq 0.1658$.
\end{enumerate}
\EndLemma

\Proof
By a result of Adams and Schoenfeld \cite[Theorem
2.1]{AdamsSchoenfeld}, 
we may assume after an isotopy that
$\orbS$ is totally geodesic in $\orbM$. Let $\orbN =  \orbM \split \orbS$, 
and observe that the hyperbolic metric on $\orbM$ endows $\orbN$ with the structure of a (possibly disconnected) hyperbolic $3$--orbifold with totally geodesic boundary. Furthermore, $\bdy \orbN$ consists of two copies of $\orbS$, and  $\vol \orbM = \vol \orbN$.

Let $p,q,r$ be the orders of the three points of $\Sing_\orbS$, and arrange the labels so that $p \leq q \leq r$. We consider two cases.

First, suppose that $(p,q,r) \in \{(2,3,7), \, (2,3,8)\}$. In this case, we apply a volume estimate of Atkinson and Rafalski \cite[Theorem 3.4]{AtkinsonRafalski}, which builds on the work of Miyamoto \cite[Section 4]{Miyamoto}  by incorporating not only $\chi(\bdy \orbN)$ but also the largest order of torsion in $\bdy \orbN$.
Their theorem says that $\vol \orbN \geq 0.1491$ if $(p,q,r) = (2,3,7)$ and $\vol \orbN \geq 0.2204$ if $(p,q,r) = (2,3,8)$.

Next, suppose that $(p,q,r) \notin \{(2,3,7), \, (2,3,8)\}$. Since $\chi(\orbS) < 0$, we have $\frac{1}{p} + \frac{1}{q} + \frac{1}{r} < 1$. There are three subcases:
\[
p = 2, \: q = 3, \: r \geq 9;
\qquad
p = 2, \: q \geq 4, \: r \geq 5;
\qquad
p \geq 3, \:  q \geq 3, \: r \geq 4.
\]
In all three of these subcases, we have $\chi(\orbS) \leq -1/20$, hence $\chi(\bdy \orbN) \leq -1/10$. Now, Theorem \ref{MiyamotoOrbifold} implies $\vol \orbN \geq \voct/20$.

We conclude that $\vol \orbM \geq 0.1491$ if $(p,q,r) = (2,3,7)$, and $\vol \orbM \geq \voct/20 \approx 0.1831$ otherwise, establishing conclusions \eqref{TurnoverVolSmallest} and \eqref{TurnoverVolNot7}.

It remains to establish conclusion \eqref{TurnoverVolLink}. By
\eqref{TurnoverVolNot7}, 
a stronger estimate holds if $\orbS$ is not a $(2,3,7)$ turnover. If
$\orbS$ is a $(2,3,7)$ turnover, then in particular $\orbM$ is a
hyperbolic link orbifold with $7$--torsion. By a lemma of
Atkinson and Futer~\cite[Lemma 2.3]{af}, any such orbifold has volume
at least $0.1658$.
\EndProof

Observe that $\voct/12 = 0.305 \ldots > 0.1491$. Thus combining Lemma \ref{Turnover volume} with Theorem \ref{Thm:NoTurnoverMain} gives the following result, which was stated as Theorem~\ref{Thm:FirstMain} in the introduction.

\Theorem\label{Thm:VolWithTurnover}
Suppose that $\orbM$ is a closed, orientable hyperbolic $3$--orbifold with $\vol \orbM < 0.1491$. Then 
$M
\doteq\mani \orbM$ 
satisfies one of the alternatives (i)--(iv) of Theorem \ref{Thm:NoTurnoverMain}.
Furthermore, if Alternative (i) does not hold, then every node of
$\Sing_\orbM$ is dihedral.
\NoProof
\EndTheorem

We now turn our attention to link orbifolds.

\Lemma\label{essential or not}
If   $\orbM$ is a (closed, orientable) link orbifold, then every
turnover in  $\orbM$ is  essential.
\EndLemma

\Proof
  If $\orbS\subset\orbM$ is a turnover, then
  by Proposition \ref{if three}, $\orbS$ is
$\pi_1$--injective in $\orbM$. It remains to show that $\orbS$ does not
bound a discal suborbifold of $\orbM$.
Since $\orbS$ is an orientable $2$--suborbifold of $\orbM$, it follows
from an observation
made in \S\ref{Suborbifolds-of-3orbifolds} that the $2$--sphere
$S\doteq\mani\orbS$ is transverse to $\Sing_\orbM$.
If
$\orbH$ is a discal suborbifold of $\orbM$ with $\bdy\orbH=\orbS$,
then since $\orbM$ is a link orbifold, each component of the intersection
of $\Sing_\orbM$ with $H\doteq \mani \orbH$ is an arc or a simple closed
curve. This implies that $S\cap\Sing_\orbM=\Sing_\orbS$ consists of an
even number of points, which contradicts the fact that $\orbS$ is a turnover.
\EndProof

  \Custom{Constructing orbifold coverings from covering spaces}
  \label{orbi cover from mani cover}
Suppose that $\orbM$ is an orbifold and that $Q\from \tM\to M$ is a covering space of
$M\doteq\mani\orbM$. If $( U , W, \Phi, G)$ is a chart for $\orbM$ such that
$U\subset M$ is connected and evenly covered by $Q$, if $V$ is a
component of $Q^{-1}(U)$, and if $R$ denotes the inverse of the
homeomorphism $Q|V\from V\to U$, then $(V, W, R\circ\Phi, G)$ is an
orbifold chart in $\tM$ (in the abstract sense defined in
\S\ref{Def:OrbifoldChart}). The orbifold charts of this form, as
$( U , W, \Phi, G)$ ranges over charts for $\orbM$ such that
$U\subset M$ is connected and evenly covered by $Q$, and  $V$ ranges over
components of $Q^{-1}(U)$, form an orbifold atlas for $\tM$, which
defines an orbifold $\torbM$ with $\mani\torbM=\tM$. We then have $\Sing_\torbM=Q^{-1}(\Sing_\orbM)$, and $Q$ 
defines an orbifold covering $q\from \torbM\to\orbM$ having the same
degree as $Q$.
\EndCustom

\Proposition\label{uses Agol}
Suppose that   $\orbM$ is a 
hyperbolic link
orbifold containing no turnovers, and that  $Q\from \tM\to M$ is a
finite-sheeted covering space of
$M\doteq\mani\orbM$. Let  $q\from \torbM\to\orbM$ denote the orbifold
covering given by the construction of \S\ref{orbi cover from mani cover}.
Then $\torbM$ is a link
orbifold and contains no turnovers.
\EndProposition

\Proof
 Since by \S\ref{orbi cover from mani cover} we have
$\Sing_\torbM=Q^{-1}(\Sing_\orbM)$, and since $\orbM$ is
a link orbifold, $\torbM$ is also
a link orbifold. It remains to show that 
$\torbM$ contains no turnovers.

Let us fix a tubular neighborhood $J$ of the link $\Sing_\orbM$ in
$M$. Since
$\Sing_\torbM=Q^{-1}(\Sing_\orbM)$, the set $\tJ\doteq Q^{-1}(J)$ is a
tubular neighborhood of  the link $\Sing_\torbM$ in
$\tM$. We set $N=\overline{M-J}$ and $\tN=\overline{\tM-\tJ}$. The restriction
$Q|\tN \from \tN\to N$ is a covering map.  If
$\mu\subset\bdy\tN$ is a meridian of $\Sing_\torbM$ (i.e.\ a simple closed curve
bounding a \neatly embedded disk in $\tJ$ which meets $\Sing_\torbM$ transversely in exactly one point),
then $Q$ maps $\mu$ to a closed curve which, although not necessarily
simple, is homotopic in $\bdy N$ to a meridian of
$\Sing_\orbM$.

By the work of Kojima~\cite{Kojima:Deformations}, the $3$--manifold $M  -  \Sing_\orbM$
admits a finite-volume complete hyperbolic metric.
The manifold
$N$ is a compact core of $M  -  \Sing_\orbM$.  

If $\torbM$ contains a turnover, then $\tN$ contains a \neatly
embedded pair of pants $P$ whose boundary curves are meridians of 
$\Sing_\torbM$. If $P$ is not $\pi_1$--injective in $\tN$, then the
loop theorem provides a compressing disk $D$ for $P$. The boundary of
$D$ is a homotopically non-trivial simple closed curve on $P$, and is
therefore parallel to a boundary curve of $P$. This implies that some
meridian of $\Sing_\torbM$ is homotopically trivial in $\tN$, and
therefore that some
meridian of $\Sing_\orbM$ is homotopically trivial in $N$; but this
contradicts
the hyperbolicity of $M  - \Sing_\orbM$.
Hence $P$ is  $\pi_1$--injective in $\tN$.

Since $Q|\tN$ is a covering map, it now follows that  $f\doteq Q|P$ is a
$\pi_1$--injective map from the pair of pants $P$ to $N$. We have
$f(\bdy P)\subset\bdy N$; and $f$ maps each boundary curve of
$\bdy P$ to a closed curve which is homotopic in $\bdy N$ to a meridian of
$\Sing_\orbM$.  

We now apply a theorem due to Agol~\cite[Theorem 4.1]{Agol:Pants},
which asserts that if $N$ is a compact core of a finite-volume orientable
hyperbolic $3$--manifold, and $f$ is a $\pi_1$--injective map of a pair
of pants $P$ into $N$ such that $f(\bdy P)\subset\bdy N$, then
one of two possibilities must hold:

\begin{enumerate}
\item The map of pairs $f\from (P,\bdy P) \to (N,\bdy N)$ is
  homotopic to an embedding of pairs.
\item 
The interior of
$N$ is obtained by (possibly empty) Dehn filling on the exterior
$W$ of the Whitehead
link. That is, there is a $3$--manifold $K \subset\inter N$, which is
either empty or a solid torus, such that $\overline{N- K}$ is
homeomorphic to $W$. Furthermore,  $f(P)$ is disjoint from $K$, and
$f$ is homotopic in $N-K$ to a clasp of $W$, as in \cite[Figure
2]{Agol:Pants}. In particular, 
after modifying the map of pairs $f\from (P,\bdy P) \to (N,\bdy N)$
within its homotopy class, we may assume that $f$ maps the 
boundary components of $P$ homeomorphically to three simple closed curves in
the same 
component $T$ of $\bdy N$, and that these curves represent exactly
two distinct homotopy classes on $T$.
\end{enumerate}

In our setting,
  Alternative
(2) leads to a contradiction, because the simple
closed curves which are the images of the components of $\bdy P$
under $f$
must all be meridians of $\Sing_\orbM$, and hence
no two of them can represent distinct homotopy classes on the same boundary torus of
$N$.
  Now suppose that Alternative (1) holds, so that
the map of pairs $f\from (P,\bdy P) \to (N,\bdy N)$ is
  homotopic to an embedding $g$. In particular, $g|\bdy
  P\from \bdy P\to \bdy N$ is homotopic to $f|\bdy
  P\from \bdy P\to \bdy N$, so that each boundary component of the pair of
  pants $E\doteq g(P)$ is a meridian curve of $\Sing_\orbM$.   Hence
  the components of $\bdy E$ bound disjoint \neatly embedded disks in $J$,
  each of which meets $\Sing_\orbM$ transversely in one point. The
  union of $E$ with these disks is the underlying set of a turnover
  contained in $\orbM$.
  Since by hypothesis $\orbM$ contains no
  turnovers, we have a contradiction in this case as well.
\EndProof

\Theorem
\label{MainConnectSum}
Suppose that $\orbM$ is a closed,
orientable hyperbolic link orbifold, and
that $\vol \orbM \le \Vweeks / 6$. 
If  $M \doteq \mani \orbM$ is a non-trivial connected sum, then $M$ is homeomorphic  to 
either $\RR \PP^3 \# \RR \PP^3$, or to $\RR \PP^3 \# L(p,q)$, where $p > 1$ is relatively prime to $6$. 
\EndTheorem

\Proof
We first note that
since $\vol \orbM \leq \Vweeks / 6 < 0.1658$, 
it follows from Lemma~\ref{Turnover volume}
that $\orbM$ contains no essential turnovers, 
and hence by Lemma \ref{essential or not} it contains no turnovers at
all.

Since $M$ is a non-trivial connected sum, it must satisfy alternative (ii) of 
Theorem \ref{Thm:NoTurnoverMain}, i.e.\ $M$ is
 homeomorphic to a connected sum
$Y_1 \# Y_2$, where each $Y_i$ is a non-trivial lens space or $\SS^1
\times  \SS^2$.
Hence for $i = 1,2$, we have $\pi_1(Y_i) \cong \ZZ/p_i\ZZ$ for some integer $p_i\ge0$,
so $\pi_1(M) \cong \ZZ/p_1 \ZZ * \ZZ/p_2 \ZZ$. If $p_i=0$ then $Y_i$ is
homeomorphic to $\SS^2\times \SS^1$, and if $p_i>1$ then $Y_i$ is homeomorphic to $L(p_i,q_i)$ for
some positive integer $q_i$ which is relatively prime to $p_i$. 
 Since neither $Y_1$ nor $Y_2$ is the $3$--sphere, we have
$p_i \neq 1$ for $i=1,2$.

Next, we claim:

\Claim\label{however}
For each $i\in\{1,2\}$, either $p_i$ is relatively prime to $6$, or  $p_i=2$.
\EndClaim

To prove \ref{however}, it suffices by symmetry to prove that either
$p_1$ is relatively prime to $6$, or  $p_1=2$. Assume that  $p_1\ne2$
and that $p_1$ is not relatively prime to $6$.
Then $p_1$ is divisible by some $d\in\{2,3\}$. 
Set $\tp_1=p_1/d$. Then  $Y_1$ has a
$d$--sheeted covering space $\tY_1$, which is a lens space with
fundamental group isomorphic to $\ZZ/\tp_1\ZZ$ if $p_1\ne0$, and is
homeomorphic to $\SS^1\times \SS^2$ if $p_1=0$. It follows that $M$, which is
homeomorphic to $Y_1 \# Y_2$, has a $d$--sheeted covering $\tM$ which is
homeomorphic to the connected sum of $\tY_1$ with $d$ homeomorphic
copies of $Y_2$. 
Since $p_1\ne2$, we have either $d=3$ or $\tp_1 \neq 1$. Thus 
$\tM$ is homeomorphic to a connected sum of three or more
prime, non-trivial $3$--manifolds.

Since $\orbM$ is a closed orientable link
$3$--orbifold containing no turnovers, and $\tM$ is a
$d$--sheeted covering of $M$, \S\ref{orbi cover from mani cover} provides a
$d$--sheeted orbifold covering 
$\torbM$
of $\orbM$ such that $\mani\torbM=\tM$. According to Proposition \ref{uses Agol}, $\torbM$ is a
link orbifold containing no  turnovers. 
We have $\vol \torbM =d\cdot\vol \orbM\le 3
\vol \orbM < \Vweeks / 2 < \voct / 6$.
Thus the hypotheses of Theorem \ref{Thm:NoTurnoverMain}
hold with $\torbM$ in place of $\orbM$. Hence one of the alternatives (i),
(ii), (iii), (iv) of the conclusion of Theorem \ref{Thm:NoTurnoverMain} holds with $\tM$ in place of $M$. But each of these
conditions implies that  $\tM$ 
has a prime decomposition with at most two summands.
As we have already
shown that $\tM$ is 
homeomorphic 
to the
connected sum of  three or more prime, non-trivial $3$--manifolds, this
contradicts the uniqueness of the prime decomposition for a
closed oriented $3$--manifold and proves \ref{however}.

To complete the proof, we claim:

\Claim\label{one's divisible}
At least one of the integers $p_1,p_2$ is equal to $2$. 
\EndClaim

To prove \ref{one's divisible}, let us assume for a contradiction that neither $p_1$ nor
$p_2$ is equal to $2$. Then by \ref{however}, $p_1$ and $p_2$ are both
relatively prime to $6$. Since $H_1(M;\ZZ)$ is
isomorphic to $(\ZZ/p_1\ZZ)\oplus (\ZZ/p_2\ZZ)$, it follows that $M$
is a $\ZZ/6\ZZ$--homology $3$--sphere. But by hypothesis we have
$\vol\orbM\le \Vweeks / 6$; and
according to a theorem of Atkinson and Futer \cite[Theorem 1.5]{af}, a link orbifold which has volume at most $\Vweeks /
6$,  and whose underlying manifold is a
$\ZZ/6\ZZ$--homology $3$--sphere, 
must be the orbifold $\calO_L$ depicted in Figure~\ref{Fig:OL}. The underlying space of $\calO_L$ is $\SS^3$. This contradicts the hypothesis that $M$ is a non-trivial connected sum, proving \ref{one's divisible} and the theorem.
\EndProof

Combining Theorems \ref{Thm:NoTurnoverMain} and \ref{MainConnectSum}, we obtain the following corollary, which was stated in the Introduction as Theorem~\ref{Thm:LinkMain}.

\Corollary
\label{second main} Suppose that $\orbM$ is a closed,
orientable hyperbolic link orbifold, and
that $\vol \orbM \le \Vweeks / 6$. 
Set $M=\mani \orbM $. Then either  one of the alternatives (i),  (iii), or (iv)
of  Theorem \ref{Thm:NoTurnoverMain}
holds, 
or we have the following specialization of (ii):
\begin{enumerate}
\item[(ii$'$)]
$M$ is homeomorphic either  to $\RR \PP^3 \# \RR \PP^3$, or to a connected
sum $\RR \PP^3 \# L(p,q)$, where $p > 1$ is relatively prime to $6$.  \NoProof
\end{enumerate}
\EndCorollary

\bibliographystyle{plain}
\bibliography{orbifolds}

\end{document}

%% file: macros.tex
\newcommand\delete[1]{{\color{green}#1}}

\swapnumbers
\theoremstyle{definition}
\newtheorem{para}{}[section]

\newtheorem{remark}[para]{Remark}

\newtheorem{remarks}[para]{Remarks}

\newtheorem{convention}[para]{Convention}
\newtheorem{definition}[para]{Definition}
\newtheorem{example}[para]{Example}

\newcommand\Alternatives{\begin{enumerate}[(i)]}
\newcommand\EndAlternatives{\end{enumerate}}

\newtheorem{namedpara}[para]{\theoremname}
\newcommand{\theoremname}{testing}
\newenvironment{custom}[1]{\renewcommand{\theoremname}{#1}\begin{namedpara}}{\end{namedpara}}

\theoremstyle{plain}

\newtheorem{theorem}[para]{Theorem}

\newtheorem{lemma}[para]{Lemma}

\newtheorem{proposition}[para]{Proposition}

\newtheorem{corollary}[para]{Corollary}

\newtheorem{claim}[equation]{}
\numberwithin{equation}{para}
\numberwithin{figure}{section}

\newcommand\Definition{\begin{definition}}
\newcommand\EndDefinition{\end{definition}}

\newcommand\Theorem{\begin{theorem}}
\newcommand\EndTheorem{\end{theorem}}

\newcommand\Remark{\begin{remark}}
\newcommand\EndRemark{\end{remark}}
\newcommand\Remarks{\begin{remarks}}
\newcommand\EndRemarks{\end{remarks}}
\newcommand\Convention{\begin{convention}}
\newcommand\EndConvention{\end{convention}}

\newcommand\Lemma{\begin{lemma}}
\newcommand\EndLemma{\end{lemma}}
\newcommand\Proposition{\begin{proposition}}
\newcommand\EndProposition{\end{proposition}}
\newcommand\Corollary{\begin{corollary}}
\newcommand\EndCorollary{\end{corollary}}
\newcommand\Claim{\begin{claim}}
\newcommand\EndClaim{\end{claim}}
\newcommand\Proof{\begin{proof}}
\newcommand\EndProof{\end{proof}}
\newcommand\Equation{\begin{equation}}
\newcommand\EndEquation{\end{equation}}
\newcommand\NoProof{{\hfill$\square$}}

\newcommand\Example{\begin{example}}
\newcommand\EndExample{\end{example}}
\newcommand\Custom[1]{\begin{custom}{#1}}
\newcommand\EndCustom{\end{custom}}

\newcommand{\from}{\colon\thinspace}
\newcommand{\bdy}{\partial}
\newcommand{\vertbdy}{\partial_\mathrm{v}}
\newcommand{\vertbdyclas}{\partial_\mathrm{v}^{\rm cl}}
\newcommand{\horizbdyclas}{\partial_\mathrm{h}^{\rm cl}}
\newcommand{\horizbdy}{\partial_\mathrm{h}}
\newcommand\clasibun{classical $I$--bundle}
\newcommand\classical{classical}
\newcommand\Clasibun{Classical $I$--bundle}
\newcommand\sta{straightening the angles}
\newcommand\mani[1]{{{\rm X}_{#1}}}

\newcommand\Nnought{{N_0}}
\newcommand\orbNnought{{\orbN_0}}

\newcommand\orbi[1]{{\Omega_{#1}}}
\newcommand\mirror[1]{{\Sing_{#1}^{\rm mir}}}

\newcommand\neatly{neatly\ }
\newcommand\harmless{harmless}
\newcommand\spanningsize{{\rm spanning size}}
\newcommand\textspanningsize{spanning size}
\newcommand{\comp}{{\rm comp}} 

\newcommand\reflective{reflective}
\newcommand\Reflective{Reflective}
\newcommand\clean{clean}

\newcommand\locallift{local lift}

\newcommand\Clean{Clean}
\newcommand\cleansub{a clean suborbifold with corners of\ }

\newcommand\bb{\mathbb}

\newcommand\calT{{\mathcal T}}
\newcommand\calS{{\mathcal S}}
\newcommand\calY{{\mathcal Y}}
\newcommand\caln{{\mathcal N}}
\newcommand\orbN{{\caln}}
\newcommand\orbT{{\calT}}
\newcommand\orbE{{\cale}}
\newcommand\orbJ{{\calj}}
\newcommand\orbK{{\calK}}

\newcommand\orbZ{{\calz}}
\newcommand\orbW{{\calw}}
\newcommand\orbQ{{\calq}}
\newcommand\orbM{{\calM}}

\newcommand\orbC{{\calc}}
\newcommand\bigorbC{{\scrC}}
\newcommand\orbA{{\cala}}
\newcommand\orbY{{\calY}}
\newcommand\orbD{{\calD}}
\newcommand\orbH{{\calH}}

\newcommand\orbS{{\calS}}
\newcommand\torbM{{\widetilde{\calM}}}

\newcommand\torbU{{\widetilde{\orbU}}}
\newcommand\torbH{{\widetilde{\orbH}}}

\newcommand\tkappa{\widetilde{\kappa}}
\newcommand\torbS{{\widetilde{\orbS}}}
\newcommand\torbP{{\widetilde{\orbP}}}
\newcommand\torbW{{\widetilde{\orbW}}}
\newcommand\torbA{\widetilde{\cala}}
\newcommand\torbC{\widetilde{\calc}}

\newcommand\orbP{{\calp}}
\newcommand\orbF{{\calf}}

\newcommand\orbB{{\calb}}
\newcommand\orbR{{\calR}}
\newcommand\orbU{{\calU}}
\newcommand\orbV{{\calV}}
\newcommand\hatorbP{{\widehat{\calp}}}

\renewcommand\epsilon{\varepsilon}

\newcommand\Sbig{big}

\newcommand\Stab{{\rm Stab}}
\newcommand\Guts{{\rm Guts}}

\newcommand\Vp{V_{\rm p}}
\newcommand\Vb{V_{\rm b}}

\newcommand\RR{{\bb R}}
\newcommand\DD{{\bb D}}
\newcommand\PP{{\bb P}}

\newcommand{\boldC}{{\bf C}}

\newcommand{\boldD}{{\bf D}}
\newcommand{\boldB}{{\bf B}}
\newcommand{\boldA}{{\bf A}}
\newcommand{\boldE}{{\bf E}}
\newcommand{\boldF}{{\bf F}}

\newcommand{\boldH}{{\bf H}}
\newcommand{\boldK}{{\bf K}}
\newcommand{\boldL}{{\bf L}}

\newcommand\hatboldP{\widehat{\bf P}}

\newcommand{\boldT}{{\bf T}}
\newcommand{\boldV}{{\bf V}}

\newcommand{\boldtV}{\widetilde {\bf V}}

\newcommand\Sing{\Sigma}

\newcommand\hatP{{\widehat P}}

\newcommand\maybeD{D}

\newcommand\inter{\mathop{\rm int}}
\newcommand\size{\mathop{\rm size}}

\newcommand\bard{\overline{d}}
\newcommand\barsigma{\overline{\sigma}}

\newcommand\barG{\overline{G}}

\newcommand\ttau{\widetilde\tau}
\newcommand\tH{\widetilde H}
\newcommand\tM{\widetilde M}

\newcommand\tV{\widetilde V}

\newcommand\tN{\widetilde N}

\newcommand\tY{\widetilde Y}

\newcommand\tP{\widetilde P}

\newcommand\tU{{\widetilde U}}

\newcommand\tQ{\widetilde Q}
\newcommand\tq{\widetilde q}

\newcommand\barrho{\overline{\rho}}

\newcommand\tS{{\widetilde S}}
\newcommand\tw{{\widetilde w}}
\newcommand\tcalS{{\widetilde {\calS}}}

\newcommand\ta{thickened annulus}
\newcommand\talpha{{\widetilde\alpha}}

\newcommand\tF{{\widetilde F}}

\newcommand\tJ{{\widetilde J}}

\newcommand\tx{{\widetilde x}}

\newcommand\calc{{\mathcal C}}
\newcommand\calO{{\mathcal O}}

\newcommand\calD{{\mathcal D}}

\newcommand\cala{{\mathcal A}}
\newcommand\volG{\mathop{\rm vol_G}}
\newcommand\vtet{{v_{\rm tet}}}
\newcommand\bigO{{\rm O}}
\newcommand\bigSO{{\rm SO}}

\newcommand\calp{{\mathcal P}}

\newcommand\scrY{{\mathscr Y}}
\newcommand\scrZ{{\mathscr Z}}
\newcommand\scrF{{\mathscr F}}

\newcommand\calf{{\mathcal F}}

\newcommand\calK{{\mathcal K}}
\newcommand\calQ{{\mathcal Q}}

\newcommand\frakG{{\mathfrak G}}

\newcommand\calq{{\mathcal Q}}

\newcommand\calb{{\mathcal B}}

\newcommand\cale{{\mathcal E}}

\newcommand\calj{{\mathcal J}}

\newcommand\calU{{\mathcal U}}

\newcommand\baralpha{\overline{\alpha}}

\newcommand\calw{{\mathcal W}}
\newcommand\calV{{\mathcal V}}

\newcommand\id{{\rm id}}

\newcommand\arc{arc}

\newcommand\ZZ{{\mathbb Z}}

\newcommand\CC{{\mathbb C}}
\newcommand\BB{{\mathbb B}}

\renewcommand\SS{{\mathbb S}}

\newcommand\HH{{\bb H}}

\newcommand\scrA{{\mathscr A}}
\newcommand\scrB{{\mathscr B}}

\newcommand\scrbstar{{\scrB^\ast}}
\newcommand\boldbstar{\boldB^\ast}

\newcommand\scrC{{\mathscr C}}

\newcommand\scrH{{\mathscr H}}

\newcommand\scrM{{\mathscr M}}
\newcommand\scrP{{\mathscr P}}

\newcommand\scrS{{\mathscr S}}

\newcommand\calz{{\mathcal Z}}

\newcommand\weighttwo{\mathop{{\rm weight}}}
\newcommand\textweighttwo{weight}
\newcommand\ord{\mathop{\rm ord}}

\newcommand\vol{\mathop{\rm vol}}
\newcommand\area{\mathop{\rm area}}

\newcommand\voct{{v_{\rm oct}}}

\newcommand\torbN{\widetilde\orbN}

\newcommand\torbT{\widetilde\orbT}

\newcommand\torbF{\widetilde\orbF}

\newcommand\torbR{\widetilde\orbR}

\newcommand\card{\mathop{{\rm card}}}

\newcommand\isomplus{\mathop{{\rm Isom}^+}}

\newcommand\Fr{\mathop{\rm Fr}}

\newcommand\pagelike{pagelike}
\newcommand\bindinglike{bindinglike}

\newcommand\calR{{\mathcal R}}
\newcommand\calM{{\mathcal M}}
\newcommand\calH{{\mathcal H}}

\newcommand\tp{{\widetilde{p}}} 

\renewcommand{\split}{{\backslash \backslash}}

\newcommand{\Vweeks} {{V_{\rm Weeks}}}